\documentclass[reqno]{amsart}

\usepackage{graphicx} 
\usepackage{hyperref}
\usepackage[mathscr]{euscript}
\usepackage{amsmath}
\usepackage{amssymb}
\usepackage{enumerate}
\usepackage{graphicx}
\usepackage{slashed}
\usepackage{verbatim}

\newtheorem{MainThm}{Theorem}

\theoremstyle{plain}
\newtheorem{thm}[equation]{Theorem}
\newtheorem{hypthm}[equation]{Hypothetical Theorem}
\newtheorem{cor}[equation]{Corollary}
\newtheorem{lem}[equation]{Lemma}
\newtheorem{prop}[equation]{Proposition}
\newtheorem{hypothesis}[equation]{Hypothesis}

\theoremstyle{definition}
\newtheorem{definition}[equation]{Definition}
\newtheorem{example}[equation]{Example}
\newtheorem{defn}[equation]{Definition}

\theoremstyle{remark}
\newtheorem{remark}[equation]{Remark}
\newtheorem{rem}[equation]{Remark}

\numberwithin{equation}{subsection}

\newcommand{\bD}{\mathrm{D}}
\newcommand{\bK}{\mathrm{KO}}

\newcommand{\bN}{\mathbb{N}}
\newcommand{\bQ}{\mathbb{Q}}
\newcommand{\bR}{\mathbb{R}}
\newcommand{\bS}{\mathbb{S}}
\newcommand{\bZ}{\mathbb{Z}}

\newcommand{\gA}{\bold{A}}

\newcommand{\gx}{\mathbf{x}}
\newcommand{\gy}{\mathbf{y}}
\newcommand{\gz}{\mathbf{z}}

\newcommand{\cA}{\mathcal{A}}
\newcommand{\cB}{\mathcal{B}}
\newcommand{\cC}{\mathcal{C}}
\newcommand{\cD}{\mathcal{D}}
\newcommand{\cE}{\mathcal{E}}
\newcommand{\cF}{\mathcal{F}}
\newcommand{\cG}{\mathcal{G}}
\newcommand{\cK}{\mathcal{K}}
\newcommand{\cL}{\mathcal{L}}
\newcommand{\cM}{\mathcal{M}}
\newcommand{\cO}{\mathcal{O}}
\newcommand{\cP}{\mathcal{P}}
\newcommand{\cQ}{\mathcal{Q}}
\newcommand{\cR}{\mathcal{R}}
\newcommand{\cU}{\mathcal{U}}
\newcommand{\cX}{\mathcal{X}}
\newcommand{\cY}{\mathcal{Y}}
\newcommand{\cZ}{\mathcal{Z}}

\newcommand{\Conf}{\mathrm{Conf}}

\newcommand{\const}{\mathrm{const}}
\newcommand{\ad}{\mathrm{ad}}

\newcommand{\bott}{\operatorname{bott}}
\newcommand{\Bun}{\mathrm{Bun}}

\newcommand{\Cl}{\mathbf{Cl}}

\newcommand{\cstar}{\mathbf{C}^{\ast}}

\newcommand{\dom}{\mathrm{dom}}
\newcommand{\dirac}{\slashed{D}}
\newcommand{\Dir}{\slashed{\mathfrak{D}}}

\newcommand{\eps}{\epsilon}
\newcommand{\Eig}{\operatorname{Eig}}

\newcommand{\grotimes}{\hat{\otimes}}
\newcommand{\hofib}{\operatorname{hofib}}

\newcommand{\ind}{\operatorname{ind}}
\newcommand{\id}{\operatorname{id}}
\newcommand{\inc}{\operatorname{inc}}
\newcommand{\inddiff}{\operatorname{inddiff}} 

\newcommand{\Set}{\mathsf{Set}}
\newcommand{\ssSet}{\mathsf{ssSet}}
\newcommand{\CatNU}{\mathsf{Cat}^{\mathrm{nu}}}
\newcommand{\Sheaves}{\mathsf{Sheaves}}
\newcommand{\Top}{\mathsf{Top}}
\newcommand{\hoTop}{\mathsf{Ho(Top)}}
\newcommand{\Mfds}{\mathsf{Mfds}}

\newcommand{\KO}{\mathrm{K}\mathrm{O}}

\newcommand{\MT}{\mathrm{MT}}

\newcommand{\MTSO}{\mathrm{MTSO}}
\newcommand{\Psc}{\mathrm{P}}

\newcommand{\Mor}{\mathrm{Mor}}
\newcommand{\morita}{\operatorname{morita}}
\newcommand{\norm}[1]{\| #1 \|}

\newcommand{\Ob}{\mathrm{Ob}}

\newcommand{\Path}{P}
\newcommand{\pr}{\operatorname{pr}}

\newcommand{\Riem}{\cR}
\newcommand{\res}{\mathrm{res}}

\newcommand{\lst}{\mathrm{lst}}
\newcommand{\rst}{\mathrm{rst}}

\newcommand{\psc}{\mathrm{psc}}
\newcommand{\spann}{\mathrm{span}}
\newcommand{\spec}{\operatorname{spec}}
\newcommand{\Spin}{\mathrm{Spin}}
\newcommand{\supp}{\operatorname{supp}}
\newcommand{\spinor}{\slashed{\mathfrak{S}}}
\newcommand{\scal}{\mathrm{scal}}

\newcommand{\twomatrix}[4]{\begin{pmatrix} #1 & #2 \\ #3 & #4  \end{pmatrix}}
\newcommand{\tor}{\mathrm{tor}}

\newcommand\lra{\longrightarrow}

\newcommand\Diff{\mathrm{Diff}}
\newcommand\Emb{\mathrm{Emb}}

\newcommand\colim{\operatorname*{colim}}

\newcommand{\round}{\circ}

\newcommand{\inter}[1]{\mathrm{int}{(#1)}}

\newcommand{\hAut}{\mathrm{hAut}}

\newcommand{\CircNum}[1]{\ooalign{\hfil\raise .00ex\hbox{\scriptsize #1}\hfil\crcr\mathhexbox20D}}
\newcommand{\ev}{\mathrm{ev}}

\newcommand{\SE}{\mathrm{SE}}

\newcommand{\Cob}{\mathcal{C}}
\newcommand{\PCob}{\mathcal{P}}

\newcommand{\scpr}[1]{\langle #1 \rangle}

\newcommand{\cstarred}{\mathbf{C}^*_{\mathrm{r}}}
\newcommand{\op}{\mathrm{op}}

\newcommand{\Simp}{\mathrm{Simp}}
\newcommand{\inj}{\mathrm{inj}}

\newcommand{\Conc}{\mathcal{Q}} 
\newcommand{\rep}[1]{| #1 |}
\newcommand{\cDs}{\hat{\mathcal{D}}} %

\newcommand{\riem}{\mathrm{riem}}
\newcommand{\sh}{\mathrm{sh}}
\newcommand{\trg}{\mathrm{trg}}

\newcommand{\Inddiff}{\mathrm{Inddiff}} 

\usepackage[all]{xy}
\SelectTips{cm}{10}
\CompileMatrices

\title[The psc cobordism category]{The positive scalar curvature\\cobordism category}

\author{Johannes Ebert}
\thanks{J. Ebert was partially supported by the SFB 878, by EPSRC grant EP/M027783/1, and by the Deutsche Forschungsgemeinschaft under Germany's Excellence Strategy EXC 2044 390685587, Mathematics M\"unster: Dynamics--Geometry--Structure.}
\email{johannes.ebert@uni-muenster.de}
\address{
Mathematisches Institut der Westf{\"a}lischen Wilhelms-Universit{\"a}t M{\"u}nster\\
Einsteinstr. 62\\
DE-48149 M{\"u}nster\\
Germany
}

\author{Oscar Randal-Williams}
\thanks{O. Randal-Williams was supported by EPSRC grant EP/M027783/1, the ERC under the European Union's Horizon 2020 research and innovation programme (grant agreement No. 756444), and a Philip Leverhulme Prize from the Leverhulme Trust.}
\email{o.randal-williams@dpmms.cam.ac.uk}
\address{
Centre for Mathematical Sciences\\
Wilberforce Road\\
Cambridge CB3 0WB\\
UK}

\date{\today}

\begin{document}

\begin{abstract}
We prove that many spaces of positive scalar curvature metrics have the homotopy type of infinite loop spaces. Our result in particular applies to the path component of the round metric inside $\Riem^+ (S^d)$ if $d \geq 6$. 

To achieve that goal, we study the cobordism category of manifolds with positive scalar curvature. Under suitable connectivity conditions, we can identify the homotopy fibre of the forgetful map from the psc cobordism category to the ordinary cobordism category with a delooping of spaces of psc metrics. This uses a version of Quillen's Theorem B and instances of the Gromov--Lawson surgery theorem. 

We extend some of the surgery arguments by Galatius and the second named author to the psc setting to pass between different connectivity conditions. Segal's theory of $\Gamma$-spaces is then used to construct the claimed infinite loop space structures. 

The cobordism category viewpoint also illuminates the action of diffeomorphism groups on spaces of psc metrics. We show that under mild hypotheses on the manifold, the action map from the diffeomorphism group to the homotopy automorphisms of the spaces of psc metrics factors through the Madsen--Tillmann spectrum. This implies a strong rigidity theorem for the action map when the manifold has trivial rational Pontrjagin classes. 

A delooped version of the Atiyah--Singer index theorem proved by the first named author is used to moreover show that the secondary index invariant to real $K$-theory is an infinite loop map. These ideas also give a new proof of the main result of our previous work with Botvinnik. 
\end{abstract}

\maketitle

\newpage

\tableofcontents

\newpage

\section{Introduction}

For a closed manifold $M$, let $\Riem^+ (M)$ denote the space of Riemannian metrics of positive scalar curvature (psc) on $M$. Using Chernysh's improvement \cite{Chernysh} of the Gromov--Lawson surgery method \cite{GL}, one can show that as long as $\dim (M)=\dim (N)\geq 3$ there is a connected sum map
\[
 \Riem^+ (M) \times \Riem^+ (N) \lra \Riem^+ (M \sharp N),
\]
well-defined up to homotopy. In particular, if $M=N=S^d$, one obtains a multiplication 
\[
\mu:\Riem^+ (S^d) \times \Riem^+ (S^d) \lra \Riem^+ (S^d), 
\]
which can be shown to be homotopy unital (the round metric is a homotopy unit), homotopy associative, and homotopy commutative. 

Much more precisely, Walsh \cite{WalshH} has shown that up to homotopy $\Riem^+ (S^d)$ admits an action of the little $d$-disc operad and hence is an $E_d$-space. The underlying $H$-space structure of Walsh's $E_d$-structure is given by the map $\mu$. The existence of such an $E_d$-structure implies that the path component of $\Riem^+ (S^d)$ which contains the unit is a $d$-fold loop space. Slightly better, if $\Riem^+ (S^d)^{\mathrm{st}} \subset \Riem^+ (S^d)$ denotes the subspace of those elements which are invertible up to homotopy (which is a union of path components), then Walsh's results imply that $\Riem^+ (S^d)^{\mathrm{st}}$ has the homotopy type of a $d$-fold loop space. In this paper, we go beyond Walsh's work and prove:
\begin{MainThm}\label{Main:infiniteloopspacetheoremsphere}
As long as $d \geq 6$ the space $\Riem^+ (S^d )^{\mathrm{st}}$ has the homotopy type of an infinite loop space.
\end{MainThm}
We do \emph{not} claim that the entire space $\Riem^+ (S^d)$ is an $E_\infty$-space: our techniques will apply directly to the subspace $\Riem^+ (S^d )^{\mathrm{st}}$. Walsh's theorem is geometrically quite plausible, especially after using Chernysh's theorem to replace $\Riem^+ (S^d)$ with the homotopy equivalent space $\Riem^+ (D^d)_{g_\round^{d-1}}$ of metrics on the $d$-disc which are collared and agree with the round metric on the boundary. Our infinite loop space structure is less geometrically clear, and is akin to Tillmann's theorem \cite{Tillmann} that the plus-constructed stable mapping class group, which has a geometrically evident double loop space structure, in fact has an infinite loop space structure. As in the case of Tillmann's theorem, it is difficult to compare our infinite loop space structure with Walsh's $d$-fold loop space structure: we do \emph{not} claim that ours extends his.

\subsection{Stable metrics and a generalisation}
Theorem \ref{Main:infiniteloopspacetheoremsphere} is a special case of a more general result, which needs some further preliminaries to state. Firstly, let us recall the stability conditions for psc metrics on cobordisms which we introduced in \cite {ERWpsc2}. For a cobordism $W: M_0 \leadsto M_1$ and $g_i \in \Riem^+ (M_i)$, we let $\Riem^+ (W)_{g_0,g_1}$ be the space of all psc metrics on $W$ which are of the form $dx^2+g_i $ near $M_i$, with respect to some (given) collars. For $h \in \Riem^+ (W)_{g_0,g_1}$, there are composition maps 
\[
 \mu (h,\_) : \Riem^+ (V)_{g_1,g_2} \lra \Riem^+ (W \cup V)_{g_0,g_2}
\]
and 
\[
 \mu(\_,h): \Riem^+ (V')_{g_{-1} ,g_0} \lra \Riem^+ (V'\cup W)_{g_{-1},g_1},
\]
defined for cobordisms $V': M_{-1} \leadsto M_0$ and $V:M_1 \leadsto M_2$ and $g_i \in \Riem^+ (M_i)$. We say that $h$ is \emph{right stable} if $\mu(h,\_)$ is a weak equivalence for all such cobordisms $V$, and \emph{left stable} if $\mu(\_,h)$ is a weak equivalence for all such $V'$, and \emph{stable} if it is both left and right stable. It turns out that a right stable metric on $[0,1] \times M$ is also left stable, and cylinder metrics $dx^2+g$ are right stable. The space
\[
 \Riem^+ ([0,1] \times M)_{g,g}^{\mathrm{st}} \subset \Riem^+ ([0,1] \times M)_{g,g}
\]
of all stable psc metrics on $[0,1] \times M$ is a union of path components of $ \Riem^+ ([0,1] \times M)_{g,g}$. The above composition maps extend to a map
$$\mu(\_,\_) : \Riem^+ ([0,1] \times M)_{g,g} \times \Riem^+ ([0,1] \times M)_{g,g} \to \Riem^+ ( [0,2]\times M)_{g,g} \cong \Riem^+ ([0,1] \times M)_{g,g}$$
equipping $\Riem^+ ([0,1] \times M)_{g,g}$ with the structure of an $H$-space (and in fact of an $E_1$-space), and $\Riem^+ ([0,1] \times M)_{g,g}^{\mathrm{st}}$ is the subspace of those elements which are invertible up to homotopy. With this vocabulary introduced, we can state the more general version of Theorem \ref{Main:infiniteloopspacetheoremsphere}.
\begin{MainThm}\label{Main:infiniteloopspacetheorem}
Let $M^{d-1}$ be a closed manifold and $g \in \Riem^+ (M)$. Assume that
\begin{enumerate}[(i)]
 \item $d \geq 6$,
 \item there is a cobordism $W: \emptyset \leadsto M$ such that $M \to W$ is $2$-connected and such that
 \item there is a right stable metric $h \in \Riem^+ (W)_g$.
\end{enumerate}
Then the space $\Riem^+ ([0,1] \times M)_{g,g}^{\mathrm{st}}$ has the homotopy type of an infinite loop space, with underlying $E_1$-space structure given by $\mu$.
\end{MainThm}
Again, we do \emph{not} claim that $\Riem^+ ([0,1] \times M)_{g,g}$, without the stability condition,  is an $E_\infty$-space. Theorem \ref{Main:infiniteloopspacetheorem} implies Theorem \ref{Main:infiniteloopspacetheoremsphere}: by \cite{Chernysh}, there is a homotopy equivalence
$$\Riem^+ (S^d) \simeq \Riem^+ ([0,1] \times S^{d-1})_{g_\round^{d-1}, g_\round^{d-1}}$$
under which the subspaces $\Riem^+ (S^d)^{\mathrm{st}}$ and $\Riem^+ ([0,1] \times S^{d-1})_{g_\round^{d-1}, g_\round^{d-1}}^{\mathrm{st}}$ correspond. 

\begin{rem}
The most classical example of an $E_d$-space is $\coprod_{n \geq 0} \Conf^n (\bR^d)$, the space of unordered configurations in $\bR^d$. Its group completion is the very interesting $d$-fold loop space $\Omega^d S^d$. It follows from the general theory of $E_d$-spaces that the subspace of invertible elements in $\coprod_{n \geq 0} \Conf^n (\bR^d)$ is a $d$-fold loop space, but in this case the subspace of invertible elements is just $\Conf^0 (\bR^d)=*$, and so the statement that it is a $d$-fold loop space is vacuous. This situation is quite common and shared by many classical examples, so the reader guided by analogies may ask whether the same is true in the situation of Theorem \ref{Main:infiniteloopspacetheorem}.

This is not true, at least not if $M$ carries a spin structure, by the main result of \cite{BERW}. By Theorem B of that paper, there is, for even $d \geq 6$, a map $\rho:\Omega^{\infty+1} \MT \Spin (d) \to \Riem^+ ([0,1] \times M)_{g,g}$ such that the composition with the index difference map $\inddiff_{g+dt^2}: \Riem^+ ([0,1] \times )_{g,g} \to \Omega^{\infty+d+1} \KO$ is highly nontrivial in homotopy. 
Even though this is not explicitly stated in \cite{BERW}, the map $\rho$ can be chosen to map the base point component of its source to an arbitrarily chosen component of $\Riem^+ ([0,1] \times M)_{g,g}$, for example to the component containing $g+dt^2$. But $g+dt^2$ lies in $\Riem^+ ([0,1] \times M)_{g,g}^{\mathrm{st}}$; therefore, we obtain a homotopically nontrivial map from the unit component of $\Omega^{\infty+1}_0\MT \Spin (d)$ into $\Riem^+ ([0,1] \times M)_{g,g}^{\mathrm{st}}$. This explains why $\Riem^+ ([0,1]\times M^{d-1})_{g,g}^{\mathrm{st}}$ is not contractible for even $d \geq 6$, and the case of odd $d \geq 7$ can be dealt with similarly by an inspection of the proof of Theorem C of \cite{BERW}.

In \S 8 of this paper, we shall give a new proof of the main results of \cite{BERW} (and of \cite{ERWpsc2}), which applies equally to odd-dimensional manifolds and directly to the subspace $\Riem^+ ([0,1]\times M)_{g,g}^{\mathrm{st}}$. This shows that $\Riem^+ ([0,1]\times M)_{g,g}^{\mathrm{st}} \not \simeq *$ when $M$ has a spin structure.
\end{rem}

\subsection{Outline of the proof of Theorem \ref{Main:infiniteloopspacetheorem}}
The proof of Theorem \ref{Main:infiniteloopspacetheorem} employs the \emph{cobordism category of manifolds with positive scalar curvature metrics}. The definition involves the notion of a tangential structure, so let us fix a fibration $\theta: B \to B\mathrm{O}(d)$, and denote by $\gamma_\theta:= \theta^* \gamma_d$ the pullback of the universal vector bundle along $\theta$. The (ordinary) cobordism category $\Cob_\theta$ has as its objects the closed $(d-1)$-dimensional $\theta$-manifolds and as its morphisms the $d$-dimensional $\theta$-cobordisms $W: M_0 \leadsto M_1$. Galatius, Madsen, Tillmann, and Weiss \cite{GMTW} have identified the homotopy type of the classifying space $B \Cob_\theta$ as the infinite loop space $\Omega^{\infty-1} \MT \theta$ of the Thom spectrum of the virtual vector bundle $-\gamma_\theta$. 

In their work \cite{GRW} on the homology of diffeomorphism groups of high-dimensional manifolds, Galatius and the second named author introduced some important subcategories of $\Cob_\theta$. The first is $\Cob_\theta^\kappa \subset \Cob_\theta$, which is the wide subcategory\footnote{A subcategory $\cB \subset \cC$ is \emph{wide} if $\cB$ contains all objects of $\cC$.} whose morphisms are the $\theta$-cobordisms $W: M_0 \leadsto M_1$ such that the inclusion map $M_1 \to W$ is $\kappa$-connected.
A further subcategory is the full subcategory $\Cob_\theta^{\kappa,l} \subset \Cob_\theta^{\kappa}$ on all objects whose structure maps $M \to B$ are $(l+1)$-connected\footnote{The definition is phrased in a slightly different way in \cite{GRW}, and this is responsible for the degree shift. In order to avoid confusion with the terminology of \cite{GRW}, we chose to stick to this somewhat unnatural convention.}. For suitable values of $\kappa$ and $l$, Theorems 3.1 and 4.1 of \cite{GRW} imply that the classifying spaces of these categories are weakly equivalent. For us it is the case $(\kappa,l)=(2,1)$ that is relevant, and in this case the maps
\begin{equation}\label{eqn:grwtheoremintroduction}
B \Cob_\theta^{2,1} \lra B \Cob_\theta^2 \lra B \Cob_\theta \simeq \Omega^{\infty-1} \MT \theta
\end{equation}
are weak equivalences provided that $d \geq 6$ and $B$ satisfies Wall's \cite{WallFin} finiteness condition $(F_2)$. The latter condition is satisfied for example if $B = B \Spin (d) \times BG$ when $G$ is a finitely presented group, or more generally if $\theta$ is the tangential $2$-type of a compact mani\-fold. 

The definition of the psc cobordism categories $\PCob_\theta^{\kappa,l}$ is straightforward: an object consists of a pair $(M,g)$ of an object $M$ of $\Cob_\theta^{\kappa,l}$ and a psc metric $g \in \Riem^+ (M)$, and a morphism $(W,h): (M_0,g_0) \leadsto (M_1,g_1)$ consists of a morphism $W: M_0 \leadsto M_1$ in $\Cob_\theta^{\kappa,l}$, together with a psc metric $h \in \Riem^+ (W)_{g_0,g_1}$. There is a suitable topology on $\PCob_\theta^{\kappa,l}$, which we will not describe in this introduction. There are wide subcategories $\PCob_\theta^{\kappa,l,\rst} \subset \PCob_\theta^{\kappa,l}$ whose morphisms are the pairs $(W,h)$ where $h$ is right stable. The forgetful functor
\[
 F: \PCob_\theta \lra \Cob_\theta
\]
restricts to functors $F^{\kappa}, F^{\kappa,l}$, and $F^{\kappa,l,\rst}$ between the respective subcategories. The first major step towards the proof of Theorem \ref{Main:infiniteloopspacetheorem} is to identify the homotopy fibres of (some of) those forgetful functors. To describe those homotopy fibres, we introduce the \emph{concordance category} $\Conc (M)$ of psc metrics on a closed manifold. Roughly, an object of $\Conc (M)$ is a psc metric $g \in \Riem^+ (M)$, and a morphism $g_0 \leadsto g_1$ is a \emph{concordance}, i.e. $h \in \Riem^+([0,1] \times M)_{g_0,g_1}$. There is a subcategory $\Conc (M)^{\rst}$ of right stable concordances.
\begin{MainThm}\label{Main:fibretheorem}
Let $M$ be an object of $\Cob_\theta^{2,1}$ and assume that $d \geq 6$. Then there are weak homotopy equivalences
\[
 B \Conc (M) \simeq \hofib_{M} (BF^{2,1} : B\PCob_\theta^{2,1} \to B\Cob_\theta^{2,1})
\]
and 
\[
 B \Conc (M)^{\rst} \simeq \hofib_{M} (BF^{2,1,\rst} : B\PCob_\theta^{2,1, rst} \to B\Cob_\theta^{2,1})
\]
(for a map $f:X \to Y$, the symbol $\hofib_y (f)$ denotes the homotopy fibre taken at the base-point $y \in Y$).
\end{MainThm}
This is a special case of the more precise and general Theorem \ref{thm:Fibre} below. The ingredients for the proof are the existence results for right stable metrics (Theorems E and D of \cite{ERWpsc2}) and a version of Quillen's Theorem B for topological categories (Theorem 4.9 of \cite{SxTech}). 

The relation of $B \Conc (M)^{\rst}$ with actual spaces of psc metrics is described as follows. A standard delooping argument (given in Theorem \ref{thm:Concordance}) shows that the tautological map
\begin{equation}\label{eq:firstdelooping}
 \Riem^+ ([0,1] \times M)_{g,g}^{\mathrm{st}}= \Riem^+ ([0,1] \times M)_{g,g}^{\rst}  \stackrel{\sim}{\lra} \Omega_g B \Conc (M)^{\rst}
\end{equation}
is a weak equivalence. By the equivalences \eqref{eqn:grwtheoremintroduction} the target space of $BF^{2,1,\rst}$ is equivalent to an infinite loop space. We would like to argue that $B \PCob_\theta^{2,1,\rst}$ is an infinite loop space and that $BF^{2,1,\rst}$ is an infinite loop map and hence conclude that $B \Conc (M)^{\rst}$ is an infinite loop space, thereby proving Theorem \ref{Main:infiniteloopspacetheorem}. The space $ B \Cob_\theta$ is a special $\Gamma$-space in the sense of Segal \cite{Segcat}: the composition law is given by disjoint union of the manifolds. Taking disjoint unions preserves the connectivity of the inclusion maps $M_1 \to W$, so that $B \Cob_\theta^2$ is a $\Gamma$-space as well. Similarly, $B \PCob_\theta^{2,\rst}$ is a special $\Gamma$-space and $BF^{2,\rst}$ is a map of $\Gamma$-spaces. It follows that the base-point component of the homotopy fibre $\hofib_\emptyset BF^{2,\rst}$ is an infinite loop space. 
However, $B \Cob_\theta^{2,1}$ and $B \PCob_\theta^{2,1,\rst}$ are \emph{not} $\Gamma$-spaces: taking disjoint union does not preserve the connectivity of the structure maps. To overcome this problem, we carry over some of the parametrised surgery methods of \cite{GRW} to the psc cobordism category. More precisely, we shall prove:
\begin{MainThm}\label{surgery-theorem-introcution}
If $2(l+1) < d$, $l \leq \kappa$, $l \leq d-\kappa-2$, $d-l-1 \geq 3$ and if $B$ is of type $(F_{l+1})$ in the sense of \cite{WallFin}, then the maps
\[
B \PCob_\theta^{\kappa,l} \lra B \PCob_\theta^{\kappa} \; \text{and} \;  B \PCob_\theta^{\kappa,l,\rst} \lra B \PCob_\theta^{\kappa,\rst}
\]
are weak homotopy equivalences.
\end{MainThm}
In our case of interest $(\kappa,l)=(2,1)$, Theorem \ref{surgery-theorem-introcution} holds as long as $d \geq 5$. It follows that the base-point component of $\hofib_\emptyset BF^{2,1,\rst}$ is an infinite loop space, and if $M$ and $g$ are chosen as in the hypothesis of Theorem \ref{Main:infiniteloopspacetheorem}, we can identify the homotopy fibres of $BF^{2,\rst}$ over $M$ and over $\emptyset$, which finishes our outline of the proof of Theorem \ref{Main:infiniteloopspacetheorem}. 

\subsection{Diffeomorphism group actions}
The cobordism category approach is also useful to illuminate the action by diffeomorphisms on spaces of metrics of positive scalar curvature. More specifically, let $W^d$ be a compact manifold with boundary $M = \partial W$ and let $g \in \Riem^+ (M)$. Let $\hAut(\Riem^+ (W)_g)$ be the monoid of homotopy automorphisms of the space $\Riem^+ (W)_g$. The action of the diffeomorphism group $\Diff_\partial (W)$ on $\Riem^+ (W)_g$ by pulling metrics back gives a map
\[
A:\Diff_\partial (W)\to \hAut (\Riem^+ (W)_g)
\]
of topological monoids and hence a map $B \Diff_\partial (W) \to B \hAut (\Riem^+ (W)_g)$ on classifying spaces. For each $h \in \Riem^+ (W)_g$, we furthermore get the orbit map $\sigma_h: \Diff_\partial (W) \to \Riem^+ (W)_g$, $f \mapsto f^* g$, which clearly factors through $A$. 

We prove that under favorable circumstances, $A$ factors through an infinite loop space. For that, we have to assume that $(W,M)$  is $2$-connected. Let $\theta:B \to B\mathrm{O}(d)$ be the tangential $2$-type of $W$. There is a natural map
\[
\alpha_W: B \Diff_\partial (W) \lra   \Omega_{\emptyset,M} B \Cob_\theta^2 \simeq  \Omega_{\emptyset,M} B \Cob_\theta \simeq \Omega^\infty \MT \theta,
\]
which we shall explain in \eqref{eqn:defnalphaW}. Consider the Borel construction
\[
 E\Diff_\partial (W) \times_{\Diff_\partial (W)} \Riem^+ (W)_g \lra B \Diff_\partial (W)
\]
given by the pullback action of diffeomorphisms on psc metrics. We shall prove the following result (a more general version is given in Theorem \ref{thm:diff-action-comesfromMT}):
\begin{MainThm}\label{mainthmintro:diffaction}
If $W$ and $\theta$ are as stated and $d \geq 6$, there exists a homotopy cartesian diagram
\[
\xymatrix{
E\Diff_\partial (W) \times_{\Diff_\partial (W)} \Riem^+ (W)_g \ar[r] \ar[d] & \mathcal{X} \ar[d]  \\
B \Diff_\partial (W) \ar[r]^-{\alpha_W} & \Omega^\infty \MT \theta.
}
\]
for a certain space $\mathcal{X}$. 

The map $A: \Diff_\partial (W) \to \hAut(\Riem^+ (W)_g)$ factors up to homotopy through $H$-space maps $\Diff_\partial (W) \simeq \Omega B \Diff_\partial (W) \stackrel{\Omega \alpha_W}{\to} \Omega^{\infty+1} \MT \theta$.
\end{MainThm}

Since $\pi_0 (\Omega^{\infty+1} \MT \theta)$ is an abelian group, we may conclude for example that the image of the group homomorphism $A_*:\pi_0 (\Diff_\partial (W)) \to \pi_0 (\hAut(\Riem^+ (W)_g)$ is abelian.

In certain special cases, an analogue of Theorem \ref{mainthmintro:diffaction} was implicitly proven in \cite[\S 4]{BERW} and \cite[\S 4]{ERWpsc2}, for even-dimensional manifolds only, by obstruction theory. The key ingredient for the obstruction argument was to first prove, by different means, that the image of $A_*:\pi_0 (\Diff_\partial (W)) \to \pi_0 (\hAut(\Riem^+ (W)_g)$ is abelian. In this paper, the logic is reversed. 

\begin{rem}
There are two ways in which Theorem \ref{mainthmintro:diffaction} is not expected to be the optimal result in this direction. Firstly, one might try to get rid of the hypothesis that $M\to W$ is $2$-connected, but we did not succeed in doing so using the techniques of this paper. The methods developed by Perlmutter \cite{PerlmutterMorse} \cite{PerlmutterPsc} seem better suited to this situation.

Secondly, the kernel of the action map $\pi_0 (\Diff_\partial (W)) \to \pi_0 (\hAut(\Riem^+ (W)_g))$ is in general larger than the kernel of $\pi_0 (\Diff_\partial (W)) \to \pi_1 (\MT \theta)$. In fact, it contains the kernel of the mapping torus map $\pi_0 (\Diff_\partial (W)) \to \Omega^\theta_{d+1}$ to the cobordism group of $(d+1)$-dimensional $\theta$-manifolds. This is shown in the PhD thesis of Georg Frenck \cite{Frenck}, by a fairly direct Morse-theoretic argument. 
\end{rem}

Theorem \ref{mainthmintro:diffaction} can be used to prove rigidity theorems for the action of the diffeomorphism group on spaces of psc metrics. As a sample for such results, in \S \ref{subsec:finitenessheorem} we prove the following theorem.

\begin{MainThm}\label{thm:finiteness-of-orbitmap}
Let $W$ be a simply connected manifold of dimension $d \geq 6$, with $2$-connected boundary inclusion $M \to W$ and assume that all rational Pontrjagin classes of $TW$ are trivial. Then for each $g \in \Riem^+ (M)$, each $h \in \Riem^+ (W)_g$ and each $k \geq 0$, the image of the map
\[
\sigma_h : \pi_k (\Diff_\partial (W)) \lra \pi_k (\Riem^+ (W)_g)
\]
induced by the orbit map $\sigma_h$ is finite. 
\end{MainThm}

\begin{rem}
The case of Theorem \ref{thm:finiteness-of-orbitmap} in which $W=D^d$ is a disc and $k$ is small compared to $d$ has previously been obtained by Botvinnik, Hanke, Schick and Walsh \cite{BHSW}. For even $d$ and the manifolds $\sharp^g (S^n \times S^n) \setminus D^{2n}$, this might be extracted from \cite{BERW}, \S 4. For the case $k=0$, see \cite{Frenck} for a more direct argument.
\end{rem}

\subsection{Index-theoretic consequences}
In the case where the manifolds have spin structures, we also prove index-theoretic results. Let $G$ be a finitely presented group and consider $\theta:B\Spin (d) \times BG \to B\mathrm{O}(d)$. A $\theta$-structure on a $d$-manifold $W$ is the same as a spin structure on $W$ and a map $W \to BG$. These ingredients may be used to define the Rosenberg--Dirac operator on $W$, which is linear over the group $\cstar$-algebra $\cstar(G)$ of $G$ (everything applies to both, the reduced and the maximal group version of $\cstar(G)$). Like the usual spin Dirac operator, it satisfies the Lichnerowicz--Schr\"odinger formula. If $g_0 \in \Riem^+ (\partial W)$ and $h_0 \in \Riem^+ (W)_{g_0}$, one defines the secondary index invariant
$$\inddiff_{h_0}: \Riem^+ (W)_{g_0} \lra \Omega^{\infty+d+1} \KO(\cstar(G))$$
as in \cite[\S 5]{ERWpsc2}. 
\begin{MainThm}\label{thm:indddiff-infiniteloopmap-introduction}
If $(M,g_0)$ are as in Theorem \ref{Main:infiniteloopspacetheorem}, and in addition $M$ is spin and has fundamental group $G$, then the map 
\[
\inddiff_{dx^2+g_0}: \Riem^+ ([0,1] \times M)^{\rst}_{g_0,g_0} \lra \Omega^{\infty+d+1} \KO(\cstar(G))
\]
is an infinite loop map, where the source space carries the infinite loop space structure from Theorem
\ref{Main:infiniteloopspacetheorem}.
\end{MainThm}

To achieve the proof of Theorem \ref{thm:indddiff-infiniteloopmap-introduction}, we construct a ``delooped version'' of $\inddiff_{h_0}$. In \cite{JEIndex2}, the first named author constructed an index map
\[
\ind_1: B \Cob_\theta \lra  \Omega^{\infty+d-1} \KO(\cstar(G))
\]
given in operator-theoretic terms. It is a delooping of the family index of the Rosenberg--Dirac operator on closed manifolds, in the following sense. When we compose $\Omega \ind_1$ with the tautological map $\tau: \Cob_\theta (\emptyset,\emptyset) \to \Omega B \Cob_\theta$ from the space of closed $d$-dimensional $\theta$-manifolds, we obtain the family index of the Rosenberg--Dirac operator, interpreted as a map
\[
\ind_0:  \Cob_\theta (\emptyset,\emptyset) \lra \Omega^{\infty+d} \KO(\cstar(G)).
\]
The analytical description of $\ind_1$ enables us to construct a nullhomotopy of the composition 
\[
 \ind_1 \circ BF: B \PCob_\theta \lra \Omega^{\infty+d-1} \KO(\cstar(G)),
\]
by an application of the Lichnerowicz--Schr\"odinger formula. This nullhomotopy yields a map
\[
\hofib_M (BF) \lra \Omega^{\infty+d} \KO(\cstar(G)),
\]
whose homotopy class depends on the choice of $g_0 \in \Riem^+ (M)$. Composing with the obvious map $B \Conc (M) \to \hofib_M BF$, we obtain 
\[
 B \inddiff_{h_0} : B \Conc (M) \lra \Omega^{\infty+d} \KO(\cstar(G)),
\]
the \emph{delooped index difference} (here $h_0=dx^2 + g_0$). It is not hard to show that the restriction of $B \inddiff_{h_0}$ to $B \Conc(M)^{\rst}$ is an infinite loop map. An elementary, but tedious computation (Theorem \ref{thm:comparison-delooped-index-diference1}) proves that the composition 
\[
 \Riem^+ ([0,1] \times M)_{g_0,g_0}^{\rst} \stackrel{\sim}{\lra} \Omega_{g_0} B \Conc (M)^{\rst} \xrightarrow{\Omega B \inddiff_{h_0}} \Omega^{\infty+d+1} \KO (\cstar(G)) 
\]
is homotopic to $\inddiff_{h_0}$, which concludes the proof of Theorem \ref{thm:indddiff-infiniteloopmap-introduction}. We moreover prove (Theorem \ref{thm:comparison-delooped-index-diference2}) that the composition 
\[
 \Riem^+ (M) \lra B \Conc(M) \xrightarrow{B \inddiff_{h_0}} \Omega^{\infty+d} \KO(\cstar(G))
\]
is homotopic to $\inddiff_{g_0}$ (the first map is the inclusion of $0$-simplices). Another application of those ideas gives a new proof of the main results of \cite{ERWpsc2} and \cite{BERW} which also works for odd-dimensional manifolds (Theorem \ref{berw-pscpi1-newproof}). 

\subsection{Implication of the concordance-implies-isotopy conjecture}

In chapter \S \ref{sec:concimpliesisotopy-consequences}, we explain how an affirmative solution of the concordance-implies-isotopy problem for psc metrics leads to a cleaner formulation of many of our main results. In short, it would imply, for $M$ as in Theorem \ref{Main:infiniteloopspacetheorem}, that $\Riem^+ ([0,1] \times M)=\Riem^+ ([0,1] \times M)^{\mathrm{st}}$. However, the \emph{proofs} would not simplify at all, except for shorter notation.

\subsection*{Outline of the paper}

\S \ref{sec:psc-recollection} is of preparatory nature; we mainly recall the stability condition for psc metrics from \cite{ERWpsc2} and prove some auxiliary elementary lemmas about stable psc metrics. In \S\ref{sec:psc-cob-category}, we introduce the psc cobordism category. There are many ways to write down point-set topological models for the cobordism categories, and the proofs in the subsequent sections employ several of them. This slightly unfortunate fact makes \S \ref{sec:psc-cob-category} relatively long. 
In \S \ref{sec:fibre-theorem}, we prove Theorem \ref{Main:fibretheorem} and the equivalence \eqref{eq:firstdelooping}. In \S \ref{sec:loopinf-structure}, we give the parametrised surgery proof of Theorem \ref{surgery-theorem-introcution}. The proof is written to be as parallel as possible as the proofs in \cite[\S 4, \S 6]{GRW}, and this section is written with a reader who is fully familiar with that paper in mind. While Theorem \ref{surgery-theorem-introcution} is crucial for all our results, we have written the rest of the paper so that the reader can take Theorem \ref{surgery-theorem-introcution} as a black box. 
In \S \ref{sec:6}, we put the strands from the previous sections together and complete the proof of Theorem \ref{Main:infiniteloopspacetheorem}, after giving a review of the theory of $\Gamma$-spaces following Segal \cite{Segcat}. 
In \S \ref{sec:diffactiononpscspace}, we prove Theorem \ref{mainthmintro:diffaction}, using the results from \S \ref{sec:fibre-theorem}, Theorem \ref{surgery-theorem-introcution} and some basic semi-simplicial technique. The reader of the index-theoretic part, \S \ref{sec:indextheory}, needs to know the theory of \cite{JEIndex2}. In \S \ref{sec:concimpliesisotopy-consequences}, we show that an affirmative solution of the concordance-implies-isotopy conjecture affects the formulation of Theorem \ref{Main:infiniteloopspacetheorem}. This makes strong use of the existence theorems for stable psc metrics from \cite{ERWpsc2} but is otherwise self-contained. In Appendix \ref{appendixsheaves}, we prove a version of the key technical ingredient for the proofs in \cite{GRW} in the context of sheaves, which is used in \S \ref{sec:loopinf-structure}.

\section{Recollections on spaces of psc metrics}\label{sec:psc-recollection}

\subsection{Spaces of psc metrics on manifolds with boundaries}

For a closed mani\-fold $M$, we let $\Riem (M)$ be the space of all Riemannian metrics, equipped with the usual Fr\'echet topology and we let $\Riem^+ (M) \subset \Riem (M)$ be the open subspace of all Riemannian metrics with positive scalar curvature. 

Let $W$ be a compact manifold with boundary $M$. We assume that the boundary of $W$ comes equipped with a collar $c:  [0,\infty) \times M \to W$. The collar identifies $ [0,\infty)  \times M$ with an open subset of $W$ and we usually use this identification without further mentioning.

For $\eps>0$, we denote by $\Riem^+ (W)^{\eps}$ the space of all Riemannian metrics $h$ on $W$ with positive scalar curvature such that $c^* h = dx^2+g$ on $ [0,\eps] \times M$ for some metric $g$ on $M$, with the usual Fr\'echet topology. We let
\[
 \Riem^+ (W):= \colim_{\eps \to 0} \Riem^+ (W)^\eps.
\]
Elements in $\Riem^+ (W)$ are psc metrics $h$ which are of the form $dx^2+g$ near $M$, and the scalar curvature of $g$ is positive. Hence assigning to $h$ the boundary value $g$ defines a continuous map 
\[
\res: \Riem^+ (W) \lra \Riem^+ (M) .
\]
We define
\[
 \Riem^+ (W)_g := \res^{-1}(g) \subset \Riem^+ (W).
\]
\begin{thm}[Chernysh \cite{Chernysh2}]\label{thm:improved-chernysh-theorem} 
For a compact manifold $W$ with collared boundary $M$, the restriction map $\res:\Riem^+ (W) \to \Riem^+ (M)$ is a Serre fibration.
\end{thm}
In \cite{Chernysh2} it is only shown that $\res$ is a quasifibration. The version as stated is proven as Theorem 1.1 of \cite{EbFrenck}. 

\subsection{The Gromov--Lawson--Chernysh theorem}\label{subsec:gromov-lawson-recap}

\begin{defn}\label{defn:torpedo-metric}
By $g_\round^{k-1} \in \Riem (S^{k-1})$, we denote the round metric on $S^{k-1}$, i.e. the metric induced from the euclidean metric by the standard inclusion $S^{k-1} \subset \bR^{k}$. It has constant scalar curvature $\scal (g_\round^{k-1}) \equiv (k-1)(k-2)$. Let $\delta >0$. A \emph{$\delta$-$R$-torpedo metric} $g_\tor^k$ on $\bR^{k}$, $k \geq 3$, is an $\mathrm{O}(k)$-invariant metric such that 
$\scal (g_\tor^k)\geq \frac{1}{\delta^2}(k-1)(k-2)$ and such that 
\[
\varphi^* g_\tor^k = dr^2 + \delta g_\round^{k-1}
\]
near $[R,\infty)\times S^{k-1}$, where $\varphi:(0,\infty) \times S^{k-1} \to \bR^k \setminus 0$ is the diffeomorphism defined by $(r,x)\mapsto rx$. For more details, see \S 2.3 of \cite{Walsh01}.
\end{defn}

\begin{defn}
Let $d-k \geq 3$, let $W^d$ be a compact manifold with collared boundary $M$, let $V^k$ be a compact manifold with collared boundary $N$ and let $\phi: V \times \bR^{d-k} \to W$ be an embedding. Assume that $\phi^{-1} ( [0,\infty) \times M) =  [0,\infty)  \times N$ and such that inside the collar, $\phi$ is of the form $\id \times \varphi$ for some embedding $\varphi:N \times \bR^{d-k} \to M$. 

Let $h_V$ be a Riemannian metric on $V$ which is collared near the boundary, pick $\delta>0$ such that $\scal (h_V) + \frac{1}{\delta^2}(k-1)(k-2) >0$ and fix a $\delta$-$1$-torpedo metric $g_\tor^{d-k}$ on $\bR^{d-k}$. By
\[
\Riem^+ (W,\phi) \subset \Riem^+ (W),
\] 
we denote the space of all $h \in \Riem^+ (W)$ such that $\phi^* h = h_V + g_\tor^{d-k}$ near $V \times D^{d-k} \subset V \times \bR^{d-k}$. Furthermore, we let
\[
 \Riem^+ (W,\phi)_g := \Riem^+ (W,\phi) \cap \Riem^+ (W)_g \subset  \Riem^+ (W)_g,
\]
where $g \in \Riem^+ (M,\varphi)$ is a suitable boundary condition. 
\end{defn}

The following result due to Chernysh \cite{Chernysh, Chernysh2} is a sharpening of a famous result by Gromov--Lawson \cite{GL}, and is of crucial importance for this paper:
\begin{thm}[Chernysh]\label{thm:chernysh-theorem}
Assume that $d-k \geq 3$. Then
\begin{enumerate}[(i)]
\item the inclusion
\[
 \Riem^+ (W,\phi) \lra \Riem^+ (W)
\]
is a weak homotopy equivalence, and
\item for each $g \in \Riem^+ (M,\varphi)$, the inclusion
\[
\Riem^+ (W,\phi)_g \lra \Riem^+ (W)_g
\]
is a weak homotopy equivalence.
\end{enumerate}
\end{thm}
A complete and self-contained exposition of the proof (which also corrects some minor flaws) appears in \cite{EbFrenck}. 

\begin{defn}
Let $W$ be a compact $d$-dimensional manifold, possibly with boundary $M$. A surgery datum (i.e.\ embedding) $\phi:S^{k} \times \bR^{d-k} \to \inter{W}$ is \emph{admissible} if $2 \leq k \leq d-3$. Two compact $d$-manifolds $W$ and $W'$ with the same boundary are \emph{admissibly cobordant} if one can obtain $W'$ from $W$ by a sequence of admissible surgeries in the interior.
\end{defn}
Let
\[
W_\phi:= (W\setminus \phi(S^k \times D^{d-k})) \cup_{S^k \times S^{d-k-1} } (D^{k+1} \times S^{d-k-1})
\]
be the result of performing a surgery along $\phi$. The following easy consequence of Theorem \ref{thm:chernysh-theorem} is Theorem 2.5 of \cite{BERW}.
\begin{cor}\label{cor:chernysh-theorem}
An admissible surgery datum $\phi$ determines a preferred homotopy class of weak homotopy equivalences
\[
\SE_\phi: \Riem^+ (W)_g \simeq \Riem^+ (W_{\phi})_g,
\]
the \emph{surgery equivalence} determined by $\phi$.
\end{cor}
We remark that $\SE_\phi$ is not explicitly given, but only a zig-zag of maps. This is not a problem for our purposes: we only use $\SE_\phi$ to identify the sets of path components of both spaces. 

\subsection{The stability condition}

In our previous work \cite{ERWpsc2}, we proved a gene\-ralisation of Theorem \ref{thm:chernysh-theorem}, which is also a key ingredient in the present paper and which is therefore recalled here. For composable sequences 
$$M_{-1} \stackrel{V'}{\leadsto} M_0 \stackrel{W}{\leadsto} M_1 \stackrel{V}{\leadsto} M_2$$
of cobordisms, $g_i \in \Riem^+ (M_i)$ and $h \in \Riem^+ (W)_{g_0,g_1}$, there are \emph{gluing maps} 
\[
 \mu (\_,h): \Riem^+ (V')_{g_{-1},g_0} \lra \Riem^+ (V'\cup W)_{g_{-1},g_1}
\]
and 
\[
  \mu (h,\_): \Riem^+ (V)_{g_{1},g_2} \lra  \Riem^+ (W \cup V)_{g_0,g_2}.
\]
\begin{defn}\label{defn:right-stable}
Let $W: M_0 \leadsto M_1$ be a cobordism and let $h \in \Riem^+ (W)_{g_0,g_1}$. Then $h$ is called \emph{left-stable} if the map $\mu (\_, h): \Riem^+ (V)_{g_{-1},g_0} \to \Riem^+ (V \cup W)_{g_{-1},g_1}$ is a weak equivalence for all cobordisms $V:M_{-1}\leadsto M_0$ and all boundary conditions $g_{-1}$. Dually, $h$ is \emph{right-stable} if the map $\mu (h,\_): \Riem^+ (V)_{g_1,g_2} \to \Riem^+ (W \cup V)_{g_0,g_2}$ is a weak equivalence for all cobordisms $V: M_1 \leadsto M_2$ and all boundary conditions $g_2$. Finally, $h$ is \emph{stable} if it is both left and right stable. By
\[
 \Riem^+ (W)^{\rst} \subset \Riem^+ (W) \quad \text{and} \quad  \Riem^+ (W)_{g_0,g_1}^{\rst} \subset \Riem^+ (W)_{g_0,g_1}, 
\]
we denote the subspaces of right stable psc metrics, and define $\Riem^+ (W)_{g_0,g_1}^{\mathrm{st}}$ similarly. 
\end{defn}

The following result encapsulate most instances of the Gromov--Lawson surgery method that we shall use in this paper.

\begin{thm}[Theorem 3.1.2 of \cite{ERWpsc2}]\label{thm:StabMetrics}
Let $d \geq 6$ and let $W: M_0 \leadsto M_1$ be a $d$-dimensional cobordism.
\begin{enumerate}[(i)]
\item If the pair $(W, M_1)$ is $2$-connected then for each $g_0 \in \Riem^+ (M_0)$, there is $g_1 \in \Riem^+ (M_1)$ and a right stable $h \in \Riem^+ (W)_{g_0,g_1}$. 
\item If the pairs $(W, M_0)$ and $(W, M_1)$ are both 2-connected, then every right stable $h \in \Riem^+ (W)_{g_0,g_1}$ is also left stable.
\end{enumerate}
\end{thm}

Let us collect some fairly straightforward but important facts about stable metrics.
The following simple observation is immediate from the definitions and will be used repeatedly.

\begin{lem}[Lemma 3.3.1 of \cite{ERWpsc2}]\label{lem:2-out-of-three}
Let $(W, h): (M_0, g_0) \leadsto (M_1, g_1)$ and $(W', h'): (M_1, g_1) \leadsto (M_2, g_2)$ be psc cobordisms. Then
\begin{enumerate}[(i)]
\item If $h$ and $h'$ are left-stable, then so is $h\cup h'$. 
\item\label{it:2of3ii} If $h$ and $h'$ are right-stable, then so is $h\cup h'$. 
\item If $h'$ and $h\cup h'$ are left-stable, then so is $h$.
\item If $h$ and $h\cup h'$ are right-stable, then so is $h'$.
\end{enumerate}
\end{lem}

\begin{lem}\label{lem:stability-homotopy-invariant}
The subspaces $\Riem^+ (W)_{g_0,g_1}^{\rst} \subset \Riem^+ (W)_{g_0,g_1}$ and $\Riem^+ (W)^{\rst} \subset \Riem^+ (W)$ are unions of path components. The same holds for left stable metrics.
\end{lem}

\begin{proof}
If $h,h' \in \Riem^+ (W)_{g_0,g_1}$ lie in the same path component, then $\mu(h,\_)$ and $\mu(h',\_)$ are homotopic, which already shows the first claim. For the second claim, let $M_0 \stackrel{W}{\leadsto} M_1 \stackrel{V}{\leadsto} M_2$ be two cobordisms. Consider the commutative diagram
\[
\xymatrix{
\ar[d]^{\id \times r_2}  \Riem^+ (W) \times_{\Riem^+ (M_1)} \Riem^+ (V) \ar[r]^-{\mu(\_,\_)} & \Riem^+ (W \cup V)\ar[d]^{ (r_0,r_2)}\\
\Riem^+ (W) \times \Riem^+ (M_2) \ar[r]^-{r_0 \times \id} & \Riem^+ (M_0)\times \Riem^+ (M_2)
}
\]
where $r_i$ denotes the suitable restriction map. For $(h,g_2) \in \Riem^+(W) \times \Riem^+ (M_2)$, the fibre map 
\[
(\id \times r_2)^{-1} (h,g_2) \to (r_0,r_2)^{-1} (h,g_2) 
\]
is precisely the map $\mu(h,\_):\Riem^+ (V)_{r_1 (h),g_2}\to \Riem^+ (W \cup V)_{r_0(h),g_2}$. The vertical maps are fibrations, by Theorem \ref{thm:improved-chernysh-theorem}. Therefore, if $h'$ is in the same component of $\Riem^+ (W)$ as $h$, then the fibre map over $(h,g_2)$ is a weak equivalence if and only if the fibre map over $(h',g_2)$ is. 

This proves the Lemma for right stable metrics, and the proof for left stable metrics is completely analogous.
\end{proof}

Lemma \ref{lem:stability-homotopy-invariant} has a generalisation where the underlying manifolds are allowed to be varied continuously too, as follows. 

\begin{lem}\label{lem:stability-invariant-under-continuously-changing-W}
Let $\pi:E \to B$ be a bundle of compact manifolds with boundary. Assume that the boundary bundle $\partial E$ is decomposed into two parts $\partial_0 E$ and $\partial_1 E$ (so that $E$ can be viewed as a bundle of cobordisms). Let $W_x:= \pi^{-1}(x)$ and $M_{i,x} := W_x \cap \partial_i E$ (so that each $W_x$ is a cobordism $M_{0,x} \leadsto M_{1,x}$). 
Let $(h_x)_{x \in B}$ be a continuous family of psc metrics on the fibres of $\pi$. 

Then if $B$ is path-connected and if $h_y \in \Riem^+ (W_y)$ is right stable for one $y \in B$, then $h_x \in \Riem^+ (W_x)$ is right stable for each other $x \in B$. The same holds for left stability.
\end{lem}

\begin{proof}
By pulling back the bundle along a path from $y$ to $x$, we find that it is enough to consider the case $B=[0,1]$, $y=0$ and $x=1$. But then the bundle $\pi$ can be trivialised, so that $E\cong  [0,1] \times W_0$. The family $(h_x)$ becomes a continuous map $[0,1] \to \Riem^+ (W_0)$, and the claim follows immediately from Lemma \ref{lem:stability-homotopy-invariant}.
\end{proof}

\begin{lem}[Lemma 3.3.3 of \cite{ERWpsc2}]\label{lem:stability-under-surgery}
Let $W^d: M_0 \leadsto M_1$ be a cobordism and let $\phi: S^{k} \times D^{d-k} \to \inter{W}$ be an admissible surgery datum. Let $[h]\in \pi_0(\Riem^+ (W)_{g_0,g_1})$ and $[h'] \in  \pi_0(\Riem^+ (W_{\phi})_{g_0,g_1})$ correspond under the weak equivalence $\SE_{\phi}$. Then $h'$ is left stable (right stable) iff $h$ is left stable (right stable).
\end{lem}

Later on, we need a sharpening of Theorem \ref{thm:StabMetrics} which involves tangential structures. 

\begin{defn}\label{defn:tangentialstructure}
Let $\theta:B \to B\mathrm{O}(d)$ be a fibration and let $\gamma_\theta$ be the pullback of the universal $d$-dimensional vector bundle along $\theta$. A \emph{$\theta$-structure} on a $d$-dimensional manifold $W$ is a bundle map $\ell: TW \to \gamma_\theta$, and a $\theta$-structure on a $(d-1)$-dimensional manifold $M$ is a bundle map $\ell : \bR \oplus TM \to \gamma_\theta$. A $\theta$-structure $\ell$ on a cobordism $W: M_0 \leadsto M_1$ induces $\theta$-structures on $M_i$, namely $\bR \oplus TM_i \cong TW|_{M_i} \stackrel{\ell}{\to} \gamma_\theta$, where the first isomorphism is induced by the collar\footnote{Here we need the collar around $M_0$ to be of the form $ [0,\infty) \times M_0  \to W$ and around $M_1$ to be of the form $ (-\infty,1] \times M_1  \to W$.} of $W$.

The fibration $\theta$ is \emph{once-stable} if it is pulled back from a fibration $\theta'$ over $B\mathrm{O}(d+1)$. 
\end{defn}

A $\theta$-structure $\ell: T W \to \gamma_\theta$ covers a map $W \to B$, sometimes referred to as the \emph{structure map}. Slightly abusing notation, we often denote the structure map also by $\ell: W \to B$. This should not cause confusion.

\begin{thm}\label{prop:right-stable-on-2.1.cob-stable}
Let $\theta$ be once-stable. Let $(W,\ell_W) : (M_0, \ell_{M_0}) \leadsto (M_1, \ell_{M_1})$ be a $\theta$-cobordism with $d \geq 6$. Assume that the inclusion $i_1:M_1 \to W$ and the structure maps $\ell_{M_i}: M_i \to B$ are $2$-connected. Then the following statements hold.
\begin{enumerate}[(i)]
\item A psc metric on $W$ is left stable if and only if it is right stable.
\item For each $g_0 \in \Riem^+ (M_0)$, there exists $g_1 \in \Riem^+ (M_1)$ and a stable $h \in \Riem^+ (W)_{g_0,g_1}$.
\item For each $g_1 \in \Riem^+ (M_1)$, there exists $g_0 \in \Riem^+ (M_0)$ and a stable $h \in \Riem^+ (W)_{g_0,g_1}$.
\end{enumerate}
\end{thm}

For the proof, we need a result which is very similar to \cite[Theorem 2.2]{Rosenberg1986} and to a result in \cite[Appendix B]{HebestreitJoachim} (which were used for similar purposes in those papers).

\begin{lem}\label{lem:surgery-lemma-fabian-michael}
Let $\theta$ be once-stable and let $W: M_0 \leadsto M_1$ be a $d$-dimensional $\theta$-cobordism, $d \geq 6$, such that the structure maps $\ell_{M_0}:M_0 \to B$ and $\ell_W: W \to B$ are both $2$-connected. Then $W$ is admissibly $\theta$-cobordant (relative to its boundary) to a cobordism $W': M_0 \leadsto M_1$ such that the inclusion $M_0 \to W'$ is $2$-connected. If $M_1 \to W$ was $2$-connected, then so is $M_1 \to W'$.
\end{lem}

\begin{proof}
The maps $M_0 \overset{i_0}\to W \overset{\ell_W}\to B$ both induce isomorphisms on $\pi_1$, and we write $\pi$ for the common fundamental group. The long exact homotopy sequence of the pair $(W,M_0)$ and the maps to $B$ induce a diagram with exact row
\[
 \xymatrix{
 \pi_2 (M_0)\ar@{->>}[dr]_-{(\ell_{M_0})_*} \ar[r]^-{(i_0)_*} & \pi_2 (W) \ar@{->>}[d]^-{(\ell_W)_*}\ar[r]^-{q_*} & \pi_2 (W,M_0) \ar[r] & 0\\
 & \pi_2 (B). & & 
 }
\]
As the inclusion $i_0 : M_0 \to W$ induces an isomorphism on fundamental groups, the $\bZ[\pi]$-module $\pi_2 (W,M_0)$ is finitely generated by \cite[\S 1]{WallFin}. Hence we can pick elements $y_1 ,\ldots,y_r \in \pi_2 (W)$ such that $\{q_* (y_i)\}_{i\leq r}$ generate $\pi_2 (W,M_0)$. By a diagram chase, we can moreover pick these $y_i$ so that $(\ell_W)_* (y_i)=0 \in \pi_2 (B)$. Since $\dim W \geq 5$, we can represent each $y_i$ by an embedded $2$-sphere, and since the image in $\pi_2 (B)$ is zero, the normal bundle of this sphere is trivial, and hence we find embeddings $\phi_i: S^2 \times D^{d-2} \to \inter{W}$ representing the $y_i$. Doing $\theta$-surgeries on all these spheres yields a new $\theta$-cobordism $W'$. Since the surgery on an embedded $S^2 \times D^{d-2}$ as well as the opposite surgery on $S^{d-3} \times D^3$ is in codimension at least $3$, $W'$ is admissibly $\theta$-cobordant to $W$. If $d \geq 6$, then the initial $2$-connectivity of $M_1 \to W'$ is not destroyed, by a general position 
argument.
\end{proof}

\begin{proof}[Proof of Theorem \ref{prop:right-stable-on-2.1.cob-stable}]
Let $W'$ be a cobordism as provided by Lemma \ref{lem:surgery-lemma-fabian-michael}. Since $W$ and $W'$ are admissibly cobordant, there is a surgery equivalence
\[
 \SE: \Riem^+ (W)_{g_0,g_1} \simeq \Riem^+ (W')_{g_0,g_1}
\]
and left and right stable metrics on $W$ and $W'$ correspond under $\SE$, by Lemma \ref{lem:stability-under-surgery}. Thus part (i) follows from Theorem \ref{thm:StabMetrics} (ii). Part (ii) is immediate from part (i) and Theorem \ref{thm:StabMetrics} (i). For part (iii), we use a dual version of Theorem \ref{thm:StabMetrics} (i) (i.e.\ apply it to the reversed cobordism): given $g_1 \in \Riem^+ (M_1)$, there is $g_0 \in \Riem^+ (M_0)$ and a left-stable $h \in \Riem^+ (W')_{g_0,g_1}$, because $i_0: M_0 \to W'$ is $2$-connected. By part (i), $h$ is also right stable.
\end{proof}

\section{The psc cobordism category}\label{sec:psc-cob-category}

\subsection{Cobordism categories}\label{sec:CobCat}

We shall make use of the theory of cobordism categories as developed in \cite{GMTW, GRW}. More specifically, we use the version $\Cob_\theta$ defined in \cite{GRW}, which is a non-unital topological category. Let us recall this category in a form which will be most convenient for our needs. We shall fix a fibration $\theta : B \to B\mathrm{O}(d)$ as in Definition \ref{defn:tangentialstructure}. 

An object of $\Cob_\theta$ consists of a pair $(M, \ell)$ of a $(d-1)$-dimensional closed submani\-fold $M \subset \bR^\infty$ and a $\theta$-structure $\ell$ on $M$. A morphism from $(M_0, \ell_0)$ to $(M_1, \ell_0)$ in $\Cob_\theta$ consists of a triple $(t, W, \ell)$ of a real number $t>0$, a $d$-dimensional submani\-fold $W \subset [0,t] \times \bR^\infty$ which is equal to $[0,t] \times M_0$ near $\{0\} \times \bR^\infty$ and equal to $[0,t] \times M_1$ near $\{t\} \times \bR^\infty$, and a $\theta$-structure $\ell : TW \to \gamma_\theta$ which restricts to $\ell_i$ on $M_i$. Composition of morphisms is given by translation and union of subsets of $\bR \times \bR^\infty$. This may be seen to give sets of objects and morphisms in bijection with those of \cite[Definition 2.6]{GRW}, and they are topologised using this bijection. We now explain this topology 
in more familiar terms.

For a $(d-1)$-manifold $M$, we let the space of embeddings $\Emb(M, \bR^\infty)$ be given the Whitney $C^\infty$-topology. This is contractible, has an obvious free action by the diffeomorphism group $\Diff(M)$ and the quotient map $\Emb(M;\bR^\infty) \to \Emb(M;\bR^\infty)/\Diff(M)$ is a universal $\Diff(M)$-principal bundle \cite{BinzFischer}. Furthermore, let $\Bun(\bR \oplus TM, \gamma_\theta)$ be the space of bundle maps with the compact-open topology and define
$$
\mathcal{M}^\theta(M) := \bigl(\Emb(M, \bR^\infty) \times \Bun(\bR \oplus TM, \gamma_\theta) \bigr) / \Diff(M).
$$
The space of objects of $\Cob_\theta$ can be described as
$$
\Ob(\Cob_\theta) = \coprod_{[M]} \mathcal{M}^\theta(M),
$$
where the disjoint union is over closed $(d-1)$-manifolds, one in each diffeomorphism class. Similarly, for a $d$-dimensional cobordism $W$ from $M_0$ to $M_1$ equipped with collars $e_0 : [0,1) \times M_0 \hookrightarrow W$ and $e_1 : (0,1] \times M_1 \hookrightarrow W$ and $\epsilon>0$, we let $\Emb_\epsilon(W, [0,1]\times\bR^\infty)$ denote the space of $\epsilon$-collared embeddings, and $\Diff_\epsilon(W)$ denote the group of $\epsilon$-collared diffeomorphisms, both with the Whitney $C^\infty$-topology, and put
$$
\mathcal{M}^\theta(W)_\epsilon := \bigl(\Emb_\epsilon(W, [0,1]\times\bR^\infty) \times \Bun(TW, \gamma_\theta)\bigr)/\Diff_\epsilon(W)
$$
and
\[
 \mathcal{M}^{\theta}(W) :=\colim_{\epsilon\to 0}  \mathcal{M}^{\theta}(W)_\epsilon.
\]
The space of morphisms of $\Cob_\theta$ can be described as
$$
\Mor(\Cob_\theta) = (0,\infty) \times \coprod_{[W]} \colim\limits_{\epsilon \to 0}\mathcal{M}^\theta(W)_\epsilon,
$$
where the disjoint union is over $d$-dimensional cobordisms, one in each diffeomorphism class.

In \cite{GRW}, certain subcategories $\Cob_{\theta}^{\kappa,l} \subset \Cob_\theta$ were considered, and since they will be crucial for us as well, we recall their definition.

\begin{defn}
For $\kappa \geq 0$, we define $\Cob_{\theta}^{\kappa} \subset \Cob_\theta$ as the wide subcategory whose morphism space $\Mor (\Cob^{\kappa}_\theta) \subset \Mor (\Cob_\theta)$ consists of all $(t,W,\ell)$ such that the inclusion $W|_t \to W$ is $\kappa$-connected. 

For $l \geq -1$, we define $\Cob_\theta^{\kappa,l} \subset \Cob_\theta^{\kappa}$ as the full subcategory on those $(M,\ell) \in \Ob (\Cob_\theta^{\kappa})$ such that the map $\ell: M \to B$ is $(l+1)$-connected.
\end{defn}

\begin{rem}
In \cite{GRW}, a variant of $\Cob_\theta^{\kappa,l}$ was defined which only makes sense when all manifolds are required to contain an embedded copy of a certain manifold $L$ or $L \times [0,t]$. The present definition is better suited to our needs. The index shift (that the objects in $\Cob_\theta^{\kappa,l}$ are required to be $(l+1)$-connected relative to $B$) has its origins in \cite{GRW} and we decided not to change the notation, in order to avoid confusion with that paper. 
\end{rem}

\begin{thm}[Theorem 3.1 of \cite{GRW}]\label{thm:grw-thm3.1}
If $2 \kappa \leq d-2$, the inclusion $B \Cob_{\theta}^{\kappa} \to B \Cob_{\theta}^{\kappa-1}$ is a weak homotopy equivalence.
\end{thm}

The following is essentially Theorem 4.1 of \cite{GRW}, with a small modification. Recall that Wall \cite{WallFin} has defined a topological space $B$ to have \emph{type ($F_n$)} if there exists a finite CW-complex $K$ and and $n$-connected map $K \to B$. 

\begin{thm}[Galatius--Randal-Williams]
\label{thm:grw-thm4.1}
Assume that
\begin{enumerate}[(i)]
\item $2(l+1) < d$,
\item $l \leq \kappa$,
\item $l \leq d-\kappa-2$,
\item the space $B$ is of type ($F_{l+1}$).
\end{enumerate}
Then the inclusion map $B \Cob_\theta^{\kappa,l}\to B \Cob_\theta^{\kappa,l-1}$ is a weak homotopy equivalence. 
\end{thm}

\begin{remark}
As we already said, in \cite{GRW}, there are variant categories $\Cob_{\theta,L}^{\kappa,l}$ considered, where the categories we are considering correspond to setting $L = \emptyset$. The formulation of \cite[Theorem 4.1]{GRW} is meaningful only if $L \neq \emptyset$ as it requires $\ell_L : L \to B$ to be $(l+1)$-connected. 
We indicate in Remark \ref{rem:proofof:thm:grw-thm4.1} the minor way in which the proof in \cite{GRW} needs to be adjusted to prove Theorem \ref{thm:grw-thm4.1} as stated. 
This variant has also been observed by Hebestreit--Perlmutter \cite[Theorem 3.3.2 et seq.]{HP}.
\end{remark}

Important for us is the case where $(\kappa,l)=(2,1)$. In that case Theorems \ref{thm:grw-thm3.1} and \ref{thm:grw-thm4.1} both apply as long as $d \geq 6$ and $B$ has type ($F_2$), showing that the inclusions $\Cob_{\theta}^{2,1} \to \Cob_\theta^2 \to \Cob_\theta$ induce weak equivalences on classifying spaces. 

\subsection{Positive scalar curvature cobordism categories}\label{sec:PSCCobCat}

We shall now describe an analogue of this where all manifolds are equipped with psc metrics.

\begin{defn}
Let $\PCob_\theta$ be the non-unital topological category described as follows. For each closed $(d-1)$-manifold $M$ let
$$
\mathcal{M}_{\psc}^\theta(M) := (\Emb(M, \bR^\infty) \times \Bun(\bR \oplus TM, \gamma_\theta) \times \mathcal{R}^+(M))/\Diff(M),
$$
and set
$$
\Ob(\PCob_\theta) := \coprod_{[M]} \mathcal{M}_{\psc}^\theta(M)
$$
where the disjoint union is over closed $(d-1)$-manifolds, one in each diffeomorphism class. For a $d$-dimensional cobordism $W$ from $M_0$ to $M_1$ equipped with collars $e_0 : [0,1) \times M_0 \hookrightarrow W$ and $e_1 : (0,1] \times M_1 \hookrightarrow W$ let
$$
\mathcal{M}^\theta_{\psc}(W)_\epsilon := (\Emb_\epsilon(W, [0,1]\times\bR^\infty) \times \Bun(TW, \gamma_\theta) \times \mathcal{R}^+(W)^{\epsilon})/\Diff_\epsilon(W),
$$
and set
$$
\Mor(\PCob_\theta) = (0,\infty) \times \coprod_{[W]} \colim\limits_{\epsilon \to 0}\mathcal{M}^\theta_{\psc}(W)_\epsilon
$$
where the disjoint union is over $d$-dimensional cobordisms, one in each diffeomorphism class. As in the case of the ordinary cobordism category, composition is defined by translation and union. This defines a non-unital topological category analogously to $\Cob_\theta$, and forgetting psc metrics defines a continuous functor 
\[
 F: \PCob_\theta \lra \Cob_\theta.
\]
\end{defn}
We can write objects in $\PCob_\theta$ as $(M,\ell,g)$, with $(M,\ell) \in \Ob (\Cob_\theta)$ and $g \in \Riem^+ (M)$. A morphism from $(M_0,\ell_0,g_0)$ to $(M_1,\ell_1,g_1)$ is a tuple $(t,W,\ell,h)$, with $(t,W,\ell): (M_0,\ell_0)\leadsto (M_1,\ell_1)$ a morphism in $\Cob_\theta$ and $h \in \Riem^+ (W)_{g_0,g_1}$. 
For a pair of objects $(M_0, \ell_0, g_0), (M_1, \ell_1, g_0) \in \Ob(\PCob_\theta)$ and $(t,W, \ell) \in \Cob_\theta(M_0, M_1)$ the space 
$$\PCob_\theta((M_0, \ell_0, g_0), (M_1, \ell_1, g_0)) \cap F^{-1}(t, W, \ell)$$
is canonically identified with the space $\Riem^+(W)_{g_0, g_1}$ of psc metrics on $W$ subject to the boundary conditions $g_0$ and $g_1$.

\begin{defn}
We define $\PCob_\theta^{\kappa} \subset \PCob_\theta$ as the wide subcategory containing all morphisms in $\PCob_\theta$ whose image under $F: \PCob_\theta \to \Cob_\theta$ lie in $\Cob_\theta^\kappa$. 
We define $\PCob_\theta^{\kappa,l} \subset \PCob_\theta^{\kappa}$ is the full subcategory containing all objects whose image under $F$ lie in $\Cob_\theta^{\kappa,l}$. Finally, we define $\PCob_\theta^{\kappa,\rst} \subset \PCob_\theta^{\kappa}$ and $\PCob_\theta^{\kappa,l,\rst} \subset \PCob_\theta^{\kappa,l}$ as the wide subcategories containing all morphisms $(W,\ell,t,h)$ such that the psc metric $h$ is right-stable. The restrictions of the forgetful functor $F$ are denoted
\[
 F^{\kappa}, \; F^{\kappa,l}, \; \text{and} \; F^{\kappa,l,\rst}.
\]
\end{defn}
One of our main goals is the identification of the homotopy fibres of the maps
\[
 BF^{\kappa,l}: B \PCob_{\theta}^{\kappa,l}\lra B \Cob_\theta^{\kappa,l} \quad\text{and}\quad   BF^{\kappa,l,\rst}: B \PCob_{\theta}^{\kappa,l,\rst}\lra B \Cob_\theta^{\kappa,l}.
\]
To achieve that goal, we shall apply some of the categorical techniques described in \cite{SxTech}, such as Quillen's Theorem A and B for non-unital topological categories. The following discussion shows that one of the main technical hypotheses of these theorems is satisfied. Let us first recall some definitions from \cite{SxTech}.

\begin{defn}\label{def:SxTechRecollections}
Let $\cC$ be a non-unital topological category, with source map $s=d_1:\Mor(\cC) \to \Ob (\cC)$ and target map $t=d_0:\Mor(\cC) \to \Ob (\cC)$. We say that
\begin{enumerate}[(i)]
 \item $\cC$ is \emph{left fibrant} if $s$ is a fibration, \emph{right fibrant} if $t$ is a fibration and \emph{fibrant} if $(s,t): \Mor(\cC)\to \Ob (\cC) \times \Ob (\cC)$ is a fibration.
 \item $\cC$ has \emph{soft right units} if the under category $c \backslash \cC$ has contractible classifying space for each $c \in \Ob (\cC)$ and $\cC$ has \emph{soft left units} if the over-category $ \cC/c$ has contractible classifying space for each $c \in \Ob (\cC)$.
 \item $\cC$ has \emph{weak left units} if for each object $b \in \Ob (\cC)$, there is a morphism $u: b\to b'$ in $\cC$ so that the map
\[
\cC (-,b) := d_0^{-1} (b)\stackrel{u \circ -}{\lra}\cC (-,b')
\]
is a weak homotopy equivalence. Dually, $\cC$ has \emph{weak right units} if for each object $b \in \cC_0$, there is a morphism $u:b'\to b$ in $\cC$ such that
\[
\cC (b,-):=d_1^{-1}(b) \stackrel{- \circ u}{\lra} \cC (b',-)
\]
is a weak homotopy equivalence. 
\end{enumerate}
\end{defn}

The notion of weak units is a technical one, and will be used only as a tool to establish that a category has soft units, via \cite[Lemma 3.14]{SxTech}. That lemma says that if $\cC$ has weak left units and is right fibrant, then it has soft left units, and similarly with left and right interchanged.

\begin{prop}\label{prop:Fibrant}\mbox{}
\begin{enumerate}[(i)]
\item The forgetful functor $F$ induces fibrations 
$$F_0: \Ob (\PCob_\theta) \lra \Ob (\Cob_\theta)\quad \text{ and } \quad F_1: \Mor (\PCob_\theta) \lra \Mor (\Cob_\theta).$$
\item The non-unital topological categories $\Cob_\theta$ and $\PCob_\theta$ are fibrant.
\end{enumerate}
\end{prop}

In order to prove this proposition we shall use the following general result.

\begin{lem}\label{sublemma:lem:fibrant}
\mbox{}
\begin{enumerate}[(i)]
\item Let $G$ be a topological group, let $p:E \to B$ a $G$-principal bundle and let $f: X\to Y$ be a $G$-equivariant map which is a Serre fibration. Then the induced map on Borel constructions $\id \times_G f: E \times_G X \to E \times_G Y$ is a Serre fibration.
\item Let $\varphi:G_0 \to G_1$ be a homomorphism of topological groups which is also a Serre fibration, let $p_i:E_i\to E_i /G_i$ be $G_i$-principal bundles and let $\psi: E_0 \to E_1$ be a Serre fibration and $\varphi$-equivariant. Then the induced map $\psi':E_0/G_0 \to E_1 / G_1$ is a Serre fibration.
\end{enumerate}
\end{lem}

\begin{proof}
The first part follows easily from the fact that being a Serre fibration is a local property.

For the second part note that $E_0 \overset{p_0}\to E_0/G_0 \overset{\psi'}\to E_1 / G_1$ is equal to $p_1 \circ \psi$, a composition of Serre fibrations, so is a Serre fibration. But $p_0$ is a Serre fibration with non-empty fibres, so $\psi'$ is a Serre fibration too. (This is standard: given a lifting problem $(a: D^n \times [0,1] \to E_1/G_1, b : D^n \times \{0\} \to E_0/G_0$, we may lift $b$ to $b' : D^n \times \{0\} \to E_0$ as $D^n$ is a cell and the fibres of $p_0$ are nonempty. We may then solve the lifting problem $(a,b')$, as $\psi' \circ p_0$ is a Serre fibration, which in particular solves the original lifting problem.)
\end{proof}

\begin{proof}[Proof of Proposition \ref{prop:Fibrant}]
Part (i) is immediate from the definitions, Lemma \ref{sublemma:lem:fibrant} (i) and from the fact that $\Emb(M,\bR^\infty) \to \Emb(M,\bR^\infty) / \Diff(M)$ is a principal bundle. 

To show part (ii) in the case of $\PCob_\theta$, it is enough to show that for each cobordism $W$ from $M_0$ to $M_1$, the source/target map $\cM_{\psc}^\theta (W) \to \cM_{\psc}^{\theta} (M_0) \times \cM_{\psc}^{\theta} (M_1)$ is a Serre fibration. To simplify notation, we may assume $M_0=\emptyset$ and set $M:=M_1$. We abbreviate $B_0 := \Bun (TW, \gamma_\theta)$ and $B_1 := \Bun(\bR \oplus TM , \gamma_\theta)$. The restriction map $B_0 \to B_1$ is a Serre fibration because $M \to W$ is a cofibration. Furthermore, if we write $R_0^\eps := \Riem^+ (W)^{\epsilon}$, $R_0 := \colim_{\epsilon \to 0} R_0^{\epsilon}$ and $R_1 := \Riem^+ (M)$, then the restriction map $R_0 \to R_1$ is a Serre fibration, by Theorem \ref{thm:improved-chernysh-theorem}. Let $E_0^\epsilon := \Emb_\epsilon(W, [0,1]\times\bR^\infty)$, $E_0 := \colim_{\epsilon\to 0} E_0^{\epsilon}$ and $E_1 := \Emb(M,\bR^\infty)$, and similarly  $G_0^{\epsilon}:= \Diff_\eps(W)$, $G_0:= \colim_{\eps \to 0} G_0^\epsilon$ and $G_1 := \Diff(M)$. Then the restriction maps
$$E_0 \lra E_1 \quad\text{ and }\quad G_0 \lra G_1$$
are also Serre fibrations. This may be proved in parallel with the proof of Theorem \ref{thm:improved-chernysh-theorem} given in \cite{EbFrenck}: in brief, using the collar structure it is easy to see that both maps have the homotopy lifting property for smooth homotopies, by pushing the ``graph" of the homotopy into the collar; one then establishes an analogue of \cite[Lemma 5.1]{EbFrenck}, using the local Fr{\'e}chet space structure of $\Emb(M,\bR^N)$ or $\Diff(M)$, to reduce lifting arbitrary homotopies to lifting smooth ones.

We have to prove that the map
\[
 \colim_{\eps \to 0} (E_0^{\eps} \times B_0 \times R_0^{\eps}) / G_0^{\eps} \lra (E_1 \times B_1 \times R_1)/G_1
\]
is a Serre fibration. But the source of this map is the same as\footnote{We implicitly work in the category of compactly generated spaces as in \cite{Strick} to interchange products and colimits.} 
\begin{align*}
  \colim_{\eps \to 0} (E_0^{\eps} \times B_0 \times R_0^{\eps}) / G_0^{\eps} &=  \colim_{\eps,\delta \to 0, \delta \leq \epsilon} (E_0^{\eps} \times B_0 \times R_0^{\delta}) / G_0^{\eps}\\
 &=   \colim_{\eps \to 0} (E_0^{\eps} \times B_0 \times R_0) / G_0^{\eps} = (E_0 \times B_0 \times R_0)/G_0,
\end{align*}
and so it remains to prove that $(E_0 \times B_0 \times R_0)/G_0 \to (E_1 \times B_1 \times R_1)/G_1$ is a fibration. With all the things said above, this is a consequence of Lemma \ref{sublemma:lem:fibrant} (ii). The case of the category $\Cob_\theta$ is by the same argument, ignoring the spaces $R_i$ throughout. 
\end{proof}

\begin{lem}\label{lem:soft-units}
The categories $\Cob_\theta$ and $\PCob_\theta$ have weak right and left units, and therefore soft right and left units.
\end{lem}

\begin{proof}
As both categories are fibrant by Proposition \ref{prop:Fibrant} (ii),  having weak units implies having soft units by Lemma 3.14 of \cite{SxTech}. Let us prove that $\PCob_\theta$ has weak left units; the other cases are analogous. Using Remark 3.13 of \cite{SxTech}, this means that for each $(M_1,\ell_1,g_1) \in \Ob (\PCob_\theta)$, there is a morphism $u: (M_1,\ell_1,g_1) \to (M_2,\ell_2,g_2)$, such that for all $(M_0,\ell_0,g_0)\in \Ob (\PCob_\theta)$, the composition map
\[
 \PCob_\theta ((M_0,\ell_0,g_0),(M_1,\ell_1,g_1)) \stackrel{u \circ -}{\lra}  \PCob_\theta ((M_0,\ell_0,g_0),(M_2,\ell_2,g_2))
\]
is a weak equivalence. Let $ (M_2,\ell_2,g_2):=(M_1,\ell_1,g_1)$ and let $u:=(1, [0,1 ]\times M_1 , \ell_1',dx^2+g_1)$, where $\ell_1'$ is the canonical extension of $\ell_1$ to the cylinder. To check that $u \circ -$ is a weak equivalence, use Corollary 2.2 (ii) of \cite{BERW}.
\end{proof}

\begin{defn}
A subcategory $\cB \subset \cC$ of a topological category is \emph{clopen} if both $\Ob (\cB) \subset \Ob (\cC)$ and $\Mor (\cB) \subset \Mor (\cC)$ are open and closed.
\end{defn}

\begin{lem}\mbox{}
\begin{enumerate}[(i)]
\item The subcategories $\Cob_\theta^{\kappa} , \Cob_\theta^{\kappa,l} \subset \Cob_\theta$ are clopen and have weak right and left units.
\item The subcategories $\PCob_\theta^{\kappa} , \PCob_{\theta}^{\kappa,\rst}, \PCob_\theta^{\kappa,l}, \PCob_\theta^{\kappa,l,\rst} \subset \PCob_\theta$ are clopen and have weak right and left units.
\end{enumerate}
\end{lem}
\begin{proof}
The clopenness in (i) is clear, and in (ii) follows from Lemma \ref{lem:stability-homotopy-invariant}. For the weak units, we take the cylinder cobordisms from the proof of Lemma \ref{lem:soft-units}.
\end{proof}

It is clear that a clopen subcategory of a fibrant category is fibrant, so Proposition \ref{prop:Fibrant} gives the following.

\begin{cor}\label{psccob-fibratnt}
The categories $\Cob_\theta^{\kappa,l}$, $\PCob_\theta^{\kappa,l}$, and $\PCob_\theta^{\kappa,l,\rst}$ are fibrant and have weak units. The forgetful functors $F^{\kappa,l}: \PCob_\theta^{\kappa,l}\to \Cob_\theta^{\kappa,l}$ and $F^{\kappa,l,\rst}: \PCob_\theta^{\kappa,l,\rst} \to \Cob_\theta^{\kappa,l}$ are fibrations on object and morphism spaces. 
\end{cor}

\subsection{Models for the psc cobordism category}

The point-set topological model for $\PCob_\theta^{\kappa,l}$ described above is appropriate for the proof of Theorem \ref{Main:fibretheorem}, which identifies the homotopy fibre of the forgetful maps $BF^{2,1,\rst}$ and $BF^{2,1}$. However, for most other purposes in this paper it is better to use different models. We shall use variants of the ``poset model'' for the cobordism categories defined in \cite[Section 2.6]{GRW}, but it is more convenient for us to use the language of sheaves as in \cite{MadsenWeiss} and \cite{GMTW} instead of point-set topology. Let us give a brief review of that formalism. For details, the reader is referred to \cite[\S 2.1, \S 2.4, \S 4.1 and Appendix A]{MadsenWeiss}.

\subsubsection{The language of sheaves}\label{sec:LangSh}

Let $\Mfds$ be the category of all smooth manifolds and smooth maps, referred to as \emph{test manifolds}. A \emph{sheaf} on $\Mfds$ is a contravariant functor $\cF: \Mfds \to \Set$, which satisfies the gluing condition for open covers $(U_i)_{i \in I}$ of a manifold $X$. In other words, if $z_i \in \cF(U_i)$ is a collection of elements (often called ``sections'') such that $z_i|_{U_i \cap U_j}= z_j|_{U_i \cap U_j}$ for all $i,j \in I$, then there is a unique $z \in \cF(X)$ such that $z|_{U_i}=z_i$ for all $i$. The sheaves on $\Mfds$ are the objects of a category $\Sheaves$. 

The \emph{points} of $\cF$ are the elements of the set $\cF(\ast)$. A section $z \in \cF(X)$ defines a function $f_z:X \to \cF(\ast)$, namely $x \mapsto j_x^* (z)$, where $j_x:\ast \to X$ is the inclusion of $x$ into $X$. Most (but not all) sheaves we meet in this paper have the property that a section $z$ is determined by the function $f_z$. 

For a subset $\cG(*) \subset \cF(*)$ and a test manifold $X$, we let $\cG(X) \subset \cF(X)$ be the set of those $z \in \cF(X)$ such that $j_x^* (z) \in \cG(*)$ for each $x \in X$. This defines a subsheaf $\cG \subset \cF$.

Let $\Delta_e^p := \{(x_0, \ldots, x_p) \in \bR^{p+1} \, | \, \sum_{i=0}^p x_i = 1\}$ denote the extended $p$-simplex; for varying $p$ these define a cosimplicial object $\Delta^\bullet_e : \Delta \to \Mfds$ in the category of smooth manifolds. To a sheaf $\cF$ one can therefore associate a simplicial set $[p] \mapsto \cF(\Delta_e^p)$, and the \emph{representing space} $\rep{\cF}$ of $\cF$ is by definition the fat\footnote{In \cite{MadsenWeiss}, the ordinary geometric realisation is used. Since we will have to deal with sheaves of \emph{semi}-simplicial sets in this paper, it is more convenient to use the fat geometric realisation. In any case, the homotopy types agree.} geometric realisation
\[
 \rep{[p] \mapsto \cF(\Delta_e^p)}
\]
of this simplicial set. We say that a map of sheaves is a weak equivalence if the induced map of representing spaces is a weak equivalence of spaces.

\begin{rem}
Evaluation defines functions $\Delta^p \times \cF(\Delta_e^p) \to \cF(\ast)$ which are compatible for varying $p$, giving a (surjective) function
$$\rep{\cF} \lra \cF(\ast).$$
We may endow $\cF(\ast)$ with the quotient topology from this map, which is easily seen to have the property that the functions $f_z : X \to \cF(\ast)$ become continuous for every $z \in \cF(X)$. However this topology might be somewhat pathological in that it receives more continuous maps than these, even up to homotopy: proving it does not involves establishing an ``approximation" theorem for maps into $\cF(\ast)$ (endowed with this topology). In any case $\rep{\cF}$ always has the correct homotopy type, and one can avoid a certain amount of anguish by working solely with sheaves and their representing spaces.
\end{rem}

Elements $z_0, z_1 \in \cF(X)$ are \emph{concordant} if there is a $z \in \cF(X \times \bR)$ which agrees with $\mathrm{pr}_X^*(z_0)$ on a neighbourhood of $X \times (-\infty,0]$ and with $\mathrm{pr}_X^*(z_1)$ on a neighbourhood of $X \times [1,\infty)$. This defines an equivalence relation on $\cF(X)$, and we write $\cF[X]$ for the set of equivalence classes. It is shown in Proposition 2.17 of \cite{MadsenWeiss} that there is a natural bijection
$$[X; \rep{\cF}] \overset{\sim}\lra \cF[X]$$
from the set of homotopy classes of continuous maps from $X$ to $\rep{\cF}$ to the set of concordance classes of sections of $\cF$ over $X$. 
The following criterion is useful to prove that a map of sheaves is a weak equivalence.

\begin{prop}[Proposition 2.18 of \cite{MadsenWeiss}]\label{prop:SurjCrit}
Let $f: \cE \to \cF$ be a map of sheaves. For a closed subset $A \subset X$, we let 
$$\cE(A \subset X) :=\colim_{\substack{U \supset A\\ U \text{ open}}} \cE(U)$$ 
denote the set of germs of sections of $\cE$ near $A$, and for $z \in \cE(A \subset X)$, we let $\cE[X,A;z]$ be the set of concordance classes (relative to $A$) of sections which agree with $z$ near $A$. If for each $X \in \Mfds$, $A \subset X$ closed, and $z\in \cE(A \subset X)$ the map 
\[
 f_* : \cE[X, A; z] \lra \cF[X, A; f(z)]
\]
is surjective, then $f: \cE \to \cF$ is a weak equivalence. (This is simply an instance of the philosophy that injectivity is relative surjectivity.)
\end{prop}

\subsubsection{The sheaf version of the cobordism categories}

\begin{defn}\label{ex:SheafPsi}
Let $U$ be a manifold and $d \in \bN_0$. Let $\theta: B \to B\mathrm{O}(d)$ be a Serre fibration, and let $\gamma_\theta:= \theta^* \gamma_d$ be the pullback of the universal vector bundle. 
The sheaf $\psi_\theta(U)$ assigns to a test manifold $X$ the set of all pairs $(M,\ell)$, where 
\begin{enumerate}[(i)]
 \item $M \subset U \times X$ is a submanifold which is closed as a subspace, such that the projection map $\pi=\pr_X: M \to X$ is a submersion with $d$-dimensional fibres.
 \item $\ell$ is a $\theta$-structure on the vertical tangent bundle $T_v M := \ker (d\pi) \to M$, in other words a bundle map $\ell: T_v M \to \gamma_\theta$.
\end{enumerate}
Pullback along a smooth map $f: X_0 \to X_1$ yields a map $\psi_\theta(U) (X_1) \to \psi_\theta(U)(X_0)$.
\end{defn}

\begin{rem}\label{rem:SP}
Note that $\psi_\theta(U)(*)$ is the set of all $d$-dimensional submanifolds $M \subset U$ which are closed as a subspace and are equipped with a $\theta$-structure. In Sections 2.1 and 2.3 of \cite{GRW10} a topology on the set $\psi_\theta(U)(*)$ is defined, with the resulting topological space denoted by $\Psi_\theta(U)$. By Lemma 2.17 of \cite{GRW10}, the function $f_z:X \to \Psi_\theta (U)$ induced by an element $z \in \psi_\theta(U)(X)$ is ``smooth'' and so in particular continuous. The canonical evaluation function $\rep{\psi_\theta(U)} \to \Psi_\theta(U)$ is therefore continuous, and in fact is a weak homotopy equivalence by the smooth approximation theorem, Lemma 2.18 of \cite{GRW}. In Theorem A.3 of \cite{SPInv} Schommer-Pries has shown that the quotient topology on $\psi_\theta(U)(*)$ from $\rep{\psi_\theta(U)}$ agrees with that of $\Psi_\theta(U)$.
\end{rem}

A \emph{fibrewise Riemannian metric} on a submersion $\pi:M\to X$ is a smooth bundle metric $g$ on $T_v M$. The restriction of $g$ to a fibre $\pi^{-1}(x)$ is a Riemannian metric $g_x$ on $\pi^{-1}(x)$.

\begin{defn}
The sheaf $\psi_\theta^{\riem}(U)$ assigns to a test manifold $X$ the set of all triples $(M,\ell, g)$, where 
\begin{enumerate}[(i)]
\item $(M, \ell) \in \psi_\theta(U)(X)$, and
\item $g$ is a fibrewise Riemannian metric on $M$.
\end{enumerate}
\end{defn}

A \emph{semi-simplicial sheaf} or \emph{sheaf of semi-simplicial sets} is a functor $\cF_\bullet : \Mfds \to \ssSet$ such that each $\cF_q :\Mfds \to \Set$ is a sheaf of sets in the sense described above.
Then $[q] \mapsto \rep{\cF_q}$ is a semi-simplicial space, whose fat geometric realisation we denote by $\norm{\cF_\bullet}$. If $\cC :\Mfds \to \CatNU$ is a sheaf of (small) non-unital categories (i.e.\ a functor such that $\Ob(\cC), \Mor(\cC) : \Mfds \to \Set$ are sheaves of sets in the sense described above), then taking nerves gives a semi-simplicial sheaf $N_\bullet \cC$.

We may easily encode the categories $\Cob_\theta^{\kappa,l}$, $ \PCob_\theta^{\kappa,l}$, $ \PCob_\theta^{\kappa,l, \psc}$, and $ \PCob_\theta^{\kappa,l, \rst}$ from Section \ref{sec:PSCCobCat} in the language of sheaves, as their spaces of objects and morphisms all have a natural notion of ``smooth maps from a manifold".

\begin{definition}
Let $\Ob(\sh\cC_\theta) : \Mfds \to \Set$ be the sheaf which assigns to a test manifold $X$ the set of pairs $(M, \ell_M)$ of a submanifold $M \subset \bR^\infty \times X$ such that the projection $\pi_X : M \to X$ is a proper submersion with $(d-1)$-dimensional fibres, and $\ell_M : \bR \oplus T_v M \to \gamma_\theta$ is a bundle map.

For a smooth function $t: X \to (0,\infty)$, we denote
\[
 [0,t] \times \bR^\infty \times X := \{ (s,v,x) \in \bR \times \bR^\infty \times X\, \vert\, 0 \leq s \leq t(x)\}
\]
and 
\[
 \{t\} \times \bR^\infty \times X := \{ (s,v,x) \in \bR \times \bR^\infty \times X \,\vert\, s = t(x)\} \cong  \{0\} \times \bR^\infty \times X .
\]
Let $\Mor(\mathrm{sh}\cC_\theta) : \Mfds \to \Set$ be the sheaf which assigns to a test manifold $X$ the set of triples $(t, W, \ell_W)$ of a smooth function $t : X \to (0,\infty)$, a submanifold $W \subset [0,t] \times \bR^\infty \times X$ such that the projection $\pi : W \to X$ is a proper submersion with fibres $W_x := \pi^{-1}(x)$ being $d$-dimensional manifolds whose boundary lies in $ (\{0\} \cup \{t\})  \times \bR^\infty \times X$, and $\ell_W : T_v W \to \gamma_\theta$ is a bundle map, such that there is a smooth $\epsilon : X \to (0,\infty)$ so that $(W_x,\ell_W\vert_{W_x}) \in \cM^\theta(W_x)_{\epsilon(x)}$ for each $x \in X$.

The source and target maps are given by restriction to $\{0\} \times \bR^\infty \times X$ or $\{t\} \times \bR^\infty \times X$ and the obvious identification. The composition map is given by translating and gluing. This defines a sheaf of non-unital categories, which is very much similar to the one defined in \cite[Definition 2.8]{GMTW}.
\end{definition}

Note that $(N_p \mathrm{sh}\cC_\theta)(*) = N_p \cC_\theta$ as sets, and with the topology on $N_p \cC_\theta$ described in Section \ref{sec:CobCat} the evaluation maps $X \times (N_p \mathrm{sh}\cC_\theta)(X) \to N_p \cC_\theta$ are continuous: this follows using Lemma 2.17 of \cite{GRW10}, as discussed in Remark \ref{rem:SP} above. Thus there is a continuous map $\rep{N_p \mathrm{sh}\cC_\theta} \to N_p \cC_\theta$ given by evaluation. This map is a weak equivalence: this follows using the smooth approximation, Lemma 2.18 of \cite{GRW}, as discussed in Remark \ref{rem:SP} above. Thus taking the fat geometric realisation gives a weak equivalence $\norm{N_\bullet \mathrm{sh}\cC_\theta} \overset{\sim}\to B\cC_\theta$. Similarly, if we define $\mathrm{sh}\cC_\theta^{\kappa, l}$ as sub-(sheaves of non-unital categories) of $\mathrm{sh}\cC_\theta$ by imposing the connectivity conditions pointwise, then we have 
\begin{equation*}
\norm{N_\bullet \mathrm{sh}\cC_\theta^{\kappa, l}} \overset{\sim}\lra B\cC_\theta^{\kappa, l}. 
\end{equation*}

\begin{definition}
Let $\Ob(\mathrm{sh}\cP_\theta) : \Mfds \to \Set$ be the sheaf which assigns to a test manifold $X$ the set of triples $(M, \ell_M, g_M)$ where $(M, \ell_M) \in \Ob(\mathrm{sh}\cC_\theta)$ and $g_M$ is a fibrewise Riemannian metric on $ M$ with positive scalar curvature (that is, $(g_M)_x$ has positive scalar curvature for each $x \in X$).

Let $\Mor(\mathrm{sh}\cP_\theta) : \Mfds \to \Set$ be the sheaf which assigns to a test manifold $X$ the set of quadruples $(t, W, \ell_W, g_W)$ where $(t, W, \ell_W) \in \Mor(\mathrm{sh}\cC_\theta)$ and $g_W$ is a fibrewise Riemannian metric on $ W$ with positive scalar curvature such that there is a smooth $\epsilon : X \to (0,\infty)$ so that $(W_x, \ell_W\vert_{W_x}, g_W\vert_{W_x}) \in \cM^\theta_{\psc}(W_x)_{\epsilon(x)}$ for all $x \in X$.
\end{definition}

Again maps into each $N_p \cP_\theta$ have smooth approximations, so there are weak equivalences $\norm{N_\bullet \mathrm{sh}\cP_\theta} \overset{\sim}\to B\cP_\theta$. If we define $\mathrm{sh}\cP_\theta^{\kappa, l, (\rst)}$ as sub-(sheaves of non-unital categories) of $\mathrm{sh}\cP_\theta$ by imposing the connectivity conditions (and perhaps right-stability) pointwise, then we have equivalences 
\begin{equation}\label{eqn:equivalence-cobcat-vs-sheafcobcatpsc}
\norm{N_\bullet \mathrm{sh}\cP_\theta^{\kappa, l, (\rst)}} \overset{\sim}\lra B\cP_\theta^{\kappa, l, (\rst)}. 
\end{equation}

\subsubsection{Poset models}\label{sec:PosetModels}

We wish to establish models analogous to the $\norm{\cD_\theta^{\kappa,l}(\bR^N)_\bullet} \simeq B \Cob_\theta^{\kappa,l}$ of \S 2.6 of \cite{GRW}, but in the context of sheaves.

\begin{defn}\label{defn:nakedDspace}
The sheaf $\cD_\theta(\bR^N)$ assigns to a test manifold $X$ the set of all $(W,\ell)\in \psi_\theta (\bR \times \bR^N)(X)$ such that $W \subset \bR \times (-1,1)^N \times X$. 
We also define
\[
\cD_\theta:= \colim_{N \to \infty} \cD_\theta (\bR^N),
\]
where the colimit is taken in the category of sheaves.
\end{defn}

By Section 4 of \cite{GMTW}, there is a weak equivalence
\begin{equation}\label{eq:longmanioflds-versus-ordcobcat}
 \rep{\cD_\theta} \simeq B \Cob_\theta.
\end{equation}

\begin{defn}\label{defn:cylindrical}
For $(W,\ell) \in \cD_\theta(\bR^N)(X)$, we denote by $x_1: W \to \bR$ the restriction of the projection to $\bR$. Note that $x_1$ is fibrewise proper, i.e. $(x_1,\pi):W \to \bR \times X$ is proper. For $J \subset \bR$, we write $W|_J := x_1^{-1}(J)= W \cap\bigl( J \times \bR^N \times X \bigr)$. If $I \subset \bR$ is an interval, we say that $W$ is \emph{cylindrical over $I$}, if for every $t \in I$, we have the equality $W|_I = I \times W|_t$ as subsets of $\bR \times \bR^{N} \times X$ and as $\theta$-manifolds. 
We say that $W$ is \emph{cylindrical over $J \subset \bR$} if it is cylindrical over each interval $I \subset J$, and we say that $W$ is \emph{cylindrical near $a \in \bR$} if there is an open neighbourhood of $a$ over which $W$ is cylindrical. 

Let in addition $g$ be a Riemannian metric on $W$. We say that $(W,\ell,g)$ is \emph{cylindrical over $J$} if $(W,\ell)$ is cylindrical over $J$ in the above sense and if for each interval $I \subset J$ and each $t \in I$, the metric $g|_{W_I}$ is equal to $dx_1^2 + g|_{W|_t}$. Similarly, we define $(W,\ell,g)$ to be cylindrical near $a \in \bR$ if is cylindrical over some neighbourhood of $a$.
\end{defn}

\begin{defn}\label{defn:Dspaces}
The semi-simplicial sheaf $\cD_\theta^{\kappa,l}(\bR^N)_\bullet$ has its sheaf of $p$-simplices given as follows. An element of $\cD_\theta^{\kappa,l} (\bR^N)_p(X)$ is a tuple $(W,\ell_W,a,\eps)$ where
\begin{enumerate}[(i)]
\item $(W,\ell_W ) \in \cD_\theta(\bR^N)$,
\item $a= (a_0,\ldots,a_p)$ and $\epsilon=(\eps_0, \ldots, \eps_p)$ are tuples of smooth functions $X \to \bR$, with $\eps_i>0$ and $a_{i-1} + \eps_{i-1} < a_i - \eps_i$ (for each $i$),
\end{enumerate}
such that after restriction to each point of $X$ it satisfies
\begin{enumerate}[(i)]
\setcounter{enumi}{2}
\item $W$ is cylindrical over each interval $[a_i-\eps_i,a_i+\eps_i]$, and for any $t \in \cup_{i=0}^p[a_i-\eps_i, a_i+\eps_i]$  we have that $(W|_t, \ell_W\vert_{W|_t})$ is an object of $\Cob_{\theta}^{\kappa,l}$, and for two such $t < t'$ we have that $(W|_{[t,t']}, \ell_W\vert_{W|_{[t,t']}})$ is a morphism of $\Cob_{\theta}^{\kappa,l}$.
\end{enumerate}
The face map $d_i: \cD_\theta^{\kappa,l} (\bR^N)_p(X) \to \cD_\theta^{\kappa,l} (\bR^N)_{p-1}(X)$ forgets $a_i$ and $\eps_i$.  We write $\cD_{\theta,p}^{\kappa,l}$ for the colimit of these sheaves for $N \to \infty$.
\end{defn}
There is a weak equivalence
\begin{equation}\label{eqn:poset-model-for-ordcobcat}
\norm{\cD_{\theta,\bullet}^{\kappa,l}} \simeq B\Cob_\theta^{\kappa,l} 
\end{equation}
which is proven completely analogously to Proposition 2.14 of \cite{GRW} but adapted to sheaves. Next, we introduce the psc version of $\cD_{\theta,\bullet}^{\kappa,l}$. 

\begin{defn}\label{defn:d-space}
The semi-simplicial sheaf $\cD_\theta^{\kappa,l,\psc} (\bR^N)_\bullet$ has its sheaf of $p$-simplices given as follows. An element of $\cD_\theta^{\kappa,l,\psc} (\bR^N)_p(X)$ is a tuple $(W,\ell_W,a,\eps,g_W)$ where $(W,\ell_W,a,\eps) \in \cD_\theta^{\kappa,l} (\bR^N)_p(X)$ and $g_W$ is a fibrewise Riemannian metric on $W$ such that after restriction to each point of $X$
\begin{enumerate}[(i)]
\item the Riemannian $\theta$-manifold $(W,\ell_W,g)$ is cylindrical over each of the intervals $[a_i-\frac{3}{4}\eps_i,a_i+\frac{3}{4}\eps_i]$ and 
\item the scalar curvature of $g_W$ is positive over $W|_{[a_0,a_p]}$. 
\end{enumerate}
We define $\cD_\theta^{\kappa,l,\rst} (\bR^N)_p(X)  \subset \cD_\theta^{\kappa,l,\psc} (\bR^N)_p(X)$ as the subset of all those elements $(W,\ell_W,a,\eps,g_W)$ such that after restriction to each point of $X$ the psc metrics $g_W|_{[a_i,a_{i+1}]}$ are right stable (for $i=0, \ldots,p-1$). The $i$th face map again forgets $a_i$ and $\epsilon_i$, and we again write $\cD_{\theta,p}^{\kappa,l,\psc}$ and $\cD_{\theta,p}^{\kappa,l,\rst}$ for the colimits of these sheaves as $N \to \infty$.
\end{defn}
 
There are weak equivalences 
\begin{equation}\label{eqn:comparison-pscbobmodelssheaf}
\norm{\cD_{\theta,\bullet}^{\kappa,l, \psc}} \simeq \norm{N_\bullet\mathrm{sh}\PCob_\theta^{\kappa,l}} \quad \text{ and } \quad  \norm{\cD_{\theta,\bullet}^{\kappa,l,\rst}} \simeq \norm{N_\bullet\mathrm{sh}\PCob_\theta^{\kappa,l, \rst}} 
\end{equation}
which are also proven completely analogously to Proposition 2.14 of \cite{GRW}, using that $\Riem^+ (W)^\eps \simeq \Riem^+ (W)$ by Corollary 2.2 of \cite{BERW}. Combining them with the equivalences 
\eqref{eqn:equivalence-cobcat-vs-sheafcobcatpsc} gives equivalences
\begin{equation}\label{eqn:poset-model-for-psccobcat}
\norm{\cD_{\theta,\bullet}^{\kappa,l, \psc}} \simeq B\PCob_\theta^{\kappa,l} \quad \text{ and } \quad  \norm{\cD_{\theta,\bullet}^{\kappa,l,\rst}} \simeq B\PCob_\theta^{\kappa,l, \rst}.
\end{equation}

\subsubsection{Long and infinitesimal collars}

There are two variants of $\cD_{\theta,\bullet}^{\kappa,l}$ (and its psc versions) which are sometimes useful. In the first variant, we prescribe strict lower bounds for the collar lengths $\epsilon_i$. 

\begin{defn}\label{defn:longcollars}
Fix a sequence of real numbers with $0<c_0 \leq c_1 \leq \ldots$. Let $\cD_{\theta,p,c_p}(\bR^N)(X) \subset \cD_{\theta,p}(\bR^N)(X)$ be the subset of those $(W,\ell_W,a,\eps)$ such that $\eps_i(x) \geq c_p$ for all $i \in [p]$ and all $x \in X$. These form the $p$-simplices of a sub-semi-simplicial sheaf $\cD_{\theta,\bullet, c}(\bR^N)$ of $\cD_{\theta,\bullet} (\bR^N)$. 

Similar definitions are made for the colimit $N \to \infty$, for the connectivity conditions and for the decorations $\psc$ and $\rst$.
\end{defn}

\begin{lem}\label{collar-strtching-cobcat}
The inclusion $\cD_{\theta,p,c_p}^{\kappa,l} (\bR^N) \to  \cD_{\theta,p}^{\kappa,l}(\bR^N)$ is a weak equivalence for each $p$, and the same is true with the decorations $\psc$ and $\rst$.
\end{lem}

\begin{proof}
We will use the criterion of Proposition \ref{prop:SurjCrit}, so must show that
$$\cD_{\theta,p,c_p}^{\kappa,l} (\bR^N)[X, A; z] \lra \cD_{\theta,p}^{\kappa,l}(\bR^N)[X, A; \mathrm{inc}(z)]$$
is surjective for all $X$, $A \subset X$ closed, and germs $z$ of $\cD_{\theta,p,c_p}^{\kappa,l} (\bR^N)$ near $A$.

An element of the right-hand side is represented by a tuple $(W, \ell_A, a, \epsilon)$ such that on some open neighbourhood $U \supset A$ we have $\epsilon_i \vert_U \geq c_p$ for all $i$. We may choose a smooth map
\begin{align*}
X \times \bR \times \bR &\lra \bR\\
(x, t, s) &\longmapsto h_{x,t}(s)
\end{align*}
such that each $h_{x,t}$ is a diffeomorphism, and
\begin{enumerate}[(i)]
\item $h_{x,t}(s)=s$ on a neighbourhood of $X \times (-\infty,0] \times \bR$,
\item there is a neighbourhood $U \supset V \supset A$ with $h_{x,t}(s)=s$ for $x \in V$ and all $t$,
\item for each $x \in X$ and each $q\in (0,1)$ the functions 
$$t \mapsto \begin{cases}h_{x,t} (a_i(x)+q\eps_i(x)) - h_{x,t} (a_i(x))\\
h_{x,t} (a_i(x))-h_{x,t} (a_i(x)-q\eps_i(x))
\end{cases}$$
are equal and nondecreasing in $t$,
\item  we have $|h_{x,t} (a_i(x)\pm\eps_i(x))- h_{x,t} (a_i(x))| \geq c_p$ on a neighbourhood of $X \times [1,\infty) \subset X \times \bR$,
\item  we have $h'_{x,t}(s) = 1$ on a neighbourhood of
$$\{(x,t,  s) \in X \times \bR \times \bR \, | \, s= a_i(x)\pm \eps_i(x)\}.$$
\end{enumerate}

Pulling back the data $\xi := (W, \ell_W, a, \epsilon)$ along the smooth map 
$$\psi: (s,v,x,t) \mapsto (h_{x,t}(s),v,x) :  \bR \times \bR^N \times X \times \bR \to  \bR \times \bR^N \times X$$
gives an element $\xi' := (W', \ell_{W'}, a', \epsilon') \in \cD_{\theta,p}^{\kappa,l}(\bR^N)(X \times \bR)$ with $W'= \psi^{-1}(W)$, $\ell_{W'} = \ell_W \circ D\psi$, $a_i'(x,t)= h_{x,t}(a_i(x))$, and $\epsilon'_i(x,t) = h_{x,t} (a_i(x)+q\eps_i(x)) - h_{x,t} (a_i(x))$. This indeed represents an element, by (iii). 

The restriction of $\xi'$ to a neighbourhood of $X \times (-\infty,0]$ is the pullback of $\xi$ by (i), whose restriction to a neighbourhood of $A \times \bR$ is constant by (ii), and whose restriction to a neighbourhood of $X \times [1,\infty)$ is pulled back from a $\xi'' \in \cD_{\theta,p, c_p}^{\kappa,l}(\bR^N)(X)$ by (iv).

If in addition $W$ is equipped with a Riemannian metric $g_W$ which is psc over $[a_0-\eps_0, a_p+\eps_p]$ when we may endow $W'$ with a Riemannian metric $g_{W'}$ which is equal to $\psi^*g_W$ outside of the $[a'_i-\eps'_i, a'_i+\eps'_i]$ and is equal to the evident cylindrical metric on $W'|_{a'_i} \times [a'_i-\eps'_i, a'_i+\eps'_i]$, which indeed fit together by (v). With this choice the above concordance works with the decoration $\psc$. Furthermore, if each $g_W|_{[a_i, a_j]}$ is right-stable so is each $g_{W'}|_{[a'_i, a'_j]}$, as the underlying manifolds are diffeomorphic and under this diffeomorphism the metrics differ by stretching a collar.
\end{proof}

In the second variant of $\cD_{\theta,\bullet}^{\kappa,l}$ the collar lengths are not recorded (but must exist), and the regular values $a_i$ are allowed to coincide.

\begin{defn}\label{defn:inifniteseminalcollars}
Let $\cDs_{\theta,p}^{\psc} (\bR^N)(X)$ be the set of all tuples $(W,\ell,a,g)$ where $(W,\ell) \in \cD_\theta(\bR^N)(X)$, $a=(a_0, \ldots, a_p)$ is a tuple of smooth functions $X \to \bR$ such that $a_0 \leq \ldots \leq a_p$, and $g$ is a fibrewise Riemannian metric on $W$. We require that after restricting to a point of $X$, $(W_x,g_x)$ is cylindrical near each $a_i(x)$ and that $g_x|_{W|_{[a_0(x),a_p(x)]}}$ has positive scalar curvature. There is a version $\cDs_{\theta,p}^{\kappa,l,\psc}(\bR^N)$ with the appropriate connectivity conditions, and a version $\cDs_{\theta,p}^{\kappa,l,\rst}(\bR^N)$ requiring right-stability. The sheaves $\cDs_{\theta,p}^{\kappa,l,\psc}$ and $\cDs_{\theta,p}^{\kappa,l,\rst}$ are obtained by taking the colimit $N \to \infty$. A similar definition is made for $\cDs_{\theta,\bullet}^{\kappa,l}$, without psc metrics. 

This has a semi-simplicial structure by forgetting the $a_i$'s, which extends to a \emph{simplicial} structure by doubling the $a_i$'s.
\end{defn}

\begin{lem}\label{collar-shrinking-cobcat}
The forgetful map $ \cD_{\theta,p}^{\kappa,l}(\bR^N)\to  \cDs_{\theta,p}^{\kappa,l}(\bR^N)$ is a weak equivalence for each $p$, and the same is true with the decorations $\psc$ and $\rst$.
\end{lem}

\begin{proof}
Again, we use the criterion of Proposition \ref{prop:SurjCrit}, so we must show that
$$\cD_{\theta,p}^{\kappa,l}[X, A; z] \lra \cDs_{\theta,p}^{\kappa,l}[X, A;\mathrm{inc}(z)]$$
is surjective (the argument with the decorations $\psc$ or $\rst$ is the same). An element of the target is represented by a family $(W,\ell_W,a)$ over $X$ and a tuple $\eps= (\eps_0, \ldots,\eps_p)$ of smooth functions $\eps_i:U \to (0,\infty)$ defined on an open neighbourhood of $A$. These satisfy the conditions that $a_i(x)+\eps_i (x) < a_{i+1}(x)-\eps_{i+1}(x)$ and that $W_x$ is cylindrical over the intervals $[a_i(x)-\eps_i (x),a_i(x)+\eps_i (x)]$, for all $x \in U$. We construct a concordance which proves that the element lies in the image and this concordance does not change the $\theta$-manifold $(W,\ell_W)$. 

There is a smooth function $\eta$ on $X$ which is positive outside $U$ and vanishes near $A$ such that $W_x$ is cylindrical over $[a_i(x)-\eta(x),a_i(x)+\eta(x)]$, for all $x \in X$. During the first part of the concordance, we change the function $a_i$ linearly to $a_i + i c \eta$, for some small $c>0$. The result is that we can assume that all $a_i$'s are distinct. We can then extend the functions $\eps_i$ (after restricting them to a smaller neighbourhood of $A$) appropriately to find the preimage. 

The cases with the decorations $\psc$ or $\rst$ are entirely analogous.
\end{proof}

\subsubsection{Flexible models}

We wish to establish models analogous to the $\norm{X_\bullet^{\kappa,l}} \simeq \norm{\cD_{\theta, \bullet}^{\kappa,l}}$ of Section 2.8 of \cite{GRW}, but in the context of sheaves and with psc metrics. This makes use of the following weakening of Definition \ref{defn:cylindrical}.

\begin{defn}\label{defn:producttype}
Let $(W, g)$ be a Riemannian manifold and let $x_1 : W \to \bR$ be a smooth function. We say that $g$ is of \emph{product type} over an open subset $J \subset \bR$ with respect to $x_1$ if
\begin{enumerate}[(i)]
\item $J$ consists of regular values of $x_1$, and $\norm{dx_1}\equiv 1$ on $x_1^{-1}(J)$,
\item The Lie derivative $\cL_{\frac{\partial}{\partial x_1}}g$ of $g$ along the vector field dual to $dx_1$ vanishes near $x_1^{-1}(t)$.
\end{enumerate}
\end{defn}

We remark that if $x_1$ is proper, the flow generated by $\frac{\partial}{\partial x_1}$ identifies $x_1^{-1}(J)$ with $J \times x_1^{-1}(t)$ for a small interval $J$ containing $t$. With respect to this identification, the metric $g$ takes the form $ dx_1^2+g_0$ for $g_0:= g|_{x_1^{-1} (t)}$. In this sense, Definition \ref{defn:producttype} is a weakening of Definition \ref{defn:cylindrical}. The following is a psc version of \cite[Definition 2.18]{GRW}, adapted to sheaves.

\begin{defn}\label{defn:x-space}
The semi-simplicial sheaf $X^{\kappa,l, \psc}_\bullet$ has its sheaf of $p$-simplices given as follows. An element of $X^{\kappa,l,\psc}_p(X)$ is a tuple $(a, \epsilon, (W,\ell_W, g_W))$ such that $a=(a_0, \ldots, a_p)$ and $\epsilon=(\epsilon_0, \ldots, \epsilon_p)$ are tuples of smooth maps $X \to \bR$, with $\eps_i>0$ and $a_{i-1} + \eps_{i-1} < a_{i} - \eps_{i}$,
$$W \subset \{(t, v, x) \in \bR \times \bR^N \times X \, | \, a_0(x)-\eps_0(x) < t < a_p(x)+\eps_p(x) \}$$
is a closed subset such that $\pi = \mathrm{pr}_X : W \to X$ is a submersion with $d$-dimensional fibres, $\ell_W : T_v W \to \gamma_\theta$ is a bundle map, and $g_W$ is a fibrewise Riemannian metric on $W$, such that after restriction to each point of $X$ it satisfies
\begin{enumerate}[(i)]

\item $W \subset \bR \times (-1,1)^N$,

\item for each pair of regular values $t_0 < t_1 \in \cup_i (a_i-\epsilon_i, a_i+\epsilon_i)$, the cobordism $(W\vert_{[t_0, t_1]}, W\vert_{t_1})$ is $\kappa$-connected,

\item for each regular value $t \in (a_i-\epsilon_i, a_i+\epsilon_i)$, the map $\ell\vert_t : W\vert_t \to B$, is $(l+1)$-connected,

\item the Riemannian metric $g_{W}$ has positive scalar curvature,
 
\end{enumerate}
and on some open neighbourhood $U$ of each point of $X$ it satisfies
\begin{enumerate}[(i)]
\setcounter{enumi}{4}
\item for each $i \in [p]$ there exist $s_i \in \bR$ and $\delta_i > 0$ such that after restriction to any $x \in U$ we have $a_i-\epsilon_i \leq s_i - \delta_i < s_i + \delta_i \leq a_i+\epsilon_i$ and $(W, g_{W})$ is of product type over $(s_i-\delta_i,s_i+\delta_i)$ with respect to $x_1 : W \to (a_0-\epsilon_0, a_p+\epsilon_p)$.

\end{enumerate}

Define $X^{\kappa,l,\rst}_\bullet$ as the sub-semi-simplicial sheaf of $X^{\kappa,l,\psc}_\bullet$, consisting of those elements which in addition satisfy
\begin{enumerate}[(i)]
\setcounter{enumi}{5}
\item for $i \leq j$ and any $s_i< s_j$ as in (v), the metric $(g_{W})\vert_{[s_i, s_j]}$ is right stable.
\end{enumerate}
\end{defn}

The following is easily proved by adapting the idea of Proposition 2.20 of \cite{GRW} to sheaves. 

\begin{lem}\label{lem:x-space-equivalent-to-d-space}
The natural maps 
\[
\norm{ \cD_\theta^{\kappa,l,\psc} (\bR^N)_{\bullet} } \lra \norm{X^{\kappa,l, \psc}_\bullet} \quad \text{ and } \quad  \norm{ \cD_\theta^{\kappa,l,\rst} (\bR^N)_{\bullet}} \lra \norm{  X^{\kappa,l, \rst}_\bullet}
\]
induced by inclusion are weak homotopy equivalences.
\end{lem}

\section{The fibre theorems}\label{sec:fibre-theorem}

\subsection{Statement of results}
The goal of this section is to identify the homotopy fibres of the maps
\begin{align*}
BF^{2,1} : B\cP_\theta^{2,1} &\lra B\cC_\theta^{2,1}\\
BF^{2,1,\rst} : B\cP_\theta^{2,1,\rst} &\lra B\cC_\theta^{2,1}
\end{align*}
in terms of certain categories, which we call concordance categories, constructed from spaces of psc metrics on a fixed long manifold. The result is given in Theorem \ref{thm:Fibre} below, but first we must define the concordance categories.

\begin{defn}\label{defn:concroidancecateogiry}
Let $(W,\ell) \in \cD_\theta(*)$ and let $J \subset \bR$ be an open subset over which $W$ is cylindrical. Let $\mathcal{P}(W,J)$ be the non-unital category with objects given by the set of pairs $(t, g)$ with $t \in J $ and a psc metric $g \in \Riem^+(W\vert_t)$. A morphism from $(t_0, g_0)$ to $(t_1, g_1)$ can exist only if $t_0 < t_1$ in which case it is given by a psc metric $h \in \Riem^+(W\vert_{[t_0, t_1]})_{g_0, g_1}$. Composition of morphisms is given by gluing. 
We let $\cP(W,J)^{\rst} \subset \cP (W,J)$ be the wide subcategory whose morphisms $(t_0,g_0)\to (t_1,g_1)$ are the right-stable psc metrics $h \in \Riem^+ (W|_{[t_0,t_1]})_{g_0,g_1}$, and define $\cP(W,J)^{\lst} \subset \cP (W,J)$ similarly. 
\end{defn}

These categories do not depend on $\ell$. We have not yet defined a topology on $\cP(W,J)$. To this end, we consider $J$ as a non-unital category via $<$ and equip it with its natural topology. There is a forgetful functor $G : \mathcal{P}(W,J) \to J$ and there is also a functor $K : J \to \Cob_{\theta}$ given by sending $t$ to $(W\vert_t, \ell\vert_{W\vert_t})$ and $t<t'$ to $(W\vert_{[t,t']}, \ell\vert_{W\vert_{[t,t']}})$. Similarly, there is a functor $H: \mathcal{P}(W,J) \to \PCob_{\theta}$ which sends an object $(t, g)$ to $(W\vert_t, \ell\vert_{W\vert_t}, g)$ and a morphism $h \in \Riem^+ (W|_{[t_0,t_1]})_{g_0,g_1}$ to $(W\vert_{[t,t']}, \ell\vert_{W\vert_{[t,t']}}, g\vert_{W\vert_{[t,t']}})$. These functors fit into a commutative diagram
\begin{equation}\label{diagram:defining-topology-on-conccate}
\begin{gathered}
\xymatrix{
\cP(W,J) \ar[r]^-{H} \ar[d]^{G} &\PCob_{\theta} \ar[d]^{F}\\
J \ar[r]^-{K} & \Cob_\theta
}
\end{gathered}
\end{equation}
of categories. The induced squares on sets of objects and morphisms are cartesian, so that \eqref{diagram:defining-topology-on-conccate} is a pullback square of categories. We define the topology on $\cP(W,J)$ so that \eqref{diagram:defining-topology-on-conccate} is a pullback square of topological categories. 

If for each $t_{0}<t_1 \in J$, $(W|_{[t_0,t_1]},\ell|_{W|_{[t_0,t_1]}})$ is a morphism in $\Cob_\theta^{\kappa,l}$, the functor $H$ factors through the subcategory $\PCob_\theta^{\kappa,l}$; similarly for $K$. Observe also that then $H$ restricts to a functor $H^{\rst}: \cP (W,J)^{\rst} \to \PCob_{\theta}^{\kappa,l,\rst}$ on the subcategory of right-stable metrics, and we similarly denote by $G^{\rst}$ the restriction of $G$. 

We call the categories $\cP(W,J)$ and $\cP^{\rst}(W,J)$ \emph{concordance categories}, for the following reason. First note that if $J$ is an interval, then $\cP(W,J)= \cP( \bR \times M,J)$. The set $\pi_0 (B \cP( \bR \times M,\bR))$ is the set of all psc metrics $g \in \Riem^+ (M)$, with $g_0$ and $g_1$ identified if and only if there exists a concordance $h \in \Riem^+ ( [0,1]\times M)_{g_0,g_1}$, so is the set of concordance classes of psc metrics on $M$. For sake of readability, we abbreviate
\begin{equation}\label{eqn:defn:concordancecategory}
\Conc (M):= \cP( \bR\times M,\bR) \quad \text{ and }\quad  \Conc (M)^{\rst}:= \cP( \bR  \times M,\bR)^{\rst}.  
\end{equation}
Now we are ready to state the main results of this section.
\begin{thm}\label{thm:Fibre}
Let $W \in \cD_\theta (*)$ and let $J \subset \bR$ be a disjoint union of finitely many open intervals over which $W$ is cylindrical. Suppose that for each $x < y \in J$ the cobordism $W\vert_{[x,y]}$ defines a morphism in $\Cob_{\theta}^{2,1}$. Assume furthermore that $d \geq 6$. Then the squares
\begin{equation*}
\begin{gathered}
\xymatrix{
B\cP(W,J) \ar[r]^-{BH} \ar[d]^{BG} & B\PCob_{\theta}^{2,1} \ar[d]^{BF^{2,1}}\\
 BJ \ar[r]^-{BK} & B\Cob_{\theta}^{2,1}
}
\end{gathered}
\end{equation*}
and
\begin{equation}\label{diag:mainfibretheorem2}
\begin{gathered}
\xymatrix{
B\cP(W,J)^{\rst} \ar[r]^-{BH^{\rst}} \ar[d]^{BG^{\rst}} & B\PCob_{\theta}^{2,1,\rst} \ar[d]^{BF^{2,1,\rst}}\\
 BJ \ar[r]^-{BK} & B\Cob_{\theta}^{2,1}
}
\end{gathered}
\end{equation}
are homotopy cartesian. 
\end{thm}
If $J \neq \emptyset$, the space $BJ$ is contractible (by a straightforward application of Theorem 6.2 of \cite{GRW} to the augmentation map $N_\bullet J \to *$). Hence Theorem \ref{thm:Fibre} describes the homotopy fibre of $BF^{2,1}$ and $BF^{2,1,\rst}$. Applying it in the situation $W = \bR \times M$ and $J=\bR$ proves Theorem \ref{Main:fibretheorem}. 

The reason we are interested in the category $\PCob_\theta^{2,1,\rst}$ is the following theorem, which relates $B \cP^{\rst}(W,J)$ to actual spaces of psc metrics. For each topological category $\cC$ and objects $a$ and $b$ of $\cC$, there is a tautological map 
\begin{align*}
\tau=\tau_{a,b}:\cC(a,b) &\lra \Omega_{a,b} B \cC \\
f &\longmapsto \left(t \mapsto (f,te_1 + (1-t)e_0)\in N_1 \cC \times \Delta^1 \to B \cC\right)
\end{align*}
to the space of all paths from $a$ to $b$ in $B \cC$. For $(t_0,g_0), (t_1,g_1) \in \Ob (\cP(W,J)^{\rst})$, we have $\cP(W,J)^{\rst} ((t_0,g_0),(t_1,g_1)) = \Riem^+(W\vert_{[t_0,t_1]})_{g_0, g_1}^{\rst}$.

\begin{thm}\label{thm:Concordance}
Let $W \in \cD_\theta(*)$ be cylindrical over the open subset $J \subset \bR$ and let $t_0< t_1 \in J$. Assume that the inclusion $\{t_0,t_1\} \to J$ is $0$-connected. Then for each choice of $g_i \in \Riem^+ (W|_{t_i})$, the tautological map 
\[
\tau: \Riem^+(W\vert_{[t_0,t_1]})_{g_0, g_1}^{\rst} \lra \Omega_{(t_0, g_0),(t_1, g_1)} B\cP(W,J)^{\rst}
\]
is a weak equivalence.
\end{thm}

Theorem \ref{thm:Concordance} is true without any dimension or connectivity hypotheses. It has the consequence that $\Riem^+ (W|_{[t_0,t_1]})^{\rst}_{g_0,g_1}$ has the homotopy type of a loop space, provided that it is non-empty. In particular, Theorem \ref{thm:Concordance} implies that
\begin{equation*}
\Riem^+ (M \times [0,1])^{\rst}_{g,g} \simeq \Omega_{(0,g),(1,g)} B \Conc(M)^{\rst}.
\end{equation*}

\begin{rem}
An $H$-space structure on $\Riem^+ (M \times [0,1])_{g,g}$ is easily explained: one concatenates the psc metrics and scales appropriately; this construction clearly extends to an $E_1$-algebra structure. One can show that $\tau: \Riem^+ (M \times [0,1])_{g,g} \to \Omega_{(0,g),(1,g)} B \Conc(M)$ is a map of $E_1$-algebras. Essentially by definition (together with Theorem \ref{thm:StabMetrics} (ii)), the sub-$E_1$-algebra $\Riem^+ (M \times [0,1])^{\rst}_{g,g}$ is group complete (i.e.\ $\pi_0 (\Riem^+ (M \times [0,1])^\rst_{g,g})$ is a group). 
One can interpret the map $\Riem^+ (M \times [0,1])_{g,g}) \to \Omega_{(0,g),(0,g)} B \Conc (M)$ as a group completion. However, it is unclear to us whether there is the required calculus of fractions to apply the Group Completion Theorem in this situation.
\end{rem}

The next result shows that the forgetful map $BF^{2,1,\rst}$ aligns nicely with the Borel construction given by the action of diffeomorphism groups on spaces of psc metrics. It is the first step towards the proof of Theorem \ref{mainthmintro:diffaction}.

\begin{thm}\label{cor:berw-new-proof}
Let $d \geq 6$ and let $(W : M_0 \leadsto M_1)$ be a morphism in $\Cob_{\theta}^{2,1}$. Then the square
\[
\xymatrix{
\mathcal{M}^\theta_{\psc}(W)^{\rst}_{g_0, g_1} \ar[d]\ar[r]^-{\tau} & \Omega_{(M_0,g_0),(M_1,g_1)} B \PCob_{\theta}^{2,1,\rst} \ar[d]^{\Omega B F^{2,1,\rst}}\\
\mathcal{M}^\theta(W) \ar[r]^-{\tau} & \Omega_{M_0,M_1} B \Cob_{\theta}^{2,1}
}
\]
is homotopy cartesian, where the horizontal maps are the tautological maps.
\end{thm}
To understand the horizontal maps, note that $\mathcal{M}^\theta(W)$ is a subspace of the morphism space $\Cob_\theta^{2,1} (M_0,M_1)$ (and similarly for the version with stable psc metrics). 

\subsection{The concordance category}

We shall prove Theorems \ref{thm:Fibre} and \ref{thm:Concordance} by the topological version of Quillen's Theorems A and B which we established in Section 4 of \cite{SxTech}, and here we mainly verify some of the formal hypotheses of those theorems. 
The proof of (the topological version of) Quillen's Theorems A and B for a continuous functor $F: \cA \to \cB$ between non-unital topological categories uses a certain bi-semi-simplicial space $(F/\cB)_{\bullet,\bullet}$ whose definition we will have to recall, since certain properties of it belong to the hypotheses of those theorems. 

The following bi-semi-simplicial spaces are an instance of Definition 4.1 of \cite{SxTech}. Let $(F/\cB)_{p,q} \subset N_p \cA \times N_{q+1}\cB$ be the subspace of all elements of the form $(a_0 \to \cdots \to a_p, F(a_p) \to b_0 \to \cdots \to b_q)$. There is an augmentation map
\begin{align*}
 \eta_{p,q}^F: (F/\cB)_{p,q} &\lra N_q \cB\\
 (a_0 \to \cdots \to a_p, F(a_p) \to b_0 \to \cdots \to b_q) &\longmapsto (b_0 \to \cdots \to b_q).
\end{align*}
Dually, we let $(\cB/F)_{p,q} \subset N_p \cA \times N_{q+1}\cB$ be the space of all elements of the form $((a_0 \to \cdots \to a_p,  b_0 \to \cdots \to b_q \to F(a_0))$, with augmentation
\begin{align*}
\zeta_{p,q}^F: (\cB/F)_{p,q} &\lra N_q \cB\\
 (a_0 \to \cdots \to a_p,  b_0 \to \cdots \to b_q \to F(a_0)) &\longmapsto (b_0 \to \cdots \to b_q).
\end{align*}

The terms in following lemma were recalled in Definition \ref{def:SxTechRecollections}.

\begin{lem}\label{lem:concordance-category-formal-properties} In the setting of Definition \ref{defn:concroidancecateogiry}:
\begin{enumerate}[(i)]
\item The categories $J$ and $\cP(W,J)$ are right and left fibrant.
\item The categories $J$ and $\cP(W,J)$ have soft right and left units.
\item The forgetful functor $G:\cP(W,J) \to J$ induces fibrations $G_0: \Ob (\cP(W,J)) \to \Ob (J)$ and $G_1: \Mor (\cP(W,J)) \to \Mor (J)$.
\item The map $\eta_{0,0}^G$ is a fibration. 
\item The same statements are true for the subcategories $\cP(W,J)^{\rst}$ and $\cP(W,J)^{\lst}$.
\end{enumerate}
\end{lem}

\begin{proof}
Part (i): we prove that both categories are right fibrant, i.e. that the target maps $d_0$ are fibrations, left fibrancy is proven by the same argument. It is clear that the target map $\Mor (J) \to \Ob (J)$ is a fibration: it is even a locally trivial fibre bundle, because $J \subset \bR$ is open. To treat the category $\cP(W,J)$, consider a lifting problem 
\[
 \xymatrix{
 D^n \times \{0\} \ar[d] \ar[r]^-{a} & \Mor(\cP (W,J)) \ar[d]^{d_0}\\
 D^n \times [0,1] \ar[r]^-{b} \ar@{..>}[ur]^{c} & \Ob (\cP(W,J)).
 }
\]
The map $b$ determines a continuous function $f_1: D^n \times [0,1] \to J$, and the lift $a$ of $b|_{D^n \times \{0\}}$ determines a function $f_0': D^n \to J$ with $f_0' < f_1|_{D^n \times \{0\}}$. Extend $f_0' $ to a continuous function $f_0: D^n \times [0,1] \to J$ with $f_0 < f_1$. Now the projection
\[
\{  (v,t,p) \in D^n  \times [0,1] \times W \,\,\vert\,\,  f_0 (v,t) \leq x_1 (p) \leq f_1 (v,t)\}\lra D^n \times [0,1]
\]
is a smooth fibre bundle, and hence trivial since $D^n \times [0,1]$ is contractible. The psc data contained in the maps $a$ and $b$ determine (compatible) psc metrics on the restriction of the bundle to $D^n \times \{0\}$ and on the upper boundary bundle 
$$\{  (v,t,p) \in D^n  \times [0,1] \times W \,\,\vert\,\,   x_1 (p) = f_1 (v,t)\}.$$
An application of Theorem \ref{thm:improved-chernysh-theorem} shows that we can extend these psc metrics to the whole bundle, which provides the solution to the lifting problem.

Part (ii): by part (i) and Lemma 3.14 of \cite{SxTech}, it is enough to prove that the categories $J$ and $\cP(W,J)$ have weak units. We will show that they have weak left units; the other case is analogous. If $t \in \Ob (J)$, choose $\eps>0$ such that $[t,t+\eps] \subset J$. Then the morphism $t<t+\eps$ is a weak left unit in $J$. If $(t,g) \in \Ob (\cP(W,J))$, the psc metric $dx_1^2+g$ on $W|_{[t,t+\eps] }$ is a morphism in $\cP(W,J)$ and a weak left unit in $\cP(W,J)$ by Corollary 2.2 (ii) of \cite{BERW}. 

Part (iii): clear from the definition of the topology on $\cP(W,J)$ and from Proposition \ref{prop:Fibrant}. 

Part (iv): this follows from parts (i), (iii) and Lemma 4.6 (iii) of \cite{SxTech}.

Part (v): observe that by Lemma \ref{lem:stability-homotopy-invariant}, $\cP(W,J)^{\rst} \subset \cP(W,J)$ is a clopen subcategory and contains the weak units. Therefore, the proofs just given carry over to the case of $\cP(W,J)^{\rst}$ without change, and the same remark applies to $\cP(W,J)^{\lst}$.
\end{proof}

\begin{lem}\label{lem:VaryingJ}
Let $\emptyset \neq J \subset I \subset \bR$ be finite disjoint unions of open intervals over which $W$ is cylindrical. Assume that for all $t_0 < t_1 \in I$, the cobordism $W|_{[t_0,t_1]}$ is a morphism in $\Cob_\theta^{2,1}$ and assume that $d \geq 6$. Then the inclusions $\iota : \cP(W, J) \to \cP(W, I)$ and $\iota^{\rst} : \cP(W, J)^{\rst} \to \cP(W, I)^{\rst}$ induce weak equivalences on classifying spaces.
\end{lem}

\begin{rem}
Lemma \ref{lem:VaryingJ} is the only place in the proof of Theorem \ref{thm:Fibre} where the existence result for right stable metrics (Theorem \ref{thm:StabMetrics}) is used, through Theorem \ref{prop:right-stable-on-2.1.cob-stable}. This is also responsible for the hypothesis $d \geq 6$ and for the connectivity assumption on $W$. These hypotheses are used in the proofs of \emph{both} conclusions of the lemma. 
\end{rem}

\begin{proof}[Proof of Lemma \ref{lem:VaryingJ}]
We give the details for the case of $\iota$; the case of $\iota^{\rst}$ will be analogous. We shall apply the version of Quillen's Theorem A which we have established in Theorem 4.7 of \cite{SxTech}, but to start we make some elementary simplifications. Firstly, it is enough to discuss the case where $J$ is a single interval. Secondly, an inclusion $J \subset I$ of sets as in the statement, with $J$ an interval, may be written as a composition of maps of the following type:
\begin{enumerate}[(i)]
\item $J \to I$ is an isotopy equivalence.
\item $J$ is clopen in $I$ and if $s \in J$ and $t \in I\setminus J$, then $s<t$.
\item $J$ is clopen in $I$ and if $s \in J$ and $t \in I\setminus J$, then $s>t$.
\end{enumerate}
Hence it is enough to prove the lemma in these three cases: we prove the first case and the second case, the third is done by using the dual form of Quillen's Theorem A, i.e.\ Theorem 4.8 of \cite{SxTech}.

For the first case, if $J \subset I$ is an isotopy equivalence then $N_p \cP(W,J) \to N_p \cP(W,I)$ is a weak equivalence for each $p \geq 0$ and hence $B \cP (W,J) \to B \cP (W,J)$ is a weak equivalence. 

For the second case, choose $t_0 \in J$, and write $J_0 := J \cap (-\infty,t_0)$ and $I_0 := J_0 \cup (I \setminus J)$. The inclusions $J_0 \to J$ and $I_0 \to I$ are isotopy equivalences, which means that it is enough to prove that $B\cP(W,J_0) \to B\cP(W,I_0)$ is a weak equivalence. What we have achieved in this step is to ensure that $W$ is cylindrical near $t_0= \sup (J_0)$. 

We claim that the inclusion $\iota:\cP(W,J_0) \to  \cP(W,I_0)$ satisfies the hypotheses (i) and (iv) of Theorem 4.7 of \cite{SxTech}, which will complete the proof. Both categories have soft units and are left and right fibrant, by Lemma \ref{lem:concordance-category-formal-properties} (i) and (ii). 
Since $J_0$ is clopen in $I_0$, the functor $\iota$ is a fibration on both object and on morphism spaces. It follows from Lemma 4.6 (iii) of \cite{SxTech} that the map $\eta_{0,0}^\iota$ is a fibration. Together these verify hypothesis (iv) of Theorem 4.7 of \cite{SxTech}. 

It remains to verify hypothesis (i), i.e.\ that the fibre categories $\iota / (t,g)$ have contractible classifying spaces for all $(t,g) \in \Ob (\cP(W,I_0))$. Suppose first that $t \in J_0$. Then $\iota /(t,g)$ is the over-category $\cP(W,J_0)/(t,g)$, which has a contractible classifying space since $\cP(W,J_0)$ has soft left units by Lemma \ref{lem:concordance-category-formal-properties} (ii). Now suppose that $t  \in I_0 \setminus J_0$. The cobordism $W|_{[t_0,t]}: W|_{t_0} \leadsto W|_t$ is a morphism in $\Cob_\theta^{2,1}$ by assumption. Thus, by Theorem \ref{prop:right-stable-on-2.1.cob-stable}, there is a psc metric $g_0 \in \Riem^+ (W|_{t_0})$ and a \emph{stable} psc metric $h \in \Riem^+ (W|_{[t_0,t]})_{g_0,g}$ (it is here where we are using the hypothesis that $d \geq 6$). The composition $h \circ -$ gives a functor $\cP(W,J)/(t_0,g_0) \to \iota /(t,g)$, which as $h$ is stable induces a levelwise weak equivalence on nerves and hence a weak equivalence on classifying spaces. Therefore, since the 
space $B (\cP(W,J)/(t_0,g_0))$ is contractible (by Lemma \ref{lem:concordance-category-formal-properties} (ii) again), it follows that $B(\iota/(t,g))\simeq *$, as claimed. This finishes the proof for $\iota$, and the same argument applies to $\iota^{\rst}$.
\end{proof}

The next goal is the proof of Theorem \ref{thm:Concordance}. The proof is a version of a well-known delooping argument, adapted to the technical details of the situation. 

\begin{proof}[Proof of Theorem \ref{thm:Concordance}]
Write $x_i:=(t_i,g_i)$. Let $I := J_{<t_1}$, let $\iota^{\rst}: \cP(W,I)^{\rst} \to \cP(W,J)^{\rst}$ be the inclusion functor, and let $H: \iota^{\rst}/ x_1 \to \cP(W,I)^{\rst}$ be the source functor. There is a commutative square
\begin{equation}\label{diag:proofstable-psc-space-is-loopspace}
\begin{gathered}
\xymatrix{
B(\iota^{\rst}/ x_1) \ar[r]^-{K'} \ar[d]^{BH} & P_{x_1} B \cP(W,J)^{\rst} \ar[d]^{\ev_0}\\
B \cP(W,I)^{\rst} \ar[r]^-{B\iota^{\rst}} & B \cP(W,J)^{\rst}.
} 
\end{gathered}
\end{equation}
Here $\Path_{x_1}B \cP(W,J)^{\rst}$ denotes the space of paths in $B \cP(W,J)^{\rst}$ which end at $x_1$, and $\ev_0$ is the map that evaluates at $0$. The map $K'$ is adjoint to the homotopy $K: [0,1] \times B (\iota^{\rst}/x_1) \to B \cP(W,J)^{\rst}$ coming from the natural transformation $\kappa: \iota^{\rst} \circ H \Rightarrow {x_1}$ to the constant functor $x_1$. 
It is clear that $BH^{-1}(x_0) = \Riem^+ (W_{[t_0,t_1]})_{g_0,g_1}^{\rst} $, and the restriction of $K'$ to $BH^{-1}(x_0)\to \ev_0^{-1}(x_0) = \Omega_{x_0,x_1} B \cP(W,J)^{\rst}$ is the map $\tau$, which we want to show is a weak equivalence.

The space $B(\iota^{\rst}/ x_1)$ is contractible since $\cP(W,J)^{\rst}$ has soft left units by Lemma \ref{lem:concordance-category-formal-properties}. Since $\{t_0,t_1\} \to J$ is $0$-connected, the inclusion $I \to J$ is an isotopy equivalence. As in the proof of Lemma \ref{lem:VaryingJ}, it follows that $B\iota^{\rst}$ is a weak equivalence. Because the path space is contractible as well, \eqref{diag:proofstable-psc-space-is-loopspace} is homotopy cartesian. Since $\ev_0$ is a fibration, the claim will follow once we can prove that $BH$ is a quasifibration. The map $BH$ is the geometric realisation of the semi-simplicial map $N_\bullet H$, and we shall prove that each $N_p H$ is a fibration and $N_\bullet H$ is homotopy cartesian in the sense of Definition 2.9 of \cite{SxTech}. It then follows by applying the Dold--Thom criterion \cite[Satz 2.2, Hilfssatz 2.10 and Satz 2.12]{DoldThom} (a convenient reference is \cite[Lemma 4.K.3]{MR1867354}) that $BH$ is a quasifibration. This will finish the proof.

The map $N_p H: N_p (\iota^{\rst}/ x_1) \to N_p (\cP(W,I)^{\rst})$ is given by 
\[
(y_0 \stackrel{h_1}{\to} \cdots \stackrel{h_p}{\to} y_p, y_p \stackrel{h}{\to} x_1 ) \longmapsto (y_0 \stackrel{h_1}{\to} \cdots \stackrel{h_p}{\to} y_p)
\]
(since $\iota^{\rst}$ is an inclusion, we write $y_p$ for $\iota^{\rst}(y_p)$, slightly abusing notation). Because $\cP(W,I)^{\rst}$ is left fibrant by Lemma \ref{lem:concordance-category-formal-properties}, the map $N_p H$ is a fibration. The fibre over $(y_0 \stackrel{h_1}{\to} \cdots \stackrel{h_p}{\to} y_p)$ is, writing $y_i=(s_i,k_i)$, just $\Riem^+ (W_{[s_p,t_1]})_{k_p,g_1}^{\rst}$. The maps between the fibres induced by the face maps are either the identity, or they are given by $\mu(h_p, \_)$ with $h_p \in \Riem^+ (W_{[s_{p-1},s_p]})_{k_{p-1},k_p}^{\rst}$. By the definition of right stability, these maps are weak equivalences. Hence $N_\bullet H$ is homotopy cartesian. 
\end{proof}

\subsection{Proof of Theorems \ref{thm:Fibre} and \ref{cor:berw-new-proof}}

We shall apply the version of Quillen's Theorem B which we have established as Theorem 4.9 of \cite{SxTech}. We will only write down the proof for the first case of Theorem \ref{thm:Fibre}, the proof for the second case is entirely analogous, with decoration ${}^{\rst}$ everywhere. In order to prove the first case of Theorem \ref{thm:Fibre} have to verify the following hypotheses:
\begin{enumerate}[(i)]
\item $J$ and $\Cob_\theta^{2,1}$ are left fibrant and have soft right units (this has been done in Lemma \ref{lem:concordance-category-formal-properties}, Proposition \ref{prop:Fibrant} and Lemma \ref{lem:soft-units}).
\item $\cP(W,J)$ and $\PCob_\theta^{2,1}$ are right fibrant (this has been done in Lemma \ref{lem:concordance-category-formal-properties} and Proposition \ref{prop:Fibrant}). The maps 
$$\eta_{0,0}^G: (G/J)_{0,0} \to \Ob (J) \quad\text{ and }\quad \eta_{0,0}^{F^{2,1}}: (F^{2,1}/\Cob_\theta^{2,1})_{0,0} \to \Ob (\Cob_\theta^{2,1})$$
are fibrations. This follows from the fact that the maps 
$$G_0 : \Ob(\cP(W,J)) \lra \Ob(J) \quad\text{ and }\quad F_0^{2,1} : \Ob(\PCob_{\theta}^{2,1}) \lra \Ob(\Cob_{\theta}^{2,1})$$
are fibrations by Lemma \ref{lem:concordance-category-formal-properties} and Proposition \ref{prop:Fibrant}), and the fact that $J$ and $\Cob_{\theta}^{2,1}$ are right fibrant by Lemma \ref{lem:concordance-category-formal-properties} (i) and Corollary \ref{psccob-fibratnt}, using Lemma 4.6 (iii) of \cite{SxTech}.
\item For each morphism $(t,W): M \leadsto N$ in $\Cob_\theta^{2,1}$, the fibre transition $W \circ - :F^{2,1}/M \to F^{2,1}/N$ induces a weak equivalence on classifying spaces.
\item For each $t \in J$, the functor $G/t \to F^{2,1}/K(t)$ induced by $K$ and $H$ induces a weak equivalence on classifying spaces.
\end{enumerate}
It remains to verify hypotheses (iii) and (iv). To prepare for that, we first describe the homotopy types of the fibre categories $F^{2,1} /M$ in terms of concordance categories, as follows. Let $M \in \Ob(\Cob_{\theta}^{2,1})$. For $a < b \in \bR$ there is a functor
\begin{equation}\label{eq:Comparison}
L : \cP(\bR \times M, (a,b)) \lra F^{2,1}/M
\end{equation}
given on objects by (ignoring the tangential structures in the notation) 
$$L(t, g) := ((M,g), [0,b-t] \times M :  M \leadsto   M)$$
and on morphisms by 
$$L(t_0<t_1, h ) := ([0, t_1-t_0] \times M, h) : L(t_0, g_0) \to L(t_1, g_1).$$
It restricts to a functor 
\begin{equation}\label{eq:Comparison2}
L^{\rst} : \cP(\bR \times M, (a,b))^{\rst} \lra F^{2,1,\rst}/M
\end{equation}
on the subcategories of right stable psc metrics. 

\begin{lem}\label{lem:OverCatConc}
The functors \eqref{eq:Comparison} and \eqref{eq:Comparison2} induce weak equivalences on classifying spaces.
\end{lem}
\begin{proof}
We present the proof for $L$; the argument for $L^{\rst}$ is completely analogous. We shall apply our dual form of Quillen's Theorem A (Theorem 4.8 of \cite{SxTech}). 

Hypothesis (i) of Theorem 4.8 of \cite{SxTech} is that the fibre category $x / L$ has contractible classifying space, for each object $x$ of $F^{2,1}/M$. Let $x=((t, W) : (N, g) \leadsto M) \in \Ob(F^{2,1}/M)$. There is a largest $\eps>0$ such that $W$ is cylindrical over $(t-\epsilon, t]$, and for some $\delta>0$, $W$ is cylindrical over $[0, \delta]$. Consider the elongation $\hat{W}: = ( (-\infty,0]  \times N)\cup W \cup ( [t, \infty) \times M)$. Let $J = (t-\eps,t)$ and $I = J \cup (-\infty, \delta)$. The fibre category $x / L$ is isomorphic to the fibre category $(0,g) / \iota$ of the functor $\iota:\cP (\hat{W}, J) \to \cP (\hat{W},I)$ over the object $(0,g)$. It was shown in the proof of Lemma \ref{lem:VaryingJ} that $(0,g) / \iota$ has contractible classifying space. Hence so does $x / L$. 

We will verify hypotheses (ii) and (iii) of Theorem 4.8 of \cite{SxTech} by establishing the conditions of (iv) of that theorem: $\cP(\bR \times M, (a,b))$ is left fibrant,  $F^{2,1}/M$ is right fibrant and has soft left units, and $\zeta_{0,0}^L : ((F^{2,1}/M)/L)_{0,0} \to \Ob (F^{2,1}/M)$ is a fibration.

The category $\cP(\bR \times M, (a,b))$ is left fibrant by Lemma \ref{lem:concordance-category-formal-properties} (ii).

To see that the category  $F^{2,1}/M$ is right fibrant, note that there is a cartesian square
\begin{equation*}
\xymatrix{
\Mor(F^{2,1}/M) \ar[d]^-{t} \ar[r] & \Mor(\PCob_{\theta}^{2,1}) \ar[d]^-{t}\\
\Ob(F^{2,1}/M) \ar[r] & \Ob(\PCob_{\theta}^{2,1})
}
\end{equation*}
given by the source functor $F^{2,1}/M \to \PCob_{\theta}^{2,1}$. The right vertical map is a fibration by Proposition \ref{prop:Fibrant}, and hence the target map $t:\Mor (F^{2,1}/M) \to \Ob (F^{2,1}/M)$ is a fibration. 

As $F^{2,1}/M$ is right fibrant by the above, to see that it has soft right units it is enough, by Lemma 3.14 of \cite{SxTech}, to show that it has weak left units. A weak left unit for an object $((t,W) : (N, g) \leadsto M) \in \Ob(F^{2,1}/M)$ may be given by choosing a small $\epsilon>0$ so that $W$ is cylindrical over $[0,\epsilon]$, in which case
$$(t,W) = (t-\epsilon, W\vert_{[\epsilon, t]}-\epsilon \cdot e_1) \circ (\epsilon, [0,\epsilon] \times N)$$
and $(\epsilon, [0,\epsilon] \times N)$ may be given the metric $ dx_1^2+g$. This gives a morphism
$$(\epsilon, [0,\epsilon] \times N,dx_1^2+g) : (t,W) \lra (t-\epsilon, W\vert_{[\epsilon, t]}-\epsilon \cdot e_1)$$
in $F^{2,1}/M$, which is stable by Corollary 2.2.2 of \cite{BERW} and so induces a weak equivalence on morphism spaces. Together with Lemma 3.14 of \cite{SxTech}, the last two facts show that $F^{2,1}/M$ has soft left units.

The map $\zeta_{0,0}^L : ((F^{2,1}/M)/L)_{0,0} \to \Ob (F^{2,1}/M)$ is \emph{not} a Serre fibration (which would be the hypothesis of Theorem 4.8 of \cite{SxTech} as stated). To understand the problem, note that $((F^{2,1}/M)/L)_{0,0}$ is the space of all $(t,g',W,h,s)$, where $ t \in (a,b)$, $g' \in \Riem^+ (M)$, $(W,s)$ is a $\theta$-cobordism $N \leadsto M$ for some $N$, $s >b-t$, and $W$ is cylindrical over $[s-(b-t),s]$. Finally, $h$ is a psc metric on $W|_{[0,s-(b-t)]}$ which equals $g'$ on the boundary $W|_{s-(b-t)}=M$. The map $\zeta^L_{0,0}:  ((F^{2,1}/M)/L)_{0,0} \to \Ob(F^{2,1}/M)$ sends such a point to $(W,s,g)$, where $g= h|_{N}$. Figure \ref{fig:4} explains why this is not a Serre fibration.

\begin{figure}[h]
\begin{center}
\includegraphics{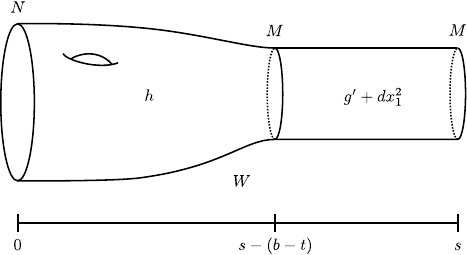}
\caption{A point in $((F^{2,1}/M)/L)_{0,0}$, which is sent by $\zeta^L_{0,0}$ to $(W,s, h\vert_N) \in \Ob(F^{2,1}/M)$. A path in $\Ob(F^{2,1}/M)$ starting at this point which immediately deforms the middle of $W\vert_{[s-(b-t),s]}$ to be non-cylindrical cannot be lifted, as a lift would have to vary $t$ continuously and hence must remain cylindrical over a slightly smaller interval.}\label{fig:4}
\end{center}
\end{figure}

Fortunately, by Remark 3.5 of \cite{SxTech}, the hypothesis that $\zeta_{0,0}^L$ is a Serre fibration can be replaced by the weaker assumption that it is a Dold--Serre fibration, which we claim to be the case. In other words, given a lifting problem
\begin{equation}\label{eq:liftingproblemdoldserre}
\begin{gathered}
\xymatrix{
D^{n} \times \{0\}  \ar[r]^-{G} \ar[d] & ((F^{2,1}/M)/L)_{0,0} \ar[d]^-{\zeta_{0,0}^L} \\
D^n \times [0,1] \ar[r]^-{H} \ar@{..>}[ur] & \Ob(F^{2,1}/M),
}
\end{gathered}
\end{equation}
we can find a vertical homotopy of $G$ to a map $G'$ for which the lifting problem can be solved. The bottom map $H$ of the diagram \eqref{eq:liftingproblemdoldserre} is given by a family over $D^n \times [0,1]$ of cobordisms $W(x,x')$ from manifolds $N(x,x')$ to $M$, of length $s(x,x')$, together with psc metrics $g(x,x')$ on $N(x,x')$. The top map $G$ is given by a function $t:D^{n} \to (a,b)$, psc metrics $h(x,0)$ on $W(x,0)|_{[0,s(x,0)-(b-t(x,0))]}$ which coincide with $g(x,0)$ on $N(x,0)$ (for all $x \in D^{n}$). There exists $\epsilon>0$ such that $W(x,x')|_{[s(x,x')-\epsilon,s(x,x')]}$ is cylindrical for each $(x,x') \in D^n \times [0,1]$. The vertical homotopy of $G$ adjusts the map $G$ until the function $t$ is constant and satisfies $t(x) \geq b-\epsilon$. 
The manifolds $W(x,0)$, the numbers $s(x,0)$ and the psc metrics $g(x,0)$ remain unchanged, and a suitable deformation of $h$ is done by pulling back with a collar-stretching isotopy of $W(x,0)$. After this preliminary homotopy, the lifting problem \eqref{eq:liftingproblemdoldserre} can be solved, by Theorem \ref{thm:improved-chernysh-theorem}. 
\end{proof}

\begin{proof}[Completion of the proof of Theorem \ref{thm:Fibre}]
To prove Theorem \ref{thm:Fibre} it remains to verify hypotheses (iii) and (iv) of Theorem 4.9 of \cite{SxTech}. For hypothesis (iii), let a cobordism $((t,W) : M_0 \leadsto M_1) \in \Mor(\Cob_{\theta}^{2,1})$ be given, having a collar of width $\epsilon$, and form the elongation
$$\hat{W} = (-\infty,0] \times M_0) \cup W \cup ([t,\infty) \times M_1) \subset \bR \times I^{\infty-1}.$$
We may form the diagram
\begin{equation}\label{diag:proof-fibretheorem}
\begin{gathered}
\xymatrix{
\cP(\hat{W}, (-\epsilon,0)) \ar[r] \ar[d]^-L & \cP(\hat{W}, (-\epsilon,0) \sqcup (t-\epsilon,t)) \ar[d]^-{L'} & \cP(\hat{W}, (t-\epsilon, t)) \ar[l] \ar[ld]^-L\\
F^{2,1}/M_0 \ar[r]^-{W \circ - } & F^{2,1}/M_1
}
\end{gathered}
\end{equation}
where the unnamed functors are inclusions, which each induce equivalences on classifying spaces by Lemma \ref{lem:VaryingJ}. The functors marked $L$ are those of Lemma \ref{lem:OverCatConc}, where we use that $\hat{W}\vert_{(-\epsilon,0)} = (-\epsilon,0) \times M_0 \subset \bR \times M_0$ and $\hat{W}\vert_{(t-\epsilon, t)} = (t-\epsilon, t) \times M_1 \subset \bR \times M_1$, which by that lemma both induce weak equivalences on classifying spaces. The functor $L'$ is defined similarly: on objects and some morphisms, its definition is forced by the commutativity of the diagram. The remaining morphisms are of the form $(-\epsilon < a < 0 < t-\epsilon < b < t, h \in \Riem^+(\hat{W}\vert_{[a,b]}))$, and such a morphism is sent to
$$(\hat{W}\vert_{[a,b]}, h):((M_0, h\vert_a), W \circ ([0,-a] \times M_0)) \lra ((M_1, h\vert_b), [0,t-b] \times M_1).$$
With this definition Diagram \eqref{diag:proof-fibretheorem} commutes, and so $W \circ - : F^{2,1}/M_0 \to F^{2,1}/M_1$ induces a weak equivalence on classifying spaces.

For hypothesis (iv), we must show that the functor $H : G/t \to F/K(t)$ induces an equivalence on classifying spaces. The manifold $W$ is cylindrical near $t$, and $J$ is open, so we may suppose it is cylindrical near some open interval $t \in (a,b) \subset J$, where it agrees with $\bR \times K(t)$. The category $G/t$ is isomorphic to $\cP(W, J_{<t})$, and we have a commutative diagram
\begin{equation*}
\xymatrix{
\cP(W, (a,t)) \ar@{=}[r] \ar[d] & \cP(\bR \times K(t), (a,t)) \ar[dd]^-L\\
\cP(W, J_{<t}) \ar[d]^{\cong}\\
G/t \ar[r] & F^{2,1}/K(t).
}
\end{equation*}
The right-hand map is an equivalence on classifying spaces by Lemma \ref{lem:OverCatConc}. The top left-hand map is an equivalence by Lemma \ref{lem:VaryingJ} and hence so is the bottom map, as desired. This finishes the proof of Theorem \ref{thm:Fibre} in the first case. The second case is analogous, using the other cases of Lemma \ref{lem:OverCatConc} and Lemma \ref{lem:VaryingJ}.
\end{proof}

\begin{proof}[Proof of Theorem \ref{cor:berw-new-proof}]
A point in $\mathcal{M}^\theta(W)$ is a $\theta$-cobordism $(V,\ell):M_0 \leadsto M_1$ which is diffeomorphic (relative to the boundary) to $W$. We extend the diagram of Theorem \ref{cor:berw-new-proof} as follows:
\begin{equation}\label{diag:berw-new-proof-proof}
\begin{gathered}
 \xymatrix{
\Riem^+ (V)_{g_0,g_1}^{\rst} \ar[r] \ar[d] &\mathcal{M}^\theta_{\psc}(W)^{\rst}_{g_0, g_1} \ar[d]^{\pi} \ar[r]^-{\tau} & \Omega_{(M_0,g_0),(M_1,g_1)} B \PCob_{\theta}^{2,1,\rst} \ar[d]^{\Omega B F^{2,1,\rst}}\\
\ast \ar[r]^{(V,\ell)} & \mathcal{M}^\theta(W) \ar[r]^{\tau} & \Omega_{M_0,M_1} B \Cob_{\theta}^{2,1}.
 }
\end{gathered}
\end{equation}
The left square is homotopy cartesian, and if we can prove that the outer rectangle is homotopy cartesian, it follows that the right square induces a weak equivalence $\hofib_{(V,\ell)} (\pi) \to \hofib_{\tau (V,\ell)} (\Omega BF^{2,1,\rst})$. If we do so for each point $(V,\ell) \in \mathcal{M}^\theta(W)$, then the right square is homotopy cartesian as required. 

The elongation $\hat{V}:= ((-\infty ,0] \times M_0) \cup V \cup ([1,\infty ) \times M_1)$ is cylindrical over $J:=(-\infty,\epsilon) \cup (1-\epsilon,\infty)$ for some $\epsilon$. On taking path spaces \eqref{diag:mainfibretheorem2} gives the right square of the diagram
\begin{equation}\label{diag:berw-new-proof-proof1}
\begin{gathered}
\xymatrix{
\Riem^+ (V)_{g_0,g_1}^{\rst} \ar[r]^-{\tau} \ar[d] & \Omega_{(0,g_0),(1,g_1)} B\cP(\hat{V},J)^{\rst} \ar[r]^-{\Omega BH^{ \rst}} \ar[d]^{\Omega BG^{\rst}} &  \Omega_{(M_0,g_0),(M_1,g_1)}B\PCob_{\theta}^{2,1,\rst} \ar[d]^{\Omega BF^{2,1,\rst}}\\
\ast \ar[r]^{j} & \Omega_{0,1} BJ \ar[r]^-{\Omega BK} & \Omega_{M_0,M_1} B\Cob_{\theta}^{2,1}.
}
\end{gathered}
\end{equation}
The map $\tau$ is the tautological map from Theorem \ref{thm:Concordance} and hence a weak equivalence by that theorem. The map $j$ sends the point to the path $(t \mapsto ((0<1), ((1-t)e_0,te_1)) \in N_1 J \times \Delta^1 \subset BJ$. The left square commutes by inspection and is hence homotopy cartesian (because $BJ \simeq *$). The right square of \eqref{diag:berw-new-proof-proof1} is homotopy cartesian by Theorem \ref{thm:Fibre}. Hence the outer rectangle of \eqref{diag:berw-new-proof-proof1} is homotopy cartesian. On the other hand, it is straightforward to verify that the outer rectangles of diagrams \eqref{diag:berw-new-proof-proof} and \eqref{diag:berw-new-proof-proof1} agree. Hence \eqref{diag:berw-new-proof-proof} is homotopy cartesian, as desired. 
\end{proof}

\section{The surgery theorem for psc cobordism categories}\label{sec:loopinf-structure}

\subsection{Statement of the result}

In this section we will show how the ``surgery on objects below the middle dimension" technique of \cite{GRW} may be adapted to prove that the inclusion maps $ B\PCob_{\theta}^{2,1, \rst} \to B\PCob_\theta^{2, \rst}$ and $ B\PCob_{\theta}^{2,1} \to B\PCob_\theta^{2}$ are weak homotopy equivalences, just as the map $B\Cob_{\theta}^{2,1} \to B\Cob_\theta^{2}$ was shown to be in that paper. More precisely, we wish to prove the analogue of Theorem 4.1 of \cite{GRW} for both versions (right stable or not) of the psc cobordism categories. At one point we depart slightly from \cite{GRW}: namely, in that paper, the cobordism categories $\Cob_{\theta}^{\kappa,l}$ are not considered at all, but only a version $\Cob_{\theta,L}^{\kappa,l}$ where all objects are required to contain a fixed copy of a certain $(d-1)$-manifold-with-boundary $L$ and all morphisms contain a copy of $L \times [0,1]$. In this paper, we are interested in the case $L=\emptyset$, but unfortunately Theorem 4.1 of \cite{GRW} is stated in a way that does not apply to that case, since it requires the structure map $\ell_L : L\to B$ to be $(l+1)$-connected. We will replace this requirement with a finiteness property of $B$, item (iv) in the following theorem (item (v) is specific to the context of psc metrics and comes from the codimension restriction for Gromov--Lawson surgery).

\begin{thm}\label{thm:SurgBelowMid}
Suppose that the following are satisfied.
\begin{enumerate}[(i)]
\item $2(l+1) < d$,
\item $l \leq \kappa$,
\item $l \leq d-\kappa-2$,
\item the space $B$ is of type ($F_{l+1}$),
\item $d-l-1 \geq 3$.
\end{enumerate}
Then the maps
$$B\PCob_{\theta}^{\kappa,l} \lra B\PCob_{\theta}^{\kappa,l-1} \quad\text{ and }\quad B\PCob_{\theta}^{\kappa,l, \rst} \lra B\PCob_{\theta}^{\kappa,l-1, \rst}$$
are weak homotopy equivalences.
\end{thm}

\subsection{Proof of Theorem \ref{thm:SurgBelowMid}}

The proof of Theorem \ref{thm:SurgBelowMid} follows the same strate\-gy as the proof of Theorem 4.1 of \cite{GRW}, and we have written it to be as parallel as possible. In order to avoid repetitions of large portions of that paper, we assume full familiarity with the relevant parts of \cite{GRW} for the rest of this section. There are two cases of Theorem \ref{thm:SurgBelowMid}: one for right stable metrics and one for arbitrary psc metrics. We present the proof in the case of right stable metrics; the other case has the same proof, except that the condition of right-stability is dropped whenever it occurs. The formal framework of the proof is summarised in the following diagram
\begin{equation}\label{proof-thm-surgbelow:overview}
\begin{gathered}
 \xymatrix{
\norm{ D_{\theta,\bullet}^{\kappa,l,\rst} } \ar[r]^-{\sim} \ar@{^(->}[d] & \norm{X_{\bullet}^{\kappa,l, \rst}} \ar@{^(->}[d] \\
 \norm{ D_{\theta,\bullet,0}^{\kappa,l,\rst}} \ar[ur]^-{\mathscr{S}(1,-)} \ar[r]_-{\sim}^-{\mathscr{S}(0,-)}  & \norm{X_{\bullet}^{\kappa,l-1, \rst}},
 }
\end{gathered}
\end{equation}
and we first give an approximate explanation of the terms of this diagram.
\begin{enumerate}[(i)]
 \item The spaces $\norm{X_{\bullet}^{\kappa,l, \rst}}$ and $\norm{X_{\bullet}^{\kappa,l-1,\rst}}$ in the right hand column are the flexible models for $B \PCob_\theta^{\kappa,l,\rst}$ and $B \PCob_\theta^{\kappa,l-1,\rst}$ introduced in Definition \ref{defn:x-space} and the right vertical map is given by inclusion. 
 \item The space $\norm{ D_{\theta,\bullet}^{\kappa,l,\rst} }$ has been introduced in Definition \ref{defn:d-space}; the upper horizontal map is given by inclusion and is a weak equivalence by Lemma \ref{lem:x-space-equivalent-to-d-space}.
 \item The space $\norm{ D_{\theta,\bullet,0}^{\kappa,l,\rst} }$ is a ``space of surgery data'' or better a ``space of psc manifolds equipped with surgery data''. The left vertical map is defined by including empty surgery data. The lower horizontal map forgets the surgery data, and is a weak equivalence, by Theorem \ref{thm:Contractibility} below. 
 \item Finally, we construct a surgery homotopy $\mathscr{S}: [0,1] \times  \norm{ D_{\theta,\bullet,0}^{\kappa,l,\rst}} \to \norm{X_{\bullet}^{\kappa,l-1,\rst}}$, such that $\mathscr{S}(0,-)$ is the forgetful map and such that $\mathscr{S}(1,-)$ factors as indicated. Moreover, the construction is such that $\mathscr{S}$ is the constant homotopy on the subspace $\norm{ D_{\theta,\bullet}^{\kappa,l,\rst} }$.
\end{enumerate}
These properties together imply that the right vertical map is a weak equivalence; combining this with Lemma \ref{lem:x-space-equivalent-to-d-space} and \eqref{eqn:poset-model-for-psccobcat} finishes the proof of Theorem \ref{thm:SurgBelowMid}.

\vspace{2ex}

In the following, we will equip the ``standard family'' that was used in \cite{GRW} to construct the surgery homotopy with suitable psc metrics. Then we will define the space of surgery data and prove that the map which forgets surgery data is a weak equivalence. Finally (and this will be almost the same as in \cite{GRW}), we show how to implement the surgery homotopy.

\subsubsection*{The standard family}

We begin by recalling some definitions and constructions from Section 4.2 of \cite{GRW}. Let $K \subset \bR^{d-l} \times \bR^{l+1}$ be the manifold defined in that section. Its essential features are summarised in Figure \ref{fig:1} below. It is a $d$-dimensional submanifold of $\bR \times \bR^{d-l-1} \times D^{l+1}$ which outside the set $D_{\sqrt{2}}^{d-l} \times D^{l+1}$ coincides with $\bR^{d-l} \times S^l$.

\begin{figure}[htb!]
\begin{center}
\includegraphics{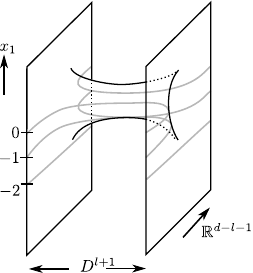}
\caption{The manifold $K$ with $d=2$ and $l=0$.}\label{fig:1}
\end{center}
\end{figure}

The projection $x_1 : K \to \bR$ onto the first coordinate has exactly two critical points on $K$, with critical values $\pm 1$ and of indices $(l+1)$ and $(d-l-1)$. We write $K|_I := x_1^{-1}(I) \cap K$ for a subset $I \subset \bR$. For the proof of Theorem \ref{thm:SurgBelowMid}, only the part $K\vert_{(-6,0)}$ is relevant. 
We fix an isotopy $\lambda_s : \bR \to \bR$ of diffeomorphism such that
\begin{enumerate}[(i)]
 \item $\lambda_0 = \id$, 
 \item $\lambda_s|_{(-\infty,-5)} = \id$,
 \item $\lambda_1 (-4)= -1$,
 \item $\lambda_1 (-3)= -\frac{1}{2}$,
 \item $\lambda_1 (-2) = - \frac{1}{4}$,
 \item $\lambda_s \leq \lambda_{s'}$ for $s \leq s'$,
 \item $\lambda'_t \equiv 1$ near $\lambda_t^{-1}(-1)$.
\end{enumerate}
The restriction of $\lambda_s\vert_{(-6,-2)} : (-6,-2) \to (-6,0)$ is an embedding and serves the same purpose as the isotopy with the same name in \cite[p.\ 306]{GRW} (it is convenient for us to have $\lambda_s$ defined on all of $\bR$ and to satisfy conditions (iv)--(vii)). As in \cite{GRW}, we define a 1-parameter family of submanifolds
\[
 \cP_t := (\lambda_t\vert_{(-6,-2)} \times \id_{\bR^d})^{-1}(K) \subset (-6,-2) \times \bR^{d-l-1} \times \bR^{l+1}.
\]
From a given tangential structure $\ell: TK|_{(-6,0)} \to \gamma_\theta$, a tangential structure on the 1-parameter family $\cP_t$ of manifolds is constructed in \cite[p.\ 306]{GRW}. There is no need for us to keep track of the tangential structure in this section: the surgery move that we will construct will be the same as the one in \cite{GRW} on the underlying manifolds, and so all issues about tangential structures are already addressed in that paper. We choose a 1-parameter family of increasing diffeomorphisms
\[
 \varphi_t (a_i, \eps_i, a_p, \eps_p): \bR \lra \bR,
\]
depending smoothly on $(a_i, \eps_i, a_p,\eps_p)\in \bR^4$ with $\eps_i,\eps_p>0$ and $a_i+\eps_i < a_{p}-\eps_{p}$, such that $\varphi_t(a_i, \eps_i, a_p, \eps_p)$ maps
\begin{enumerate}[(i)]
\item $(-6,-2)$ diffeomorphically onto $(a_i-\eps_i, a_p + \eps_p)$, 
\item $-4$ to $a_i - \frac{1}{2} \eps_i$, and $-5$ to $a_i - \frac{3}{4} \eps_i$,
\item $\varphi_t(a_i,\eps_i,a_p,\eps_p)' \equiv 1$ near $\lambda_t^{-1}(-1)$
\end{enumerate}
(the only difference to the maps with the same name in \cite[p. 308]{GRW} is that these are extended to diffeomorphisms of the whole real line, and we have introduced the last condition, which is convenient for us, and hence introduced a $t$-dependence). Now we define, depending on the parameters $a_i, a_p, \eps_i, \eps_p$ and $t \in [0,1]$, the manifold 
\[
 \bar{\cP}_t := (\varphi_t (a_i, \eps_i, a_p, \eps_p)\vert_{(a_i-\eps_i, a_p + \eps_p)} \times \id_{\bR^d}) (\cP_t) \subset (a_i-\eps_i, a_p + \eps_p) \times \bR^{d-l-1} \times \bR^{l+1},
\]
which is diffeomorphic to $\cP_t$, but with height function rescaled by $\varphi_t(a_i, \eps_i, a_p, \eps_p)$. The manifolds $\cP_t$ and $\bar{\cP}_t$ are the same as to the manifolds with the same name in Sections 4.2 and 4.4 of \cite{GRW}, and in particular Proposition 4.2 of \cite{GRW} holds for these manifolds. 

We will now equip the manifold $\bar{\cP}_t$ with suitable psc metrics. We first fix a $1$-$2$-torpedo metric $g_\tor^{d-l-1}$ on $\bR^{d-l-1}$ (see Definition \ref{defn:torpedo-metric}). 

\begin{lem}\label{lem:psc-metric-on-standard-family}
There exist psc metrics $h=h_{(a_i, \eps_i,a_p,\eps_p,t)}$ on $\bar{\cP}_t=\bar{\cP}_t (a_i, \eps_i, a_p, \eps_p)$, depending smoothly on the parameters $a_i, \eps_i, a_p, \eps_p$ and $t$, with the following properties:
\begin{enumerate}[(i)]
\item In the region 
$$\bar{\cP}_t \setminus (a_i-\eps_i, a_p+\eps_p) \times D_2^{d-l-1} \times S^{l} = (a_i-\eps_i, a_p+\eps_p) \times (\bR^{d-l-1} \setminus D_2^{d-l-1}) \times S^{l}$$
the metric $h$ is equal to $dx_1^2 + g_\tor^{d-l-1} + g_\round^{l}$.
\item In a neighbourhood of the region 
$$\bar{\cP}_t \cap ((a_i-\eps_i, a_i- \tfrac{3}{4} \eps_i] \times \bR^{d-l-1} \times D^{l+1}) = (a_i-\eps_i, a_i- \tfrac{3}{4} \eps_i] \times \bR^{d-l-1} \times S^l$$
the metric $h$ is equal to $dx_1^2 + g_\tor^{d-l-1} + g_\round^{l}$.
\item $h$ is of product type with respect to the height function $x_1$ except on an interval of length\footnote{This condition is vacuous unless $a_p+\eps_p-(a_i-\eps_i)>1$ and we are only interested in this construction if $\eps_i$ is large enough.} $1$.
\item If $s_0 < s_1 \in (a_i-\eps_i, a_p+ \eps_p)$ are two points over which $\bar{\cP}_t$ and $h$ are of product type, the restriction of $h$ to the manifold $U:=\bar{\cP_t} \cap ([s_0,s_1] \times D_3^{d-l-1} \times \bR^{l+1})$ has the following property: if $(M, g_M)$ is any compact $(d-1)$-dimensional psc manifold and $\phi : U\vert_{s_0} \hookrightarrow M$ is an embedding such that $\phi^*(g_M) = h\vert_{U\vert_{s_0}}$, then writing $L := M \setminus \phi(U\vert_{s_0})$ the psc cobordism
$$(U \cup_{[s_0,s_1] \times \partial D_3^{d-l-1} \times S^l} ([s_0,s_1] \times L), h\vert_U \cup (dx_1^2 + g_M\vert_L))$$
is right-stable.
\end{enumerate}
\end{lem}

\begin{proof}
For the construction of the psc metric $h$, it is convenient for us to consider the variant $\tilde{K}$ of $K$ depicted in Figure \ref{fig:5} (a). We then shall construct a psc metric $g_0$ on $\tilde{K}$ and a 1-parameter family of embeddings 
\[
 \eta_t: \bar{\cP}_t (a_i, \eps_i, a_p, \eps_p) \lra \tilde{K},
\]
and take $h:= \eta_t^* g_0$. 

For the construction of $\tilde{K}$, choose a smooth map $\gamma : \bR \to \bR$ which is equal to $0$ on $[0,\infty)$, is equal to the identity on $(-\infty,-\tfrac{1}{4}]$, and has $\gamma'>0$ on $(-\infty,0)$. Set 
$$\tilde{K} := (\gamma \times \id_{\bR^{d}} )^{-1}(K).$$
This manifold is depicted in Figure \ref{fig:5} (a), satisfies 
\begin{align*}
\tilde{K}|_{(-\infty,-\frac{1}{4})} &= K_{(-\infty,-\frac{1}{4})},\\
\tilde{K}|_{[0, \infty)} &= K|_{0} \times [0,\infty),\\
\tilde{K}|_{(-\infty,-2]} &= (-\infty,-2]  \times  \bR^{d-1-l}\times S^l,
\end{align*}
the map 
$$
\gamma\vert_{(-6,0)} \times \id_{\bR^{d}}  : K\vert_{(-6,0)} \lra \tilde{K}\vert_{(-6,0)}
$$ 
is a diffeomorphism, and the only critical point of the height function $x_1 : \tilde{K} \to \bR$ has critical value $-1$ and index $(l+1)$. This manifold is the long trace of a surgery of index $(l+1)$ on a $(d-1)$-manifold.

\begin{figure}[htb!]
\begin{center}
\includegraphics{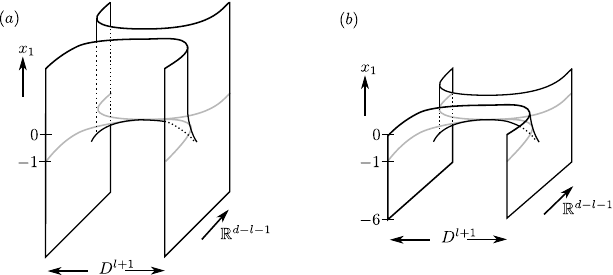}
\caption{The manifolds (a) $\tilde{K}$, and (b) $K\vert_{(-6,0)}$, with $d=2$ and $l=0$.}\label{fig:5}
\end{center}
\end{figure}

By Lemma 3.4.1 of \cite{ERWpsc2} (here the condition $d-l-1 \geq 3$ is used) we may choose a psc metric $g_0$ on $\tilde{K}$ such that
\begin{enumerate}[(i)]
\item $g_0$ is of product type with respect to $x_1$ except on $\tilde{K}|_{[-\frac{3}{2},-\frac{1}{2}]}$,
\item $g_0$ is equal to $dx_1^2 + g_\tor^{d-l-1} + g_\round^{l}$ outside the region $\bR  \times D_2^{d-l-1} \times S^l$,
\item if $t_0 <t_1 \in \bR \setminus [-\frac{3}{2},-\frac{1}{2}]$, then the metric $g_0$ restricted to
$$\tilde{K} \cap ([t_0, t_1] \times D_3^{d-l-1} \times \bR^{l+1})$$
satisfies the analogue of condition (iv) (i.e.\ it yields a right-stable metric whenever it is glued in as the trace of a surgery on a psc manifold).
\end{enumerate}

Let $\nu_t=\nu_{(t,a_i,\eps_i,a_p,\eps_p)}$ be the composition
\[
\bar{\cP}_t (a_i, \eps_i, a_p, \eps_p) \xrightarrow{(\varphi_t (a_i, \eps_i, a_p, \eps_p)\vert_{(a_i-\eps_i, a_p + \eps_p)} \times \id_{\bR^d})^{-1}} \cP_t \xrightarrow{\lambda_t\vert_{(-6,-2)} \times \id_{\bR^d}} K|_{(-6,0)}
\]
which is an embedding. We will next construct a 1-parameter family of embeddings 
\begin{equation}\label{eq:EmbMu}
\mu_t = \mu_{(t,a_i,\eps_i,a_p,\eps_p)} : K|_{(-6,0)} \lra \tilde{K} 
\end{equation}
and put $\eta_t := \nu_t \circ \mu_t$. To this end, let $m_t: \bR \to \bR$ be the diffeomorphism (depending also on $a_i, \eps_i$ and $a_p,\eps_p)$ with $m_t (-1)=-1$ and 
\[
m'_t = (\varphi_t(a_i,\eps_i,a_p,\eps_p)\circ\lambda_t^{-1})'.
\]
By the properties of the functions $\lambda_t$ and $\varphi_t(a_i,\eps_i,a_p,\eps_p)$, we have that $m_t$ is the identity near $-1$. There are, depending smoothly on the data $(a_i,\eps_i,a_p,\eps_p)$, embeddings $\mu_t$ as in \eqref{eq:EmbMu} which
\begin{enumerate}[(i)]
\item are embeddings over the function $m_t: \bR \to \bR$, and
\item satisfy $\mu_t=m_t \times \id_{\bR^{d-l-1}\times S^l}$ outside $(-2,0)\times D_2^{d-l-1} \times S^l$.
\end{enumerate}
These are chosen to be the identity near height $-1$ (where $m_t$ is the identity and near which the manifolds $K\vert_{(-6,0)}$ and $\tilde{K}$ are equal, cf. Figure \ref{fig:5}), and may be constructed using flows of the vector field dual to $dx_1$ elsewhere. 

With all these choices being made, the metric $h:=(\nu_t \circ \mu_t)^* g_0$ satisfies all requirements.
\end{proof}

\subsubsection*{Surgery data}

In this section we construct a psc version of the semi-simplicial space of surgery data from Section 4.3 of \cite{GRW}. The key property is stated as Theorem \ref{thm:Contractibility} below, which is the analogue of Theorem 4.5 of \cite{GRW}. In order to facilitate the proof of that result, we present the definition in two steps. This allows us to deal with the topological and the metric aspects of the proof separately. The first part will be almost the same as the proof in \cite{GRW}, but adapted to the context of sheaves.

\begin{defn}\label{defn:Xpsc-1}
Fix an infinite set $\Omega$. Let $X$ be a test manifold and let $\Gamma=(a,\epsilon,(W, \ell_W, g_W)) \in D^{\kappa,l,\rst}_{\theta}(\bR^N)_p(X)$. Define the set $\bar{Y}^{\rst}_0(\Gamma)$ to consist of tuples $(\Lambda, \delta, e, \ell)$ where 
\begin{enumerate}[(i)]
\item $\Lambda \subset \Omega$ is a finite subset,
\item $\delta: \Lambda \to [p]$ is a function,
\item $e: \Lambda \times (-6,-2) \times \bR^{d-l-1} \times D^{l+1} \times X \to \bR \times (-1,1)^N \times X$ is a smooth embedding over $X$ and
\item $\ell: T_v (\Lambda \times K|_{(-6,0)} \times X) \to \gamma_\theta$ is a bundle map,
\end{enumerate}
such that after restriction to each point (which will not be explicitly denoted) of $X$ the conditions (i)--(iv) listed below are satisfied. 

As a matter of notation, write $\Lambda_i := \delta^{-1}(i)$ and 
\[
 e_i : \Lambda_i \times (a_i-\eps_i, a_p + \eps_p) \times \bR^{d-l-1} \times D^{l+1} \lra \bR \times (-1,1)^N
\]
for the embedding $e|_{\Lambda_i} \circ (\id_{\Lambda_i} \times \varphi (a_i, \eps_i,a_p,\eps_p)^{-1} \times \id_{\bR^{d-l-1} \times D^{l+1}})$. 

Now we require the following conditions:
\begin{enumerate}[(i)]
\item $e^{-1}(W) = \Lambda \times (-6,-2) \times \bR^{d-l-1} \times S^l $ and we write $\partial e$ for the embedding $\Lambda \times (-6,-2) \times \bR^{d-l-1} \times S^l \to W$ obtained by restriction.
\item For $t \in \cup_{k=i}^p (a_k-\eps_k,a_k+\eps_k) $, we have $(x_1 \circ e_i)^{-1} (t)= \Lambda_i \times \{t\} \times \bR^{d-l-1} \times D^{l+1}$.
\item The composition $\ell_{W} \circ D \partial e : T(\Lambda \times K|_{(-6,-2)}) \to \gamma_\theta$ coincides with the restriction of $\ell$.
\item If $\ell_i$ denotes the restriction of $\ell$ to $T(\Lambda_i \times K|_{(-6,0)})$, then the datum $(e_i, \ell_i)$ is enough to perform $\theta$-surgery on $M_i := W|_{a_i}$. The resulting $\theta$-manifold is denoted $\bar{M_i}$ and we require that its structure map $\bar{M_i} \to B$ be $(l+1)$-connected.
\end{enumerate}
We define a semi-simplicial sheaf $\bar{D}_{\theta}^{\kappa,l,\rst}(\bR^N)_{\bullet, 0}$ by
$$\bar{D}_{\theta}^{\kappa,l,\rst}(\bR^N)_{p,0}(X) = \{(\Gamma,y) \, | \, \Gamma \in D^{\kappa,l,\rst}_{\theta}(\bR^N)_p(X), y \in \bar{Y}^{\rst}_0(\Gamma)\}.$$
The $i$th face map forgets the data $e_i$, $a_i$ and $\eps_i$. There is a forgetful map 
$$\bar{D}_{\theta}^{\kappa,l,\rst}(\bR^N)_{\bullet, 0} \lra D_{\theta}^{\kappa,l-1,\rst}(\bR^N)_{\bullet}.$$
\end{defn}

This definition is entirely parallel to Definition 4.3 of \cite{GRW}, adapted to sheaves, with two small differences. Firstly the underlying manifolds carry psc metrics, but as there is no condition that couples the psc metric with the surgery data this is a very mild change. Secondly in \cite{GRW} the last condition is replaced by the requirement that after performing the surgeries, the resulting structure map $\bar{M_i} \to B$ is injective on homotopy groups in degrees $\leq l$. We will explain below how to deal with this difference.

\begin{figure}[htb!]
\begin{center}
\includegraphics{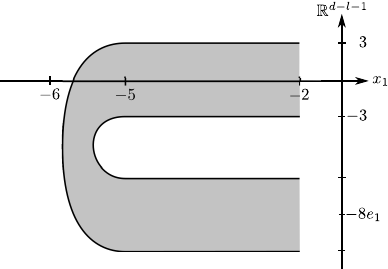}
\caption{The manifold $U$, with the $S^l$-direction not drawn.}\label{fig:picture-of-U}
\end{center}
\end{figure}

In order to be able to do psc surgery we will really be interested in a sub-semi-simplicial sheaf $D_{\theta}^{\kappa,l,\rst}(\bR^N)_{\bullet,0} \subset \bar{D}_{\theta}^{\kappa,l,\rst}(\bR^N)_{\bullet,0}$ in which the psc metrics are coupled to the surgery data in a precise way. In order to define this condition we need another technical detail. We fix an embedding $\alpha: (-10,-2) \times D_3^{d-l-1} \to (-6,-2) \times \bR^{d-l-1}$, such that
\begin{enumerate}[(i)]
 \item for $(t,y) \in (-5,-2) \times D_3^{d-l-1}$, we have $\alpha (t,y)= (t,y)$,
 \item for $(t,y) \in (-10,-7) \times D_3^{d-l-1}$, we have $\alpha (t,y)= (-t-12, -8 e_1 - \bar{y})$, where $\bar{y}:=(-y_1,y_2, y_3,  \ldots,y_{d-l-1})$,
 \item if $t \in (-7,-5)$, then $\alpha (t,y) \in (-6,-5) \times \bR^{d-l-1}$. 
\end{enumerate}
The image of $\alpha \times \id_{S^l}$ is denoted $U \subset (-6,-2)\times \bR^{d-l-1} \times S^l$, and is depicted in Figure \ref{fig:picture-of-U}. 

We have to reparametrise $\alpha$ by $\varphi(a_i,\eps_i,a_p,\eps_p)$, in the following way. Fix diffeomorphisms $\psi(a_i,\eps_i,a_p,\eps_p): (2a_i-2\eps_i-a_p-\eps_p,a_p+\eps_p) \cong (-10,-2)$, smoothly depending on these data, such that
\[
\psi(a_i,\eps_i,a_p,\eps_p)(x)=
\begin{cases}
\varphi (a_i, \eps_i,a_p,\eps_p)^{-1} (x) & a_i - \frac{3}{4}\epsilon_i \leq x \leq a_p+\epsilon_p\\
-6 & x = a_i - \eps_i \}\\
\end{cases}
\]
and such that the function
\[
 x \mapsto \psi(a_i,\eps_i,a_p,\eps_p)(x-(a_i-\eps_i))+6
\]
is odd. We get a new embedding 
\[
\alpha': (2a_i-a_p-\eps_p-2 \eps_i,a_p+\eps_p) \times D_3^{d-l-1} \to  (a_i-\eps_i,a_p+\eps_p) \times \bR^{d-l-1} 
\]
given by
\[
\alpha' (t,y) := (\varphi (a_i, \eps_i, a_p, \eps_p) \times \id_{\bR^{d-l-1}})\circ  \alpha \circ (\psi(a_i, \eps_i, a_p, \eps_p) \times \id_{D_3^{d-l-1}})
\]
which has the same qualitative properties as $\alpha$. This embedding is used in the following definition of the space of psc surgery data. 

\begin{defn}\label{defn:Xpsc-2}
We define $D_{\theta}^{\rst,\kappa,l}(\bR^N)_{p,0}(X) \subset \bar{D}_{\theta}^{\rst,\kappa,l}(\bR^N)_{p,0}(X)$ as the subset of those tuples $(a,\epsilon,(W, \ell_W, g_W), \Lambda, \delta, e,\ell)$ such that after restricting to each point of $X$ we have
\begin{enumerate}[(i)]
\item $\eps_i \geq 2 (p+2)$, and
\item $(\partial e_i \circ (\id_{\Lambda_i} \times \alpha' \times \id_{S^l}))^* g_W = dx_1^2 + g_\tor^{d-l-1} + g_\round^l$ on $\Lambda_i \times (2a_i-2\eps_i-a_p-\eps_p,a_p+\eps_p)\times D^{d-l-1}_3 \times S^l$
\end{enumerate}
for all $i \in [p]$.
\end{defn}

For the purpose of performing surgeries, we only need (ii) to be satisfied on the piece $\Lambda_i \times (a_i-\frac{3}{4}\eps_i , a_p + \eps_p) \times D^{d-l-1}_3 \times S^l$, but for the proof of the following theorem the whole of $\Lambda_i \times (2a_i-2\eps_i-a_p-\eps_p,a_p+\eps_p)\times D^{d-l-1}_3 \times S^l$ will be used. The reason is as follows: working over a single point (i.e.\ $X=\ast$) the embedding
$$\partial e_i \circ (\id_{\Lambda_i} \times \alpha' \times \id_{S^l}) : \Lambda_i \times (2a_i-a_p-\eps_p-2 \eps_i,a_p+\eps_p) \times D_3^{d-l-1} \times S^l \lra W$$
intersects each cobordism $W\vert_{[t_0,t_1]}$ with $t_0, t_1 \in \cup_{k=1}^p (a_k - \eps_k, a_k+\eps_k)$ in a finite set of thickened compact and neatly embedded submanifolds of codimension $d-l-1 \geq 3$ with collars, so by Theorem \ref{thm:chernysh-theorem} there is no homotopical cost to deforming the psc metric on $W$ to be standard on the image of $\partial e_i \circ (\id_{\Lambda_i} \times \alpha' \times \id_{S^l})$. (Note that this argument cannot be applied to $\partial e_i$ itself, as the intersection of the image this embedding with a $W\vert_{[t_0,t_1]}$ generally fails to have the above form as it has components which are thickened \emph{non-compact} submanifolds, to which Theorem \ref{thm:chernysh-theorem} does not apply.)

\begin{thm}\label{thm:Contractibility}
Under hypotheses (i)--(v) of Theorem \ref{thm:SurgBelowMid}, the forgetful map
\[
\norm{ D_{\theta}^{\kappa,l,\rst}(\bR^N)_{\bullet,0}} \lra \norm{ D_{\theta}^{\kappa,l-1,\rst}(\bR^N)_{\bullet}}
\]
is a weak homotopy equivalence. More precisely:
\begin{enumerate}[(i)]
\item If hypotheses (i)--(iv) of Theorem \ref{thm:SurgBelowMid} are satisfied, then the forgetful map
 \[
\norm{ \bar{D}_{\theta}^{\kappa,l, \rst}(\bR^N)_{\bullet,0}} \lra \norm{ D_{\theta}^{\kappa,l-1,\rst}(\bR^N)_{\bullet}}
\]
is a weak homotopy equivalence.
\item If hypothesis (v) of Theorem \ref{thm:SurgBelowMid} is satisfied, then the inclusion map 
\[
D_{\theta}^{\kappa,l,\rst}(\bR^N)_{p,0} \lra \bar{D}_{\theta}^{\kappa,l,\rst}(\bR^N)_{p,0}
\]
is a weak homotopy equivalence for each $p \geq 0$ and hence gives a weak homotopy equivalence after geometric realisation.
\end{enumerate}
\end{thm}

The proof is similar in spirit as the proof of Theorem 4.5 of \cite{GRW}. The key technical tool in that paper is the simplicial technique of Theorem 6.2 of \cite{GRW}. As we work with sheaves instead of topological spaces, we cannot literally apply the simplicial technique described in Theorem 6.2 of \cite{GRW}. Instead, we will use the analogue Theorem \ref{thm:flag-complextheoremn} of that result in the context of sheaves, which is stated and proven in the appendix.

\begin{proof}[Proof of Theorem \ref{thm:Contractibility} (i)]
We follow very closely the proof of Theorem 4.5 of \cite{GRW} given in Sections 6.1 and 6.4 of \cite{GRW} and only comment on the few differences that arise. As in Section 4 of \cite{GRW}, the semi-simplicial sheaf $\bar{D}_{\theta}^{\kappa,l,\rst}(\bR^N)_{\bullet,0}$ forms part of a bi-semi-simplicial sheaf $\bar{D}_{\theta}^{\kappa,l,\rst}(\bR^N)_{\bullet,\bullet}$ which is augmented over $D_{\theta}^{\kappa,l-1,\rst}(\bR^N)_{\bullet}$ (in the second simplicial direction, one takes disjoint sets of surgery data; for details, the reader is referred to \cite[p. 309]{GRW}). One then proves that the maps
\[
 \norm{ \bar{D}_{\theta}^{\kappa,l,\rst}(\bR^N)_{\bullet,0}} \stackrel{}{\lra}  \norm{ \bar{D}_{\theta}^{\kappa,l,\rst}(\bR^N)_{\bullet,\bullet}} \stackrel{}{\lra} \norm{ D_{\theta}^{\kappa,l-1,\rst}(\bR^N)_{\bullet}}
\]
(the first is given by inclusion of $0$-simplices and the second is the augmentation/forgetful map) are both weak homotopy equivalences. 

The proof that the inclusion of $0$-simplices is a weak equivalence is exactly as in Section 6.1 of \cite{GRW}, straightforwardly adapted to sheaves, and uses the fact that the augmentation map is a weak equivalence.\footnote{In \cite{HP} it is observed that the argument in Section 6.1 of \cite{GRW} is not quite correct, as the map displayed at the bottom of page 327 of \cite{GRW} is not equal to the augmentation map. As explained in the proof of Lemma 6.2.5 of \cite{HP} this may be remedied by taking the \emph{thin} geometric realisation of the target of this map; in the context of sheaves one may do the same, and furthermore there are no point set considerations necessary to conclude that the fat and thin realisations are homotopy equivalent.}

The proof that the augmentation map is a weak equivalence differs a bit more from that in Section 6.4 of \cite{GRW}, though the general strategy is the same. We prove that $\norm{ \bar{D}_{\theta}^{\kappa,l,\rst}(\bR^N)_{p,\bullet}} \to  D_{\theta}^{\kappa,l-1,\rst}(\bR^N)_p$ is a weak equivalence for each $p$. To that end, we factor the augmentation map ${ \bar{D}_{\theta}^{\kappa,l,\rst}(\bR^N)_{p,\bullet}} \stackrel{}{\lra} { D_{\theta}^{\kappa,l-1,\rst}(\bR^N)_{p}}$ as 
$${ \bar{D}_{\theta}^{\kappa,l,\rst}(\bR^N)_{p,\bullet}} \stackrel{}{\lra} { \tilde{D}_{\theta}^{\kappa,l,\rst}(\bR^N)_{p,\bullet}} \stackrel{}{\lra} { D_{\theta}^{\kappa,l-1,\rst}(\bR^N)_{p}}.$$
Similarly to Definition 6.14 of \cite{GRW}, the intermediate sheaf $\tilde{D}_{\theta}^{\kappa,l,\rst}(\bR^N)_{p,\bullet}$ is defined like $\bar{D}_{\theta}^{\kappa,l,\rst}(\bR^N)_{p,\bullet}$, except that we only require $e$ in Definition \ref{defn:Xpsc-1} to be a smooth map which restricts to a smooth family of embeddings of an open neighbourhood of the core $X \times \Lambda \times (-6,-2) \times \{0\} \times D^{l+1}$. The proof that the inclusion ${ \bar{D}_{\theta}^{\kappa,l,\rst}(\bR^N)_{p,\bullet}} \to { \tilde{D}_{\theta}^{\kappa,l,\rst}(\bR^N)_{p,\bullet}}$ is a levelwise weak equivalence is as the proof of Proposition 6.15 of \cite{GRW}, adapted to sheaves. 

This leaves to show that the augmentation 
\begin{equation}\label{eqn:augmentationmap-surgerycomplex}
\tilde{D}_{\theta}^{\kappa,l,\rst}(\bR^N)_{p,\bullet} \stackrel{}{\lra} { D_{\theta}^{\kappa,l-1,\rst}(\bR^N)_{p}}
\end{equation}
induces a weak equivalence. We now verify that \eqref{eqn:augmentationmap-surgerycomplex} satisfies the hypotheses of Theorem \ref{thm:flag-complextheoremn}. 

It is straightforward to verify that the augmentation ${ \tilde{D}_{\theta}^{\kappa,l,\rst}(\bR^N)_{p,\bullet}} \to { D_{\theta}^{\kappa,l-1,\rst}(\bR^N)_{p}}$ is a topological flag complex in the sense of Definition \ref{defn:topologicalflagcomplex}. The verification of hypothesis (i) of Theorem \ref{thm:flag-complextheoremn} is exactly as the proof of Proposition 6.16 of \cite{GRW}, but adapted to the sheaf language. Verifying hypothesis (iii) of Theorem \ref{thm:flag-complextheoremn} may be done by general position as in Proposition 6.17 of \cite{GRW}. It remains to establish hypothesis (ii), i.e.\ to prove the analogue of Proposition 6.18 of \cite{GRW}. It is here that there is a slight difference, because of the altered formulation of the connectivity condition. 
(The presence of psc metrics does not change the argument, precisely because the definition of $ \tilde{D}_{\theta}^{\rst,\kappa,l}(\bR^N)_{\bullet,0}$ does \emph{not} contain any condition that relates the psc metrics with the surgery data.) In \cite{GRW}, the connectivity conditions are
\begin{enumerate}[(i)]
\item In the definition of $ D_{\theta}^{\kappa,l-1}(\bR^N)_{\bullet}$: $\ell_M: M \to B$ is injective on $\pi_k$ for $k \leq l-1$.
\item In the formulation of Theorem 4.1 of \cite{GRW}: $L \to B$ is $(l+1)$-connected.
\item In the definition of the space $D_{\theta}^{\kappa,l-1}(\bR^N)_{\bullet,0}$ of surgery data: after performing surgeries on the manifolds $M_i$ along the embedding $e_i$, the new structure map $\bar{M_i} \to B$ is injective on $\pi_k$ for $k \leq l$.
\end{enumerate}

One now has to check that the proof of Proposition 6.18 of \cite{GRW} goes through when these conditions are replaced by\footnote{Given these conditions, it might be more natural to denote the category $\Cob_{\theta}^{\kappa,l}$ by $\Cob_{\theta}^{\kappa,l+1}$. We chose not to change the notation here, in order to avoid confusion with the notation of \cite{GRW}.}

\begin{enumerate}[(i)]
\item In the definition of $ D_{\theta}^{\kappa,l-1}(\bR^N)_{\bullet}$: $\ell_M: M \to B$ is $l$-connected.
\item In the formulation of \cite[Theorem 4.1]{GRW}: $B$ is of type ($F_{l+1}$).
\item In the definition of the space $D_{\theta}^{\kappa,l-1}(\bR^N)_{\bullet,0}$ of surgery data: after performing surgeries on the manifolds $M_i$ along the embedding $e_i$, the new structure map $\bar{M_i} \to B$ is $(l+1)$-connected.
\end{enumerate}

Only the first part of the proof of Proposition 6.18 of \cite{GRW} is affected. In this part, a finite set $\Lambda_i$ and maps $g_i: \Lambda_i \times S^l \to M_i$ are constructed so that the structure map $\ell_i : M_i \to B$ extends to an $(l+1)$-connected map $M_i \cup_{g_i} (\Lambda_i \times D^{l+1}) \to B$. This is done using that $L \to B$ is $(l+1)$-connected and that $M \to B$ is injective on homotopy groups up to degree $l-1$. But this can also be achieved under the conditions that $M \to B$ is $l$-connected and $B$ is of type ($F_{l+1}$). This is by the results of Section 1 of \cite{WallFin}, which might be stated by saying that the following conditions on a space $B$ are equivalent:
\begin{enumerate}[(i)]
 \item $B$ is of type $(F_n)$. 
 \item If $k \leq n$ and if $f:X \to B$ is a map from a finite complex which is $(k-1)$-connected, then we can find finitely many maps $g_j: S^{k-1} \to X$, $j=1, \ldots,r$, such that $f$ extends to a $k$-connected map $X \cup_{\coprod_j g_j} (\underline{r} \times D^n) \to B$.
\end{enumerate}
The rest of the proof goes through without any further change.
\end{proof}

\begin{proof}[Proof of Theorem \ref{thm:Contractibility} (ii)]
It is helpful to factor the map in question as a composition
\begin{equation}\label{proff-of-thm:Contractibility2}
 D_{\theta}^{\kappa,l,\rst}(\bR^N)_{p,0} \stackrel{\mathrm{inc}}{\lra} \bar{D}_{\theta}^{\kappa,l,\rst}(\bR^N)'_{p,0}\lra \bar{D}_{\theta}^{\kappa,l,\rst}(\bR^N)_{p,0}.
\end{equation}
The middle sheaf is defined to be the subsheaf of $\bar{D}_{\theta}^{\rst,\kappa,l}(\bR^N)_{p,0}$ such that only the first of the two conditions of Definition \ref{defn:Xpsc-2} is satisfied (in other words, we require the cylindrical parts to be long, but do not insist on having a compatibility of the psc metric with the surgery datum). The second map in \eqref{proff-of-thm:Contractibility2} is a weak equivalence by the argument used in the proof of Lemma \ref{collar-strtching-cobcat}.

The first of the two maps of \eqref{proff-of-thm:Contractibility2} is a weak homotopy equivalence as well. For the proof, we will use the criterion of Proposition \ref{prop:SurjCrit}, so for any $X \in \Mfds$, closed $A \subset X$ and germ $z$ of $D_{\theta}^{\kappa,l,\rst}(\bR^N)_{p,0}$ near $A$ we must show that
$$D_{\theta}^{\kappa,l,\rst}(\bR^N)_{p,0}[X,A;z] \lra \bar{D}_{\theta}^{\kappa,l,\rst}(\bR^N)_{p,0}'[X,A;\mathrm{inc}(z)]$$
is surjective.

An element of the codomain is represented by data $\xi=(W, \ell_W, a,\epsilon,g_W,\Lambda,\delta,e,\ell)$ such that for all $x$ in some open neighbourhood $U \supset A$ we have
$$(\partial e_i\vert_x \circ (\id_{\{x\} \times \Lambda_i} \times \alpha' \times \id_{S^l}))^* g_{W_x} = dx_1^2 + g_\tor^{d-l-1} + g_\round^l;$$
we must find a concordance, relative to $A$, to a new family where this holds for all $x \in X$. We shall do this without changing the underlying $\theta$-manifold or the surgery data at all, so first neglecting metrics we take
$$\xi' = (W', \ell_{W'}, a', \epsilon',\Lambda',\delta',e',\ell')) = \mathrm{pr}_X^*(\xi) \in D_{\theta}^{\kappa,l}(\bR^N)_{p,0}(X \times \bR)$$
which we now endow with a family of psc metrics. 

For each $i=1,2,\ldots, p$ the projections
$$\pi: W_i := \{(x,s,v) \in W \,|\, a_{i-1}(x)-\eps_{i-1}(x) \leq s \leq a_i(x)+\eps_i(x)\} \lra X$$
are smooth bundles of cobordisms with collared boundaries, and there is an associated fibre bundle $\cR^{+}_X(W_i)^{\eps} \to X$ of spaces of psc metrics (collared near the boundary). The restriction of $g_W$ determines sections $g_i : X \to \cR^+_X(W_i)^{\eps}$ of this bundle, and the sections $g_i$ and $g_{i+1}$ agree on restriction to $W_i \cap W_{i+1}$.

Inside each $W_i$ the image of the embeddings $\partial e_i \circ (\id_{X \times \Lambda_i} \times \alpha' \id_{\times S^l})$ consists of a disjoint collection of thickened compact submanifolds $V_i$ of codimension $d-l-1\geq 3$, neatly embedded with collars, giving embeddings $\phi_i :V_i \times \bR^{d-l-1} \hookrightarrow W_i$ over $X$. We may then form the subbundle
$$\cR^+_X(W_i, \phi_i)^{\eps} \subset \cR^+_X(W_i)^{\eps}$$
of those psc metrics whose restriction along $\phi$ is of the form required by Definition \ref{defn:Xpsc-2} (ii). The inclusion of this subbundle is a homotopy equivalence on fibres, by Theorem \ref{thm:chernysh-theorem} (and because the fibres have the homotopy type of CW complexes, being homeomorphic to open subspaces of Fr\'echet spaces), so any section of $\cR^+_X(W_i)^{\eps}$ is fibrewise homotopic to a section of $\cR^+_X(W_i, \phi_i)^{\eps}$, and such a fibrewise homotopy maybe chosen to stay in $\cR^+_X(W_i, \phi_i)^{\eps}$ at points (such as those near $A$) which start there.

We then proceed as follows, always leaving things fixed over $A$. We first fibrewise homotope $s_1 = s_1(0) : X \to \cR^+_X(W_1)^{\eps}$ to a section $s_1(1) : X \to \cR^+_X(W_1, \phi_1)^{\eps}$; this may be extended to a path of Riemannian metrics on
$$\{(x,s,v) \in W \,|\, s \leq a_0(x)-\eps_0(x)\},$$
having no curvature constraints. This gives a homotopy of $s_1(t)\vert_{W_1 \cap W_2}$ starting from $s_1 \vert_{W_1 \cap W_2} = s_2\vert_{W_1 \cap W_2}$, which by Theorem \ref{thm:improved-chernysh-theorem} may be extended to a fibrewise homotopy $s_2(t)$ starting from $s_2(0)=s_2 : X \to \cR^+_X(W_2)^{\eps}$. Now $s_2(1)$ is already correct over $W_1 \cap W_2$, so combining Theorems \ref{thm:improved-chernysh-theorem} and \ref{thm:chernysh-theorem} there is a fibrewise homotopy from $s_2(1)$ to a $s_2(2) : X \to \cR^+_X(W_2, \phi_2)^{\eps}$, constant over $W_1 \cap W_2$. Continuing in this way, concatenating all of these we obtain a path of families of psc metrics on
$$\{(x,s,v) \in W \,|\, s \leq a_p(x)+\eps_p(x)\},$$
from $g_W$ to a family of psc metrics satisfying the condition of Definition \ref{defn:Xpsc-2} (ii); this may be extended to a smooth path of Riemannian metrics on the whole of $W$, which we write as $t \mapsto g_W(t)$ for $t \in [0,1]$. It is easy to arrange that the extension of this for $t \in \bR$ which is constant for $t \leq 0$ and $t \geq 1$ is smooth.

Given this discussion, we endow $W' = W \times \bR$ with the family of psc metrics $g_{W'}$ which on $W \times \{t\}$ is $g_W(t)$. Then we have
$$(W', \ell_{W'}, a', \epsilon', g_{W'})  \in \bar{D}_{\theta}^{\kappa,l, \psc}(\bR^N)'_{p,0}(X \times \bR),$$
and as on each $W_i$ the psc metric is only changed by an isotopy, by Lemma \ref{lem:stability-homotopy-invariant} this in fact lies in $\bar{D}_{\theta}^{\kappa,l, \rst}(\bR^N)'_{p,0}(X \times \bR)$. It is a concordance relative to $A$ starting at $\xi$ and ending at a $\xi'' \in {D}_{\theta}^{\kappa,l, \rst}(\bR^N)_{p,0}(X, A, \mathrm{inc}(z))$, as required.
\end{proof}

\subsubsection*{Performing the surgeries}

We have now produced the outer square of diagram \eqref{proof-thm-surgbelow:overview}, and we already know that the two horizontal maps are weak homotopy equivalences, by Theorem \ref{thm:Contractibility} and Lemma \ref{lem:x-space-equivalent-to-d-space}. The left vertical map is given by inclusion of empty surgery data. In order to complete the proof of Theorem \ref{thm:SurgBelowMid}, we still need to construct the surgery homotopy, and this is parallel to Section 4.4 of \cite{GRW}. We give some details. Let $(a,\epsilon,(W, \ell_W, g_W), \Lambda, \delta, e,\ell) \in D_{\theta}^{\rst, \kappa,l}(\bR^N)_{p,0}(X)$. Over each point $x \in X$, for each $i =0, 1, \ldots,p$, we have an embedding
\[
e_i(x): \Lambda_i \times (a_i-\eps_i, a_p + \eps_p) \times \bR^{d-l-1} \times D^{l+1} \lra \bR \times (-1,1)^N 
\]
as in Definition \ref{defn:Xpsc-1}, which is compatible with the psc metrics in the sense spelled out in Definition \ref{defn:Xpsc-2}. Furthermore, there is a bundle map $\ell_i(x) : T(\Lambda_i \times K\vert_{(-6,0)}) \to \gamma_\theta$ as in Definition \ref{defn:Xpsc-1} (iv). 

Now we define manifolds $\mathcal{K}_{e_i(x),\ell_i(x)}^{t}(W_x) \subset (a_0(x)-\eps_0(x), a_p(x)+\eps_p(x)) \times \bR^N$, depending on $t = (t_0, t_1, \ldots, t_p) \in [0,1]^{p+1}$, as in \cite[p. 309 f]{GRW}. This manifold is equal to $W_x$ outside the image of $e_i(x)$, and equal to $e_i(x)(\Lambda_i \times \bar{\cP}_{t_i})$ on the image of $e_i(x)$. A suitable tangential structure on this manifold is defined in \cite{GRW}, giving a $\theta$-submanifold
$$\mathcal{K}_{e_i(x),\ell_i(x)}^{t}(W_x, \ell_{W_x}) \subset (a_0(x)-\epsilon_0(x), a_p(x)+\epsilon_p(x)) \times  \bR^N.$$
The union of these gives a smooth submanifold
$$\mathcal{K}_{e_i,\ell_i}^{t}(W, \ell_{W}) \subset \{(x,s,v) \in X \times \bR \times \bR^N \, | \, a_0(x)-\epsilon_0(x) < s < a_p(x)+\epsilon_p(x)\}$$
submersing to $X$, with a $\theta$-structure on its fibrewise tangent bundle.

It remains to define a suitable family of psc metrics on $\mathcal{K}_{e_i,\ell_i}^{t}(W)$. The constructions above and in \cite {GRW} are designed so that $\mathcal{K}_{e_i(x),\ell_i(x)}^{t}(W_x)$ agrees with $W_x$ even outside of the set $e_i(\Lambda_i \times (a_i-\frac{3}{4}\eps_i,a_p+\eps_p     ) \times D_2^{d-l-1})$. Therefore, we may glue in the psc metric $h_{(a_i(x),\eps_i(x),a_p(x),\eps_p(x))}$ constructed in Lemma \ref{lem:psc-metric-on-standard-family}, giving $ \mathcal{K}^t_{e_i(x),\ell_i(x)}(W_x,\ell_{W_x},g_{W_x})$. By applying the above process to all the surgery data, with $i$ running from $0$ to $p$, and collating all fibres together we obtain
\[
\mathcal{K}^t_{e,\ell}(W,\ell_W,g_W).
\]
The analogue of Lemma 4.6 of \cite{GRW} is as follows.

\begin{lem}\label{lem:psc-surgeriy-has.right-properties}\mbox{}
\begin{enumerate}[(i)]
\item For each $t = (t_0, t_1, \ldots, t_p) \in [0,1]^{p+1}$, $ \mathcal{K}^t_{e,\ell}(W,\ell_W,g_W)$, together with the choices of $(a_i)$ and $(\frac{1}{2}\eps_i)$, is an element of $X_p^{\kappa,l-1,\rst}(X)$. 
\item If $t_i=1$ (so that the surgery on the $i$th level is fully done), then for each regular value $b \in (a_i(x)-\frac{1}{2}\eps_i(x),a_i(x)+\frac{1}{2}\eps_i(x))$, the structure map 
$$\mathcal{K}^{t}_{e(x),\ell(x)} (W_x,\ell_{W_x},g_{W_x})|_b \lra B$$
is $(l+1)$-connected. 
\end{enumerate}
\end{lem}

\begin{proof}
The manifold part of the argument is essentially the same as the proof of Lemma 4.6 of \cite{GRW}. Because the connectivity conditions considered here are different from that in \cite{GRW}, some things need to be watched. Namely, one needs to prove that performing an arbitrary $\theta$-surgery of index $(l+1)$ on a manifold $M^{d-1}$ whose structure map is either $l$-connected or $(l+1)$-connected does not destroy this property. This holds as long as $2 (l+1) < d$, by the same argument as the one given in Lemma 4.6 of \cite{GRW}. Part (ii) and the manifold part of (i) follow from this observation and from Definition \ref{defn:Xpsc-1}.

For the metric part of (i) we must first show that each point of $X$ has a neighbourhood $U$ such that for each $i = 0,1,\ldots, p$ there is a $s_i \in \bR$ and $\delta_i>0$ such that for every $x \in U$ we have $a_i(x)-\frac{1}{2} \eps_i(x) < s_i-\delta_i < s_i+\delta_i < a_i(x)+\frac{1}{2}\eps_i(x)$ and the psc metric on $\mathcal{K}^t_{e(x),\ell(x)}(W_x,\ell_{W_x},g_{W_x})$ is of product type over $(s_i-\delta_i, s_i+\delta_i)$: this will mean that we have defined an element of $X_p^{\kappa, l-1,\psc}(X)$. To see this, first observe that each of the surgeries destroys the property that a metric is of product type only over an interval of length $\leq 1$, by Lemma \ref{lem:psc-metric-on-standard-family} (iii). Furthermore, for each $j=0,1,\ldots,p$ the surgeries indexed by $\Lambda_j$ all happen at the same height over $(a_i-\epsilon_i, a_i+\epsilon_i)$, by condition (ii) of Definition \ref{defn:Xpsc-1}. Now choose a $x \in X$. By the above the interval in $(a_i(x)-\frac{1}{2} \eps_i(x), a_i(x)+\frac{1}{2}\eps_i(x))$ over which the psc metric does not have product type has length at most $(p+1)< 2(p+1) \leq \epsilon_i(x)$, so there is always an interval left over which the metric is of product type. Furthermore, after perhaps shrinking it this interval also works on an open neighbourhood of $x$, by continuity of the $a_i$ and $\epsilon_i$.

Finally, we must show that we have in fact defined an element of $X_p^{\kappa, l-1,\rst}(X) \subset X_p^{\kappa, l-1,\psc}(X)$, so must show that if $s_i \in (a_i(x) -\eps_i(x), a_i(x)+\eps_i(x))$ and $s_j \in (a_j(x) -\eps_j(x), a_j(x)+\eps_j(x))$ with $s_i < s_j$ are values near which the metric is of product type, then the psc metric on $ \mathcal{K}^t_{e(x),\ell(x)}(W_x,\ell_{W_x},g_{W_x})|_{[s_i,s_j]}$ is right stable. This follows by construction, Lemma \ref{lem:psc-metric-on-standard-family} (iv),  and Lemma \ref{lem:2-out-of-three} (ii) and (iv). 
\end{proof}

The rest of the proof of Theorem \ref{thm:SurgBelowMid} follows the argument given in \cite[p. 311 ff]{GRW} verbatim.

\begin{remark}\label{rem:proofof:thm:grw-thm4.1}
We can now indicate the changes to Section 4 and 6.4 of \cite{GRW} which are necessary to show Theorem \ref{thm:grw-thm4.1} as stated above. Condition (iv) of Definition 4.3 of \cite{GRW} is replaced by the requirement that $\bar{M} \to B$ is $(l+1)$-connected. The same change is done to Lemma 4.6 of \cite{GRW}, whose proof goes through without change, as we explained during the proof of Lemma \ref{lem:psc-surgeriy-has.right-properties}. The other change is to the first part of the proof of Proposition 6.18 of \cite{GRW}, where one has to use the results of Section 1 of \cite{WallFin}, as we explained at the end of the proof of the first claim of Theorem \ref{thm:Contractibility}.

Finally, the argument also goes through for the inclusion $B \Cob_{\theta,L}^{\kappa,l}\to B \Cob_{\theta,L}^{\kappa,l-1}$ considered in \cite{GRW}. The condition on $L$ which is needed for that is condition (v) of Theorem 4.1 of \cite{GRW}, i.e.\ that $ L$ admits a handle decomposition using handles of index at most $d-l-2$.
\end{remark}

\section{Infinite loop space structures on spaces of psc metrics}\label{sec:6}

\subsection{Segal's \texorpdfstring{$\Gamma$}{Gamma}-spaces}

We will produce infinite loop space structures using Segal's theory of $\Gamma$-spaces \cite{Segcat}, and in this section we recall its main points. 

We write $\Gamma^{\op}$ for the category of finite pointed sets and base-point preserving maps (Segal describes $\Gamma$ explicitly, but we prefer to use only $\Gamma^{\op}$). We denote the base-point of an object $S$ of $\Gamma^{\op}$ by $+ \in S$, and write $S_o:= S \setminus \{+\}$ for the remaining elements. For $S,R \in \Ob (\Gamma^{\op})$, we define the \emph{wedge sum} $S \vee T$ and \emph{smash product} $S \wedge V$ as usual. Let $n_+ := \{1 , \ldots, n, +\} \in \Ob (\Gamma^{\op})$. 
The object $0_+ = \{+\}$ is initial and terminal in $\Gamma^{\op}$; we denote by $\upsilon_S: 0_+ \to S$ the unique morphism. For $S \in \Ob (\Gamma^{\op})$ and $s \in S_o$, we define $\iota_s: S \to 1_+$ by 
\[
\iota_s (t):=
\begin{cases}
1 & t=s \\
+ & \text{otherwise}.
\end{cases}
\]
Furthermore, we let
\begin{align*}
\mu: 2_+ \lra 1_+
\end{align*}
be given by $\mu (1)= \mu(2)=1$.

\begin{defn}\label{defn:defGammaspace}\mbox{}
\begin{enumerate}[(i)]
\item A \emph{$\Gamma$-space} is a covariant functor $A:\Gamma^{\op} \to \Top$, and a map of $\Gamma$-spaces is a natural transformation of functors.
\item A $\Gamma$-space $A$ is \emph{special} if for all pointed finite sets $S$ and $T$ the map 
\[
p_* \times q_*: A(S \vee T) \lra A(S) \times A(T),
\]
induced by the collapse maps $p: S \vee T\to S$ and $q: S \vee T \to T$ is a weak homotopy equivalence.
\end{enumerate}
\end{defn}

Equivalently, $A$ is special if $A(0_+) \simeq *$ and if the maps
\[
\prod_{i=1}^p(\iota_j)_*:A(p_+) \lra A(1_+)^p
\]
are weak homotopy equivalences for all $p \geq 0$. It is a useful guideline to think of a $\Gamma$-space as the space $A(1_+)$, equipped with additional structure. 

\begin{defn}
A \emph{pointed $\Gamma$-space} is a $\Gamma$-space $A$ together with a base-point $a_0 \in A(0_+)$. This determines base-points $a_S = (\upsilon_S)_* (a_0) \in A(S)$, and for each morphism $\eta:S \to T$, we have $\eta_* (a_S)=a_T$.
\end{defn}

If $A$ is special then $A(0_+)$ is contractible, so there is a contractible choice of base-points. In the zig-zag
\[
A(1_+) \times A(1_+) \xleftarrow{((\iota_1)_* ,( \iota_2)_*)} A(2_+) \xrightarrow{\mu_*} A(1_+),
\]
the left map is a weak equivalence when $A$ is special. This in particular gives a map $\pi_0 (A(1_+)) \times \pi_0 (A(1_+)) \to \pi_0 (A(1_+))$ defining the structure of an abelian monoid, with unit given by $[a_1] \in \pi_0 (A(1_+))$.

\begin{defn}
A $\Gamma$-space $A$ is \emph{very special} or \emph{group complete} if it is special and the monoid $\pi_0 (A(1_+))$ is a group. 
\end{defn}

There is a functor $\Lambda:\Delta^{\op} \to \Gamma^{\op}$ defined as follows. It sends $[p] \in \Delta^{op}$ to $p_+$, and a monotone function $f: [p] \to [q]$ is mapped to the function $f^*: q_+ \to p_+$ given by
\[
f^* (j)=
\begin{cases}
i, & f(i-1) < j \leq f(i),\\
+, & \text{otherwise.}
\end{cases}
\]
Via $\Lambda$ each $\Gamma$-space $A$ yields a simplicial space. Segal \cite{Segcat} defines the classifying space $BA$ of a $\Gamma$-space $A$ to be the $\Gamma$-space
\[
(BA) (S) := \norm{[p] \mapsto A(S \wedge p_+)}.
\]
\begin{rem}
In fact Segal uses a different geometric realisation to the fat geometric realisation $\norm{-}$ that we use, described in Appendix A of \cite{Segcat}. However, he shows in Proposition A.3 of that paper that these two choices are homotopy equivalent.
\end{rem}

For pointed $\Gamma$-spaces, Segal also constructs a natural map $A(1_+) \to \Omega BA (1_+)$. In this way, out of a (pointed) $\Gamma$-space $A$, one obtains a connective spectrum $\mathrm{B}^{\infty}A$ with $n$th term $B^n A(1_+)$. The following is Proposition 1.4 of \cite{Segcat}, apart from (i) which holds simply because $BA(1_+)$ is path-connected whenever $A$ is special.

\begin{thm}[Segal]\label{thm:segals-theorem}\mbox{}
\begin{enumerate}[(i)]
\item If $A$ is special, then $BA$ is very special.
\item If $A(1_+)$ is $k$-connected, then $BA(1_+)$ is $(k+1)$-connected.
\item If $A$ is very special, then $A(1_+) \to \Omega BA(1_+)$ is a weak equivalence.
\end{enumerate}
\end{thm}

It follows from (iii) that if $A$ is very special then $A(1_+)$ has the homotopy type of the infinite loop space of the spectrum $\mathrm{B}^{\infty}A$. 

\begin{defn}
Let $A$ be a pointed $\Gamma$ space. The \emph{unit component} of $A$ is the sub-$\Gamma$-space $A^0$ which is defined by letting $A^0(S) \subset A(S)$ be the path-component containing $a_S$, and whose structure maps are given by restricting those of $A$.
\end{defn}
If $A$ is special then $A^0$ is very special, because each $A^0(S)$ and in particular $A^0(1_+)$ is connected by definition. It follows that $A^0 (1_+)$ has the homotopy type of an infinite loop space if $A$ is special.
The $\Gamma$-spaces in this paper will arise as the geometric realisations of certain \emph{simplicial} spaces. 

\begin{defn}\label{defn:simplicialgammapsace}
A \emph{simplicial} $\Gamma$-space is a simplicial object $A_\bullet$ in the category of $\Gamma$-spaces (this is the same as a functor $\Delta^{\op} \times \Gamma^{\op} \to \Top$, $([p],S) \mapsto A_p (S)$). The \emph{geometric realisation of $A_\bullet$} is the $\Gamma$-space $\norm{A_\bullet}$ defined by $\norm{A_\bullet}(S):= \norm{ [p] \mapsto A_p (S)}$. 
\end{defn}

\begin{lem}\label{lem:geometric-realization-simplicial-gammaspace}
Let $A_\bullet$ be a simplicial $\Gamma$-space which is levelwise special in the sense that the $\Gamma$-space $S \mapsto A_p (S)$ is special for each $p$. Then the geometric realisation $\norm{A_\bullet}$ is special. 
\end{lem}

\begin{proof}
First observe that $\norm{A_\bullet}(0_+)$ is the geometric realisation of the levelwise contractible simplicial space $p \mapsto A_p (0_+) \simeq *$ and is hence contractible. Secondly, the map (defined as in \ref{defn:defGammaspace}) $(p_*,q_*): \norm{A_\bullet (S \vee T)} \to \norm{A_\bullet (S) } \times \norm{A_\bullet (T)}$ is the composition of the two natural maps
\[
\norm{A_\bullet (S \vee T)} \lra \norm{ A_\bullet (S) \times A_\bullet (T)} \lra \norm{ A_\bullet (S)  }\times \norm{A_\bullet (T)}.
\]
The first comes from applying $p_* \times q_*$ levelwise. Since $A$ was assumed to be special, the first map is the fat geometric realisation of a levelwise weak equivalence and hence is a weak equivalence. The second map comes from the two projections and is a weak equivalence by Theorem 7.2 of \cite{SxTech} (it is here that we use that $A_\bullet$ is simplicial and not merely semi-simplicial).
\end{proof}

\subsection{Homotopy fibres and base-points}

Recall that the homotopy fibre $\hofib_y f$ of a map $f:X \to Y$ of spaces over $y \in Y$ is the space of all pairs $(x,\gamma)$, where $x \in X$ and $\gamma: f(x) \leadsto y$ is a path in $Y$ from $f(x)$ to  $y$. When $x_0 \in X$ is a base-point and $y_0=f(x_0)$, the space $\hofib_{y_0}(f)$ is given the base-point $(x_0,\const_{y_0})$. 

\begin{defn}
Let $f:A \to C$ be a map of $\Gamma$-spaces. The \emph{homotopy fibre} of $f$ at a base-point $c_0 \in C(0_+)$ is the $\Gamma$-space $\hofib_{c_0}f$ which is defined on objects $S \in \Gamma^{\op}$ by
\[
(\hofib_{c_0} f)(S):= \hofib_{c_S} (f(S) : A(S) \to C(S)).
\]
A morphism $\eta:S \to T$ induces a map $(\hofib_{c_0} f)(S) \to (\hofib_{c_0} f)(T)$, namely the comparison map of vertical homotopy fibres of the commutative diagram
\[
\xymatrix{
A(S) \ar[r]^{\eta_*} \ar[d]^{f(S)} & A(T) \ar[d]^{f(T)}\\
C(S) \ar[r]^{\eta_*} & C(T)
}
\]
over $c_S$ and $c_T = \eta_*(c_S)$.
\end{defn}

If $A$ is also a pointed $\Gamma$-space such that $f$ is a pointed map then $\hofib_{c_0} f$ is pointed with base-point $(a_0,\const_{c_0}) \in (\hofib_{c_0} f)(0_+) = \hofib_{c_0} (f(0_+))$.

\begin{lem}\label{lem:hofib-of-special-gamme}\mbox{}
Let $f:A \to C$ be a pointed map of pointed special $\Gamma$-spaces. 
\begin{enumerate}[(i)]
\item The $\Gamma$-space $\hofib_{c_0} f$ is special.
\item If in addition $A$ is very special then so is $\hofib_{c_0} f$.
\item The homotopy fibre $\hofib_{c_0}(f^0)$ of the restriction of $f$ to $f^0 : A^0 \to C^0$ is very special.
\end{enumerate}
\end{lem}

\begin{proof}
It is clear that $(\hofib_{c_0} f)(0_+)=\hofib_{c_0} (f(0_+) : A(0_+) \to C(0_+))$ is contractible since both source and target are. Let $S,T \in \Ob (\Gamma^{\op})$, with collapse maps $p:S \vee T \to S$ and $q: S \vee T \to T$. Consider the commutative diagram
\[
\xymatrix{
A(S \vee T) \ar[r]^-{(p_*,q_*)} \ar[d]^{f(S \vee T)}& A(S) \times A( T) \ar[d]^{f(S) \times f(T)} \\
C(S \vee T) \ar[r]^-{(p_*,q_*)} & C(S) \times C(T).
}
\]
If $A$ and $C$ are special, the horizontal maps are both weak equivalences, and so the diagram is homotopy cartesian. It follows that the map
\[
\hofib_{c_{S \vee T}} (f(S\vee T)) \lra \hofib_{(c_S,c_T)} (f(S) \times f(T)) \cong \hofib_{c_S} f(S) \times \hofib_{c_T} f(T)
\]
is a weak equivalence as desired. 

For part (ii), let $A^1 = f^{-1} C^0$ be the preimage of the unit component and let $f':= f|_{A^1}$. The $\Gamma$-spaces $C^0$ and $A^1$ are very special, and because $\hofib_{c_0} (f)= \hofib_{c_0} (f')$, it is enough to carry out the proof of claim (ii) if $C$ is very special as well. Next, consider the following piece of the long exact homotopy sequence for the map $f(1_+)$:
\[
\pi_1 (C(1_+), c_1) \stackrel{\partial}{\lra} \pi_0 (\hofib_{c_1} f(1_+)) \lra \pi_0 (A(1_+)) \lra \pi_0 (C(1_+)). 
\]
The terms are abelian monoids, and all maps are maps of abelian monoids. (This is clear for all maps but $\partial$. Here it follows from the usual Eckmann--Hilton argument: since $C(1_+)$ is an $H$-space, the $H$-space multiplication induces the same composition on $\pi_1 (C(1_+), c_1)$ as the usual product in the fundamental group. By construction of the long exact sequence, $\partial$ is clearly a monoid homomorphism when the monoid structure on $\pi_1 (C(1_+), c_1)$ coming from the $H$-space structure is used.) All terms but $\pi_0 (\hofib_{c_1} (f(1_+))$ are known to be groups: a simple diagram chase then proves that this is also a group.

Part (iii) follows from part (ii), since $A^0$ and $C^0$ are very special.
\end{proof}

In the application we aim at, we will need to consider homotopy fibres over a point other than the base-point. We shall do this by means of the following general construction.
Let $f:X \to Y$ be a map of spaces. A path $\delta: y_0 \leadsto y_1$ in $Y$ determines a map 
\begin{align*}
T(\delta): \hofib_{y_0} f &\lra \hofib_{y_1} f\\
 (x,\alpha) &\longmapsto (x, \alpha \ast \delta),
\end{align*}
and a path $\gamma: x_0 \leadsto x_1$ in $X$ determines
\begin{equation}\label{eqn:basepointtransport}
\Lambda ( \gamma) := T(f \circ \gamma): \hofib_{y_0} f \lra \hofib_{y_1} f.
\end{equation}
The homotopy class of $T(\delta)$ only depends on the homotopy class of $\delta$ (rel endpoints) and so $T$ defines a functor
\[
 T: \Pi_1 (Y) \to \hoTop;\; T([\delta]) := [T(\delta)]
\]
from the fundamental groupoid\footnote{A morphism in $\Pi_1 (Y)(y_0,y_1)$ is a homotopy class of paths $y_0 \leadsto y_1$, so that composition in $\Pi_1 (Y)$ is $\delta\circ \delta':= \delta' \ast \delta$.} of $Y$ to the homotopy category of spaces. Similarly, $\Lambda$ induces a functor $\Pi_1 (X) \to \hoTop$. 

\begin{lem}\label{lem:homotopyfibre-fundamentalgroupoid}\mbox{}
\begin{enumerate}[(i)]
\item The functor $\Lambda$ is weakly base-point-preserving in the sense that 
$$\Lambda(\gamma): \hofib_{f(x_0)} f \lra \hofib_{f(x_1)} f$$ 
sends the base-point to the component containing the base-point.
\item If $X$ and $Y$ are $H$-spaces with units $x_0$ and $y_0=f(x_0)$ and $f$ is an $H$-map, then for each loop $\gamma$ in $X$ based at $x_0$, the map $\Lambda(\gamma) $ is homotopic to the identity.
\end{enumerate}
\end{lem}

\begin{proof}
For (i), the base-point of $\hofib_{f(x_0)} f$ is sent by $\Lambda (\gamma)$ to the point $(x_0,f\circ (\const_{x_0} * \gamma))$, which lies in the same path component as $(x_0,f\circ \gamma)$. Define a path $\gamma^t : \gamma (t) \leadsto x_1$ by $\gamma^t (s):= \gamma (t+s(1-t))$. Then $t \mapsto (\gamma(t), f \circ \gamma^t)$ is a path in $\hofib_{f(x_1)}(f)$ from $(x_0,f \circ \gamma)$ to $(x_1,\const_{x_1})$, as desired.

For (ii), we may assume that $x_0$ and $y_0$ are strict units. If $f$ is strictly compatible with the multiplications, define a homotopy
\begin{align*}
H_t:\hofib_{f(x_0)} f &\lra \hofib_{f(x_0)} f\\
(x,\alpha) &\longmapsto (x \cdot \gamma(t),\alpha_t),
\end{align*}
where $\alpha_t$ is the path given by the formula
\[
\alpha_t (s):=
\begin{cases}
\alpha (2s)\cdot f(\gamma(t)) & s \leq \frac{1}{2}\\
f(\gamma(t+(1-t)(2s-1))) & s \geq \frac{1}{2}.
\end{cases}
\]
This is a homotopy from $\Lambda(\gamma)$ to a map which is visibly homotopic to the identity. If $f$ is only compatible with the multiplications up to homotopy, one has to compose the path $\alpha_t$ with a path $f(x \cdot \gamma(t)) \leadsto f(x) \cdot f(\gamma(t))$ coming from such a homotopy.
\end{proof}

\begin{cor}\label{cor.infiniteloopspace-over-nonbasepoint}
Let $f:A \to C$ be a pointed map of pointed special $\Gamma$-spaces. Let $a \in A(1_+)$ be a point lying in the component of the base-point $a_1$. The homotopy class 
\[
\Lambda(\gamma): \hofib_{f(a)} (f(1_+))^0 \lra \hofib_{f(a_1)} (f(1_+))^0
\]
does not depend on the choice of a path $\gamma: a \leadsto a_1$ in $A(1_+)$. Since $\hofib_{f(a_1)} (f(1_+))^0$ has the homotopy type of an infinite loop space, this equips $\hofib_{f(a)} (f(1_+))^0$ with the structure of an infinite loop space in a canonical way. \qedhere
\end{cor}

\subsection{The (psc) cobordism categories as \texorpdfstring{$\Gamma$}{Gamma}-spaces}

We now recall how the classifying spaces of the cobordism categories and related spaces are equipped with the structure of $\Gamma$-spaces. This is due to Madsen--Tillmann \cite{MadTill}, see also Nguyen \cite{Nguyen}.

We shall work with the poset models for cobordism categories described in Section \ref{sec:PosetModels}. In particular, we use the simplicial sheaves $\cDs_{\theta,\bullet}^{\kappa}$, $\cDs_{\theta,\bullet}^{\kappa,\psc}$ and $\cDs_{\theta,\bullet}^{\kappa,\rst}$ introduced in Definition \ref{defn:inifniteseminalcollars}. We will therefore need to work with simplicial $\Gamma$-objects in the category $\Sheaves$, and will produce simplicial $\Gamma$-spaces by applying the representing space functor $\rep{-} : \Sheaves \to \Top$.

We will first describe the structure of a $\Gamma$-space on $\norm{\cDs_{\theta, \bullet}^{\kappa}}$ for any $\kappa$. (It will be clear that the construction does not generalise to $\norm{\cDs_{\theta, \bullet}^{\kappa, l}}$ for $l \geq 0$.)

\begin{definition}\label{def:CobCatGamma}
For each $p \geq 0$ define a functor $\cDs_{\theta, p}^{\kappa}(\_) : \Gamma^{op} \to \Sheaves$ on a $S \in \Gamma^{op}$ and $X \in \Mfds$ by letting $\cDs_{\theta, p}^{\kappa}(S)(X)$ be the set of tuples $(W, \ell_W, a, f)$ with $(W, \ell_W, a) \in \cDs_{\theta, p}^{\kappa}(X)$ and $f : W \to S_o$ a locally constant function. Then $\cDs_{\theta, p}^{\kappa}(S)$ is a sheaf by restriction.

On a pointed map $\eta : S \to T$ the induced map $\eta_* : \cDs_{\theta, p}^{\kappa}(S)(X) \to \cDs_{\theta, p}^{\kappa}(T)(X)$ is given by letting
$$\eta_* W := f^{-1} (\eta^{-1}(T_o)) \subset W$$
and  setting $\eta_* (W,\ell_W, a, f):= (\eta_* W, \ell|_{\eta_* W}, a, \eta \circ f|_{\eta_* W})$. This defines a $\Gamma$-sheaf $\cDs_{\theta, p}^{\kappa}(\_)$.

These assemble into a simplicial $\Gamma$-sheaf $\cDs_{\theta, \bullet}^{\kappa}(\_)$, where the $i$th face map forgets $a_i$ and the $i$th degeneracy map doubles $a_i$. 
\end{definition}

Applying the representing space functor gives $\Gamma$-spaces $\rep{\cDs_{\theta, p}^{\kappa}}(\_)$ assembling into a simplicial $\Gamma$-space, and taking the (fat) geometric realisation in the simplicial direction gives a $\Gamma$-space $\norm{\cDs_{\theta, \bullet}^{\kappa}}(\_) : \Gamma^{op} \to \Top$, with $\norm{\cDs_{\theta, \bullet}^{\kappa}}(1_+) = \norm{\cDs_{\theta, \bullet}^{\kappa}}$. This is the required $\Gamma$-space structure on $\norm{\cDs_{\theta,\bullet}^{\kappa}}$.

\begin{lem}\label{lem:simplicial-D-gamma-is-special}
The $\Gamma$-space $\norm{\cDs_{\theta, \bullet}^{\kappa}}(\_)$ is special. 
\end{lem}

\begin{proof}
We defined $\norm{\cDs_{\theta, \bullet}^{\kappa}}(\_)$ to be the geometric realisation of the simplicial $\Gamma$-space $\rep{\cDs_{\theta, \bullet}^{\kappa}}(\_)$. Therefore by Lemma \ref{lem:geometric-realization-simplicial-gammaspace} it is enough to verify that $\rep{\cDs_{\theta, p}^{\kappa}}(\_)$ is special for each $p$. This amounts to proving that the map
$$\rep{\cDs_{\theta, p}^{\kappa}}(S \vee T) \lra \rep{\cDs_{\theta, p}^{\kappa}}(S) \times \rep{\cDs_{\theta, p}^{\kappa}}(T)$$
induced by the collapse maps is a weak equivalence. The target is equivalent to the representing space of the sheaf $X \mapsto \cDs_{\theta, p}^{\kappa}(S)(X) \times \cDs_{\theta, p}^{\kappa}(T)(X)$, so we must show that the map of sheaves
$$\cDs_{\theta, p}^{\kappa}(S \vee T) \lra \cDs_{\theta, p}^{\kappa}(S) \times \cDs_{\theta, p}^{\kappa}(T)$$
is a weak equivalence. Using the criterion of Proposition \ref{prop:SurjCrit}, this is an easy consequence of general position: given two families of $\theta$-manifolds in $\bR \times (-1,1)^\infty\times X$, one labelled by $S$ and one by $T$, we may change them by a concordance so that they are disjoint, giving a lift to $\cDs_{\theta, p}^{\kappa}(S \vee T)[X]$. This argument goes through relative to any subset of $X$.
\end{proof}

\begin{rem}\label{rem:CobCatVerySpecial}
If $2 \kappa \leq d-2$, then Theorem 3.1 of \cite{GRW} (quoted above as Theorem \ref{thm:grw-thm3.1}) applies and shows that $\norm{\cDs_{\theta, \bullet}^{\kappa}} \simeq B \Cob_\theta^\kappa \simeq B \Cob_\theta$. In those cases, $\norm{\cDs_{\theta, \bullet}^{\kappa}}(\_)$ is very special. By the elaboration of the main result of \cite{GMTW} described in \cite{Nguyen}, there is a weak equivalence
\[
 \mathrm{B}^{\infty} \norm{\cDs_{\theta, \bullet}^{\kappa}} \simeq \Sigma (\MT \theta \langle -2 \rangle ) \simeq (\Sigma \MT \theta) \langle -1 \rangle
\]
of spectra\footnote{For a spectrum $E$, we let $E \langle n \rangle $ denote the $n$-\emph{connected} covering.}. The base-point component $\norm{\cDs_{\theta, \bullet}^{\kappa}}^0$ is always very special, no matter what $d$ and $\kappa$ are, but when $2 \kappa \leq d-2$ there is a weak equivalence
\[
  \mathrm{B}^{\infty} \norm{\cDs_{\theta, \bullet}^{\kappa}}^0 \simeq \Sigma (\MT \theta \langle -1 \rangle) \simeq (\Sigma \MT \theta) \langle 0 \rangle.
\]
\end{rem}

We now define a $\Gamma$-space structure on $\norm{\cDs_{\theta,\bullet}^{\kappa,\psc}}$ and $\norm{\cDs_{\theta,\bullet}^{\kappa,\rst}}$ in a completely analogous way.

\begin{defn}\label{defn:PscCObCatGamma}
We define $\cDs_{\theta,p}^{\kappa,\psc} (S)(X)$ as the set of all $(W,\ell_W,a, g_W,f)$ such that $(W,\ell_W,a,g_W) \in \cDs_{\theta,p}^{\kappa, \psc}(X)$ and $f : W \to S_o$ is a locally constant function. The sheaf and $\Gamma$-structure is as in Definition \ref{def:CobCatGamma}, and $[p] \mapsto \cDs_{\theta,p}^{\kappa,\psc} (\_)$ has the structure of a simplicial $\Gamma$-sheaf as in that definition.

The simplicial $\Gamma$-sheaf $[p] \mapsto \cDs_{\theta,p}^{\kappa,\rst}(\_)$ is defined analogously, using the simplicial sheaf $\cDs_{\theta,\bullet}^{\kappa, \rst}$. 
\end{defn} 

Just as in the previous section, we have 
$$\norm{D_{\theta,\bullet}^{\kappa,\psc}}(1_+) = \norm{D_{\theta,\bullet}^{\kappa,\psc}}\quad \text{ and } \quad \norm{D_{\theta,\bullet}^{\kappa,\rst}}(1_+) = \norm{D_{\theta,\bullet}^{\kappa,\rst}}.$$

\begin{thm}\label{thm.psc-gamma-special}
\mbox{}
\begin{enumerate}[(i)]
\item The inclusion and forgetful maps 
$$\norm{\cDs_{\theta,\bullet}^{\kappa,\rst}}(\_) \lra \norm{\cDs_{\theta,\bullet}^{\kappa,\psc}}(\_)\lra \norm{\cDs_{\theta,\bullet}^{\kappa}}(\_)$$ are maps of $\Gamma$-spaces.
\item The $\Gamma$-spaces $\norm{\cDs_{\theta,\bullet}^{\kappa,\psc}}(\_)$ and $\norm{\cDs_{\theta,\bullet}^{\kappa,\rst}}(\_)$ are special.
\item $\norm{\cDs_{\theta,\bullet}^{\psc}}(\_)$ is very special, provided that $\theta$ is once-stable.
\end{enumerate}
\end{thm}

\begin{proof}
Part (i) is trivial and part (ii) is proven exactly as Lemma \ref{lem:simplicial-D-gamma-is-special}. For part (iii), note that 
\[
\pi_0 (\norm{\cDs_{\theta,\bullet}^{\psc}}) = \pi_0 (B\PCob_\theta)
\]
is the monoid of psc-$\theta$-cobordism classes of $(d-1)$-dimensional psc manifolds. This is group-complete: if $(M,\ell,g)$ is a $(d-1)$-dimensional psc-$\theta$-manifold, consider the cobordism $(M \times [0,1],\ell',g+dt^2)$ with the extended $\theta$-structure. Bending this cobordism, one can consider it as a cobordism $\emptyset \leadsto (M,\ell,g) \coprod (M,\ell^{op},g)$ ($\ell^{op}$ is the opposite $\theta$-structure, which exists because $\theta$ is once-stable). 
\end{proof}

\begin{rem}
Using a bit of Gromov--Lawson surgery, one can show that the map 
\[
\pi_0 (\norm{\cDs_{\theta,\bullet}^{\kappa,\psc}}) \lra \pi_0 (\norm{\cDs_{\theta,\bullet}^{\psc}})
\]
is bijective, as long as $2 \kappa+1\leq d$, $\kappa \leq d-3$ and $\theta$ is once-stable. It follows that in those cases, $\norm{\cDs_{\theta,\bullet}^{\kappa,\psc}} (\_ )$ is very special as well. We do not need to know this fact, though. 
\end{rem}

\begin{defn}\label{defn:pscfibreGammaspace}
We define very special $\Gamma$-spaces
\[
\Psc(\theta) := \hofib_\emptyset ( \norm{\cDs_{\theta,\bullet}^{\psc}} \to \norm{\cDs_{\theta,\bullet}})
\]
and 
\[
\Psc^{\kappa,\rst}(\theta)_0 := \hofib_\emptyset ( \norm{\cDs_{\theta,\bullet}^{\kappa,\rst}}^0 \to \norm{\cDs_{\theta,\bullet}^\kappa}^0).
\]
\end{defn}
These are indeed very special, by Theorem \ref{thm.psc-gamma-special}, Remark \ref{rem:CobCatVerySpecial}, Lemma \ref{lem:simplicial-D-gamma-is-special} and Lemma \ref{lem:hofib-of-special-gamme}. 

\subsection{Conclusion}

Now we can prove Theorem \ref{Main:infiniteloopspacetheorem}, the more precise version of which reads as follows. 

\begin{thm}\label{thm:infinite-loopspace-theorem}
Let $M$ be a closed $(d-1)$-manifold and $g \in \Riem^+ (M)$. Assume that 
\begin{enumerate}[(i)]
\item $d \geq 6$, 
\item there exists a cobordism $W: \emptyset \leadsto M$ such that $(W,M)$ is $2$-connected and such that
\item $\Riem^+ (W)_g^{\rst} \neq \emptyset$.
\end{enumerate}
Let
\[
 M \stackrel{\ell_M}{\lra} B \stackrel{\theta_M}{\lra} B\mathrm{O}(d)
\]
be the Moore--Postnikov 2-stage of the stabilised tangent bundle $\bR \oplus TM$. Then there exists a weak homotopy equivalence
\[
\Theta:\Riem^+ (M \times [0,1])_{g,g}^{\rst} \simeq \Omega \Psc^{2,\rst}(\theta_M)_0 \simeq \Omega^{\infty+1} \mathrm{B}^\infty \Psc^{2,\rst}(\theta_M)_0,
\]
where $\Psc^{2,\rst}(\theta_M)_0$ is the very special $\Gamma$-space introduced in Definition \ref{defn:pscfibreGammaspace}. 
\end{thm}

The proof gives a preferred construction of $\Theta$ which uses $W$ and a metric in $\Riem^+ (W)_g^{\rst}$, but also shows that the homotopy class of $\Theta$ is independent of these choices. 

\begin{example}
Theorem \ref{thm:infinite-loopspace-theorem} applies to $M=S^{d-1}$. In that case, we can take $g=g_\round^{d-1}$, $W=D^d$ (note that $g_\tor^d \in \Riem^+ (D^d)_{g_\round^{d-1}}^{\rst}$), and the theorem implies that 
\[
 \Riem^+ (S^{d-1} \times [0,1])^{\rst}_{g_\round^{d-1},g_\round^{d-1}} 
\]
has the homotopy type of an infinite loop space. On the other hand, a simple application of Theorem \ref{thm:chernysh-theorem} shows that $ \Riem^+ (S^{d-1} \times [0,1])_{g_\round^{d-1},g_\round^{d-1}}  \simeq \Riem^+ (S^d)$. Hence a certain union of connected components of $\Riem^+ (S^d)$ has the homotopy type of an infinite loop space. This proves Theorem \ref{Main:infiniteloopspacetheoremsphere} from the introduction.
\end{example}

\begin{proof}[Proof of Theorem \ref{thm:infinite-loopspace-theorem}]
The $\Gamma$-space $\Psc^{2,\rst}(\theta_M)_0$ is very special, so we have an equivalence $\Psc^{2,\rst}(\theta_M)_0 \simeq \Omega^{\infty} \mathrm{B}^\infty \Psc^{2,\rst}(\theta_M)_0$ by Theorem \ref{thm:segals-theorem}.
Under the weak equivalences
\[
\norm{ \cDs^{2,\rst}_\theta} \simeq B \PCob_\theta^{2,\rst} \quad\text{ and }\quad \norm{ \cDs^{2}_\theta} \simeq B \Cob_\theta^{2}
\]
from \eqref{eqn:poset-model-for-psccobcat}, \eqref{eqn:poset-model-for-ordcobcat} and Lemma \ref{collar-shrinking-cobcat}, the forgetful map $ \norm{\cDs_{\theta,\bullet}^{2,\rst}}(1_+)^0 \to \norm{\cDs_{\theta,\bullet}^2}(1_+)^0$ corresponds to $(BF^{2,\rst})^0: (B \PCob_\theta^{2,\rst})^0 \to (B \Cob_\theta^2)^0$, and therefore 
\begin{equation*}
\hofib_{\emptyset} (BF^{2,\rst})^0 \simeq \Psc^{2,\rst}(\theta_M)_0. 
\end{equation*}
To finish the proof of the Theorem, we therefore have to establish a weak equivalence $\Riem^+ ([0,1] \times M)_{g_0,g_0}^{\rst} \simeq \Omega \hofib_{\emptyset} (BF^{2,\rst})^0$. This involves the change of base-points, and we shall use the functor $\Lambda$ defined in \eqref{eqn:basepointtransport} for that purpose.

The map $\ell_M$ is a $2$-connected cofibration and $\theta_M$ is a $2$-coconnected fibration. It follows that $B$ is of type $(F_2)$, and the map $\ell_M$ provides a $\theta$-structure on $M$. Let $W: \emptyset \leadsto M$ be a cobordism as required by hypothesis (ii) of the theorem. By obstruction theory, the $\theta_M$-structure $\ell_M$ extends to a $\theta_M$-structure $\ell_W$ on $W$. We obtain a morphism $(W,\ell_W):\emptyset \leadsto (M,\ell_M)$ in $\Cob_{\theta_M}^2$. 
By hypothesis (iii) of the theorem, there is a right stable $h \in \Riem^+ (W)_g$, and $(W,\ell_W,h): \emptyset \leadsto (M,\ell_M,g)$ is a morphism in $\PCob_{\theta_M}^{2,\rst}$. From that morphism, we get a path $\gamma: (M,\ell_M,g) \leadsto \emptyset$ in $B \PCob_{\theta_M}^{2,\rst}$ (in particular, $(M,\ell_M,g)$ is a point in $(\PCob_{\theta_M}^{2,\rst})^0$). This induces a weak equivalence
\begin{equation*}
\Lambda (\gamma): \hofib_{(M,\ell_M)} (BF^{2,\rst})^0 \lra \hofib_\emptyset (BF^{2,\rst})^0
\end{equation*}
which is weakly base-point-preserving and whose homotopy class does not depend on the choice of $\gamma$ (and hence $W$), by Lemma \ref{lem:homotopyfibre-fundamentalgroupoid}. By construction, $(M,\ell_M)$ is an object of $\Cob_{\theta_M}^{2,1}$, so that the symbol $\hofib_{(M,\ell_M)} (BF^{2,1,\rst})^0$ is defined. By Theorems \ref{thm:SurgBelowMid} and \ref{thm:grw-thm4.1}, the natural map
\begin{equation*}
\hofib_{(M,\ell_M)} (BF^{2,1,\rst})^0 \lra \hofib_{(M,\ell_M)} (BF^{2,\rst})^0
\end{equation*}
is a weak equivalence. Altogether, we have constructed a weak equivalence
\begin{equation}\label{eqn:lastmap-proofinfiniteloopspacestheorem0}
\hofib_{(M,\ell_M)} (BF^{2,1,\rst})^0 \simeq \Psc^{2,\rst}(\theta_M)_0,
\end{equation}
whose homotopy class does not depend on the choices involved. The base-point in $  \hofib_{(M,\ell_M)} (BF^{2,1,\rst})^0$ is the point $(M,\ell_M,g) \in (B \PCob_{\theta_M}^{2,1,\rst})^0$, together with the constant path at its image point in $(B \Cob_{\theta_M}^{2,1})^0$. Moreover, Theorems \ref{thm:Fibre} and \ref{thm:Concordance} provide a weak equivalence
\begin{equation}\label{eqn:lastmap-proofinfiniteloopspacestheorem}
\Riem^+ (M \times [0,1])_{g,g}^{\rst} \simeq \Omega  ( \hofib_{(M,\ell_M)} BF^{2,1,\rst})^0
\end{equation}
(the loop space in the target is taken at the base-point just specified). The loop space of the source of \eqref{eqn:lastmap-proofinfiniteloopspacestheorem0} is not the same as the target of \eqref{eqn:lastmap-proofinfiniteloopspacestheorem}. However, there is a general fact which finishes the proof of the Theorem. If $f:(X,x_0) \to (Y,y_0)$ is a based map and $X^0$, $Y^0$ are the path components containing the base-points, and $f^0$ is the restriction of $f$ to $X^0 \to Y^0$, there is a natural map
\[
 (\hofib_{y_0} f)^0 \lra \hofib_{y_0} f^0
\]
from the unit component of the homotopy fibre to the homotopy fibre of the restriction to the unit component. This is in general not an equivalence, but it is $0$-coconnected and so induces an equivalence after taking loop spaces at the base-points, which is all that matters for our purpose. The map $\Theta$ is defined as the composition of the above maps.
\end{proof}

\begin{rem}
We shall use the map $\Theta$ from Theorem \ref{thm:infinite-loopspace-theorem} to define an infinite loop space structure on $\Riem^+ ([0,1]\times M)_{g_0,g_0}^{\rst}$. In particular, $\Theta$ is, tautologically, an infinite loop map.
\end{rem}

\section{The action of the diffeomorphism group on psc metrics}\label{sec:diffactiononpscspace}

\subsection{Statement of the result}\label{subsec:statementdiffeoaction}

As we mentioned in the introduction, the cobordism category methods we have been using can be used to find significant constraints on the action of $\Diff_\partial(W)$ on $\Riem^+(W)_{g_M}$ when $W$ is a nullbordism of a psc manifold $(M, g_M)$. We will first formulate a more general result, and then explain how to derive Theorem \ref{mainthmintro:diffaction} from it.

\begin{thm}\label{thm:diff-action-comesfromMT}
Let $\theta:B \to B\mathrm{O}(d)$ be a fibration, $d \geq 6$. Let $(N,g_N)$ and $(M,g_M)$ be two objects of $\PCob_\theta$ and assume that the structure map $\ell_M : M \to B$ is $2$-connected. Then there exists a homotopy cartesian square
\[
\xymatrix{
\PCob_\theta^2 ((N,g_N),(M,g_M)) \ar[d]^{F} \ar[r] & \mathcal{X} \ar[d] & \\
\Cob_\theta^2 (N,M) \ar[r]^{\tau} & \Omega_{N,M} B \Cob_\theta^2 \ar[r]^{\simeq} & \Omega^\infty \MT \theta.
}
\]
\end{thm}
The space $\mathcal{X}$ will be constructed in the course of the proof. The left vertical map is the forgetful map whose fibre over a point $W$ is of course $\Riem^+ (W)_{g_N,g_M}$. The first bottom horizontal map $\tau$ is the tautological map and the second one is the homotopy equivalence arising from Theorem 3.1 of \cite{GRW} and the main result of \cite{GMTW}.

\begin{proof}[Proof of Theorem \ref{mainthmintro:diffaction} from Theorem \ref{thm:diff-action-comesfromMT}]
For a compact manifold $W$ of dimension $d \geq 6$ with boundary $M=\partial W$ such that $(W,M)$ is $2$-connected, we let $\theta:B \to B\mathrm{O}(d)$ be the Moore--Postnikov $2$-stage of the Gauss map of $W$. 

Let $\Cob_\theta^2 (\emptyset,M)_W \subset \Cob_\theta^2 (\emptyset,M)$ be the union of path components given by those cobordisms $V$ which are diffeomorphic to $W$ relative $M$. There is a forgetful map
\begin{equation}\label{eqn:forgetfulmap-forgettangential2structures}
\rho: \Cob_\theta^2 (\emptyset,M)_W \lra B \Diff_\partial (W)
\end{equation}
which is a weak equivalence, since its fibre is the space of solutions to the lifting problem
\begin{equation*}
 \xymatrix{
 M \ar[d] \ar[r]^-{\ell_W\vert_M} & B\ar[d]^-{\theta} & \\
 W \ar[r] \ar@{-->}[ru] & B\mathrm{O}(d),
}
\end{equation*}
which is contractible by elementary obstruction theory, as the left hand map is $2$-connected, the right hand map is $2$-coconnected, and a lift $\ell_W$ exists.

The map $\alpha_W$ is defined to be the composition
\begin{equation}\label{eqn:defnalphaW}
B \Diff_\partial (W) \stackrel{\rho^{-1}}{\lra}  \cC_\theta^2 (\emptyset,M)_W \stackrel{\tau}{\lra} \Omega_{\emptyset,M} B \cC_\theta^2 \simeq \Omega^\infty \MT \theta.
\end{equation}

We define 
\[
\PCob_\theta^2 (\emptyset,(M,g_M))_W :=  F^{-1}(\Cob_\theta^2 (\emptyset,M)_W) \subset \PCob_\theta^2 (\emptyset,(M,g_M)).
\]
There is a forgetful map
\[
\PCob_\theta^2 (\emptyset,M)_W  \lra E \Diff_\partial (W) \times_{\Diff_\partial (W)} \Riem^+ (W)_{g_M} 
\]
similar to the map $\rho$ from \eqref{eqn:forgetfulmap-forgettangential2structures}, and it is a weak equivalence by the same reason. We obtain a commutative square
\[
\xymatrix{
\PCob_\theta^2 (\emptyset,(M,g_M))_W \ar[r] \ar[d]^{F} &  E \Diff_\partial (W) \times_{\Diff_\partial (W)} \Riem^+ (W)_{g_M} \ar[d]^{F'}\\
\Cob_\theta^2 (\emptyset,M)_W  \ar[r] & B \Diff_\partial (W) 
}
\]
whose horizontal maps are weak equivalences. Restricting the left column of the diagram of Theorem \ref{thm:diff-action-comesfromMT} to path-components and replacing the left hand column and the lower right corner by homotopy equivalent spaces, we arrive at a homotopy cartesian square
\begin{equation*}
\begin{gathered}
 \xymatrix{
 E \Diff_\partial (W) \times_{\Diff_\partial (W)} \Riem^+ (W)_{g_M}\ar[d]^{F'} \ar[r] & \mathcal{X} \ar[d] & \\
 B \Diff_\partial (W) \ar[r]^{\alpha_W} &  \Omega^\infty \MT \theta.
}
\end{gathered}
\end{equation*}

This finishes the construction of the square. May's classification theory for fibrations \cite[\S 9]{May} now implies that the map $B\Diff_\partial (W) \to B\hAut(\Riem^+ (W)_{g_M})$ classifying the fibration $F'$ factors through $\alpha_W$, as required. After taking loop spaces, this classifying map is precisely the map $A: \Diff_\partial (W) \to \hAut(\Riem^+ (W)_g)$, and this proves the claim about $A$. The claim about the orbit map is an immediate consequence.
\end{proof}

\subsection{Proof of Theorem \ref{thm:diff-action-comesfromMT}}

The guiding idea of the proof of Theorem \ref{thm:diff-action-comesfromMT} is that the forgetful map $\Omega BF^{2,\rst}: \Omega B \PCob_\theta^{2,\rst} \to \Omega B \Cob_\theta^2$ (whose homotopy fibre over a point $M$ is weakly equivalent to $\Riem^+ (M \times [0,1])^{\rst}_{g_M,g_M}$ by Theorem \ref{cor:berw-new-proof}, Theorem \ref{thm:SurgBelowMid} and Theorem 3.1 of \cite{GRW}) could be considered as a ``principal bundle'' for the grouplike monoid $\Riem^+ (M \times [0,1])_{g_M,g_M}^{\rst}$. One should therefore be able to form the Borel construction 
\[
\text{``$\mathcal{X}= (\Omega B \PCob_\theta^{2,\rst} ) \times_{\Riem^+ (M \times [0,1])_{g_M,g_M}^{\rst}} \Riem^+ (M \times [0,1])_{g_M,g_M}$''}
\]
and hence obtain a homotopy cartesian diagram as stated in Theorem \ref{thm:diff-action-comesfromMT}. 

This idea will be made precise by establishing a diagram
\begin{equation}\label{eq:Sec7Main}
\xymatrix{
\PCob_\theta^{2}((N,g_N),(M,g_M)) \ar[d]^-{F} & \ar[l]_-{\norm{\varphi_\bullet}}^-{\simeq} \norm{Z_\bullet} \ar[d]^-{\norm{\lambda_\bullet}} \ar[r] & \norm{Y_\bullet} \ar[d]^-{\epsilon \circ \norm{\xi_\bullet}} & \mathcal{X} \ar[d] \ar[l]\\
\Cob_\theta^{2}(N,M)  \ar@/_2.0pc/[rrr]^-{\tau} & \ar[l]_-{\norm{\zeta_\bullet}}^-{\simeq} \norm{X_\bullet} \ar[r]^-{\epsilon \circ \norm{\eta_\bullet}} & \hofib_{M } BF^{2,\rst} & \Omega_{N,M} B \Cob_\theta^2 \ar[l]_-{\sigma}
}
\end{equation}
in which the three squares commute and are homotopy cartesian (the rightmost will be a homotopy pullback by definition), and the bottom semi-circle commutes up to preferred homotopy. 
This data determines a map $\norm{Z_\bullet} \to \mathcal{X}$ covering $\tau \circ \norm{\zeta_\bullet} : \norm{X_\bullet} \to \Omega_{N,M} B \Cob_\theta^2$, which together with the weak equivalences in the leftmost square gives the claimed result. As we shall explain in Remark \ref{rem:interpretationasbarconstruct}, the semi-simplicial spaces appearing in \eqref{eq:Sec7Main} can be interpreted as two-sided bar constructions.

\subsubsection{Construction of the diagram}

\begin{defn}
Let $X_p$ consist of $((t, W, \ell_W),(a_0,g_0), \ldots, (a_p,g_p), h_0, \ldots, h_p)$ where 
\begin{enumerate}[(i)]
\item $(t, W, \ell_W) \in \Cob_\theta^2 (N,M)$; we then write 
$$V := ((-\infty,-t] \times N) \cup (W-t e_1) \cup ([0,\infty) \times M) \subset \bR \times I^{\infty-1},$$
\item $0 < a_0 < \ldots < a_p<1$,
\item $g_i \in \Riem^+ ( \{a_i\} \times M)$,
\item $h_0 \in \Riem^+ (V|_{[-t,a_0]})_{g_N,g_0}^{\rst}$ and
\item $h_i \in \Riem^+ (V|_{[a_{i-1},a_i]})_{g_{i-1},g_i}^{\rst}$ (note that $V|_{[a_{i-1},a_i]} =  [a_{i-1},a_i] \times M$). 
\end{enumerate}
We topologise $X_p$ as a subspace of $N_{p+1} \PCob_\theta$. The $X_p$ form a semi-simplicial space by forgetting $a_i$'s and gluing metrics. There is an augmentation $\zeta_\bullet: X_\bullet \to \Cob_\theta^2 (N,M)$ given by recording only $(t,W,\ell_W)$.
\end{defn}

\begin{defn}
Let $Z_p$ consist of $((t,W,\ell_W),(a_0,g_0),\ldots, (a_p,g_p), h_0, \ldots, h_p,h)$ such that $((t,W,\ell_W),(a_0,g_0),\ldots, (a_p,g_p), h_0, \ldots, h_p) \in X_p$ and
\begin{enumerate}[(i)]
\setcounter{enumi}{5}
\item $h \in \Riem^+ (V\vert_{[a_p,1]})_{g_p,g_M}$ is an arbitrary psc metric.
\end{enumerate}
The map $\lambda_p : Z_p \to X_p$ forgets the datum $h$. We topologise $Z_p$ as a subspace of $N_{p+2} \PCob_\theta$. The $Z_p$ form a semi-simplicial space by forgetting $a_i$'s and gluing metrics. There is an augmentation map 
$\varphi_p : Z_p \to \PCob_\theta^2 ((N,g_N),(M,g_M))$ given by forgetting all $a_i$ and glueing the psc metrics.
\end{defn}
The square
\[
 \xymatrix{
 \norm{Z_\bullet} \ar[r]^-{\norm{\varphi_\bullet}} \ar[d]^-{\norm{\lambda_\bullet}}&  \PCob_\theta^2 ((N,g_N),(M,g_M)) \ar[d]^-{F} \\
 \norm{X_\bullet} \ar[r]^-{\norm{\zeta_\bullet}} & \Cob_\theta^2 (N,M)
 }
\]
commutes. We will show later that the horizontal maps are a weak equivalences, but first develop some more of the diagram \eqref{eq:Sec7Main}.

\begin{defn}
Let $Y_p$ consist of $((a_0,g_0),\ldots, (a_p,g_p), h_1, \ldots,h_p,h))$ where $0<a_0 < \ldots a_p < 1$, $g_i \in \Riem^+ ( \{a_i\}\times M)$, $h_i \in \Riem^+ ([a_{i-1}, a_i]\times M)_{g_{i-1},g_i}^{\rst}$ and $h \in \Riem^+ ([a_p,1]\times M)_{g_p,g_M}$. We topologise $Y_p$ as a subspace of $N_{p+1} \Conc(M)$. The $Y_p$ form a semi-simplicial space by forgetting $a_i$'s and gluing metrics. There is a semi-simplicial map $\xi_p: Y_p \to N_p (\cP( \bR \times M, (0,1))^{\rst}$ given by forgetting the datum $h$.
\end{defn}

\begin{rem}\label{rem:interpretationasbarconstruct}
The semi-simplicial spaces $X_\bullet$, $Y_\bullet$ and $Z_\bullet$ can informally be interpreted as two-sided bar constructions as follows. Let 
\[
G: (\cP(M \times \bR,(0,1))^{\rst})^{\mathrm{op}} \lra \Top
\]
be the functor which sends an object $(a,g)$ to the space $\Riem^+ (M \times [a,1])_{g,g_M}$; on morphisms it is defined by gluing. Let 
\[
H: \cP(M \times \bR,(0,1))^{\rst} \lra \Top
\]
be the functor that sends $(a,g)$ to the subspace $H(a,g) \subset X_0$ consisting of all $((t, W, \ell_W),(a_0,g_0),h_0)$ with $(a_0,g_0)=(a,g)$. (This is homotopy equivalent to $\PCob_\theta^{2, \rst} ((N,g_N),(M,g))$.)

With these notations in place, we have 
\begin{align*}
X_\bullet &= B_\bullet ( *,\cP(M \times \bR,(0,1))^{\rst}, H ), \\
Y_\bullet &= B_\bullet ( G,\cP(M \times \bR,(0,1))^{\rst}, * ), \text{ and} \\
Z_\bullet &= B_\bullet ( G,\cP(M \times \bR,(0,1))^{\rst}, H ).
\end{align*}
However, this is an informal interpretation: one easy way of saying what a continuous functor out of a topological category with non-discrete object space (such as $\cP(M\times \bR,(0,1))$ should be is by writing down its bar construction as a (semi)-simplicial space.
\end{rem}

Given those definitions, we observe that there is a pullback square
\begin{equation}\label{diag.definition-of-zbullet}
\begin{gathered}
  \xymatrix{
 Z_p \ar[r] \ar[d]^{\lambda_p} & Y_p \ar[d]^{\xi_p} \\
 X_p \ar[r]^-{\eta_p} & N_p \cP ( \bR \times M, (0,1))^{\rst},
  }
\end{gathered}
\end{equation}
where the semi-simplicial map $\eta_\bullet$ is given by forgetting $(t,W, \ell_W)$.

We obtain the middle square of the diagram \eqref{eq:Sec7Main} by forming the geometric realisation of the square \eqref{diag.definition-of-zbullet} of semi-simplicial spaces and composing with the map
\begin{equation}\label{eq:epsilon}
\epsilon : B\cP (\bR \times M, (0,1))^{\rst} \lra \hofib_{M } BF^{2,\rst}
\end{equation}
defined as follows. It sends a point $((a_0,g_0), \ldots, (a_p,g_p),h_1, \ldots, h_p; u) \in N_p \cP(\bR \times M,(0,1)) \times \Delta^p $ to the point in $\hofib_M BF^{2,\rst}$ given by 
\begin{align*}
&\Bigl( (M ,g_0) \xrightarrow{( [0,a_1-a_0] \times M, h_1)} \cdots \xrightarrow{([0,a_p-a_{p-1}] \times M, h_p)} (M,g_p); u\Bigr)\\
&\quad\quad\quad\quad\quad\quad\quad\quad\quad\quad\quad\quad\quad\quad\quad\quad\quad\quad\quad\quad\quad    \in N_p \PCob_\theta^{2,\rst} \times \Delta^p \subset B \PCob_\theta^{2,\rst},
\end{align*}
equipped with the path from the image of this point in $B \Cob_\theta^{2}$ to the base-point $M$ given by
\begin{align*}
& t \longmapsto  \Bigl(M \xrightarrow{[0,a_1-a_0] \times M} \cdots \xrightarrow{[0,a_p-a_{p-1}] \times M} M \xrightarrow{[0,1-a_p] \times M} M; (1-t)u,t\Bigr)\\
&\quad\quad\quad\quad\quad\quad\quad\quad\quad\quad\quad\quad\quad\quad\quad\quad\quad\quad\quad\quad\quad   \in N_p \Cob_\theta^{2,\rst} \times \Delta^{p+1} \subset B \Cob_\theta^{2,\rst}.
\end{align*}
Both these formulas respect the semi-simplicial identities and so descend to maps on geometric realisations.

The remaining maps in the diagram \eqref{eq:Sec7Main} are $\tau$ and $\sigma$. The map $\tau: \Cob_\theta^2 (N,M) \to \Omega_{N,M} B \Cob_\theta^2$ is the tautological one, which sends a point $W$ to the path $t \mapsto (W,(1-t,t)) \in N_1 \Cob_\theta \times \Delta^1 \to B \Cob_\theta^2$. The map $\sigma$ is the ``fibre transport'' map; it sends a path $\gamma: N \leadsto M$ in $B\Cob_\theta^2$ to the point $(N,g_N) \in B \PCob_\theta^2$, together with the path $\gamma$ from the image $N$ of $(N,g_N)$ under $BF^{2,\rst}$ to $M$. 

\subsubsection{Properties of the diagram}

We will first show that the horizontal maps in the leftmost square of \eqref{eq:Sec7Main} are weak equivalences. The bottom map can be treated quite easily.

\begin{lem}\label{lem:proof-factorzation1}
The augmentation map $\norm{\zeta_\bullet}: \norm{X_\bullet} \to \Cob_\theta^2 (N,M)$ is a weak homotopy equivalence.
\end{lem}

\begin{proof}
Each $\zeta_p$ is a fibration. Hence by Lemma 2.14 of \cite{SxTech}, it suffices to prove that $\norm{\zeta_\bullet^{-1} (t,W, \ell_W)} \simeq *$, for each $(t,W, \ell_W)  \in \Cob_\theta^2 (N,M)$. But $\zeta_\bullet^{-1}(t,W, \ell_W)$ can be identified with the nerve of the fibre category $(-t,g_N) / i$ of the inclusion functor $i: \cP(W,(0,1))^{\rst} \to \cP(W,(-t-\epsilon,-t+\epsilon)\cup (0,1))^{\rst}$. By the same argument as in the proof of Lemma \ref{lem:VaryingJ}, this fibre category has contractible nerve.
\end{proof}

To show that the top horizontal map in the leftmost square of \eqref{eq:Sec7Main} is a weak equivalence we will first develop some properties of the square \eqref{diag.definition-of-zbullet}, which will also be used in showing that the middle square of \eqref{eq:Sec7Main} is homotopy cartesian. In the following lemma we refer to a semi-simplicial map being homotopy cartesian: see Definition 2.9 of \cite{SxTech} for this notion.

\begin{lem}\label{proof-diffeomorphismactionlemma}
\mbox{}
\begin{enumerate}[(i)]
\item The square \eqref{diag.definition-of-zbullet} is homotopy cartesian, for each $p$.
\item The semi-simplicial map $\xi_\bullet$ is homotopy cartesian. 
\end{enumerate}
\end{lem}

\begin{proof}
By an application of Theorem \ref{thm:improved-chernysh-theorem} the map $\xi_p$ is a fibration, so (i) follows since \eqref{diag.definition-of-zbullet} is a pullback. The fibre of $\xi_p$ over $((a_0,g_0), \ldots, (a_p,g_p), h_0, \ldots, h_p)$ is the space $\Riem^+ ([a_p,1] \times M)_{g_p,g_M }$. Face maps either induce the identity on fibres, or they glue on $h_p$ and so are weak equivalences. 
\end{proof}

\begin{cor}\label{cor:proof-diffeomorphisaction}
The square 
\[
\xymatrix{
\norm{Z_\bullet} \ar[r] \ar[d]^{\norm{\lambda_\bullet}} & \norm{Y_\bullet} \ar[d]^{\norm{\xi_\bullet}} \\
\norm{X_\bullet} \ar[r]^-{\norm{\eta_\bullet}} & B \cP (\bR \times M, (0,1))^{\rst}
} 
\]
is homotopy cartesian, and the semi-simplicial map $\lambda_\bullet$ is also homotopy cartesian.
\end{cor}

\begin{proof}
This follows from Lemma \ref{proof-diffeomorphismactionlemma} and the following general fact: let 
\[
 \xymatrix{
 A_\bullet \ar[r]^{k_\bullet} \ar[d]^{h_\bullet}& B_\bullet \ar[d]^{g_\bullet}\\
 C_\bullet \ar[r]^{f_\bullet} & D_\bullet
 }
\]
be a commutative diagram of semi-simplicial spaces which is levelwise homotopy cartesian and assume that the semi-simplicial map $g_\bullet$ is homotopy cartesian. Then $h_\bullet$ is homotopy cartesian, and the square of geometric realisations is homotopy cartesian. This is quickly proven using Theorem 2.12 of \cite{SxTech} (which is due to Segal). 
\end{proof}

We can now show that the top horizontal map in the leftmost square of \eqref{eq:Sec7Main} is a weak equivalence.

\begin{cor}\label{cor:proof-diffeomorphisaction2}
The augmentation map $\norm{\varphi_\bullet}: \norm{Z_\bullet} \to \PCob_\theta^2 ((N,g_N),(M,g_M))$ is a weak homotopy equivalence. 
\end{cor}

\begin{proof}
Consider the diagram
\begin{equation}\label{diagram:proof-factorizationthm3}
\begin{gathered}
 \xymatrix{
 Z_0 \ar[d]^{\lambda_0}\ar[r] & \norm{Z_\bullet} \ar[r]^-{\norm{\varphi_\bullet }} \ar[d]^-{\norm{\lambda_\bullet}} & \PCob_\theta^2 ((N,g_N),(M,g_M)) \ar[d]^-{F}\\
 X_0 \ar[r] & \norm{X_\bullet} \ar[r]^{\norm{\zeta_\bullet}}_{\simeq} & \Cob_\theta^2 (N,M).
 }
\end{gathered}
\end{equation}
By Corollary \ref{cor:proof-diffeomorphisaction} the semi-simplicial map $\lambda_\bullet$ is homotopy cartesian and hence the left hand square is homotopy cartesian (using Theorem 2.12 of \cite{SxTech} again). We claim that the large rectangle is also homotopy cartesian. Since $X_0 \to \norm{X_\bullet}$ is $0$-connected for general reasons, it follows that the right hand square is homotopy cartesian, too. Since $\norm{\zeta_\bullet}$ is a weak equivalence, it follows that $\norm{\varphi_\bullet}$ is a weak equivalence.

To verify that the large rectangle is homotopy cartesian, argue as follows. Firstly, note that both $\lambda_0$ and $F$ are fibrations. Furthermore, a point $x \in X_0$ is given by $((t,W,\ell_W),(a_0,g_0),h)$, where $(t,W,\ell_W) \in \Cob_\theta^2 (N,M)$, $a_0 \in (0,1)$, and $h \in \Riem^+ (W|_{[-1,a_0]})^{\rst}_{g_N,g_0}$ is right stable. The fibre of $\lambda_0$ over such a point is $\Riem^+ ([a_0,1] \times M)_{g_0,g_M}$.

Under the bottom composition in \eqref{diagram:proof-factorizationthm3}, the point $x$ is mapped to $W \in \Cob_\theta^2(N,M)$, and the fibre over $W$ of the right vertical map is of course $\Riem^+ (V)_{g_N,g_M}$. The map on fibres can be identified with the gluing map $\mu (h,\_): \Riem^+ ([a_0,1] \times M)_{g_0,g_M} \to \Riem^+ (V)_{g_N,g_M}$, and this is a weak equivalence because $h$ was assumed to be right stable.
\end{proof}

The middle square of \eqref{eq:Sec7Main} is obtained by taking the homotopy cartesian square of Corollary \ref{cor:proof-diffeomorphisaction} and composing it with the map $\epsilon$. The following lemma therefore implies that it is again homotopy cartesian.

\begin{lem}\label{lem:proof-factorzation0}
The map $\epsilon$ in \eqref{eq:epsilon} is a weak homotopy equivalence. 
\end{lem}

\begin{proof}
The map under question is the composition of two maps 
$$B \cP(\bR \times M,(0,1))^{\rst} \lra \hofib_{M } BF^{2,1,\rst} \lra \hofib_{M } BF^{2,\rst}.$$
The first is defined by the same formula as $\epsilon$ and is a homotopy equivalence by Theorem \ref{thm:Fibre}. The second is a weak equivalence by Theorems \ref{thm:grw-thm4.1} and \ref{thm:SurgBelowMid}.
\end{proof}

The final, leftmost, square of \eqref{eq:Sec7Main} is \emph{defined} to be a homotopy pullback. The following lemma then supplies the final step of the argument.

\begin{lem}\label{lem:proof-factorzation2}
The bottom semi-circle in the diagram \eqref{eq:Sec7Main} commutes up to preferred homotopy.
\end{lem}

\begin{proof}
This is by a straightforward, but tedious checking. Let
\[
x:=((t,W,\ell_W),(a_0,g_0), \ldots, (a_p,g_p), h_0, \ldots, h_p; u_0, \ldots,u_p)\in X_p \times \Delta^p. 
\]
Depending on a homotopy parameter $r \in [0,1]$, we define a point 
\[
\gamma(x,r) \in N_{p+1}\PCob_\theta^{2,\rst} \times \Delta^{p+1}
\]
as
\begin{align*}
&\Bigl( (N,g_N) \xrightarrow{(W|_{[-1,a_0]},h_0)} (M,g_0) \xrightarrow{([a_0,a_1]\times M,h_1)} \cdots \xrightarrow{([a_{p-1},a_p]\times M,h_p)} (M,g_p) ;\\
&\quad\quad\quad\quad\quad\quad\quad\quad\quad\quad\quad\quad\quad\quad\quad\quad\quad\quad\quad\quad\quad\quad\quad\quad (1-r), ru_0, \ldots ,ru_p \Bigr),
\end{align*}
and if in addition $s \in [0,1]$, we define a point in $N_{p+2} \Cob_\theta^2  \times \Delta^{p+2}$ by the formula
\begin{align*}
&\Gamma(x,r,s):= \Bigl( N \xrightarrow{W|_{[-1,a_0]}} M \xrightarrow{[a_0,a_1]\times M} \cdots \xrightarrow{[a_{p-1},a_p]\times M} M\xrightarrow{[a_p,1]\times M} M ;\\
&\quad\quad\quad\quad\quad\quad\quad\quad\quad\quad\quad\quad\quad\quad\quad\quad\quad (1-s)(1-r), (1-s)ru_0, \ldots ,(1-s)ru_p ,s\Bigr).
\end{align*}
Both these formulas respect the semi-simplicial identifications and hence give rise to well-defined continuous maps $\gamma: \norm{X_\bullet} \times [0,1] \to  B \PCob_\theta^{2,\rst}$ and $\Gamma: \norm {X_\bullet} \times [0,1]^2 \to B \Cob_\theta^2$.
Note that $\Gamma (x,r,0)= BF^{2,\rst}(\gamma(x,r))$ and that $\Gamma(x,r,1)$ is the base-point $M \in N_0 \Cob_\theta^2$. Therefore, $\gamma$ and $\Gamma$ together define a homotopy $\Lambda:\norm{X_\bullet} \times [0,1] \to \hofib_M BF^{2,\rst}$.

For $r=0$, $\Lambda (-,0)$ is the map $\sigma \circ \tau \circ \norm{\zeta_\bullet}$, and for $r=1$, $\Lambda(-,1)$ is the map $\epsilon \circ \eta$. 
\end{proof}

\subsection{A finiteness theorem for the orbit map}\label{subsec:finitenessheorem}

\begin{proof}[Proof of Theorem \ref{thm:finiteness-of-orbitmap}]
The tangential $2$-type of $W$ is either $B \Spin (d) \to B\mathrm{O}(d)$ or $B\mathrm{SO}(d)\to B\mathrm{O}(d)$. We give the proof in the case where it is $B \mathrm{SO}(d)$, the proof in the other case is the same, up to change of notation. 
By Theorem \ref{mainthmintro:diffaction}, the map $\sigma_h$ factors up to homotopy through the map
\[
\Omega \alpha_W : \Diff_\partial (W) \lra \Omega^{\infty+1} \MTSO(d),
\]
and it will suffice to prove that the latter map has finite image on homotopy groups. Equivalently, we may consider 
\[
\alpha_W : B \Diff_\partial (W) \lra \Omega^{\infty} \MTSO(d) . 
\]
The homotopy groups of $\MTSO(d)$ are finitely generated (we have $\pi_* (\MTSO(d))=0$ for $* < -d$ and $H_k (\MTSO(d);\bZ) \cong H_{k+d} (B\mathrm{SO}(d);\bZ)$ is finitely generated for each $k \in \bZ$; the claim then follows by the Hurewicz theorem modulo the Serre class of finitely generated abelian groups).  Therefore it suffices to prove that $\alpha_W: B \Diff_\partial (W) \to \Omega^\infty \MTSO(d)$ induces the trivial map in rational homotopy (in degree $1$, this is to be understood as the statement that $\pi_1 (B \Diff_\partial W)\to \pi_1 (\Omega^\infty \MTSO(d)) \otimes \bQ$ is the zero map). 

We consider $W$ as a cobordism $\emptyset \leadsto M$. For each oriented cobordism $V: M \leadsto \emptyset$, there is a homotopy commutative diagram
\[
\xymatrix{
B \Diff_\partial (W) \ar[r]\ar[d] & \Omega_{\emptyset,M} B \Cob_{\mathrm{SO}(d)}\ar[d] \ar[r]^{\simeq} & \Omega^\infty \MTSO(d) \ar@{=}[d]\\
B \Diff^{+} (W \cup V) \ar[r] & \Omega_{\emptyset,\emptyset} B \Cob_{\mathrm{SO}(d)} \ar[r]^{\simeq}& \Omega^\infty \MTSO(d);\\
}
\]
the left vertical map is given by extending diffeomorphism over $V$ identically, the left horizontal maps are the tautological maps, and the middle vertical map is concatenation with a fixed path and hence a weak equivalence. The top composition is by definition $\alpha_W$, and the bottom composition is $\alpha_{W \cup V}$. Hence it is enough to show that for appropriate choice of $V$ the map $\alpha_{W \cup V}$ is zero on rational homotopy groups. The design criteria for $V$ are the following:
\begin{itemize}
\item $N:=W \cup V$ is connected, has trivial rational Pontrjagin classes and Euler number zero. 
\end{itemize}
We first show that for each $d$-dimensional closed oriented $N$ satisfying the above requirements, the map $\alpha_N$ is trivial on rational homotopy. 

Being homotopy equivalent to a connected infinite loop space, the path component $\Omega^\infty_N \MTSO (d)\subset \Omega^\infty \MTSO(d)$ hit by $\alpha_N$ splits rationally as a product of Eilenberg--Mac~Lane spaces, with factors given by certain Miller--Morita--Mumford classes $\kappa_c$ with $c$ a monomial in the Euler and Pontrjagin classes. To prove the claim about the triviality of $\alpha_N$ in rational homotopy, it is therefore enough to show that for each oriented fibre bundle $\pi:E \to S^k$ with fibre a closed manifold $N$ as above, all the classes $\kappa_c (E) := \pi_! (c(T_v E))\in H^* (S^k;\bQ)$ are trivial. 

The argument we give is a variant of a well-known one that appears for example in \cite[Proposition 1.9]{HSS}. Consider the Leray--Serre spectral sequence of $\pi$ (in rational cohomology). Under the map $H^{4i} (E;\bQ) \to H^{4i}(N;\bQ)$, the Pontrjagin classes $p_i (T_v E)$ map to $p_i (TN)$, similarly for the Euler class $e(T_v E)$. By our hypothesis on $N$, these images are zero. It follows from the product structure of the Leray--Serre spectral sequence that all nontrivial products of the Euler and Pontrjagin classes are trivial. Therefore, for each monomial $c$ in the classes $p_i$ and $e$ which is decomposable, the class $\kappa_c (\pi)\in H^* (S^k;\bQ)$ is zero. This leaves the possibility that one of the classes $\kappa_{p_i} (\pi)$, $i < d/2$ or $\kappa_e (\pi)$ is nonzero. For degree reasons, the only cases which can occur are the following three. Firstly, $\kappa_e (\pi) \in H^0 (S^k;\bQ)$ is just the Euler characteristic of $N$, and hence trivial by the assumption on $\chi(N)$. Secondly, if $d=4i$, we have $\kappa_{p_i}(\pi) \in H^0 (S^k; \bQ)$, but that is also trivial, since it is the Pontrjagin number of $N$. Thirdly, if $d+k = 4j$, we have $\kappa_{p_j} (\pi)\in H^k (S^k; \bQ)$. On the other hand, since $TS^k$ is stably trivial, we have for each polynomial $c$ in the Pontrjagin classes that $c(T_v E)= c(TE)$ and hence that 
\[
\langle \kappa_c(\pi), [S^k] \rangle = \langle c(T_v E), [E] \rangle = \langle c(T E), [E] \rangle.
\]
Now we use the Hirzebruch signature theorem and the fact that the $j$th Hirzebruch $L$-class has the form $L_j = a p_j + q$ with $a$ a non-zero rational number and $q$ a sum of decomposable monomials in the Pontrjagin classes. Hence 
\[
\mathrm{sign} (E)= \langle L_j (TE) ,[E]\rangle  = a \langle \kappa_{p_j} (\pi),[S^k]\rangle + \langle \kappa_q(\pi),[S^k] \rangle = a \langle \kappa_{p_j} (\pi),[S^k]\rangle,
\]
using that $q(T_v E)=0$ by the previous argument. Finally, it is a classical fact that the signature of the total space of an oriented fibre bundle over $S^k$ is zero (see \cite{ChernHirzebruchSerre} for $k \geq 2$ and \cite{Neumann} for $k=1$). 

It remains to construct the cobordism $V: M \leadsto \emptyset$. As a first approximation, we may take the opposite cobordism $W^{op}$. The double $dW = W \cup_M W^{op}$ has trivial rational Pontrjagin classes. This follows because there is a nullbordism $B: \emptyset \leadsto dW$ such that the inclusion $W \to B$ is a homotopy equivalence: one constructs $B$ by smoothing the corners of $W \times [0,1]$. Since $W$ has trivial rational Pontrjagin classes, so does $B$, and hence so does $dW$. 

To adjust the Euler number of $dW$ when $d$ is even, we take connected sum with copies of $S^k \times S^{d-k}$ in the interior of $W^{op}$. Since taking such a connected sum changes the Euler number by $\pm 2$, depending on the parity of $k$, and $\chi(dW)= 2 \chi (W)$ is even, we can force the Euler number to be zero (note that we need to assume $d \geq 4$ in order to get a connected manifold). Finally, we show that if $N$ has trivial rational Pontrjagin classes, then so does $N \sharp (S^k \times S^{d-k})$. 
The map $H^i (N \sharp (S^k \times S^{d-k});\bQ) \to H^i (N \setminus D^d;\bQ) \oplus H^i (S^k \times S^{d-k} \setminus D^d;\bQ)$ is injective, except if $i=d$. So all Pontrjagin classes of $N \sharp (S^k \times S^{d-k})$ in degrees $<d$ vanish. To see that $p_j$ in the case $d=4j$ vanishes, we again use the fact that the coefficient of $p_j$ in $L_j$ is nonzero, the signature theorem; and the fact that $\mathrm{sign}(N \sharp (S^k \times S^{d-k}))=\mathrm{sign}(N)=0$. 
\end{proof}

\section{Delooping the index difference}\label{sec:indextheory}

\subsection{Background material}

\subsubsection*{The Rosenberg--Dirac operator} 

In this section, we let $G$ be a discrete group (which in the cases of interest will be finitely presented) and consider the fibration 
\[
\theta_d: B \Spin (d) \times BG \stackrel{\mathrm{pr}}{\to} B \Spin (d) \to B\mathrm{O}(d).
\]
Let $(W,g)$ be a $d$-dimensional Riemannian manifold and let $\ell$ be a $\theta_d$-structure on $W$. This is the same data as a map $f=(f_0,f_1): W \to B \Spin (d) \times BG$, together with an isometric vector bundle isomorphism $TW \cong f_0^* \gamma_{\Spin(d)}$. Let us recall how the Rosenberg--Dirac operator $\Dir_{g,f}$ is constructed out of these data. 

We use the conventions for Clifford algebras which are spelled out in Section 2.1 of \cite{JEIndex1}. That is, $\Cl^{p,q}$ is the complex algebra generated by anticommuting elements $(e_1, \ldots,e_p,\varepsilon_1, \ldots, \varepsilon_q)$ with $\varepsilon_j^2 = - e_i^2 =1$. It has a unique $C^*$-algebra structure such that $e_i^* = -e_i$ and $\varepsilon_j^* = \varepsilon_j$. There is a unique Real structure on $\Cl^{p,q}$ such that the generators are Real, and a unique $\bZ/2$-grading such that the generators are odd. Further, we let 
\[
\spinor_d := E \Spin (d) \times_{\Spin(d)} \Cl^{d,0}\to  B \Spin (d) 
\]
be the universal spinor bundle. It comes with a canonical $\Cl^{d,0}$-valued inner product, a Real structure and a $\bZ/2$-grading. There is a natural Clifford action by the universal spin vector bundle $\gamma_{\Spin(d)}\to B \Spin (d)$, and $\spinor_W := f_0^* \spinor_d$ is the spinor bundle on $W$. 

Let $\cstar(G)$ be the group $C^*$-algebra of $G$, which could be either the reduced version $\cstarred (G)$ or the maximal version $\cstar_{\mathrm{m}}(G)$ (as all arguments in this section apply equally to both versions, there is no need for a notational distinction). There is a unique Real structure on $\cstar(G)$ such that all group elements $g \in G \subset \cstar(G)$ are Real, and $\cstar(G)$ has the trivial $\bZ/2$-grading. The universal \emph{Mishchenko line bundle} $\cL_G := EG \times_G \cstar (G) \to BG$ is a bundle of right $\cstar(G)$-modules which are free of rank $1$. It comes with a canonical $\cstar(G)$-valued inner product, which turns $\cL_G$ into a bundle of finitely generated projective Hilbert-$\cstar(G)$-modules. 
Now we write 
\[
 \spinor_\ell:= f^* ( \spinor_d \boxtimes \cL_G)\cong  \spinor_W \otimes f_1^* \cL_G \lra W,
\]
which is a bundle of Real graded finitely generated projective Hilbert $ \cstar(\theta_d)$-modules, where we used the abbreviation
\[
\cstar (\theta_d):= \Cl^{d,0} \otimes \cstar(G),
\]
with Real structure and grading induced from the two factors (since $\Cl^{d,0}$ is finite-dimensional, the tensor product is unambiguous). 

The spinor bundle $\spinor_W$ inherits a connection from the Levi-Civita connection on $W$, while $f_1^* \cL_G$ has a natural flat connection since the structure group of $\cL_G$ is the discrete group $G$. Let $\nabla$ be the tensor product of both these connections. Using $\nabla$ and the action of the Clifford algebra bundle of $TW$ on $\spinor_W$, one defines the Rosenberg--Dirac operator $\Dir_{g,\ell}$ by the usual formula \cite[\S 1]{RosNovI} \cite[\S II.5]{SpinGeometry}. In the sequel, the $\theta_d$-structure is usually fixed, which is why we write just $\Dir_g$ for $\Dir_{g,\ell}$. The Lichnerowicz--Schr\"odinger formula
\begin{equation}\label{eq:schroedinger-lichnerowicz}
 \Dir_g^2 = \nabla^* \nabla + \frac{1}{4} \scal(g)
\end{equation}
connects the Rosenberg--Dirac operator to positive scalar curvature. 

Let $\Gamma_c (W;\spinor_\ell)$ be the space of compactly supported smooth sections of $\spinor_\ell$. The $\cstar(\theta_d)$-valued inner product on $\spinor_\ell$ and the volume measure on $W$ together turn $\Gamma_c (W;\spinor_\ell)$ into a pre-Hilbert-$\cstar(\theta_d)$-module, whose completion we denote by $L^2 (W;\spinor_\ell)$, which is a Hilbert-$\cstar(\theta_d)$-module. The Dirac operator defines an unbounded symmetric operator on $L^2 (W;\spinor_\ell)$ with inital domain $\Gamma_c (W;\spinor_\ell)$. Its closure is self-adjoint and regular (in the sense of the theory of unbounded operators Hilbert modules \cite[\S 9]{Lance}) provided that there is a proper function $x: W \to \bR$ such that the commutator $[\Dir_g,x]$ is a bounded operator, by a variant of the classical Chernoff--Wolf theorem which is proven in \cite[\S 2]{HPS} or as Theorem 2.14 of \cite{JEIndex1}. Usually, $\Dir_g$ is not Fredholm (again, this term needs to be understood in the sense of Hilbert module operators), but it is if the scalar curvature of $g$ is uniformly positive outside a compact subset of $W$, by \eqref{eq:schroedinger-lichnerowicz} and e.g.\ Theorem 3.41 of \cite{JEIndex1}. In \cite[
\S 2]{JEIndex1}, it is shown how to generalise these constructions to families. 

\subsubsection*{$K$-Theory spectra}
We shall use a model for the real $K$-theory spectrum $\bK(\gA)$ of a Real graded unital $C^*$-algebra $\gA$ which is a variant of unbounded $KK$-theory and of the spectral picture of $K$-theory developed in \cite{HigsonGuent}. The precise version is described in \cite{JEIndex1}, \cite{JEIndex2}, and we refer to these papers for more details. It is a spectrum of sheaves. The $n$th sheaf in this spectrum assigns to a test manifold the set\footnote{The usual set-theoretic difficulties can be resolved by working in a Grothendieck universe, as done in \cite[\S 4]{JEIndex1}.} $\bK(\gA)_n (X)$ of all tuples $(H,\eta,c,D)$, where 
\begin{enumerate}[(i)]
\item $H$ is a countably generated continuous field of Hilbert-$\gA$-modules over $X$ with a Real structure,
\item $\eta$ is a grading and $c$ is a $\Cl^{n,0}$-structure on $H$, both compatible with the Real structure, and 
\item $D$ is a self-adjoint regular Fredholm family on $H$, which is Real, odd and $\Cl^{n,0}$-antilinear (we refer to Sections 2.1 and 4 of \cite{JEIndex1} for more details).
\end{enumerate}
It is often convenient to shorten notation and to write $(H,D)$ for the tuple $(H,\eta,c,D)$ when $\eta$ and $c$ are understood, or even just $D$. An element $(H,\eta,c,D) \in \bK(\gA)_n (X)$ is \emph{degenerate} if $D$ is invertible, and we denote by $\bD(\gA)_n \subset \bK(\gA)_n$ the subsheaf of degenerate elements. Here, $D$ being invertible means that the bounded operator family $\frac{D}{\sqrt{D^2+1}}$ is invertible. This condition is equivalent to the existence of a continuous function $c: X \to (0,\infty)$ such that $D^2 \geq c$.

As base-point in $\bK(\gA)_n$ and $\bD(\gA)_n$, we take the unique tuple with $H=0$. 

\begin{rem}\label{remark:simplifynotation}
For the rest of this section, we are concerned with maps of sheaves of the form $\Phi:\cX \to \cK$ for sheaves $\cX$ which are either $K$-theory sheaves or are given by families of manifolds together with $\theta_d$-structures, Riemannian metrics and some other data. The sheaves $\cK$ are the $K$-theory spaces for certain $C^*$-algebras or the subspaces of degenerate tuples. 
To improve the readability (and writeabilty) of this section, let us impose the following conventions. We only write down the map $\Phi_{\ast}: \cX(\ast) \to \cK(\ast)$ of sets. In each case, Example 3.28 and Theorems 3.40 and 3.41 of \cite{JEIndex1} will give the justification why $\Phi$ extends to a map of sheaves, and this will not be spelled out explicitly.

Similarly, when we want to show that two such maps $\Phi_0,\Phi_1: \cX \to \cK$ are homotopic, we have to produce a concordance $\gy$ from $\Phi_0 (\gz) $ to $\Phi_1 (\gz)$, for each $\gz \in \cX(X)$, and this concordance has to be natural in $\gz$. We only write down this concordance when $X=\ast$. In that case, it must be an element of $\cK(\bR)$. We write this down pointwise, i.e. we write elements $\gy_t \in \cK(\ast)$, depending on $t \in \bR$ and depending \emph{naturally} on $\gz$. Again, we will refer implicitly to Example 3.28 and Theorems 3.40 and 3.41 of \cite{JEIndex1} for the justification that such a procedure indeed yields a natural concordance.
\end{rem}

The spectrum structure on the collection of all $\bK(\gA)_n$ is given by the \emph{Bott map}
\[
 \bott: \bK(\gA)_n \lra \Omega \bK(\gA)_{n+1}
\]
described as follows. It sends $(H,\eta,c,D) \in \bK(\gA)_n (*)$ to the path (parametrised by $\overline{\bR}$)
\[
 t \mapsto
 \begin{cases}
  \Bigl( H \oplus H, \twomatrix{\eta}{}{}{-\eta},\twomatrix{c}{}{}{c}, \twomatrix{}{-\eta}{\eta}{},\twomatrix{D}{t\eta}{t\eta}{D}\Bigr) & t \in \bR,\\
(0,  \_,\_,\_) & t = \pm \infty.
 \end{cases}
\]
This notation is to be understood in the following way: the matrix
\[
\twomatrix{\eta}{}{}{-\eta}
\]
gives a grading on $H \oplus H$, and the vector $w + se_{n+1} \in \bR^n \oplus \bR\subset \Cl^{n+1,0}$ acts as 
\[
\twomatrix{c(w)}{-s\eta}{s\eta}{c(w)}.
\]
For an explanation what the loop space of a sheaf is, and for the rigorous definition of the Bott map, we refer to Sections 2.1 and 2.3 of \cite{JEIndex2}. It is clear that the Bott maps restricts to a map $\bD(\gA)_n \to \Omega \bD(\gA)_{n+1}$.
\begin{prop}\mbox{}
\begin{enumerate}[(i)]
\item The sheaf $\bD(\gA)_n$ is contractible \cite[Lemma 3.9]{JEIndex1}. 
\item The Bott map is a weak equivalence \cite[Theorem 3.14]{JEIndex1}. 
\end{enumerate}
\end{prop}
The second part is a consequence of the Bott periodicity theorem. 

For a graded $\cstar$-algebra, there are Morita equivalences 
\[
\morita:\bK(\gA)_n \stackrel{\simeq}{\lra} \bK(\gA \grotimes \Cl^{1,0})_{n+1}
\]
(we use the standard conventions for tensor products of graded $\cstar$-algebras and Hilbert modules as explained in Section 14 of \cite{Bla}) which commute with the Bott map. When combined with the Bott maps, these give rise to weak equivalences 
\[
| \bK(\gA \grotimes \Cl^{1,0})| \simeq \Omega |\bK(\gA) |
\]
of spectra, and in particular
\[
 |\bK(\cstar(\theta_d))| \simeq \Omega^d |\bK(\cstar(G)|.
\]
We shall need the formula for $\morita$ and its inverse when $n=0$, which are summarised in the next lemma.

\begin{lem}\label{lem:inverse-morita}
The map 
$$\morita: \bK(\gA)_0 \lra \bK(\gA \grotimes \Cl^{1,0})_{1}$$ 
sends $(H,\eta,.,D)$ to $(H \otimes \Cl^{1,0},\eta \otimes \iota, \eta \otimes e, D \otimes 1)$, where $\iota$ is the grading on $\Cl^{1,0}$ and $e$ is given by left multiplication with the generator $e \in Cl^{1,0}$. 
A homotopy inverse is given by the following recipe: let $(H,\nu,c,D) \in \bK(\gA \grotimes  \Cl^{1,0})_1$ be given. For $x \in \Cl^{1,0}$, let $\rho(x)$ be the right multiplication by $1 \otimes x \in \gA \grotimes \Cl^{1,0}$. The eigenspace $\Eig(c(e)\rho(e)\nu)$ is an $\gA$-subspace, to which $\nu$ and $D$ restrict. \qed
\end{lem}

The sheaves $\bK(\gA)_n$ have a natural $\Gamma$-space structure as follows. 
\begin{defn}
For a finite pointed set $S$, we let $ \bK(\gA)_n (S)$ be the sheaf which assigns to $X$ the set of all $(H_s,D_s)_{s \in S_o}$ where $(H_s,D_s)\in \bK(\gA)_n(X)$. For a pointed map $f:S \to T$ of finite pointed sets, we let $f_* (H_s,D_s)$ be the tuple $(H'_t,D'_t)$, with $H'_t = \bigoplus_{f(s)=t} H_s $ and $D'_t = \bigoplus_{f(s)=t} D_s$.
\end{defn}

\begin{lem}\label{lem:gammaKO-properties}
The $\Gamma$-space $|\bK(\gA)_n|$ is very special. The Bott maps are maps of $\Gamma$-spaces, and so are the Morita equivalences. 
\end{lem}

\begin{proof}
It is obvious that $\bK(\gA)_n (0_+)$ is the one-point sheaf and that the Bott map is a map of $\Gamma$-spaces. Likewise it is clear that $\bK(\gA)_n (S \vee T) \to \bK(\gA)_n (S) \times \bK(\gA)_n (T)$ is a weak equivalence (it is even an isomorphism of sheaves). That $|\bK(\gA)_n|$ is grouplike follows from the isomorphism $\pi_0 (|\bK(\gA)_n|) = KO_{-n}(\gA)$ constructed in Proposition 4.12 of \cite{JEIndex1}. 
\end{proof}

\begin{cor}\label{cor:delooping-Ktheoryspectrum}
There is a weak equivalence of spectra
\begin{equation*}
\mathrm{B}^\infty |\bK(\gA \grotimes \Cl^{d,0})_n| \simeq  \bigl( \Sigma^{n-d} |\bK(\gA)| \bigr)  \langle -1\rangle.
\end{equation*}
\end{cor}

\begin{proof}
Lemma \ref{lem:gammaKO-properties} shows that the formula $S \mapsto |\bK(\gA \grotimes \Cl^{d,0})(S)|$ defines a special $\Gamma$-spectrum in the sense of \cite[Definition 5.1]{Nguyen}. It is also ``projectively fibrant'', which means nothing else than that the spectrum $|\bK(\gA \grotimes \Cl^{d,0})(S)|$ is an $\Omega$-spectrum for each $S \in \Gamma^\op$. Hence \cite[Proposition 5.2]{Nguyen} applies and yields a weak equivalence 
\[
\mathrm{B}^\infty |\bK(\gA \grotimes \Cl^{d,0})_n| \simeq |\bK(\gA \grotimes \Cl^{d,0})|[n] \langle -1 \rangle.
\]
The symbol $[n]$ denotes the shift, and the symbol $\langle -1 \rangle$ denotes the connective cover. The fact that the Morita equivalences commute with the Bott maps means that they give a stable equivalence of spectra
\[
|\bK(\gA \grotimes \Cl^{d,0})| \simeq |\bK(\gA)|[-d]. 
\]
Hence 
\[
|\bK(\gA \grotimes \Cl^{d,0})|[n] \langle -1 \rangle \simeq |\bK(\gA)|[n-d] \langle -1 \rangle \simeq \Sigma^{n-d} |\bK(\gA)| \langle -1 \rangle.
\]
\end{proof}

\subsubsection*{The index difference}
It will be convenient for us to adopt the following definition of a generalised base-point, as we will often consider paths in $\bK(\gA)_n$ which start and end in the contractible subspace $\bD(\gA)_n$.

\begin{defn}\label{rem:generalizedbasepoint}
A \emph{generalised base-point} of a space $X$ is a contractible subspace $Z \subset X$. A map $f: X_0 \to X_1$ of spaces equipped with generalised base-points $Z_0$ and $Z_1$ is said to preserve the generalised base-points if $f(Z_0 ) \subset Z_1$. An analogous definition is made for sheaves. 

In the case of spaces, we define the loop space $\Omega_Z X$ as the space of all paths beginning and ending in $Z$, and the homotopy fibre $\hofib_{Z_1}(f)$ as the space of all pairs, consisting of a point $x_0 \in X_0$ and a path in $X_1$ from $f(x_0)$ to a point of $Z_1$. The homotopy fibre has a generalised base-point, namely the space of all pairs, consisting of a point in $Z_0$, together with a path in $Z_1$ starting at the image. 

In the case of sheaves, we define the loop sheaf $\Omega_\cZ \cX$ as the sheaf which sends a test manifold $X$ to the set of all $z \in \cX(X \times \bR)$ such that $z|_{X \times [-1,1]^c} \in \cZ(X \times [-1,1]^c)$. The homotopy fibre $\hofib_{\cZ_1}(f)$ is the sheaf which sends a test manifold $X$ to the set of all pairs $(z,y)$, where $z \in \cX_0 (X)$ and $y \in \cX_1 (X \times \bR)$ is such that $y|_{X \times 0}= f(z)$ and $y|_{X \times (1,\infty)} \in \cZ_1 (X \times (1,\infty))$. 

If the choice of $Z$ or $\cZ$ is unambiguous, it is dropped from the notation. 
\end{defn}

As a generalised base-point in $\bK(\gA)_n$, we shall always take $\bD(\gA)_n$. 
Of course, the homotopy types of these generalised loop spaces/homotopy fibres coincide with the usual ones. There are natural weak equivalences
\[
| \Omega_\cZ \cX | \lra \Omega_{|\cZ|} |\cX|
\]
and 
\[
 |\hofib_{\cZ_1} (f)| \lra \hofib_{|\cZ_1|} (|f|).
\]

With these generalities understood, we turn our attention to psc metrics and let $V:M_0 \leadsto M_1$ be a $d$-dimensional cobordism with $\theta$-structure $\ell$, let $g_i \in \Riem^+ (M_i)$ and let $h$ be a Riemannian metric on $V$ which restricts to $dx^2 +g_i$ near $M_i$. We form the elongation $W:= ((-\infty,0] \times M_0) \cup V \cup ( [1,\infty) \times M_1 )$ and extend the metric $h$ cylindrically, to a metric with the same name $h$. By the Lichnerowicz--Schr\"odinger formula, $\Dir_h$ is invertible at infinity and hence Fredholm. So
\[
 \ind (W,h) := (L^2 (W;\spinor_\ell) ,\Dir_h) \in \bK(\cstar(\theta_d))_0 (\ast)
\]
is well-defined and yields
\[
[\ind (W,h)]\in \pi_0(|\bK(\cstar(\theta_d))_0|) \cong KO_{d}(\cstar(G)).
\]
If $h$ has positive scalar curvature, then $\Dir_h$ is invertible and so $[\ind (W,h)]=0$. The construction generalises to the parametrised case, and yields a map
\[
\ind:\Riem (V)_{g_0,g_1} \lra \bK(\cstar(\theta_d))_0
\]
of sheaves (here one views $\Riem (V)_{g_0,g_1}$ as the sheaf which takes a test manifold $X$ to the set of smooth maps $X \to \Riem (V)_{g_0,g_1}$). 

\begin{defn}
Let $V$, $g_0,g_1$ and $W$ be as above, fix $h_0 \in \Riem^+ (V)_{g_0,g_1}$ and a monotone smooth function $c: \bR \to [0,1]$ with $c(0)=0$ and $c(1)=1$. For $h \in \Riem^+ (V)_{g_0,g_1}$, the path 
\[
 \bR \ni s \longmapsto \ind (W,(1-c(s)) h+ c(s) h_0) 
\]
defines an element 
\[
 \inddiff_{h_0}(h) \in (\Omega \bK(\cstar(\theta_d))_0)(*). 
\]
The construction generalises to families and defines a map
\[
\inddiff_{h_0}: \Riem^+ (W)_{g_0,g_1}  \lra \Omega \bK(\cstar (\theta_d))_0 
\]
of sheaves, the \emph{index difference} with respect to $h_0$. Its homotopy class is independent of the choice of $c$. 
\end{defn}

\subsubsection{The lower index}\label{subsubsec:deloopedindex}

A key role for the rest of this section is played by the delooped index map constructed in \cite{JEIndex2}. Hence we recapitulate its construction here. Slightly informally, there is a map 
\[
\ind_1 : B \Cob_{\theta_d} \lra |\bK (\cstar(\theta_d))_1|, 
\]
constructed analytically, which is related to the usual family index map for closed manifolds $\ind_0$ by a homotopy-commutative diagram
\[
\xymatrix{
\Cob_{\theta_d} (\emptyset,\emptyset) \ar[r]^-{\ind_0} \ar[d] & |\bK(\cstar(\theta_d))_0|\ar[d]^{\bott}\\
\Omega B \Cob_{\theta_d} \ar[r]^-{\ind_1} & \Omega |\bK(\cstar(\theta_d))_1|.
}
\]
To implement this rigorously, one has to replace the source of $\ind_1$ by a homotopy equivalent space, which is a version of the sheaf $\cD_{\theta}=\cD_{\theta_d}$ from Definition \ref{defn:nakedDspace}, but with additional data in order to define $\ind_1$. The following is a variant of Definition 3.1 of \cite{JEIndex2}.
\begin{defn}\label{defn:sheafdthetaop}
We define the sheaf $\cD_\theta^{\op}$ by assigning to a test manifold $X$ the set of all $(W,h,b,x_0,r,C)$, where
\begin{enumerate}
\item $W \in \cD_\theta(X)$ (it comes with a fibrewise proper map $x:W \to \bR$), 
\item $h$ is a fibrewise Riemannian metric on $W$, 
\item $x_0 : X \to \bR$, $r: X \to (0,\infty)$ and $b: W \to (0,\infty)$ are smooth functions and 
\item $C: X \to (0,\infty)$ is a continuous function. 
\item Let $\Dir_h$ be the (fibrewise) Rosenberg--Dirac operator on $W$ which is associated with $h$ and the $\theta$-structure $ \ell$ on $W$. We require that the commutator $[b\Dir_h b,x]$ of the Dirac operator weighted with $b$ with multiplication by the function $x$ satisfies the estimate
\begin{equation}\label{eqn:estimate-commutator-dirac}
 \norm{[b\Dir_h b,x]}\leq C.
\end{equation}
\end{enumerate}
\end{defn}
The superscript ${}^\op$ stands for ``operator''; the logic of this terminology is that the points in $\cD_\theta^{\op}$ are $\theta$-manifolds which are also equipped with Riemannian metrics (and hence Dirac operators). The other data and conditions are to make the definition of $\ind_1$ possible. The space $\cD_\theta^\op$ has an obvious structure of a special $\Gamma$-space, similar to that of Definition \ref{def:CobCatGamma}.
\begin{lem}\label{lem:forgetfulmapweakequivalence-d-dop}
The forgetful map $\cD_\theta^{\op}\to \cD_\theta$ is a weak homotopy equivalence of sheaves and a map of $\Gamma$-spaces. 
\end{lem}
\begin{proof}
It is obvious that the map is compatible with the $\Gamma$-space structures, and the proof that it is a weak equivalence is almost the same as that of Lemma 3.2 of \cite{JEIndex2}. The data $C,r,x_0$ are not present in loc.cit., but they are clearly a contractible choice and do not affect the proof that the forgetful map is a weak equivalence.
\end{proof}
Let us now give the construction of $\ind_1: \cD_\theta^\op \to \bK (\cstar(\theta_d))_1$. We describe what $\ind_1$ does on the points of $\cD_\theta^{\op}$, and refer to  \cite[\S 3.2]{JEIndex2} for the parametrised version. Let $\gx:=(W,h,b,x_0,r,C) \in \cD_{\theta}^{\op}(\ast)$. The canonical $\Cl^{1,1}$-module $\bS_1$ is given by
\[
 \bS_1 = ( \bR^2,\iota,e,\varepsilon):=\Bigl( \bR^2,\twomatrix{1}{}{}{-1},\twomatrix{}{-1}{1}{},\twomatrix{}{1}{1}{} \Bigr).
\]
Consider the Hilbert-$\cstar(\theta_d)$-module $L^2 (W;\spinor_\ell \otimes \bS_1) \cong L^2 (W;\spinor_\ell) \otimes \bS_1$ with grading $\eta \otimes \iota$, where $\eta$ is the grading on the bundle $\spinor_\ell$. We define a $\Cl^{1,0}$-action on $L^2 (W;\spinor_\ell \otimes \bS_1)$ by 
\[
e_1 \longmapsto \eta \otimes e= \twomatrix{}{-\eta}{\eta}{} 
\]
and consider the unbounded operator
\begin{equation}\label{eq:definition-operator-B}
 B:= b \Dir_h b \otimes 1 + \eta \otimes r(x-x_0) \varepsilon = \twomatrix{b \Dir_h b}{}{}{b \Dir_h b} + r(x-x_0) \twomatrix{}{\eta}{\eta}{}
\end{equation}
on $L^2 (W;\spinor_\ell \otimes \bS_1)$ with initial domain $\Gamma_c (W;\spinor_\ell\otimes \bS_1)$. This is formally self-adjoint, $\Cl^{1,0}$-antilinear and odd. The condition \eqref{eqn:estimate-commutator-dirac} implies that the closure of $B$ is a regular self-adjoint operator, see Lemma 3.5 of \cite{JEIndex2}. From loc.cit., we also collect the formula
\begin{equation}\label{eqn:square-of-operator}
B^2 = (b \Dir_h b)^2 \otimes 1 + r^2 (x-x_0)^2 + r [b \Dir_h b,x] \eta \otimes \varepsilon.
\end{equation}
It implies that\footnote{The meaning of this inequality is to be understood as in Theorem 3.40 of \cite{JEIndex1}. Namely, for each $s \in \Gamma_c (M;\spinor_\ell)$, we have $\scpr{B^2 s,s} \geq \scpr{(r^2 |x-x_0|^2-rC)s,s}$ in the $C^*$-algebra $\cstar(\theta_d)$.} $B^2 \geq r^2 |x-x_0|^2-rC$ and so $B$ is invertible at infinity and hence a Fredholm operator. For more details of this argument, we refer to Lemma 3.6 of \cite{JEIndex2}. We define 
\[
\ind_1 (\gx):=\Bigl( L^2 (W;\spinor_\ell\otimes \bS_1), \eta \otimes \iota,  \eta \otimes e, B\Bigr)\in \bK(\cstar(\theta_d))_1 (\ast).
\]
It is clear from the definitions that $\ind_1$ is a map of $\Gamma$-spaces and hence gives rise to a map of spectra
\[
\mathrm{B}^\infty \ind_1: \mathrm{B}^\infty |\cD^\op| \lra \mathrm{B}^\infty |\bK(\cstar(\theta_d))_1|.
\]
Lemma \ref{lem:forgetfulmapweakequivalence-d-dop} and the main result of \cite{Nguyen} give a weak equivalence 
\[
 \mathrm{B}^\infty |\cD^\op| \simeq \bigl( \Sigma \MT \theta_d \bigr) \langle -1 \rangle= \bigl( \Sigma \MT \Spin (d) \wedge BG_+ \bigr) \langle -1 \rangle, 
\]
and from Corollary \ref{cor:delooping-Ktheoryspectrum}, we get a weak equivalence
\[
\mathrm{B}^\infty |\bK(\cstar(\theta_d))_1| \simeq \bigl(\Sigma^{1-d}  |\bK(\cstar(G))| \bigr) \langle -1\rangle.
\]
The Atiyah--Bott--Shapiro orientation gives a spectrum map
\[
\lambda_{-d}: \MT \Spin (d) \lra \Sigma^{-d} \bK
\]
while the Novikov assembly map is a spectrum map 
\[
\nu:\Sigma^{-d} \bK \wedge BG_+ \lra \Sigma^{-d} \bK(\cstar(G)),
\]
so together they give
\[
\nu \circ (\lambda_{-d} \wedge \id_{BG_+}):  \MT \Spin (d) \wedge BG_+ \lra \Sigma^{-d} \bK(\cstar(G)).
\]
The main result of \cite{JEIndex2} is a delooped version of the Atiyah--Singer family index theorem which in the case at hand asserts that 
\begin{equation}\label{eqn:lower-index-theorem}
\begin{aligned}
&\Sigma ( \nu \circ (\lambda_{-d} \wedge \id_{BG_+})) \sim \mathrm{B}^\infty \ind_1 : \Sigma \MT \Spin (d) \wedge BG_+ \langle -1 \rangle \simeq \mathrm{B}^\infty |\cD^\op|\\
&\quad\quad\quad\quad\quad\quad\quad\quad\quad\quad\quad\quad\quad\quad\quad\quad\quad\quad \lra \Sigma |\bK(\cstar(\theta_d))|. 
\end{aligned}
\end{equation}

\subsection{The vanishing theorem}

Our goal is now to define a delooped version of the index difference, and this subsection contains the heart of the construction. We want to show that the map (which is really a zig-zag) 
\begin{equation}\label{eqn:constructionindexxidd}
\norm{ \cDs_{\theta,\bullet}^{\psc}} \stackrel{F}{\lra} |\cD_\theta| \stackrel{\simeq}{\longleftarrow} |\cD_\theta^{\op}| \stackrel{\ind_1}{\lra}|\bK(\cstar(\theta_d))_1|
\end{equation}
admits a preferred nullhomotopy (which will ultimately stem from the Lichnerowicz--Schr\"odinger formula). This nullhomotopy will induce a map 
\[
\Inddiff: \Psc(\theta_d) =\hofib_\emptyset F \lra \Omega |\bK(\cstar(\theta_d))_1|
\]
which is a map of $\Gamma$-spaces. After restriction to $\Psc^{2,\rst}(\theta_d)_0 $, identification of the latter space with $B \Conc (M)$ (where $M$ is as in Theorem \ref{thm:infinite-loopspace-theorem}) and application of $\Riem^+ ([0,1] \times M)_{g_0,g_0}^\rst \simeq \Omega B \Conc (M)$, we obtain a map
\[
\Riem^+ ([0,1] \times M)_{g_0,g_0}^\rst \lra \Omega^2 |\bK(\cstar(\theta_d))_1| \simeq \Omega |\bK(\cstar(\theta_d))_0|
\]
which by construction is an infinite loop map. The proof of Theorem \ref{thm:indddiff-infiniteloopmap-introduction} will then be completed by showing that this composition is homotopic to $\inddiff_{dx^2 + g_0}$. 
To implement this strategy, we establish in this section a commutative diagram 
\begin{equation}\label{indextheorydiagram}
\begin{gathered}
\xymatrix{
\norm{ \cDs_{\theta,\bullet}^{\psc}} \ar[d]^{F} & \norm{\cD_{\theta,\bullet}^{\psc,\op}} \ar[l]_{\simeq} \ar[d]^{F^{\op}} \ar[r] & | \bD(\cstar(\theta_d))_1|\ar[d]^{\mathrm{inc}}\\
|\cD_\theta| & |\cD_\theta^{\op}| \ar[r]^-{\ind_1} \ar[l]_{\simeq} & |\bK(\cstar(\theta_d))_1|
}
\end{gathered}
\end{equation}
of spaces. Since $| \bD(\cstar(\theta_d))_1|$ is contractible, this diagram provides the promised nullhomotopy of \eqref{eqn:constructionindexxidd}. We now give the definitions of the spaces of \eqref{indextheorydiagram}. The bottom line was defined in \S \ref{subsubsec:deloopedindex}. Recall from Definition \ref{defn:inifniteseminalcollars} that $\cDs_{\theta,p}^{\psc}(\ast)$ is the set of all $(W,a,h)$, where $W \in \cD_\theta (*)$, $a=(a_0 \leq \ldots \leq a_p)$ are real numbers, $h$ is a Riemannian metric on $W$ which is cylindrical near each $a_i$ and such that $h|_{W|_{[a_0,a_p]}}$ has positive scalar curvature. The forgetful maps $F_p : \cDs^{\psc}_{\theta,p} \to \cD_\theta$ given by $(W,a,h) \mapsto W$ form an augmentation map and induce 
\[
 F: \norm {\cDs_{\theta,\bullet}^{\psc}} \lra |\cD_\theta|.
\]
Under the weak equivalences $\norm{\cDs_{\theta,\bullet}^{\psc}} \simeq B \PCob_\theta$ (from \eqref{eqn:poset-model-for-psccobcat} and Lemma \ref{collar-shrinking-cobcat}) and $|\cD_\theta| \simeq B \Cob_\theta$ from \eqref{eq:longmanioflds-versus-ordcobcat}, the map $F$ corresponds to the forgetful map $B \PCob_\theta \to B \Cob_\theta$.
\begin{defn}
The sheaf $\cD_{\theta,p}^{\psc,\op}$ assigns to $X$ the set of all $(W,a,\eps,h,b,r,C,\kappa)$ such that 
\begin{enumerate}[(i)]
\item $(W,a,h) \in \cDs_{\theta,p}^{\psc} (X)$,
\item $\eps=(\eps_0,\ldots,\eps_p)$ is a tuple of continuous positive real-valued functions on $X$ with $a_{i}+\eps_i < a_{i+1}-\eps_{i+1}$ for all $i=0, \ldots,p-1$, such that $(W,h)$ is cylindrical over each $[a_i-\eps_i,a_i+\eps_i]$, and  
\item $b:W \to (0,\infty)$, $r: X \to \bR$ are smooth functions and $C,\kappa: X \to (0,\infty)$ are continuous functions such that \begin{align*}
b|_{W|_{[a_0-\eps_0,a_p+\eps_p]}} &\equiv 1, \\
\norm{[b \Dir_h b,x]} &\leq C,\\ 
\scal (h|_{W|_{[a_0-\eps_0,a_p+\eps_p]}}) &\geq \kappa,\\
r &\leq \frac{\kappa}{8C},\\
\eps_i^2 r^2 &\geq \max \{\kappa,\frac{2^9}{\kappa}\}.
\end{align*}
\end{enumerate}
\end{defn}
The forgetful map
\begin{align*}
 \cD_{\theta,\bullet}^{\psc,\op} &\lra \cDs_{\theta,\bullet}^{\psc}\\
 (W,a,\eps,h,b,r,C,\kappa) &\longmapsto (W,a,h)
\end{align*}
is a levelwise weak equivalence by Lemma \ref{lem:forgetfulmap-on-psc.-categories-isaweakequivalence} below. To define $F^{\op}:\norm{\cD_{\theta,\bullet}^{\psc,\op}}\to |\cD_{\theta}^{\op}|$, recall that $\norm{\cD_{\theta,\bullet}^{\psc,\op}}$ is the geometric realisation of the bi-semi-simplicial set
\[
 (p,q) \mapsto \cD_{\theta,p}^{\psc,\op} (\Delta^q_e) .
\]
We define 
\begin{align*}
 F^ {\op}_{p,q} : \cD_{\theta,p}^{\psc,\op} (\Delta^q_e) \times \Delta^p \times \Delta^q &\lra \cD_\theta^{\op} (\Delta^q_e) \times \Delta^q\\
 (W,a,\eps,h,b,r,C,\kappa;t;u) &\longmapsto (W,h,b,\sum_{j=0}^p t_i a_i,r,C;u).
\end{align*}
This clearly respects the identifications in the geometric realisation and hence defines a map $F^\op$ as asserted. 

There is an obvious structure of a $\Gamma$-space on $\norm{\cD_{\theta,\bullet}^{\psc,\op}}$, similar to that of Definition \ref{defn:PscCObCatGamma}, and the equivalence $\norm{\cD_{\theta,\bullet}^{\psc,\op}} \simeq \norm{\cDs_{\theta,\bullet}^{\psc}}$ is compatible with the $\Gamma$-space structures, where that on the latter space is as in Definition \ref{defn:PscCObCatGamma}. 

It is also obvious that the left square of \eqref{indextheorydiagram} commutes. What is missing for the construction of \eqref{indextheorydiagram} is that the upper leftward map is a weak equi\-valence (Lemma \ref{lem:forgetfulmap-on-psc.-categories-isaweakequivalence} below) and that the composition $\ind_1 \circ F^\op$ factors through $|\bD (\cstar(\theta_d)_1|$ (Lemma \ref{lem:window-of-opportunity} below). This is ultimately a consequence of the Lichnero\-wicz--Schr\"odinger formula. The definition of $\cD_{\theta,p}^{\psc,\op}$ is tailor-made so that the necessary estimates go through.

\begin{lem}\label{lem:forgetfulmap-on-psc.-categories-isaweakequivalence}
The forgetful map $\cD_{\theta,p}^{\psc,\op} \to \cDs_{\theta,p}^{\psc}$ is a weak equivalence of sheaves for each $p \geq 0$. In particular, it induces a weak equivalence after geometric realisation.
\end{lem}

\begin{proof}
The map in question factors as $\cD_{\theta,p}^{\psc,\op} \stackrel{H}{\to} \cD_{\theta,p}^{\psc} \to \cDs_{\theta,p}^{\psc}$. The second map was shown to be a weak equivalence in Lemma \ref{collar-shrinking-cobcat}. Hence it is enough to show that the first map $H$, which sends a tuple $(W,a,\eps,h,b,r,C,\kappa)$ to $(W,a,\epsilon,h)$, is a weak equivalence. 
We use the relative surjectivity criterion of Proposition 2.18 of \cite{MadsenWeiss} to recognise weak equivalences of sheaves. That is, if $A \subset X$ is a closed subset of a test manifold $X$ and $\gy$ is a germ of $\cD_{\theta,p}^{\psc,\op}$ around $A$, then
\[
\cD_{\theta,p}^{\psc,\op}[X,A,\gy] \lra \cD_{\theta,p}^{\psc} [X,A,\gy]
\]
is surjective. An element of the target is represented by $\gx = (W,a,\epsilon,h)$. There is an open neighbourhood $U$ of $A$ such that for each $z \in U$, we have numbers $C_z,\kappa_z$, $r_z$ and functions $b_z: W_z \to (0,\infty)$, satisfying the relevant conditions, and those depend smoothly on $z$. We have to show that, after applying a concordance of $\gx$ which is constant on a smaller neighbourhood of $A$, we can extend those functions to all of $X$. Firstly, choose a function $\kappa:X \to (0,\infty)$ which coincides with the given function $\kappa:U \to (0,\infty)$ near $A$, and such that $\kappa_z$ is a lower bound for the scalar curvature of the metric $h_z$ over the interval $[a_0(z),a_p(z)]$. It is clearly possible to find such a function.

Next, choose a function $C:X \to (0,\infty)$ which coincides with the given function $C:U \to (0,\infty)$ near $A$ and such that $C_z \geq \sup_{(a_0-\eps_0,a_p+\eps_p)} \norm{[\Dir_z,x]}$. Similarly, we extend the function $r:A \to (0,\infty)$ to a function $r:X \to (0,\infty)$ such that $r \leq \frac{\kappa}{8C}$, again coinciding with the given $r$ near $A$. 

Then we define the functions $b_z$ so that $[b_z \Dir_z b_z,x] \leq C_z$ for all $z$. This is possible since this is a convex condition on $b^2$, and very similar to the proof of \cite[Lemma 3.2]{JEIndex2}.

The last (and most substantial) step of the proof is to stretch the collar-like intervals, until they have length $\geq \max\{ \frac{\sqrt{\kappa}}{r}, \sqrt{\frac{2^9}{\kappa}}\}$. We also change the functions $b$ during the process: just make the region where they are constant $1$ larger. The details are very similar to the proof of Lemma \ref{collar-strtching-cobcat}. 
\end{proof}

\begin{lem}\label{lem:window-of-opportunity}
The composition $\ind_1 \circ F^{\op}: \norm{\cD_{\theta,\bullet}^{\psc,\op}} \to | \bK(\cstar(\theta_d))_1|$ goes into the subspace $|\bD (\cstar(\theta_d))_1 |$. 
\end{lem}

For the proof, we shall need another technical lemma.

\begin{lem}\label{lem:glueestimates}
Let $M$ be a manifold and let $B$ be an essentially self-adjoint and regular first order differential operator, linear over a $\cstar$-algebra $\gA$. Assume that $M= U \cup V$ is an open cover and that $\lambda >0$ is a real number so that 
\[
 \scpr{B^2 s ,s} \geq \lambda \scpr{s,s}
\]
whenever $s \in \dom (B)$ is smooth and has support in either $U$ or $V$. Let $0<c<\frac{1}{2} \sqrt{\lambda}$ and assume that there are $u \in C^\infty_c (U)$ and $v \in C^\infty_c (V)$ with $u^2 + v^2 =1$ and $\norm{[B,u]}, \norm{[B,v]} \leq c$. Then $B$ is invertible. 
\end{lem}

\begin{proof}
For $s \in \dom (B)$, we compute
\begin{align*}
 \scpr{B^2 s,s}  &= \scpr{uBs,uBs} + \scpr{vBs,vBs}\\
&= \scpr{[u,B]s,[u,B]s} +  \scpr{[v,B]s,[v,B]s} + \scpr{[u,B]s,Bus} + \scpr{[v,B]s,Bvs}\\
&\quad\quad +  \scpr{Bus,[u,B]s} +  \scpr{Bvs,[v,B]s} + \scpr{Bus,Bus} + \scpr{Bvs,Bvs}\\
&\geq \lambda \scpr{us,us} + \lambda\scpr{vs,vs} - 2c \norm{s} \norm{Bus} - 2c \norm{s} \norm{Bvs}\\
&\geq \lambda \scpr{s,s} - 2 c \norm{s} \bigl(\norm{[B,u]s}+ \norm{[B,v]s} + \norm{uBs} + \norm{vBs} \bigr)\\
&\geq  \lambda\scpr{s,s} - 4c^2 \norm{s}^2 - 4 c \norm{s} \norm{Bs}
\end{align*}
(the inequalities are inequalities of self-adjoint element of $\gA$). Together with the general inequality $\norm{Bs}^2 \geq \scpr{Bs,Bs}$, this shows that 
\[
 \bigl(\norm{Bs} + 2 c \norm{s} \bigr)^2 \geq \lambda \scpr{s,s} \in \gA,
\]
hence
\[
\norm{Bs} + 2 c \norm{s}  \geq \sqrt{\lambda} \norm{s}\in \bR
\]
or 
\[
 \norm{Bs}  \geq (\sqrt{\lambda} -2c)\norm{s}.
\]
This implies invertibility by an argument which is given e.g. in the proof of \cite[Proposition 1.21]{JEIndex1}. 
\end{proof}

\begin{proof}[Proof of Lemma \ref{lem:window-of-opportunity}]
We have to prove that if 
$$\gx=(W,a,\eps,h,b,r,C,\kappa,t) \in \cD_{\theta,p}^{\psc,\op}(X)\times \Delta^p,$$
then the operator family over $X$
\[
 B= b \Dir_h b \otimes 1 + \eta \otimes r(x-x_0) \varepsilon
\] 
defined in \eqref{eq:definition-operator-B} is invertible, where $x_0=\sum_{i=0}^p t_i a_i$. Once this is proven, we plug in the test manifolds $\Delta^q_e$ and obtain the claim, by the definition of $F^{\op}$. For the proof, we consider the case $X=*$ first. 

We shall apply Lemma \ref{lem:glueestimates} with $U:=W|_{(a_0-\eps_0,a_p+\eps_p)}$, $V:=W|_{\bR \setminus [a_0-\frac{\eps_0}{2},a_p+\frac{\eps_p}{2}]}$ and $\lambda = \frac{\kappa}{8}$. 

Formula \eqref{eqn:square-of-operator} shows that
\[
B^2 \geq  (b \Dir_h b)^2 \otimes 1 + r^2 (x-x_0)^2 - rC.
\]

Over $W|_{(a_0-\eps_0,a_p+\eps_p)}$, the scalar curvature of $h$ is bounded from below by $\kappa$ and $b \equiv 1$. Therefore, over $W|_{(a_0-\eps_0,a_p+\eps_p)}$, we have
\[
 B^2 \geq \frac{\kappa}{4}  -Cr \geq \frac{\kappa}{4} - \frac{\kappa}{8}
\]
by the Lichnerowicz--Schr\"odinger formula and by the condition $r \leq \frac{\kappa}{8C}$ which is required in the definition of $\cD_{\theta,p}^{\psc,\op}$.

Over $W|_{\bR \setminus [a_0-\frac{\eps_0}{2},a_p+\frac{\eps_p}{2}]}$, we can merely infer that
\[
B^2 \geq (b \Dir_h b)^2 \otimes 1 + r^2 (x-x_0)^2 - rC \geq   r^2 (x-x_0)^2 - rC.
\]
But the point $x_0  \in \bR$ is the convex combination $\sum_{j=0}^{p} u_i a_i$, and hence it lies in the interval $[a_0,a_p]$. It follows that 
\[
(x-x_0)^2 \geq  \min (\frac{\eps_0}{2},\frac{\eps_p}{2})^2 \geq    \frac{\kappa}{4r^2}.
\]
over $W|_{\bR \setminus [a_0-\frac{\eps_0}{2},a_p+\frac{\eps_p}{2}]}$ and so we have
\[
B^2 \geq   \frac{\kappa}{4}- rC \geq \frac{\kappa}{8}
\]
over $W|_{\bR \setminus [a_0-\frac{\eps_0}{2},a_p+\frac{\eps_p}{2}]}$.

To combine the estimates, fix a smooth function $f: \bR \to [0,\frac{\pi}{2}]$ such that $\supp (f) \subset (a_0-\eps_0,a_p,\eps_p)$, such that $f|_{[a_0-\frac{\eps_0}{2},a_p+\frac{\eps_p}{2}]} \equiv \frac{\pi}{2}$ and such that $|f'| \leq \frac{3}{2\min \{\eps_0,\eps_p\}} $. Define functions $u:= \sin (f \circ x)$ and $ v:= \cos (f \circ x)$ on $W$. Since $b|_{W|_{[a_0-\frac{\eps_0}{2},a_p+\frac{\eps_p}{2}]}}=1$ and since the metric $h$ is cylindrical in the region where $du$ or $dv$ are nonzero, we have 
\[
\norm{[B,u]} = \norm{du} \leq |f' \circ x| \leq  \frac{3}{2\min \{\eps_0,\eps_p\}} 
\]
and similarly 
\[
\norm{[B,v]} \leq \frac{3}{2\min \{\eps_0,\eps_p\}}.
\]
By the definition of $\cD_{\theta,p}^{\psc,\op}$, we have 
\[
\frac{3}{2\min \{\eps_0,\eps_p\}} < \frac{1}{2} \sqrt{\frac{\kappa}{8}},
\]
so that Lemma \ref{lem:glueestimates} applies and gives invertibility of $B$. This finishes the proof in the case $X=*$. 

For the general case, note that the proof above (and the proof of Lemma \ref{lem:glueestimates}) provides a locally uniform (in $X$) positive lower bound for $B^2_x$, which suffices for invertibility. 
\end{proof}

At this point, the construction of the diagram \eqref{indextheorydiagram} is complete. 

\subsection{The delooped index difference}

We now use the diagram \eqref{indextheorydiagram} which we established in the previous section to define the delooped index difference, and we shall do this in two ways. Recall from Definition \ref{defn:pscfibreGammaspace} that the very special $\Gamma$-space $\Psc(\theta_d)$ was defined as the homotopy fibre of the forgetful map $\norm{\cDs_{\theta_d,\bullet}^{\psc}} \to |\cD_{\theta_d}|$, at the base-point $\emptyset \in |\cD_{\theta_d}|$. 

\begin{defn}\label{defn:spectruminddiff}
We define $\Inddiff_{\theta_d}$ as the zig-zag
\[ 
\Psc(\theta_d)= \hofib_\emptyset (F) \stackrel{\sim}{\longleftarrow} \hofib_\emptyset (F^\op) \lra \Omega |\bK (\cstar(\theta_d))_1|.
\]
where the second map is induced by the right half of diagram \eqref{indextheorydiagram}. 
\end{defn}

In the above definition, we have identified the homotopy fibre of the inclusion map $\inc: |\bD ( \cstar (\theta_d))_1| \to |\bK(\cstar(\theta_d))_1|$ over the base-point up to homotopy with the loop space (taken at the generalised base-point). The map $\Inddiff_{\theta_d}$ is a (homotopy class of) maps of very special $\Gamma$-spaces, and restricts to a map
\[
\Inddiff_{\theta_d,0}^{2,\rst}:\Psc^{2,\rst}(\theta_d)_0 \lra \Omega |\bK (\cstar(\theta_d))_1|. 
\]
In particular, both $\Inddiff_{\theta_d}$ and and $\Inddiff_{\theta_d,0}^{2,\rst}$ are infinite loop maps. The drawback of this construction is that it does not refer to any particular manifold or psc metric (and the name ``$\Inddiff$'' is therefore a bit of a red herring). To get a more sensitive invariant, let us first deal with some formalities. 

\begin{defn}\label{defn:transgression}
Let 
\[
\xymatrix{
X \ar[d]^{f} \ar[r]^{g} & D \ar[d]^{j}\\
Y \ar[r]^{i} & K
}
\]
be a commutative diagram of spaces where $D$ is contractible and $j$ is an inclusion. Let $X_0 \subset X$ and $Y_0 \subset Y$ be generalised base-points such that $f(X_0)\subset Y_0$ and $i (Y_0) \subset D$. The \emph{transgression of $i$ at $Y_0$} is the map
\[
 \trg_{Y_0} (i): \hofib_{Y_0} (f) \to \hofib_D (j) \simeq \Omega K
\]
which is induced by the above diagram. Let $X_1$, $Y_1$ be another choice of generalised base-points, with $f(X_1) \subset Y_1$ and let $\gamma: [0,1] \to Y$ be a path with $\gamma(i)\in Y_i$, for $i=0,1$. We denote by
$$\tau_\gamma: \hofib_{Y_0} (f) \lra \hofib_{Y_1} (f)$$
the map given by concatenation with $ \gamma$. 
\end{defn}

The following simple observation is crucial for us: 

\begin{lem}\label{lem:changevbasepoints}
In the situation of Definition \ref{defn:transgression}, assume furthermore that the path $\gamma$ can be lifted to a path $\delta$ in $X$ with $\delta(i)\in X_i$. Then 
\[
 \trg_{Y_1}(i) \circ \tau_\gamma \sim \trg_{Y_0} (i): \hofib_{Y_0}(f) \lra \Omega K. \qed
\]
\end{lem}

From now on, let $M$ be a closed $(d-1)$-dimensional $\theta$-manifold, let $g_0 \in \Riem^+ (M)$ and let $h_0 =dx^2+g_0$ be the product metric. Our goal is to extract from \eqref{indextheorydiagram} a map\footnote{To ease notation, we will always use $|\bD (\gA)_n| \subset |\bK(\gA)_n|$ as a generalised base-point, and abbreviate $\Omega |\bK(\gA)_n| := \Omega_{|\bD(\gA)_n|}|\bK(\gA)_n|$.}
\[
 B (\inddiff_{h_0}): B \Conc(M) \lra \Omega | \bK(\cstar(\theta_d))_1| \simeq |\bK(\cstar(\theta_{d})_0)|, 
\]
where $B \Conc(M) $ is the classifying space of the concordance category from Definition \ref{defn:concroidancecateogiry}. 

In order to give the construction in a way that enables us to identify the homotopy class of $B (\inddiff_{h_0})$, we need to specify generalised base-points in all spaces of \eqref{indextheorydiagram} (except for spaces of the right column, where we pick $|\bD (\cstar(\theta_d))_1|$). This has to be done with some care. 

\begin{defn}\label{defn:generalized-basepoints-for-Binddiff}
Let $M$ be a closed $(d-1)$-manifold $M \subset \bR^\infty$ equipped with a $\theta_d$-structure $\ell_M$. We equip $\bR \times M$ with the product $\theta$-structure $\ell_{\bR\times M}$. Let furthermore $g_0 \in \Riem^+ (M)$ and let $h_0 := dx^2 + g_0$ be the product metric on $\bR \times M$. Let
\begin{equation*}
\kappa_0 := \inf (\scal (g_0))=\inf (\scal (h_0))>0 \quad\text{ and } \quad r_0:= \frac{\kappa_0}{8}
\end{equation*}
\begin{enumerate}[(i)]
\item We define
\[
\gy(M) := (\bR \times M ,\ell_{\bR \times M}) \in \cD_\theta (*) \subset |\cD_\theta|.
\]
\item We define the subsheaf $\cY(M,g_0) \subset \cD_{\theta}^{\op}$ which assigns to $X$ the set of all sections which after restriction to each point of $X$ are of the form $(\bR \times M,h,b,x_0,r,C)$ with 
\[
rC \leq \tfrac{\kappa_0}{8} , \quad h|_{[x_0- l,x_0+ l]} =h_0 \quad \text{ and } \quad   b|_{[x_0- l,x_0+ l]} \equiv 1,
\]
where $l:= \max\{ \tfrac{\sqrt{\kappa_0}}{r}, \sqrt{\tfrac{2^9}{\kappa_0}}  \}$. The sheaf $\cY(M,g_0)$ is contractible. A contraction onto the point 
\[
\gx(M,g_0):= (\bR \times M ,h_0,1,0,r_0,1) \in \cD_\theta^{\op}(\ast)
\]
can be given as follows. In the first step, we change the metric $h$ by a convex combination to $h_0$ and keep $x_0$, as well as $C$ and $r$ unchanged. The condition that $\norm{[b\Dir_h b,x]}\leq C$ is a convex condition on $b^2$, and hence we can change $b$ to keep that condition, and do not change $b$ over $[x_0-l,x_0+l]$. In the second step, we arrange that $b=1$, by using $b_t := \sqrt{(1-t)b^2 +t}$, and keep all other data fixed. In the third step, we shrink $C$ down to $1$ (note that $C\geq 1$ is implied by the conditions on an element in $\cD_\theta^{\op}(*)$ which is somewhere cylindrical). The fourth step turns $r$ linearly to $r_0$, and $x_0$ linearly to $0$. 

The inclusion $|\cY(M,g_0)| \to |\cD_\theta^\op|$ can be lifted through $\norm{\cD_{\theta,\bullet}^{\psc,\op}}$. Therefore Lemma \ref{lem:window-of-opportunity} implies that $\ind_1 (\cY(M,g_0)) \subset \bD(\cstar(\theta_d))_1$. 
 
\item We define $\cX(M,g_0) \subset \cD_{\theta,0}^{\psc,\op}$ as the subsheaf of section which after restriction to points are of the form $(\bR \times M,a,\epsilon,h,b,r,C,\kappa)$, such that in addition to the conditions required for $\cD_{\theta,0}^{\psc,\op}$, we require
\[
h|_{[a-\epsilon,a+\epsilon]} \equiv h_0, \quad \kappa \leq \kappa_0, \quad\text{ and } \quad \epsilon^2 r^2 \geq  \kappa_0.
\]
The sheaf $\cX(M,g_0)$ is contractible, by an argument analogous to that proving the contractibility of $\cY(M,g_0)$. This defines a generalised base-point $|\cX(M,g_0)| \subset |\cD_{\theta,0}^{\psc,\op}| \subset \norm{\cD_{\theta,\bullet}^{\psc,\op}}$.
\item Define $\cZ(M,g_0) \subset \cDs_{\theta,0}^{\psc}$ as the subsheaf of all sections which on points are of the form $(\bR \times M, a,h)$ such that $h \equiv h_0$ near $a$. This is contractible onto $(\bR \times M,0,h_0)$ by a similar argument as for $\cY(M,g_0)$, and yields a generalised base-point $|\cZ(M,g_0)|\subset |\cDs_{\theta,0}^{\psc}| \subset \norm{\cDs_{\theta,\bullet}^{\psc}}$. 
\end{enumerate}
\end{defn}
With the above definitions, all maps in \eqref{indextheorydiagram} preserve the generalised base-points. It is also convenient to alter the definition of the concordance category slightly.
\begin{defn}\label{defn:modifiedconccat}
We define $\cQ_p (M ):= F_p^{-1} (\gy(M)) \subset \cDs_{\theta,p}^{\psc}$. In other words, $\cQ_p (M,g_0 )(\ast)$ is the set of all $(h,a)$, where $h$ is a Riemannian metric on $\bR \times M$ and $a_0 \leq  \ldots \leq a_p$ are real numbers, such that $h$ is cylindrical near $a_i$ for each $i$, and $h|_{  [a_0,a_p]\times M}$ has positive scalar curvature. The sheaf $\cQ_p (M)$ is the sheaf of $p$-simplices of a semi-simplicial sheaf $\cQ_\bullet (M )$. 
\end{defn}
The difference to the nerve of the concordance category as defined in \eqref{eqn:defn:concordancecategory} is that here we require that the Riemannian metrics are defined over all of $\bR \times M$. There are no conditions posed on the Riemannian metric outside $[a_0,a_p]\times M$, and hence the choice of those does not affect the homotopy type. Therefore, $\norm{\cQ_\bullet (M)}$ is homotopy equivalent to $B \Conc(M)$. The sheaf $\cZ(M,g_0)$ is contained in $\cQ_0 (M)$ and so defines a generalised base-point $|\cZ(M,g_0)| \subset \norm{\cQ_\bullet (M)}$. 

The usual map from geometric to homotopy fibres yields
\begin{equation*}
\beta: \norm{\cQ_\bullet (M)} \lra \hofib_{\gy(M)} (F).
\end{equation*}

\begin{defn}\label{defn:delooped-index-difference}
Let $M$, $g_0$ and $h_0$ be as above. The \emph{delooped index difference} at $h_0$ is the homotopy class of maps 
\[
B(\inddiff_{h_0}): \norm{\cQ_\bullet (M)} \lra \Omega |\bK (\cstar(\theta_d))_1|
\]
given as the composition 
\[
\norm{\cQ_\bullet (M)} \stackrel{\beta}{\lra} \hofib_{\gy(M)} (F) \stackrel{\simeq}{\longleftarrow} \hofib_{\cY(M,g_0)} (F^\op) \xrightarrow{\trg_{\cY(M,g_0)} (\ind_1)} \Omega |\bK (\cstar(\theta_d))_1|.
\]
\end{defn}

We want to compare the delooped index difference with the ordinary index difference and shall do this in two ways. Firstly, we want to compute the composition of $\Omega B (\inddiff_{h_0})$ with the tautological map $\Riem^+ ([0,1] \times M)_{g_0,g_0} \to \Omega B \Conc (M)\simeq \Omega_{|\cZ(M,g_0)|} \norm{\cQ_\bullet (M)}$. To this end, we need a version of the tautological map that is better adapted to that goal. There is a map $\psi_0: \Riem^+ ([0,1] \times M)_{g_0,g_0} \to \cQ_1 (M)$ of sheaves, which sends $h$ to the tuple $(h,0,1)$. We obtain a map
\begin{equation}\label{eqn:comparisontheoremmap1}
\psi^{\ad}: |\Riem^+ ([0,1] \times M)_{g_0,g_0}| \times \Delta^1 \lra |\cQ_1 (M)| \times \Delta^1 \lra \norm{\cQ_\bullet (M)}
\end{equation}
by composition with the quotient map.

The map \eqref{eqn:comparisontheoremmap1} sends $|\Riem^+ ([0,1] \times M)|\times \partial \Delta^1 $ to $|\cZ(M,g_0)|$, and hence by adjunction we obtain a map
\begin{equation}\label{eqn:map-psi}
\psi:|\Riem^+ ([0,1] \times M)_{g_0,g_0}| \lra \Omega_{|\cZ(M,g_0)|} \norm{\cQ_\bullet (M)}.
\end{equation}
\begin{thm}\label{thm:comparison-delooped-index-diference1}
The composition 
\[
| \Riem^+ ([0,1] \times M)_{g_0,g_0}| \xrightarrow{\psi} \Omega_{|\cZ(M,g_0)|} \norm{\cQ_\bullet(M)} \xrightarrow{\Omega B (\inddiff_{h_0})} \Omega^2 |\bK(\cstar(\theta_d))_1|
\]
is homotopic to
\[
| \Riem^+ ([0,1] \times M)_{g_0,g_0}| \xrightarrow{\inddiff_{h_0}} \Omega |\bK(\cstar(\theta_d))_0| \xrightarrow{-\bott} \Omega^2 |\bK(\cstar(\theta_d))_1|.
\]
\end{thm}
The proof of Theorem \ref{thm:comparison-delooped-index-diference1} is by an explicit, but cumbersome computation and is deferred to \S \ref{subsec:prooof-indexcomparisontheorems}. Theorem \ref{thm:comparison-delooped-index-diference1} is the main ingredient for the proof of Theorem \ref{thm:indddiff-infiniteloopmap-introduction}. 

There is another comparison of $B (\inddiff_{h_0})$ with an ordinary index difference, which is not directly relevant for our purposes, but which we state (and prove in \S \ref{subsec:prooof-indexcomparisontheorems}) for sake of completeness. 
Namely, we wish to compare $B (\inddiff_{h_0})$ with 
\[
\inddiff_{g_0}:| \Riem^+ (M)| \lra \Omega |\bK(\cstar(\theta_{d-1}))_0| \simeq \Omega | \bK(\cstar(\theta_d))_1|.
\]
To this end, let $\varphi: \Riem^+ (M) \to \cQ_0 (M)$ be the map of sheaves that sends $g \in \Riem^+ (M)(*)$ to the point $(\bR \times M,dx^2+g,0) \in \cQ_0 (M)(*)$. It induces 
\[
|\varphi|: |\Riem^+ (M)| \lra |\cQ_0 (M)| \subset \norm{\cQ_\bullet(M)}.
\]
\begin{thm}\label{thm:comparison-delooped-index-diference2}
The composition
\begin{equation}\label{eq:comparisontheorem2}
 |\Riem^+ (M)|\stackrel{\varphi}{\lra} \norm{\cQ_\bullet (M)} \xrightarrow{(-1)^{d-1} B (\inddiff_{h_0})} \Omega |\bK(\cstar(\theta_d))_1|
\end{equation}
is homotopic to 
\[
|\Riem^+ (M)| \xrightarrow{\inddiff_{g_0}} \Omega |\bK(\cstar(\theta_{d-1}))_0| \xrightarrow{\morita} \Omega |\bK(\cstar(\theta_{d}))_1|. 
\]
\end{thm}
Again, the proof is by a lengthy computation which is carried out in \S \ref{subsec:prooof-indexcomparisontheorems}.

\subsection{The index difference map is an infinite loop map}

We can now give the proof of Theorem \ref{thm:indddiff-infiniteloopmap-introduction}, assuming Theorem \ref{thm:comparison-delooped-index-diference1}, and immediately jump into the proof. The first step is to prove that $B(\inddiff_{h_0})$ is well-behaved with respect to changing base-points. More precisely, let $(W,h): \emptyset \leadsto (M,g_0)$ be a psc-$\theta_d$-cobordism. This defines a path $\gamma_{(W,h)}: (M,g_0) \to \emptyset$ in $B  \PCob_{\theta_d}$, or alternatively a path
\[
\gamma'_{(W,h)}: [0,1] \lra \norm{\cDs_{\theta,\bullet}^{\psc}}
\]
which begins in the generalised base-point $|\cZ(M,g_0)|$ and ends in the generalised base-point $|\cZ(\emptyset)|$ which we defined in Definition \ref{defn:generalized-basepoints-for-Binddiff}. Since the forgetful map $\norm{\cD_{\theta,\bullet}^{\psc,\op}} \to \norm{\cDs_{\theta,\bullet}^{\psc}}$ is a weak homotopy equivalence by Lemma \ref{lem:forgetfulmap-on-psc.-categories-isaweakequivalence}, we can lift $\gamma'_{(W,h)}$ to a path 
\[
\gamma''_{(W,h)}: [0,1] \lra \norm{\cD_{\theta,\bullet}^{\psc,\op}}
\]
from the generalised base-point $|\cX(M,g_0)|$ to $|\cX(\emptyset)|$. Lemma \ref{lem:changevbasepoints} now shows that the right triangle in
\begin{equation}\label{eqn:anotherdiagram}
\begin{gathered}
\xymatrix{
\hofib_{\gy(M)} (F) \ar[d]_{\tau_{F \circ \gamma'_{(W,h)}}} & \ar[l]_-{\simeq} \hofib_{|\cY(M,g_0)|} (F^\op) \ar[d]_{\tau_{F^\op \circ \gamma''_{(W,h)}}} \ar[rr]^{\trg_{\cY(M,g_0)} (\ind_1)} & & \Omega |\bK (\cstar(\theta_d))_1| \\
\hofib_{\gy(\emptyset)} (F)  & \ar[l]_-{\simeq} \hofib_{|\cY(\emptyset)|} (F^\op) \ar[urr]_{\trg_{\cY(\emptyset)} (\ind_1)} & & 
}
\end{gathered}
\end{equation}
commutes up to homotopy, and the square commutes by naturality. The composition from $\hofib_{\gy(\emptyset)} (F) $ to $\Omega |\bK (\cstar(\theta_d))_1|$ can be identified with $\Inddiff_{\theta_d}$, by definition. When composed with the map $\beta: \norm{\cQ_\bullet (M)} \to \hofib_{\gy(M)} (F)$, the upper composition agrees with $B (\inddiff_{h_0})$, again by definition. Let us summarise the result of this discussion. 
\begin{lem}
With all the above notations in place, the maps 
\[
B (\inddiff_{h_0}), \Inddiff_{\theta_d} \circ \tau_{F^\op \circ \gamma''_{(W,h)}} \circ \beta: \norm{\cQ_\bullet(M)} \lra \Omega |\bK (\cstar(\theta_d))_1|
\]
are homotopic. \qed
\end{lem}

Let us next recall the hypotheses of Theorem \ref{thm:indddiff-infiniteloopmap-introduction} (which are identical to those of Theorem \ref{thm:infinite-loopspace-theorem}). Besides of what we already stated, they are
\begin{enumerate}[(i)]
\item $M$ is an object of $\Cob_{\theta_d}^{2,1}$ (that is, the map $M \to BG$ is $2$-connected),
\item $W$ is a morphism of $\Cob_{\theta_d}^{2}$ (that is, $M \to W$ is $2$-connected), and
\item $h \in \Riem^+ (W)_{g_0}$ is right-stable. 
\end{enumerate}
We now replace in \eqref{eqn:anotherdiagram} the space $\norm{\cDs_{\theta_d}^\psc}$ by $B \PCob_{\theta_d}$ and $|\cD_{\theta_d}|$ by $B \Cob_{\theta_d}$, using the equivalences \eqref{eq:longmanioflds-versus-ordcobcat} and \eqref{eqn:comparison-pscbobmodelssheaf}, and restrict to the base-point components $(B\PCob_{\theta_d}^{2,\rst})^0$, $(B\Cob_{\theta_d}^{2})^0$, respectively, of the classifying spaces of the relevant subcategories of the cobordism categories. We arrive at the following homotopy commutative diagram
\begin{equation}\label{eqn:yetanotherone}
\begin{gathered}
\xymatrix{
(B \Conc(M)^{\rst})^0 \ar[r]^-{\beta}_-{\simeq} & (\hofib_M BF^{2,\rst})^0 \ar[d]_{\Lambda(\gamma)}^{\simeq} \ar[r] & \Omega |\bK(\cstar(\theta_d))_1|\\
 & (\hofib_\emptyset BF^{2,\rst})^0 \ar[ur]_{\Inddiff_{\theta_d,0}^{2,\rst}}. &
}
\end{gathered}
\end{equation}
Here $\Lambda(\gamma)$ is the fibre transport along the path $\gamma$ in $B \Cob_{\theta_d}^{2,\rst}$ given by $W$ (and the condition that $(W,h)$ is a morphism in $\PCob_{\theta_d}^{2,\rst}$ implies, as in the proof of Theorem \ref{thm:infinite-loopspace-theorem}, that $\Lambda(\gamma)$ preserves the base-point components). The upper composition is by definition $B (\inddiff_{h_0})$, while the diagonal map is a map of very special $\Gamma$-spaces. During the proof of Theorem \ref{thm:infinite-loopspace-theorem}, we \emph{defined} the infinite loop space structure on $(B \Conc(M)^{\rst})^0$ so that the weak equivalence $\Lambda(\gamma) \circ \beta$ is an infinite loop map. Therefore, we conclude:
\begin{cor}
The map $B (\inddiff_{h_0}): (B \Conc(M)^{\rst})^0 \to \Omega |\bK(\cstar(\theta_d))_1|$ is an infinite loop map.
\end{cor}
However, the infinite loop space structure on $\Riem^+ ([0,1]\times M)_{g_0,g_0}^{\rst}$ was constructed in such a way that the weak equivalence $\Riem^+  ([0,1]\times M)_{g_0,g_0}^{\rst}\to \Omega (B \Conc(M)^\rst)^0$ is an infinite loop map by definition. To finish the proof of Theorem \ref{thm:indddiff-infiniteloopmap-introduction}, invoke Theorem  \ref{thm:comparison-delooped-index-diference1}.

\subsection{Another proof of the detection theorem for the index difference}

In this section, we apply the theory developed in this section to give a delooped versions of Theorem B of \cite{BERW} and Theorem A of \cite{ERWpsc2}. 

\begin{thm}\label{berw-pscpi1-newproof}
Let $d \geq 6$ and let $G$ be a finitely presented group. Define $\theta: B \Spin (d) \times BG \to B\mathrm{O}(d)$. Then there is a map of spectra
\[
\Xi:\bigl(\MT \Spin (d) \wedge BG_+ \bigr)\langle -1 \rangle \lra \mathrm{B}^\infty\Psc^{2,\rst}(\theta)_0
\]
such that the compositions
$$\mathrm{B}^\infty(\Inddiff_{\theta,0}^{2,\rst}) \circ \Xi: \bigl(\MT \Spin (d) \wedge BG_+\bigr) \langle -1 \rangle \lra \Sigma^{-d} \bK(\cstar(G))$$and
\[
\mu:\bigl(\MT \Spin (d) \wedge BG_+\bigr)\langle -1 \rangle \xrightarrow{\lambda_{-d} \wedge \id_{BG_+}} \Sigma^{-d} \bK \wedge BG_+ \stackrel{\nu}{\lra} \Sigma^{-d} \bK (\cstar(G))
\]
are homotopic.
\end{thm}

Before we prove this result, let us show how it leads to a detection result for the index difference, along the lines of the main results of \cite{BERW} and \cite{ERWpsc2}. 
Let $(M,g_0)$ be an object of $\PCob_\theta^{2,1,\rst}$ and assume that there is a nullbordism $(W,h_0): \emptyset \leadsto (M,g_0)$ in $\PCob_{\theta}^{2,\rst}$. 
After taking infinite loop spaces $\Omega^{\infty+1}$, using the identification $\Omega^{\infty+1} \mathrm{B}^\infty \Psc^{2,\rst}(\theta)_0 \simeq \Riem^+ ([0,1]\times M)^{\rst}_{g_0,g_0}$ from Theorem \ref{thm:infinite-loopspace-theorem} and \eqref{eqn:yetanotherone}, we obtain a map
\[
\Omega^{\infty+1} \MT \Spin (d) \wedge BG_+ \lra \Riem^+ ([0,1]\times M)_{g_0,g_0}^{\rst}
\]
such that the composition with $\inddiff_{dx^2 + g_0}$ is the same as $\Omega^{\infty+1} \mu$. This is a little stronger than Theorem A of \cite{ERWpsc2}, since it applies to odd-dimensional manifolds as well, and since it maps to right stable metrics. By exactly the same arguments as at the end of \S 5.8 of \cite{ERWpsc2}, one gets Theorem A of that paper, now for all dimensions $d \geq 6$. 

\begin{proof}
From the homotopy commutative diagram of very special $\Gamma$-spaces
\[
\xymatrix{
\Psc^{2,\rst}(\theta)_0 \ar[rr]^-{\Inddiff_{\theta,0}^{2,\rst}} \ar[d] & &\Omega |\bK(\cstar(\theta_d))_1| \ar[d]\\
(B \PCob_{\theta}^{2,\rst})^0 \ar[d] \ar[rr] &  &|\bD (\cstar(\theta_d))_1| \ar[d] \\
(B \Cob_\theta)^0 \ar[rr]^-{\ind_1} & &|\bK(\cstar(\theta_d))_1|
}
\]
whose columns are fibre sequences, we get a diagram of spectra
\[
\xymatrix{
\mathrm{B}^\infty \Psc^{2,\rst}(\theta)_0 \ar[rr]^-{\mathrm{B}^\infty(\Inddiff_{\theta,0}^{2,\rst})} \ar[d] & & \Sigma^{-d} |\bK(\cstar(G))|  \ar[d]\\
\mathrm{B}^\infty (B \PCob_{\theta}^{2,\rst})^0 \ar[d] \ar[rr] & & \ast \ar[d] \\
\Sigma \bigl( \MT \Spin (d) \wedge BG_+ \langle -1\rangle \bigr) \ar[rr]^-{\mathrm{B}^\infty (\ind_1)} &  &\Sigma^{1-d} |\bK(\cstar(G))|  
}
\]
with columns fibre sequences (here we used the main results of \cite{GMTW} and \cite{Nguyen} as explained in Remark \ref{rem:CobCatVerySpecial}). By the Puppe sequence, it yields a homotopy commutative diagram
\[
\xymatrix{
\MT \Spin (d) \wedge BG_+ \langle -1\rangle \ar[d]^{\Xi} \ar[rr]^-{\Sigma^{-1} \mathrm{B}^\infty (\ind_1)} & &\Sigma^{-d} |\bK(\cstar(G))|  \\
\mathrm{B}^\infty \Psc^{2,\rst}(\theta)_0  \ar[urr]_-{\mathrm{B}^\infty(\Inddiff_{\theta,0}^{2,\rst})}.
}
\]
From \ref{eqn:lower-index-theorem}, we get that the upper horizontal arrow is homotopic to $\mu$, as asserted.
\end{proof}

\subsection{Proof of the comparison theorems}\label{subsec:prooof-indexcomparisontheorems}

We now give the proofs of Theorems \ref{thm:comparison-delooped-index-diference1} and \ref{thm:comparison-delooped-index-diference2} (in this order). The proof of Theorem \ref{thm:comparison-delooped-index-diference1} is structured in two steps, which are stated as Lemma \ref{lem1:proof-thm.comparison-delooped-index-diference1} and \ref{lem2:proof-thm.comparison-delooped-index-diference1} below.

At this point, the reader should recall the notations from Definition \ref{defn:generalized-basepoints-for-Binddiff} and the definition of the map $\psi$ from \eqref{eqn:map-psi}. 

\begin{lem}\label{lem1:proof-thm.comparison-delooped-index-diference1}
The composition $\Omega B (\inddiff_{h_0}) \circ \psi$ is homotopic to the map induced by the map of sheaves 
\[
\phi: \Riem^+ ([0,1] \times M)_{g_0,g_0} \lra \Omega^2 \bK(\cstar(\theta_d))_1
\]
which is defined on points as follows. For $h \in \Riem^+ ([0,1]\times M)_{g_0,g_0}$, put
\[
 \kappa:= \inf \scal (h)\in (0, \kappa_0];
\]
\[
 C:= \sup_ {s \in [0,1]}\norm{[ \Dir_{(1-s) h + sh_0},x]} \in [1, \infty); \; r := \frac{\kappa}{8C}\leq \frac{\kappa_0}{8} =r_0 ;
\]
\[
 \eps=(\eps_0,\eps_1); \; \eps_i: = \max\{\frac{\sqrt{\kappa_0}}{r}, \frac{2^9}{\kappa} \}
\]
($\eps$ is a pair, not an interval) and
\[
a=(a_0,a_1); \;  a_0:=-\eps_0, a_1 := 1+\eps_1. 
\]
Further, we choose monotone smooth functions $c,d: [0,1]\to [0,1]$ such that $c(0)=d(0)=0$, $c(1)=d(1)=1$ and $c(\frac{1}{2})=1$, $d (\frac{1}{2})=0$, and define 
\[
r(s):= (1-d(s))r+d(s)r_0 \in [r,r_0].
\]
Now let $\phi(h) \in \Omega^2 \bK(\cstar(\theta_d))(*)$ be given by the continuous field of Hilbert-$\cstar(\theta_d)$-modules on $\bR^2$ with fibres $L^2 (\bR \times M; \spinor_\ell )$, obvious grading and $\Cl^{1,0}$-structure, and operator family given by the formula 
\begin{equation}\label{indextheoremproof:eq2}
\Dir_{(1-c(s))h + c(s)h_0} \otimes 1+ r(s)\bigl(x-(1-d(s)) ((1-t) a_0 + t a_1)\bigr) \eta \otimes \varepsilon 
\end{equation}
\end{lem}
Formula \eqref{indextheoremproof:eq2} defines an invertible operator if $s \leq 0$ or $s \geq 1$ or if $|t|$ is sufficiently large and so it indeed gives a point of $\Omega^2 \bK(\cstar(\theta_d))$.
To see this, one takes the square of the operator (using the formula \eqref{eqn:square-of-operator}) and has to prove that the square is strictly positive. For $s \leq 0$ and $s \geq 1$, this is immediate from the choice of $r$. For $|t| \gg 0$, one estimates separately over a neighbourhood of $[0,1] \times M$ (where the second summand of \eqref{indextheoremproof:eq2} accounts for positivity) and over $\bR \setminus [0,1] \times M$ (where the metrics are cylindrical and have positive scalar curvature and so the first summand yields positivity). These estimates are combined using Lemma \ref{lem:glueestimates}. The details are similar as in the proof of Lemma \ref{lem:window-of-opportunity}. 

\begin{lem}\label{lem2:proof-thm.comparison-delooped-index-diference1}
The sheaf map defined in Lemma \ref{lem1:proof-thm.comparison-delooped-index-diference1} is naturally concordant to the sheaf map which defines $(-\bott) \circ \inddiff_{h_0}$. 
\end{lem}

\begin{proof}[Proof of Lemma \ref{lem1:proof-thm.comparison-delooped-index-diference1}]
The main point of the proof is to resolve the zig-zag implicit in the definition of $B(\inddiff_{h_0})$. In other words, we construct a diagram
\begin{equation}\label{diag:indexthy1}
\begin{gathered}
 \xymatrix{
 |\Riem^+ ([0,1] \times M)_{g_0,g_0}| \times \Delta^1 \ar@/^2.0pc/[dddr]^{\psi_1^{\ad}}\ar[d]^{\psi_0 \times \id}  & \\
| \cQ_1 (M)| \times \Delta^1 \ar[d] & \\
\norm{\cQ_\bullet (M)} \ar[d]^{\beta} & \\
  \hofib_{\gy(M)} (F) & \hofib_{\cY(M,g_0)} (F^\op) \ar[l].
 }
\end{gathered}
\end{equation}
The subspace $|\Riem^+ ([0,1] \times M)_{g_0,g_0}| \times \partial \Delta^1$ of the source is mapped to the base-point spaces under all maps. The diagram will commute up to homotopy as a diagram of maps of space pairs. 

The vertical composition in \eqref{diag:indexthy1} is the adjoint of $(\Omega \beta) \circ \psi$. It sends $(h,v,u) \in \Riem^+ ([0,1] \times M)_{g_0,g_0} (\Delta^p_e ) \times \Delta^p \times \Delta^1$ to the point $(\bR \times M, 0,1,h,v,u) \in \cD_{\theta,1}^{\psc}(\Delta^q_e) \times \Delta^q \times \Delta^1 \subset \norm{\cD_{\theta,\bullet}^{\psc}}$, together with the constant path at $\gy(M)$ in $|\cD_\theta|$. Once we construct $\psi^\ad_1$ and prove that \eqref{diag:indexthy1} commutes up to homotopy, we obtain the adjoint diagram
\[
\xymatrix{
|\Riem^+ ([0,1] \times M)_{g_0,g_0}| \ar@/^1.0pc/[dr]^-{\psi_1} \ar[d]_{(\Omega \beta) \circ \psi} & \\
\Omega  \hofib_{\gy(M)} (F) & \Omega \hofib_{\cY(M,g_0)} (F^\op) \ar[l]
}
\]
and then verify that $\trg (\ind_1)_{(M,g_0)} \circ \psi_1$ is induced from the sheaf map $\phi$, which completes the proof. 

To define $\psi_1^{\ad}$, we consider $(h,v,u) \in \Riem^+ ([0,1] \times M)_{g_0,g_0} (\Delta^p_e ) \times \Delta^p \times \Delta^1$ and define 
\[
(\psi_1^\ad)_0 (h,v,u) := \Bigl( \bR \times M,a,\eps,h,1,r,C,\kappa,v,u \Bigr) \in \cD_{\theta,1}^{\psc,\op} (\Delta^p_e) \times \Delta^p \times \Delta^1 \subset \norm{\cD_{\theta,\bullet}^{\psc,\op}}. 
\]
Under $F^{\op}$, this is mapped to 
\[
 ( \bR \times M, h,1,u_0 a_0+ u_1 a_1 , r,C,v) \in \cD_\theta^{\op} (\Delta^p_e) \times \Delta^p \subset |\cD_\theta^{\op}|.
\]
The path $(\psi_1^\ad)_0 (h,v,u):[0,1] \to |\cD_\theta^{\op}|$ given by 
\[
s \mapsto \Bigl( \bR \times M,(1-c(s))h + c(s)h_0,1,(1-d(s))(u_0 a_0+u_1 a_1),r(s), C(s),v\Bigr) \in |\cD_\theta^{\op}|
\]
(with $C(s):= (1-d(s))C+d(s)$) connects this to 
\[
 \Bigl(  \bR \times M,h_0,1,0,r_0,1,v\Bigr)\sim \gx(M,g_0) \in |\cD_\theta^{\op}|.
\]
Hence $\psi^{\ad} := ( (\psi_1^\ad)_0,(\psi_1^\ad)_1)$ is a map $|\Riem^+ ([0,1] \times M)_{g_0,g_0}| \times \Delta^1 \to \hofib_{\cY(M,g_0)} (F^\op)$ of space pairs. 

To check that \eqref{diag:indexthy1} is homotopy commutative, we trace the fate of $(h,v,u) \in \Riem^+ ([0,1] \times M)_{g_0,g_0}(\Delta^p_e) \times \Delta^p \times \Delta^1$ under the composition of $\psi_1^\ad$ with the map $\hofib_{\cY(M,g_0)}(F^\op) \to \hofib_{\gy(M)} (F)$. The image is the point
\[
\Bigl( \bR \times M,a,h,v\Bigr) \in \hat{\cD}_{\theta,1}^{\psc} (\Delta^p_e)\times \Delta^p \times \Delta^1 \subset \norm{\hat{\cD}_{\theta,\bullet}^{\psc}},
\]
together with the constant path at the base-point $\gy (M)\in \cD_\theta$. This is almost the same as the image under the vertical map in \eqref{diag:indexthy1}, except that $a=(-\epsilon,1+\epsilon)$, and not $a=(0,1)$. But the convex homotopy $(-s\epsilon, 1+s\epsilon)$ connects the two points. Therefore \eqref{diag:indexthy1} is homotopy commutative. 

Let us next calculate the composition 
\[
| \Riem^+ ([0,1] \times M)_{g_0,g_0}|\times \Delta^1 \stackrel{\psi_1}{\lra}  \hofib_{\gx(M,g_0)}(F^\op) \xrightarrow{\trg (\ind_1)_{(M,g_0)}} \Omega  |\bK(\cstar(\theta_d))_1|.
\]
Its adjoint sends $(h,v,u) \in  \Riem^+ ([0,1] \times M)_{g_0,g_0} (\Delta^p_e) \times \Delta^p \times \Delta^1$ to the path 
\[
 s \mapsto \bigl( \Dir_{(1-c(s))h + c(s)h_0} \otimes 1+ \bigl((1 - d(s))r + d(s)r_0\bigr) \bigl(x-(1-d(s)) (t_0 a_0 + t_1 a_1)\bigr) \eta \otimes \varepsilon , v \bigr) \in 
\]
\[ 
\bK(\cstar(\theta_d))_1 (\Delta^p_e ) \times \Delta^p \subset  |\bK(\cstar(\theta_d))_1|
\]
(we de not explicitly notate the Hilbert modules, Clifford structures and grading operators here). Upon the identification $[0,1] \cong \Delta^1$ via $t \mapsto (1-t,t)$ and taking adjoints, this is precisely the map $|\phi|$, as claimed.
\end{proof}

\begin{proof}[Proof of Lemma \ref{lem2:proof-thm.comparison-delooped-index-diference1}]
The map $(-\bott) \circ \inddiff_{h_0}(h)$ is obtained from the sheaf map which sends $h \in \Riem^+ ([0,1]\times M)_{g_0,g_0}(*)$ to the family
\begin{equation}\label{indextheoremproof:eq3}
(t,  s) \longmapsto \Dir_{(1-c(s))h + c(s)h_0}  \otimes 1- t \eta \otimes \varepsilon
\end{equation}
of operators on the continuous field of Hilbert-$\cstar(\theta_d)$-modules on $\bR^2$ with fibre $L^2 (\bR \times M; \spinor_\ell)$ and the obvious grading and $\Cl^{1,0}$-structure (which is the same as the continuous field underlying $\phi(h)$). We will now provide a natural homotopy of operators from \eqref{indextheoremproof:eq2} to \eqref{indextheoremproof:eq3}. We have to guarantee that the operators are invertible for $|s|+|t|$ large. Let us first simplify notation, by writing 
\[
\Dir_s := \Dir_{(1-c(s))h + c(s)h_0} \otimes 1.
\]
The first step of the homotopy changes (linearly) $a_0$ to $0$ and $a_1$ to $1$. For $|t|$ sufficiently large or $s \not \in [0,\frac{1}{2}]$, this goes through invertible operators, because all the metrics $(1-c(s))h + c(s)h_0$ have cylindrical ends with positive scalar curvature. The result is the family
\begin{equation*}\label{indextheoremproof:eq4}
(t,  s) \longmapsto \Dir_{s}+ \bigl((1 - d(s))r + d(s)r_0\bigr) \bigl(x-(1-d(s)) t\bigr) \eta \otimes \varepsilon.
\end{equation*}
The next stage of the homotopy is the operator family (with $u$ running from $1$ to $0$)
\begin{equation}\label{indextheoremproof:eq11}
(t,  s) \longmapsto \Dir_{s}+ \bigl((1 - d(s))r + d(s)r_0\bigr) \bigl(x-(1-ud(s)) t\bigr) \eta \otimes \varepsilon.
\end{equation}
ending with 
\begin{equation}\label{indextheoremproof:eq5}
(t,  s) \longmapsto \Dir_s+ \bigl((1 - d(s))r + d(s)r_0\bigr) \bigl(x- t\bigr) \eta \otimes \varepsilon.
\end{equation}
We claim that this is an operator homotopy through operators which are invertible outside a compact set of $(s,t)$'s. 
To check this, we estimate the square of the operators:
\begin{enumerate}[(i)]
\item If $s \leq \frac{1}{2}$, then $d(s)=0$ and so the homotopy is constant there.
\item If $s \geq \frac{1}{2}$, then $\Dir_s^2 \geq \frac{\kappa_0}{4}$ because $c(s)=1$) and $\norm{[\Dir,x]} =1$, and so the square of \eqref{indextheoremproof:eq11} is bounded from below by 
\[
\frac{\kappa_0}{4}+ ((1 - d(s))r + d(s)r_0)^2 (x-(1-ud(s)) t)^2- ((1 - d(s))r + d(s)r_0) \geq \frac{\kappa_0}{4}- r_0 = \frac{\kappa_0}{8},
\]
independent of $u$.
\end{enumerate}

Next, we let $v$ run from $1$ to $0$ and consider the homotopy 
\begin{equation}\label{indextheoremproof:eq12}
(t,  s) \mapsto \Dir_s+ ((1 - d(s))r + d(s)r_0) (vx- t) \eta \otimes \varepsilon
\end{equation}
which begins with \eqref{indextheoremproof:eq5}.
Again, we claim that this goes through invertible operators as long as $s$ or $t$ are large enough:
\begin{enumerate}[(i)]
\item For $s \leq 0$, we have $c(s)=d(s)=0$, so that \eqref{indextheoremproof:eq12} is $\Dir_{h } \otimes 1+ r  (vx- t) \eta \otimes \varepsilon$. The square is at least $\frac{\kappa}{4} - vrC \geq \frac{\kappa}{8}$. 
\item For $s \geq \frac{1}{2}$, we have $c(s)=1$, and the square is 
\[
\geq \frac{\kappa_0}{4} + ((1 - d(s))r + d(s)r_0)^2 (vx- t)^2- v((1 - d(s))r + d(s)r_0)  \geq \frac{\kappa_0}{4} - r_0  \geq \frac{\kappa_0}{8}.
\]
\item If $s \in [0,\frac{1}{2}]$, we have $d(s)=0$ and hence \eqref{indextheoremproof:eq12} is $\Dir_s+ r  (vx- t) \eta \otimes \varepsilon$. We estimate the square separately on $[0,1] \times M$ and on the complement. On the complement, $(1-c(s))h + c(s)h_0= h_0$ and so the square is at least
\[
\frac{\kappa_0}{4}+ r^2  (vx- t)^2 - vr \geq \frac{\kappa_0}{8},
\]
and on $[0,1] \times M$, the square is at least 
\[
r^2  (vx- t)^2 - vrC \geq r^2(vx- t)^2 -\frac{\kappa}{8}
\]
which is strictly positive for $|t| $ large enough. This completes the proof that the homotopy \eqref{indextheoremproof:eq5} goes through invertible operators for $|t|+|s|$ large. 
\end{enumerate}
Hence the original operator family is homotopic to the family \eqref{indextheoremproof:eq12} evaluated at $v=0$, which is
\begin{equation*}\label{indextheoremproof:eq6}
(t,  s) \longmapsto \Dir_s- t ((1 - d(s))r + d(s)r_0)  \eta \otimes \varepsilon.
\end{equation*}
Now we have arrived at the situation where the two summands anticommute (and thus the squares are easy to compute). The last step of the homotopy is ($w$ running from $0$ to $1$)
\begin{equation*}\label{indextheoremproof:eq7}
(t,  s) \longmapsto \Dir_s- t (w+(1-w)((1 - d(s))r + d(s)r_0) ) \eta \otimes \varepsilon,
\end{equation*}
ending with 
\begin{equation*}\label{indextheoremproof:eq7b}
(t,  s) \longmapsto \Dir_s- t  \eta \otimes \varepsilon,
\end{equation*}
which is exactly $- \bott (\inddiff_{h_0}(h))$. 
\end{proof}

The proof of Theorem \ref{thm:comparison-delooped-index-diference1} is complete at this point, and we turn to that of Theorem \ref{thm:comparison-delooped-index-diference2}, which again is structured into two lemmas. First some notation: as before, we let 
\begin{equation}\label{eq:eqn1forlem1thm:comparison-delooped-index-diference2}
 \kappa_0 := \inf \scal(g_0)= \inf \scal (h_0); \; r_0 := \frac{\kappa_0}{8},
\end{equation}
and for $g \in \Riem^+ (M)$, we let 
\begin{equation}\label{eq:eqn2forlem1thm:comparison-delooped-index-diference2}
\kappa:= \min (\inf \scal(g), \inf \scal (g_0)) \leq \kappa_0, \; r := \frac{\kappa}{8} \leq r_0, \; \epsilon:=\max\{ \frac{\sqrt{\kappa_0}}{r},\sqrt{\frac{2^9}{\kappa}} \}.
\end{equation}

\begin{lem}\label{lem1thm:comparison-delooped-index-diference2}
The composition $B \inddiff_{h_0} \circ \varphi:  |\Riem^+ (M)|\to  \Omega |\bK(\cstar(\theta_d))_1|$ is homotopic to the map induced by the sheaf map
\[
\omega : \Riem^+ (M) \lra \Omega \bK(\cstar(\theta_d))_1
\]
defined on points as follows. To $g \in \Riem^+ (M)$, it assigns the continuous field (over $\bR$) with fibres $L^2 (\bR \times M; \spinor_\ell\otimes \bS_1)$, obvious grading and $\Cl^{1,0}$-structure, and operator family given by
\begin{equation}\label{eq.proofsecondindexthm2}
s \mapsto \Dir_{dx^2+((1-c(s))g+c(s)g_0)} \otimes 1 + ((1-c(s))r+c(s)r_0) x \eta \otimes \varepsilon.
\end{equation}
\end{lem}

\begin{lem}\label{lem2thm:comparison-delooped-index-diference2}
The sheaf map $(-1)^{d-1}\omega$ from Lemma \ref{lem1thm:comparison-delooped-index-diference2} is naturally concordant to the sheaf map defining $\morita \circ \inddiff_{g_0}$. 
\end{lem}

\begin{proof}[Proof of Lemma \ref{lem1thm:comparison-delooped-index-diference2}]
This is similar to, but easier than the proof of Lemma \ref{lem1:proof-thm.comparison-delooped-index-diference1}. As in the proof of that Lemma, we first have to resolve the zig-zag implicit in the definition of $B( \inddiff_{g_0})$. 
We consider the base-points $\gy(M) := \bR \times M \in \cD_\theta(\ast)$ and $\gx (M,g_0) :=(\bR \times M, dx^2+g_0,1,0,r_0,1)\in \cY(M,g_0)\subset \cD_\theta^\op(\ast)$. The composition 
\[
 \xi: |\Riem^+ (M)| \stackrel{\varphi}{\lra} \norm{\cQ_\bullet (M)} \stackrel{\beta}{\lra} \hofib_{\gy(M)} (F)
\]
sends $(g,v) \in \Riem^+ (M) (\Delta^p_e) \times \Delta^p$ to the point $(\bR \times M,0,dx^2+g ,v) \in \cD_{\theta,0}^{\psc} (\Delta^p_e) \times \Delta^p \subset \norm{\cD_{\theta,\bullet}^{\psc}}$, together with the constant path at $\gy(M)$ in $|\cD_\theta|$, and by definition, the composition \eqref{eq:comparisontheorem2} is equal to the composition
\[
| \Riem^+ (M)| \stackrel{\xi}{\lra} \hofib_{\gy(M)} (F) \stackrel{\sim}{\longleftarrow} \hofib_{\cY(M,g_0)} (F^{\op}) \xrightarrow{\trg(\ind_1)_{(M,g_0)}} \Omega |\bK(\cstar(\theta_d))_1|.
\]
To compute it, we construct a commutative diagram
\begin{equation}\label{eq:comparisontheorem3}
\begin{gathered}
 \xymatrix{
| \Riem^+ (M) |\ar[d]^{\xi} \ar[dr]^{\zeta} & & &\\
 \hofib_{\gy(M)} (F) & \ar[l]_{\sim} \hofib_{\cY(M,g_0)} (F^{\op}) \ar[rr]^-{\trg(\ind_1)_{(M,g_0)}} & & \Omega |\bK(\cstar(\theta_d))_1|.
 }
\end{gathered}
\end{equation}
To define the map $\zeta$, let $(g,v)  \in \Riem^+ (M) (\Delta^p_e) \times \Delta^p$ and let $\kappa$, $r$ and $\epsilon$ be as in \eqref{eq:eqn2forlem1thm:comparison-delooped-index-diference2}.
Define $\zeta_0: |\Riem^+ (M)| \to \norm{\cD_{\theta,\bullet}^{\psc,\op}}$ by assigning to $(g,v)$ the point 
\[
 ( \bR  \times M, 0, \eps, dx^2+g, 1,r,1,\kappa,v) \in \cD_{\theta,0}^{\psc,\op}(\Delta^p_e) \times \Delta^p \subset |\cD_{\theta,0}^{\psc,\op}| \subset \norm{\cD_{\theta,\bullet}^{\psc,\op}}.
\]
Under $F^\op$, $\zeta_0(g,v)$ maps to 
\[
 ( \bR \times M, dx^2+g , 0,r,1,v) \in \cD_{\theta}^{\op} (\Delta^p_e) \times \Delta^p \subset |\cD_\theta^{\op}|,
\]
so that $F^\op \circ \zeta_0: |\Riem^+ (M)| \to |\cD_\theta^\op|$ is induced from the sheaf map 
\[
\vartheta_0: g \mapsto ( \bR \times M, dx^2+g , 0,r,1).
\]
Let $c: \bR \to [0,1]$ be smooth and monotone with $c(0)=0$ and $c(1)=1$. The formula
\[
s \mapsto (\bR \times M, dx^2+(1-c(s))g+c(s)g_0 , 0,(1-c(s))r+c(s)r_0,1)
\]
defines a natural concordance $\vartheta$ from $\vartheta_0$ to the sheaf map
\[
 \vartheta_1: \Riem^+ (M) \to (\bR \times M, dx^2+g_0 , 0,r_0,1),
\]
i.e. to the constant map onto the base-point $\gx(M,g_0)$. 

We now define $\zeta$ as the map $|\zeta_0|: |\Riem^+ (M)| \to \norm{\cD_{\theta,\bullet}^{\psc,\op}}$, together with the homotopy of $F^\op \circ |\zeta_0|$ to the base-point provided by the natural concordance $\vartheta$. The choices were made so that these formulas indeed define points in the appropriate spaces. The triangle in \eqref{eq:comparisontheorem3} commutes by inspection. The composition 
\[
|\Riem^+ (M)| \stackrel{\zeta}{\lra} \hofib_\gx (F^{\op}) \xrightarrow{\trg(\ind_1)_{(M,g_0)}}  \Omega |\bK(\cstar(\theta_d))_1|
\]
is induced from $ \omega$, as one checks directly from the definitions.
\end{proof}

\begin{proof}[Proof of Lemma \ref{lem2thm:comparison-delooped-index-diference2}]
This is similar to, but more difficult than the proof of Lemma \ref{lem2:proof-thm.comparison-delooped-index-diference1}. The problem is that $\omega: \Riem^+ (M) \to \Omega \bK(\cstar(\theta_d))_1$ is defined using the Dirac operator on $\bR \times M$, while $\morita \circ \inddiff_{g_0}$ is defined using the Dirac operator on $M$. 

Before we take a closer look to this problem, let us simplify the formula for $\omega$ a bit. Recall that $\omega$ takes a point $g \in \Riem^+ (M)$ to the continuous field of Hilbert modules over $\bR$ with fibre $L^2 (\bR \times M,\spinor_\ell \otimes \bS_1)$, together with the operator family defined by the formula \eqref{eq.proofsecondindexthm2}. Since $r \leq r_0$ (by \eqref{eq:eqn2forlem1thm:comparison-delooped-index-diference2}), the formula (with $u$ running from $0$ to $1$) 
\[
s \longmapsto \Dir_{dx^2+((1-c(s))g+c(s)g_0)} \otimes 1 + ((1-(1-u)c(s))r+(1-u)c(s)r_0)x \eta \otimes \varepsilon
\]
defines operators which are invertible when $s \leq 0$ and $s \geq 1$, and hence yields a concordance beginning with \eqref{eq.proofsecondindexthm2} to 
\begin{equation}\label{eq.proofsecondindexthm3}
s \longmapsto \Dir_{dx^2+((1-c(s))g+c(s)g_0)} \otimes 1 + r x \eta \otimes \varepsilon. 
\end{equation}
Hence $\omega$ is naturally concordant to the sheaf map $\omega'$ which assigns the operator family \eqref{eq.proofsecondindexthm3} instead. 

For the rest of this proof, we deviate from our convention and consider $M \times \bR$ instead of $\bR \times M$; this will cancel out the sign $(-1)^{d-1}$ in the statement of the lemma. Now denote the Mishchenko line bundle on $M$ and on $M \times \bR$ both by the same symbol $\cL$ (as one is the pullback of the other, there is little risk of confusion). Let $\spinor_M$ and $\spinor_{M \times \bR}$ be the spinor bundles, with gradings $\nu$ and $\eta$, respectively. Furthermore, let $D$ be the Dirac operator on $\bR$, which acts on the trivial bundle with fibre $\Cl^{1,0}$, and the grading on this bundle is denoted by the symbol $\upsilon$. 

For a general Riemannian metric $g$ on $M$, we let $\dirac_{g}$ be the associated Rosenberg--Dirac operator. It acts on sections of the tensor product $\spinor_M \otimes \cL$, and it is $\cstar(\theta_{d-1})$-linear. The relation between these spinor bundles and operators is that there is a canonical isomorphism
\begin{equation}\label{eq:indextheorem2proof4}
L^2 (\bR \times M;\spinor_\ell) \cong  L^2 (M;\spinor_M \otimes \cL) \otimes L^2 (\bR; \Cl^{1,0}) , 
\end{equation}
which takes the grading $\eta$ to $\nu \otimes \upsilon $, and the Dirac operator $\Dir_{g+dx^2}$ to $\dirac_{g} \otimes 1 + \nu \otimes D$. After tensoring \eqref{eq:indextheorem2proof4} with $\bS_1$, we find that $\Dir_{g+dx^2} \otimes 1 + f \eta \otimes \varepsilon$ (for a general metric $g$ on $M$ and a general function $f$ on $\bR$) corresponds to 
\begin{equation}\label{eq.proofsecondindex8}
\dirac_g \otimes 1 \otimes 1 + \nu \otimes (D \otimes 1 + f \upsilon \otimes \varepsilon).
\end{equation}
The operator \eqref{eq.proofsecondindex8} acts on the Hilbert module 
\[
L^2 (M;\spinor_M \otimes \cL) \otimes L^2 (\bR; \Cl^{1,0}) \otimes \bS_1. 
\]
We will now show that (for the function $f= rx$) this latter Hilbert module can be decomposed as a direct sum, which is preserved by the operator, and that the operator is invertible on one summand, while the other summand is essentially $ \dirac_g$ on $L^2 (M;\spinor_M \otimes \cL)$. 
For this, we use a well-known computation in quantum mechanics that we recall first (a convenient mathematical reference is \cite[p.\ 119-122]{Roe}). For $r>0$, the differential operator 
\[
 H= -\partial^2+ r^2 x^2
\]
acting on $L^2 (\bR)$ is essentially self-adjoint with spectrum $\spec (H)=\{ (2k+1)r \vert k \in \bN_0\}$. Each point in $\spec(H)$ is an eigenvalue of multiplicity $1$, and the lowest eigenspace $\Eig(H,r)$ is spanned by the function $\psi_0 = e^{-r x^2 /2}$. We use these facts to determine the kernel and the spectrum of the operator 
\[
 D \otimes 1 + rx \upsilon \otimes \varepsilon 
\]
on $L^2 (\bR;\Cl^{1,0} ) \otimes \bS_1= L^2 (\bR )\otimes \Cl^{1,0} \otimes \bS_1$. Note that
\[
( D \otimes 1 + rx \upsilon \otimes \varepsilon )^2 = H\otimes 1 + re \upsilon \otimes \varepsilon 
\]
($e \in \Cl^{1,0}$ is the generator). The two summands commute, and the second one squares to $r^2$ and hence has spectrum $\{ \pm r\}$. It follows that
\[
 \spec (D \otimes 1 + rx \upsilon \otimes \varepsilon)^2 = r \bZ
\]
and
\[
\ker ( D \otimes 1 + rx \upsilon \otimes \varepsilon ) = (\Eig (H,r) \otimes \Cl^{1,0} \otimes \bS_1) \cap (L^2 (\bR) \otimes \Eig(e \upsilon \otimes \varepsilon,-1)).
\]
Therefore 
\begin{equation*}\label{eq.spectrum-calculation}
\ker ( D \otimes 1 + rx \upsilon \otimes \varepsilon ) = \spann \{\psi_0\} \otimes \Eig(e \upsilon \otimes \varepsilon,-1).
\end{equation*}

Exactly as in the proof of \cite[Proposition 4.6]{JEIndex2}, one now argues that the sheaf map $\omega' \sim \omega:\Riem^+ (M) \to \Omega  \bK(\cstar(\theta_d))_1$ given by 
\begin{equation*}
\begin{aligned}
 \omega ': g \longmapsto &\bigl( L^2 (\bR \times M; \spinor_\ell\otimes \bS_1),\\
&\quad\quad\quad s \mapsto \Dir_{dx^2+((1-c(s))g+c(s)g_0)} \otimes 1 + ((1-c(s))r+c(s)r_0) x \eta \otimes \varepsilon\bigr)
\end{aligned}
\end{equation*}
is naturally concordant to the sheaf map
\begin{equation*}
\begin{aligned}
g \longmapsto &\bigl(L^2 (M;\spinor_M) \otimes \Eig(e \upsilon \otimes \varepsilon,-1),\\
&\quad\quad\quad s \mapsto \dirac_{dx^2+((1-c(s))g+c(s)g_0)} \otimes\Eig(e \upsilon \otimes \varepsilon,-1)\bigr)
\end{aligned}
\end{equation*}

Now we apply the inverse $\bK(\Cl^{d,0}\otimes \cstar(G))_1 \to \bK(\Cl^{d-1,0}\otimes \cstar(G))$ of the Morita equivalence, following the recipe of Lemma \ref{lem:inverse-morita}. On $L^2 (M;\spinor_M) \otimes \Eig(e \upsilon \otimes \varepsilon,-1)$, the extra Clifford structure, the $\Cl^{1,0}$-module structure and the grading are given by
\[
c(e)= \twomatrix{}{-\nu}{\nu}{}, \; \rho(e)= \twomatrix{}{1}{-1}{}, \;  \eta= \twomatrix{\nu}{}{}{-\nu},
\]
and hence 
\[
c(e) \rho(e) \eta = \twomatrix{1}{}{}{-1}. 
\]
Hence
\begin{align*}
&\morita^{-1} \Big(L^2 (M;\spinor_M\otimes \cL) \otimes \Eig(e \upsilon \otimes \varepsilon,-1);\\
&\quad\quad\quad\quad\quad \dirac_{dx^2+((1-c(s))g+c(s)g_0)} \otimes \Eig(e \upsilon \otimes \varepsilon,-1)\Big) \\
&\quad\quad\quad\quad\quad\quad =((L^2 (M;\spinor \otimes \cL),\dirac_{dx^2+((1-c(s))g+c(s)g_0)}),
\end{align*}
which represents $\inddiff_{g_0}$. 
\end{proof}

\section{Implications of the concordance implies isotopy conjecture}\label{sec:concimpliesisotopy-consequences}

Let $M$ be a closed manifold. Two psc metrics $g_0,g_1 \in \Riem^+ (M)$ are called \emph{isotopic} if they lie in the same path component and they are called \emph{concordant} if there exists a concordance between them, in other words, if $\Riem^+ ([0,1] \times M)_{g_0,g_1}\neq \emptyset$. A well-known and simple argument (see e.g. \cite[Lemma on p.\ 184]{Gajer}) shows that isotopic metrics are concordant. The converse is a difficult problem.

\begin{defn}[Concordance-implies-isotopy conjecture]\label{conj:conoc-impl-isotopy}
A closed manifold $M$ satisfies the \emph{concordance-implies-isotopy conjecture} if the implication 
\[
\Riem^+ ( [0,1] \times M)_{g_0,g_1} \neq \emptyset \Rightarrow [g_0]=[g_1] \in \pi_0 (\Riem^+ (M)) 
\]
holds, in other words, if concordant psc metrics on $M$ are always isotopic.
\end{defn}

In this section, we assume the validity of the concordance-implies-isotopy conjecture and derive some consequences. Hence the results of this chapter are all conditional and depend on the following assumption.

\begin{hypothesis}\label{hyp:conimplisotopy}
Let $\theta:B \to B\mathrm{O}(d)$ be a $2$-coconnected tangential structure and assume that $B$ is of type $(F_2)$. We assume that all closed $d$- and $(d-1)$-dimensional $\theta$-manifolds which are $2$-connected relative to $B$ satisfy the concordance-implies-isotopy conjecture.
\end{hypothesis}

\begin{hypthm}\label{thm:all-metrics-right-stable-under-conimplisotopy}
Let $W: N \leadsto M$ be a $d$-dimensional cobordism, with $d \geq 6$ and such that $M \to W$ is $2$-connected. Let $\theta: B \to B\mathrm{O}(d)$ be the tangential $2$-type of $W$. Assume that there is a cobordism $U: \emptyset \leadsto N$ such that $N \to U$ is $2$-connected. Let $g_N \in \Riem^+ (N)$ be such that $\Riem^+ (U)_{g_N} \neq \emptyset$, and let $g_M \in \Riem^+ (M)$ be such that $\Riem^+(W)_{g_N,g_M} \neq \emptyset$. Then each psc metric in $ \Riem^+ (W)_{g_N,g_M}$ is right stable.
\end{hypthm}

\begin{proof}
Let $V:= U \cup W: \emptyset \leadsto M$ be the composite cobordism. The given $\theta$-structure on $W$ can be extended to a $\theta$-structure on $V$. This follows from obstruction theory because $N \to U$ is $2$-connected and because $\theta$ is $2$-coconnected. The extended $\theta$-structure $V  \to B$ is $2$-connected as well, because $\pi_1 (W) \to \pi_1 (V )$ is an isomorphism by the Seifert--van Kampen theorem. 
Using Theorem \ref{thm:StabMetrics} (i), we pick $g_0 \in \Riem^+ (M)$ and a right stable $h_0 \in \Riem^+ (V)_{g_0}$. We fix these for the rest of the proof.

The first step of the proof is to show that each self-concordance $h \in \Riem^+ ([0,1] \times M)_{g_0,g_0}$ is stable. Consider the manifold with corners $ [0,1]^2 \times M$ and the psc metric $ ds^2+h$ on it (the coordinates in $[0,1]^2$ are $(s,t)$). We introduce new corners over the points $(s,t)=(\frac{1}{3},0)$ and $(\frac{2}{3},0)$ and smoothen the corners over the points $(1,1)$ and $(0,1)$. In this way, we consider $[0,1]^2 \times M$ as a cobordism 
\[
([\tfrac{1}{3},\tfrac{2}{3}] \times 0) \times M \leadsto (0 \times [0,1] \cup [0,1] \times 1 \cup 1 \times [0,1]) \times M
\]
of manifolds with boundary. Note that the free boundary is $([0,\tfrac{1}{3}] \times M )\cup ([\tfrac{2}{3},1] \times M)$. 

Using the corner-smoothing technique introduced in Section 2 of \cite{ERWpsc2}, we turn $ [0,1]^2 \times M$ with the psc metric $ds^2+h$ into a psc cobordism from $k_0:= ds^2+g_0  $ to $k_1:=h \cup (ds^2+g_0) \cup \overline{h}$ ($\overline{h}$ is the concordance $h$, reversed in direction). The induced metric on the free boundary is $dt^2+g_0$. We can glue the psc cobordism $V \times [0,1]$, with the metric $ ds^2+h_0$, to both boundary cobordisms. The result is a concordance from the psc metric $h_0 \cup k_0 \cup \overline{h_0}$ to the psc metric $h_0 \cup k_1 \cup \overline{h_0}$ on the double $dV$ of $V$ (or rather a manifold canonically diffeomorphic to $dV$). Since the structure map $V \to B$, as well as the inclusion $M \to V$ are $2$-connected, the $\theta$-manifold $dV$ is $2$-connected relative to $B$. Hence by our hypothesis for $d$-manifolds, the concordance-implies-isotopy conjecture holds for $dV$, and we conclude that 
\[
[ h_0 \cup k_0 \cup \overline{h_0}] =[h_0 \cup k_1 \cup \overline{h_0}] \in \pi_0 (\Riem^+ (dV)). 
\]
Since $h_0$ is right stable (and hence $\overline{h_0}$ is left stable), it follows that $[k_0] = [k_1] \in \pi_0 ( \Riem^+ (M \times [0,1])_{g_0,g_0})$. By construction, $k_1$ lies in the same component as $h \cup \overline{h}$, and $k_0$ is (obviously) stable. Therefore $h \cup \overline{h}$ is stable. The same argument, applied to $\overline{h}$ in place of $h$ proves that $\overline{h} \cup h$ is stable as well. The same formal reasoning as in the proof of Theorem 3.1.2 (ii) of \cite{ERWpsc2} proves that $h$ is stable, which finishes the first step of the proof. 

The second step is to generalise the first step to concordances $h \in \Riem^+ ([0,1] \times M)_{g_1,g_2}$, where $g_i \in \Riem^+ (M)$ has to satisfy $\Riem^+ (V)_{g_i} \neq \emptyset$, but is otherwise arbitrary. We first show that $[g_i]=[g_0] \in \pi_0(\Riem^+ (M))$. Pick $h_i \in \Riem^+ (V)_{g_i}$. We identify $V $ with $V \cup ([0,1] \times M)$ and consider $h_i$ as an element in $\Riem^+ (V \cup (  [0,1]  \times M))_{g_i}$. Since $h_0 \in \Riem^+ (V)_{g_0}$ is right stable, we can write $[h_i]= [h_0 \cup h'] \in \pi_0 (\Riem^+ (V \cup ([0,1] \times M))_{g_i})$, for some $h' \in \Riem^+ ([0,1] \times M)_{g_0,g_1}$. Hence $g_0$ and $g_1$ are concordant and therefore, by our hypothesis for $(d-1)$-manifolds, also isotopic. 
Now we use Theorem \ref{thm:improved-chernysh-theorem}, which states that the restriction map $\Riem^+ ([0,1] \times M) \to \Riem^+ (M) \times \Riem^+ (M)$ is a Serre fibration. Lifting a path from $(g_1,g_2)$ to $(g_0,g_0)$ in $\Riem^+ (M) \times \Riem^+ (M)$ to $\Riem^+ ([0,1] \times M)$, with initial condition $h$ gives a path in $\Riem^+ ([0,1] \times M)$ from $h$ to some $h'' \in \Riem^+ ([0,1] \times M)_{g_0,g_0}$. By the first step of the proof, $h''$ is stable and Lemma \ref{lem:stability-homotopy-invariant} implies that $h$ is right stable as well. 

For the third step, we return to the notation of the statement of the theorem, and assume that $h \in \Riem^+ (W)_{g_N,g_M}=\Riem^+ (W \cup [0,1] \times M)_{g_N,g_M}$ (use an identification $W \cong W \cup [0,1] \times M$ as before). 
By Theorem \ref{thm:StabMetrics} (i), there is a $g \in \Riem^+ (M)$ and a right stable $h_W \in \Riem^+ (W)^{\rst}_{g_N,g}$. We can therefore write $[h]= [h_W \cup h'] \in \pi_0 (\Riem^+ (W \cup [0,1] \times M)_{g_N,g_M})$, for some $h' \in \Riem^+ ([0,1] \times M)_{g,g_N}$. Since we assumed that $\Riem^+ (U)_{g_N} \neq \emptyset$, both, $g$ and $g_N$ bound psc metrics on $V$, and by the second step of the proof, it follows that $h'$ is stable. Since $h_W$ is right stable by construction, it follows that $h_W \cup h' \in \Riem^+ (W)_{g_N,g_M}$ is right stable. As $[h_W \cup h']=[h]$, it follows that $h$ is right stable, as claimed.
\end{proof}

Hypothesis \ref{hyp:conimplisotopy} also leads to a strong homotopy invariance theorem for the spaces $\Riem^+ (M)$. 

\begin{hypthm}\label{thm:conditionalhomotopyinvariance}
Let $W_0$ and $W_1$ be compact $d$-manifolds, $d \geq 6$, with boundaries $M_i:= \partial W_i$, and assume that $W_0$ and $W_1$ have the same tangential $2$-type $\theta: B \to B\mathrm{O}(d)$ and that Hypothesis \ref{hyp:conimplisotopy} holds for this tangential $2$-type. Let $g_i \in \Riem^+ (M_i)$, $i=0,1$, be such that $\Riem^+ (W_i)_{g_i} \neq \emptyset$. Then $\Riem^+ (W_0)_{g_0} \simeq \Riem^+ (W_1)_{g_1}$. 
\end{hypthm}

It can be derived from \cite{Chernysh} that if $M_0$ and $M_1$ have the same tangential $2$-type $\theta:B \to BO$ and represent the same cobordism class in $\Omega_d^\theta$, then $\Riem^+ (M_0) \simeq \Riem^+ (M_1)$. For simply-connected spin manifolds, this is derived as e.g.\ Corollary 4.2 of \cite{WalshC}; the general case is shown as Theorem 1.5 of \cite{EbFrenck}. 

\begin{proof}[Proof of Theorem \ref{thm:conditionalhomotopyinvariance}]
Consider $W_i: \emptyset \leadsto M_i$ as a cobordism and pick a handlebody decomposition of $W_i$ relative to $\emptyset$. Write the cobordism $W_i$ as a composition
\[
\emptyset \stackrel{P_i}{\leadsto} N_i \stackrel{V_i}{\leadsto} M_i,
\]
where $P_i$ contains all the handles of index $\leq 2$, and $N_i$ all the handles of index $\geq 3$. The inclusion $N_i \to V_i$ is $2$-connected. Therefore, there exists, by Theorem \ref{thm:StabMetrics} (i) (or rather its obvious dual version), psc metrics $k_i \in \Riem^+ (N_i)$ and left stable metrics $h_i \in \Riem^+ (V_i)_{k_i, g_i}$. The gluing maps 
\[
\mu(\_,h_i): \Riem^+ (P_i)_{k_i} \lra \Riem^+ (W_i)_{g_i}
\]
are weak equivalences, by definition of left stable. By Theorem \ref{thm:all-metrics-right-stable-under-conimplisotopy} and hypothesis \ref{hyp:conimplisotopy}, there exists a right stable $m_i \in \Riem^+ (P_i)_{k_i}$ (here we use that $\Riem^+ (W_i)_{g_i}\neq \emptyset$), which can be viewed as a left stable metric on the opposite cobordism $P_i^{op}: N_i \leadsto \emptyset$. Hence the gluing map gives a weak equivalence
\[
\mu (\_, m_i): \Riem^+ (P_i)_{k_i} \lra \Riem^+ (P_i \cup P_i^{op}).
\]
The closed manifold $P_i \cup P_i^{op}$ is nullbordant in $\theta$-cobordism, and the structure map $P_i \cup P_i^{op} \to B$ is $2$-connected. Hence by the general version of the cobordism invariance theorem (Theorem 1.5 of \cite{EbFrenck}), there is a weak equivalence
\[
\Riem^+ (P_0 \cup P_0^{op}) \simeq \Riem^+ (P_1 \cup P_1^{op}).
\] 
The composition of the three displayed weak equivalences shows the result. 
\end{proof}

\begin{hypthm}\label{hypotheticaltheorem-infiniteloopspace}
 Let $W^d$ be a compact manifold with boundary $M$ and let $\theta:B \to B\mathrm{O}(d)$ be the tangential $2$-type of $W$. Assume that $d \geq 6$, that $\Riem^+ (W)_{g_M} \neq \emptyset$, and that Hypothesis \ref{hyp:conimplisotopy} holds for $\theta$. Then $\Riem^+ (W)_{g_M}$ has the homotopy type of an infinite loop space.
\end{hypthm}

\begin{proof}
Pick a closed, $\theta$-nullbordant $(d-1)$-manifold $N$ with $2$-connected structure map $N \to B$, and pick a nullbordism $V: \emptyset \leadsto N$ such that $N \to V$ is $2$-connected. There are psc metrics $g_N \in \Riem^+ (N)$ and $h_V \in \Riem^+ (V)_{g_N}^{\rst}$. 
By Theorem \ref{thm:infinite-loopspace-theorem}, $\Riem^+ ([0,1] \times N)^{\rst}_{g_N,g_N}$ has the homotopy type of an infinite loop space. By Theorem \ref{thm:all-metrics-right-stable-under-conimplisotopy}, $\Riem^+ ([0,1] \times N)_{g_N,g_N} = \Riem^+ ([0,1] \times N)^{\rst}_{g_N,g_N}$, and by Theorem \ref{thm:conditionalhomotopyinvariance}, there is a weak equivalence
\[
\Riem^+ (W)_{g_M} \simeq \Riem^+ ([0,1] \times N)_{g_N,g_N}.\qedhere
\]
\end{proof}

\appendix

\section{``Topological flag complexes" for sheaves}\label{appendixsheaves}

In this section, we prove an analog of Theorem 6.2 of \cite{GRW} in the context of sheaves on the category $\Mfds$ of smooth manifolds. Let us begin with some definitions. 

\begin{defn}
A sheaf $\cF$ on $\Mfds$ is \emph{concrete} if for each test manifold $X$, the map $\prod_{x \in X} j_x^*:\cF(X) \to \prod_{x \in X} \cF(*)$ given by restriction is injective. A subsheaf $\cG \subset \cF$ of a concrete sheaf is again a concrete sheaf, and we say that $\cG$ is \emph{open} if for each $\xi \in \cF(X)$, the set of all $x \in X$ such that $j_x^* (\xi) \in \cG(\ast) \subset \cF(\ast)$ is open in $X$. 
\end{defn}

\begin{defn}
Let $\cF_\bullet$ be a semi-simplicial sheaf and let $\epsilon_\bullet: \cF_\bullet \to \cF_{-1}$ be an augmentation. We say that $\epsilon_\bullet$ is a weak equivalence if the induced map $\norm{\epsilon_\bullet}: \norm{\cF_\bullet} \to |\cF_{-1}|$ of spaces is a weak equivalence. 
\end{defn}

\begin{defn}
Let $f:\cG \to \cF$ be a map of sheaves. We define a semi-simplicial augmented sheaf $\cE_\bullet (f) \to \cF$ by setting $\cE_p (f) := \cG \times_\cF \cG \times_\cF \ldots \times_\cF \cG$ ($p+1$ factors), with the obvious face maps and augmentation $\cE_0 (f) = \cG\stackrel{f}{\to} \cF$.  
\end{defn}

\begin{defn}\label{defn:topologicalflagcomplex}
Let $\eps_\bullet:\cF_\bullet \to \cF_{-1}$ be an augmented semi-simplicial sheaf, and assume that all $\cF_p$, $p \geq -1$, are concrete. We say that $\eps_\bullet$ is a \emph{topological flag complex} if the following conditions are satisfied:
\begin{enumerate}[(i)]
\item the morphism $\cF_\bullet \to \cE_\bullet(\eps_0)$ given by restriction to the vertices is an isomorphism onto a degreewise open semi-simplicial subsheaf,

\item a tuple $(\xi_0, \ldots, \xi_p) \in \cE_p (\eps_0)(\ast)$ lies in $\cF_p (\ast)$ if and only if $(\xi_i,\xi_j) \in \cF_1 (\ast)$ for all $i<j$.
\end{enumerate}
\end{defn}

\begin{thm}\label{thm:flag-complextheoremn}\label{thm:flagcomplexlemma}
Let $\eps_\bullet:\cF_\bullet \to \cF_{-1}$ be a topological flag complex. Assume that 
\begin{enumerate}[(i)]
\item $\eps_0:\cF_0 \to \cF_{-1}$ has local sections (i.e. for each $\xi \in \cF_{-1} (X)$, each $x \in X$ and $\eta \in \cF_0 (\ast)$ with $\eps_0 (\eta)=j_x^* (\xi)$, there is an open neighbourhood $U$ of $x$ and an element $\zeta \in \cF_0(U)$ such that $\eps_0 (\zeta)= \xi |_U$ and $j_x^* (\zeta)=\eta$.
\item $\epsilon_0(\ast) : \cF_0(\ast) \to \cF_{-1}(\ast)$ is surjective.
\item For any $p \in \cF_{-1}(\ast)$ and any (non-empty) finite set $\xi_1, \ldots, \xi_n \in \epsilon_0(\ast)^{-1}(p)$, there exists a $\xi \in\epsilon_0(\ast)^{-1}(p)$ such that $(\xi_i,\xi) \in \cF_1 (\ast)$ for all $i$. (We say that $\xi$ is \emph{orthogonal} to all $\xi_i$s.)
\end{enumerate}
Then $\epsilon_\bullet$ is a weak equivalence.
\end{thm}

The formulation of Theorem \ref{thm:flag-complextheoremn} is a straightforward adaption of the formulation of Theorem 6.2 of \cite{GRW}. The proof, however, is quite different. We need a tool to recognise that an augmentation map is a weak equivalence. We import this tool from \cite[Theorem 4.2]{MadsenWeiss}. In that paper, a sheaf $\beta  \cX$ is constructed out of any sheaf $\cX$ of categories and a weak equivalence $\beta \cX \simeq B \cX$ is proved. The first step in the proof is to translate the situation of Theorem \ref{thm:flag-complextheoremn} to that situation. This begins with a general simplicial construction. 

We write $\Delta_{\inj}$ for the category indexing semi-simplicial objects: its objects are the finite linear orders $[p] = \{0 < 1 < \cdots < p\}$ for $p \geq 0$ and its morphisms are the \emph{injective} order-preserving functions.

\begin{defn}
Let $X_\bullet$ be a semi-simplicial set. The \emph{simplex category} $\Simp(X_\bullet)$ is the following (unital) category. A object is a pair $([p],x)$ where $[p] \in \Ob (\Delta_{\inj})$ and $x \in X_p$. A morphism $([p],x) \to ([q],y)$ is the data of a morphism $\varphi: [p] \to [q]$ in $\Delta_{\inj}$ such that $\varphi^* y=x$. 

If the semi-simplicial set $X_\bullet$ has an augmentation $\epsilon_\bullet : X_\bullet \to X_{-1}$ then there is an augmentation $\delta_\bullet : N_\bullet\Simp(X_\bullet) \to X_{-1}$ given by $\delta_0([p],x) = \epsilon(x)$.
\end{defn}

Equivalently the objects of $\Simp(X_\bullet)$ are given by $\bigsqcup_{n \geq 0} X_n$; the element $x \in X_p \subset \bigsqcup_{n \geq 0} X_n$ corresponds to $([p], x)$ in the definition above. It will be convenient to think of the simplex category in these terms. 

\begin{lem}\label{lem:simplexcategory}
If $\epsilon_\bullet : X_\bullet \to X_{-1}$ is an augmented semi-simplicial set then there is a natural zigzag of weak equivalences between $B \Simp(X_\bullet)$ and $\norm{X_\bullet}$ over $X_{-1}$.
\end{lem}

\begin{proof}
Let $A_{\bullet,\bullet}(X_\bullet)$ be the bi-semi-simplicial set whose $(p,q)$-simplices are the tuples $(y\to x_0 \to \ldots \to x_p,y)$, where $(y \to x_0 \to \ldots \to x_p) \in N_{p+1} \Simp(X_\bullet)$ and $y \in X_q$. There is an obvious bi-semi-simplicial structure on $A_{\bullet,\bullet}(X)$, and 
\[
\eps_{p,q}: A_{p,q}(X_\bullet) \lra X_q,\; (y \to x_0 \to \ldots \to x_p) \longmapsto y 
\]
and
\[
 \eta_{p,q}:A_{p,q}(X_\bullet) \lra N_p (\Simp(X_\bullet)), \; (y\to x_0 \to \ldots \to x_p) \longmapsto  ( x_0 \to \ldots \to x_p)
\]
define augmentations. For $y \in X_q$, the preimage $\eps_{\bullet,q}^{-1}(y)$ is the nerve of the under-category $y/ \Simp (X_\bullet)$, which has contractible geometric realization since $\Simp (X_\bullet)$ is a unital category. For $x=(x_0 \to \ldots \to x_p) \in N_p (\Simp(X))$ with $x_0 \in X_r$, the semi-simplicial set $\eta_{p,\bullet}^{-1}(x)$ is isomorphic to the semi-simplicial $r$-simplex $\nabla^r_\bullet$, i.e. the semi-simplicial set $p \mapsto \Delta_{\mathrm{inj}} ([p], [r])$. Therefore $\norm{\eta_{p,\bullet}^{-1}(x)} \cong \Delta^r \simeq *$. It follows that both maps
\[
 \norm{N_\bullet (\Simp (X_\bullet))} \stackrel{\norm{\eta_{\bullet,\bullet}}}{\longleftarrow} \norm{A_{\bullet,\bullet}(X_\bullet)} \stackrel{\norm{\eps_{\bullet,\bullet}}}{\lra} \norm{X_\bullet}
\]
are weak equivalences. The maps ${A_{\bullet,\bullet}(X_\bullet)} \to X_{-1}$ induced by the augmentations $\epsilon_\bullet$ and $\delta_\bullet$ are equal.
\end{proof}

We can apply this construction to augmented semi-simplicial sheaves $\cF_\bullet \to \cF_{-1}$. One defines a sheaf $\Simp(\cF_\bullet)$ of categories, which takes a test manifold $X$ to the category $\Simp(\cF_\bullet (X))$. This gives rise to the nerve $N_\bullet (\Simp(\cF_\bullet))$ which is again a sheaf of semi-simplicial (in fact, simplicial) sets, which assigns to a test manifold $X$ the semi-simplicial set $N_\bullet (\Simp(\cF_\bullet) (X))$. The augmentation $\eps_\bullet$ induces an augmentation $\delta_\bullet: N_\bullet (\Simp(\cF_\bullet)) \to \cF_{-1}$. 

\begin{cor}\label{cor:flagcomplexlemma-reduction-toMW}
The augmentation map $\cF_\bullet \to \cF_{-1}$ is a weak equivalence if (and only if) the augmentation map $\delta_\bullet: N_\bullet (\Simp(\cF_\bullet)) \to \cF_{-1}$ is a weak equivalence.
\end{cor}

\begin{proof}
Using the construction of the proof of Lemma \ref{lem:simplexcategory}, one defines a bi-semi-simplicial sheaf $A_{\bullet,\bullet} (\cF_\bullet)$, which comes with augmentations to $N_\bullet \Simp(\cF_\bullet)$ and $\cF_\bullet$, such that the diagram
\[
 \xymatrix{
 A_{\bullet,\bullet} (\cF_\bullet) \ar[r]^{\eta_{\bullet,\bullet}} \ar[d]^{\eps_{\bullet,\bullet}} & N_\bullet \Simp(\cF_\bullet) \ar[d]^{\delta_\bullet} \\
  \cF_\bullet \ar[r]^{\epsilon_\bullet} & \cF_{-1}
  }
\]
commutes. The two maps out of $A_{\bullet,\bullet}(\cF_\bullet)$ are weak equivalences as follows. 

As weak equivalences are defined by taking (i) representing spaces of sheaves, then (ii) geometric realisation, for the left vertical map we must show that the augmentation map
$$\norm{\epsilon_{\bullet, \bullet}} : \norm{([p],[q],[n]) \mapsto A_{p,q}(\cF_\bullet)(\Delta^n_e)} \lra \norm{([q], [n]) \mapsto \cF_q(\Delta^n_e)}$$
from the geometric realisation of a tri-semi-simplicial set to that of a bi-semi-simplicial set is a weak equivalence. For this to be an equivalence it suffices to show that it is a levelwise weak equivalence in the $n$-direction, i.e.\ that for each $[n] \in \Delta_{\inj}$ the augmentation
$$\norm{\epsilon_{\bullet, \bullet}(\Delta_e^n)} : \norm{([p],[q]) \mapsto A_{p,q}(\cF_\bullet)(\Delta^n_e)} \lra \norm{[q] \mapsto \cF_q(\Delta^n_e)}$$
is a weak equivalence, but this follows from the proof of Lemma \ref{lem:simplexcategory}. 

The case of the upper horizontal map is dealt with similarly. 
\end{proof}

Corollary \ref{cor:flagcomplexlemma-reduction-toMW} reduces the proof of Theorem \ref{thm:flag-complextheoremn} to the proof that the augmentation map $N_\bullet (\Simp(\cF_\bullet)) \to \cF_{-1}$ is a weak equivalence. Now we apply Theorem 4.2 of \cite{MadsenWeiss} to the sheaf $\Simp (\cF_\bullet)$ of categories. According to that theorem, the classifying space $B \Simp (\cF_\bullet)$ is weakly equivalent to the representing space of the sheaf $\beta \Simp(\cF_\bullet)$ that we now describe explicitly, in the case at hand, with minor changes to adapt it better to our situation. 

Fix an uncountable set $\Omega$. An element in $\beta \Simp(\cF_\bullet)(X)$ is a tuple $(\cU,\varphi_*, \varphi_{*,*})$, where
\begin{enumerate}[(i)]
\item $\cU= (U_i)_{i \in \Omega}$ is a locally finite open cover of $X$. For a finite $\emptyset \neq S \subset \Omega$, we let $U_S:= \cap_{i \in S} U_i$; note that $U_S \subset U_R$ if $R \subset S$. 
\item $\varphi_*$ assigns to each finite non-empty $S \subset \Omega$ a pair $(p_S,\xi_S)$, where $p_S \in \bN_0$ and $\xi_S \in \cF_{p_S} (U_S)$.
\item $\varphi_{*,*}$ assigns to each pair of finite non-empty sets $R \subset S \subset \Omega$ a morphism $\varphi_{R,S}: [p_S] \to [p_R]$ in $\Delta_{\inj}$ such that $(\varphi_{R,S}^* \xi_R)|_{U_S}= \xi_S$. It is required that $ \varphi_{R,S} \circ \varphi_{S,T}=\varphi_{R,T}$ whenever $R \subset S \subset T$. 
\end{enumerate}

To work with augmentations, it is convenient to think of the augmentation $\delta_\bullet : N_\bullet (\Simp(\cF_\bullet)) \to \cF_{-1}$ via the corresponding map of sheaves of categories $\delta: \Simp(\cF_\bullet) \to \cF_{-1}$, where the latter is considered as a sheaf of discrete categories (i.e.\ having only identity morphisms). There is a canonical map $\beta \cF_{-1} \to \cF_{-1}$, given by gluing sections using the sheaf property of $\cF_{-1}$, and therefore a map
$$\beta \eps: \beta \Simp (\cF_\bullet) \overset{\beta \delta}\lra \beta \cF_{-1} \lra \cF_{-1}$$
of sheaves. Looking at the proof of Theorem 4.2 of \cite{MadsenWeiss} one sees that the natural equivalence $B \mathcal{X} \simeq |\beta \mathcal{X}|$ is natural for maps of sheaves of categories, and this can be used to show that the zig-zag of weak equivalences between $\rep{\beta \Simp (\cF_\bullet)}$ and $B \Simp (\cF_\bullet)$ is over $\rep{\cF_{-1}}$. Therefore to prove Theorem \ref{thm:flag-complextheoremn} it suffices to show that $\beta \eps$ is a weak equivalence. 
Another preliminary result is necessary for the proof of Theorem \ref{thm:flag-complextheoremn}. 
\begin{lem}\label{coveringlemma}
Let $\Omega$ be an uncountable set. Let $X$ be a finite-dimensional countable and locally finite simplicial complex and let $\cU^1, \cU^2 ,\ldots, $ be a sequence of open covers of the geometric realisation $|X|$, with $\cU^i = (U_{(i, s)})_{s \in \Omega }$. Then there is a locally finite open cover $\cO=(O_i)_{i \in S}$ of $|X|$ indexed by a countable set $S$ and a function $\varphi=(\varphi_1, \varphi_2):S  \to \bN \times \Omega$ such that
\begin{enumerate}[(i)]
 \item For each $i \in \bN$, $O_i \subset U_{\varphi(i)}$ and
 \item if $O_i \cap O_j \neq \emptyset$, then $\varphi_1 (i) \neq \varphi_1 (j)$. 
\end{enumerate}
\end{lem}

\begin{proof}
We first assume that $X$ is a finite complex, say of dimension $n$, and prove the statement by induction on $n$. The case $n=0$ is trivial: cover $|X|$ by singletons. Assume that the statement is proven for all finite complexes of dimension $\leq n-1$ and let $X$ be finite and $n$-dimensional. Without loss of generality, each simplex of $X$ is contained in one of the covering sets of the cover $\cU^1$ (if not, take the barycentric subdivision of $X$ sufficiently often). 

We now apply the inductive hypothesis to $|X|^{(n-1)}$ and the sequence 
 \[
\cU^2|_{|X|^{(n-1)}}, \cU^3|_{|X|^{(n-1)}} , \ldots  
 \]
of open covers of $|X|^{(n-1)}$. The result is a finite set $S_0$, an open cover $\cO'=(O'_i)_{i \in S_0}$ of $|X|^{(n-1)}$ and a map $\varphi: S_0 \to \bN_{\geq 2} \times \Omega$ with the properties spelled out in the statement of the Lemma. There is an open subset $Z \subset |X|$ containing the $(n-1)$-skeleton $|X|^{(n-1)}$ and a retraction $r: Z \to |X|^{(n-1)}$. The sets
\[
 O_i := r^{-1}(O'_i) \cap U_{\varphi(i)}
\]
are open in $|X|$, $\cup_i O_i$ contains $|X|^{(n-1)}$ and if $O_i \cap O_j$ is non-empty, then $\varphi_1 (i)\neq \varphi_1 (j)$. Now we add the finitely many open simplices $O_1, \ldots O_r$ of $X$ to the family $(O_i)_{i \in \bN}$, take $S:= S_0 \coprod \{1, \ldots, r\}$, choose $\varphi_1 (i)=1$ for $i=1, \ldots, r$ and $\varphi_2$ appropriately. This finishes the proof for finite $X$.

If $X$ is infinite, there is a sequence $\emptyset = |X_0| = Z_0 \subset |X_1| \subset Z_1 \subset |X_2| \subset Z_2 \subset \ldots \subset  |X|$, such that $X_i$ is a finite subcomplex, the union of all $X_i$ is $X$ and $Z_i$ is open in $|X|$ and admits a retraction $r_i:Z_i \to |X_i|$. We apply the lemma to the sequence $\cU^1|_{|X_1|}, \ldots$ of open covers of the finite complex $|X_1|$ and obtain a finite cover $\cO'^1$ of $|X_1|$, with function $\varphi:S_1 \to \bN \times \Omega$. This involves only finitely many of the original open covers, say $\cU^1, \ldots, \cU^{n_1}$. Using the retraction $r_i$, we extend $\cO'^1$ to a family of open sets in $|X|$ covering $|X_1|$, in the same way as in the first step of the proof. 

Starting from the sequence $\cU^{n_1+1}, \ldots $ of open covers, we construct finitely many open sets $O_1, \ldots, O_r$ covering $|X_2|$, which involves only the open covers $\cU^{n_1+1}, \ldots, \cU^{n_2}$, and set $S_2 := S_1 \coprod \{1, \ldots, r\}$. 

Starting from the sequence $\cU^{n_2+1}, \ldots $ of open covers, we construct finitely many open sets $O'_1, \ldots, O'_s$ covering $|X_3|$, which involves only the open covers $\cU^{n_2+1}, \ldots, \cU^{n_3}$. We take $O_i:= O'_i \setminus \overline{Z_1}$ and set $S_3 := S_2 \coprod \{1, \ldots, s\}$.

Continuing this way, we produce the open cover $\cO$ inductively. 
\end{proof}

\begin{proof}[Proof of Theorem \ref{thm:flag-complextheoremn}]
We use the criterion of Proposition 2.18 of \cite{MadsenWeiss} (quoted as Proposition \ref{prop:SurjCrit}) above to show that 
\[
\beta \eps: \beta \Simp(\cF_\bullet) \lra \cF_{-1} 
\]
is a weak equivalence. In the case at hand, this amounts to the following. Let $X$ be a test manifold and let $A \subset X$ be a closed subset. Let $A \subset U$ be an open neighbourhood of $A$ and let $\zeta \in \beta \Simp(\cF_\bullet) (U)$ and $\eta \in \cF_{-1}(X)$ be elements such that $\beta \eps (\zeta) = \eta_U \in\cF_{-1}(U)$. Then we have to find a lift $\psi$ of $\eta$ along $\beta \eps$ which coincides with $\zeta$ on a possibly smaller neighbourhood of $A$. 

Let $\zeta := (\cU^0,\varphi^0_*, \varphi^0_{*,*})$, and write $\varphi^0_S= (p^0_S,\xi^0_S)$. We let $\Omega_0 \subset \Omega$ be the (countable) set of those indices corresponding to nonempty elements of $ \cU^0$. 

We choose an open set $V$ with $A \subset V \subset \bar{V} \subset U$ and let $Y:= X \setminus \bar{V}$ and $x \in Y$. 

If $x \in U$ then the value $\zeta_x \in \beta\Simp(\cF_\bullet)(\{x\})$ involves, by the definition of the $\beta$-sheaf, a finite number of elements of $\cF_0 (\{x\})$, and by hypothesis (iii) of the Theorem we can pick a $\xi_x' \in \cF_0(\{x\})$ lifting $j_x^*\eta$ that is orthogonal to all these. By the local section property (hypothesis (i) of the Theorem), we can extend $\xi'_x$ to an element $\xi_x \in \cF_0(U_{1,x})$, defined on an open neighbourhood $U_{1,x}x \subset Y$ of $x$. By hypothesis (i) of Definition \ref{defn:topologicalflagcomplex} the orthogonality relation is open, so after perhaps shrinking $U_{1,x}$ we may suppose that $\xi_x$ is orthogonal to all the vertices of $\zeta$ over the neighbourhood $U_{1,x}$. 

If $x \not\in U$ then by hypothesis (ii) of the Theorem, we can find a lift $\xi'_x \in \cF_0 (\{x\})$ of $j_x^* \eta$ and by the local section property (hypothesis (i) of the Theorem) we can extend $\xi'_x$ to an element $\xi_x \in \cF_0(U_{1,x})$, defined on an open neighbourhood $U_{1,x}$ of $x$ which is disjoint from $\bar{V}$.

This produces an open cover $\cU^1=(U_{1,x})_{x \in X}$ of $X$, together with prescribed lifts $\xi_{1,x}$ of $\eta$ to $\cF_0(U_{1,x})$, for each $s \in \Omega$ which are orthogonal to all vertices of $\zeta$ over $V$. 

Continuing in this fashion, we construct a sequence of locally finite open covers $\cU^i$ of $Y$, each indexed by $X$, and for each set $U_{i,x}$ an element $\xi_{i,x} \in \cF_0 (U_{i,x})$ which lifts $\eta|_{U_{i,x}}$. The following condition has to be satisfied: if $1 \leq j<i$, then the restrictions of the elements $\xi_{i,x}$ and $\xi_{j,y}$ to the intersection $U_{i,x}\cap U_{j,y}$ are orthogonal.

Moreover, the prescribed lift $\zeta$ of $\eta$ over $U$ involves locally finite collections of $0$-simplices of $\cF_0$, and we require that the restriction of $\xi_{i,x}$ to $U_{i,x} \cap V$ is orthogonal to all the vertices occurring in the lift $\zeta$. This is possible by the hypotheses of the theorem. 

Next, we triangulate $Y$ and invoke Lemma \ref{coveringlemma} to obtain a countable cover $\cO = (O_i)_{i \in X\times \bN}$ 
of $Y$ with the properties spelled out in that lemma. We can choose the sets of $\cO$ to be disjoint from an open neighbourhood $V_0 \subset \overline{V_0} \subset V$ of $A$. 

We pick an injection $X \times \bN \to \Omega\setminus \Omega_0$ and obtain a locally finite cover of $Y$ indexed by $\Omega \setminus \Omega_0$. Together with the original open cover $\cU^0$ of $U$, intersected with $V$, we obtain a cover $\cU$ of $X$ indexed by $\Omega$, which after refinement can be assumed to be locally finite and to coincide with $\cU^0$ on $V_0$. Pick a total order on $\Omega \setminus \Omega_0$. 

A lift of $\eta$ is now given by the following data. Take a finite $\emptyset \neq S \subset \Omega$, write $S_0 = S \cap \Omega_0$ and $S_1=S \setminus S_0$, define $p_S:= p_S + |S_1|$ and $\xi_S := (\xi^0_S,\xi_{s_1}, \ldots,\xi_{s_{|S_1|}})$, where $s_1 < \ldots <s_{|S_1|}$ are the elements of $S_1$ ordered according to the total order of $\Omega\setminus \Omega_0$. This defines the assignment $\varphi_*$, and the morphisms $\varphi_{R,S}$ in $\Delta_{\inj}$ are determined by these data. 
\end{proof}

\bibliographystyle{amsalpha}
\bibliography{psc3}

\end{document}


\maketitle

We wish to explain and fill a gap in our paper \cite{ERW22}. Theorem C of that paper identifies the homotopy fibre of the map $B\PCob_\theta^{2,1, \psc} \to B\Cob_\theta^{2,1}$ induced on classifying spaces from a certain cobordism category with positive scalar curvature metrics to the ordinary cobordism category, and a variant where the decoration $\psc$ is replaced by the decoration $\rst$ (for precise definitions, we refer to the published paper). 

Here $\theta: B \to B\OO(d)$ is a tangential structure, and Theorem C is stated without any further hypothesis concerning $\theta$. However, the proof of Theorem C invokes Theorem 2.3.8 (via Lemma 4.2.2), and this result includes the hypothesis that $\theta$ be \emph{once-stable} (see Definition 2.3.7 or \cite[Definition 5.4]{GRW}). Therefore, the proof of Theorem C in the published paper is only complete for once-stable $\theta$. 

To fix this gap we explain how to deduce the following sharper version of Theorem~2.3.8, valid for all tangential structures, from Theorem~2.3.8.

\begin{customthm}{2.3.8$^\prime$}\label{prop:right-stable-on-2.1.cob-stable}
Let $\theta:B \to B\OO(d)$ and let $d \geq 6$. Let $(W,\ell_W) : (M_0, \ell_{M_0}) \leadsto (M_1, \ell_{M_1})$ be a $d$-dimensional $\theta$-cobordism. Assume that the inclusion $i_1:M_1 \to W$ and the structure maps $\ell_{M_i}: M_i \to B$ are $2$-connected. Then the following statements hold.
\begin{enumerate}[(i)]
\item A psc metric on $W$ is left stable if and only if it is right stable.
\item For each $g_0 \in \Riem^+ (M_0)$, there exists $g_1 \in \Riem^+ (M_1)$ and a stable $h \in \Riem^+ (W)_{g_0,g_1}$.
\item For each $g_1 \in \Riem^+ (M_1)$, there exists $g_0 \in \Riem^+ (M_0)$ and a stable $h \in \Riem^+ (W)_{g_0,g_1}$.
\end{enumerate}
\end{customthm}

\begin{proof}
Let 
\[
W \stackrel{\ell'_W}{\to} B' \stackrel{\theta'}{\to} B\OO(d)
\]
be the second Moore--Postnikov factorization of the Gauss map of $TW$. The tangential structure $\theta'$ is once-stable, by \cite[Lemma 5.5]{GRW}. Consider the diagram 
\[
\xymatrix{
W \ar[r]^{\ell'_W} \ar[d]_{\ell_W} & B' \ar[d]_{\theta'}\\
B \ar@{-->}[ur]^{\sigma} \ar[r]^{\theta} & B\OO(d).
}
\]
The hypotheses imply that $\ell_W : W \to B$ is 2-connected, so without loss of generality we may suppose that $B$ is obtained by attaching cells of dimension $\geq 3$ to $W$. We may also suppose that $\theta'$ is a fibration. Since $\theta'$ is $2$-co-connected by construction, the lift $\sigma$ exists by elementary obstruction theory. Then $\sigma$ is $2$-connected since $\ell_W$ and $\ell'_W$ are. 

Therefore, the $\theta'$-cobordism $(W,\ell'_W)$ satisfies the hypotheses of Theorem 2.3.8 in the published paper, and that theorem gives the desired conclusion. (The point being that the conclusion does not involve the tangential structure at all). 
\end{proof}

The above addresses the gap, but in fact the gap did not affect the validity of any of the other results stated in the introduction of the paper:

\begin{itemize}
\item Theorem B uses Theorem C in its proof, but only for a once-stable $\theta$. 
Theorem A is a special case of Theorem B. 
\item Theorem D is independent of Theorem C. 
\item Theorem E uses Theorem C where $\theta:B \to B\OO(d)$ is the tangential $2$-type of a $d$-manifold. Since by default $d$ is large, such $\theta$ are once-stable by \cite[Lemma 5.5]{GRW}. Theorem F is an application of Theorem E.
\item Theorem G uses Theorem C, but only for the special tangential structure $B \Spin (d) \times BG \to B\OO(d)$ which is once-stable (in fact stable). 
\end{itemize}

\bibliographystyle{amsalpha}
\bibliography{psc3}